\documentclass[11pt,reqno]{amsart}
\usepackage{geometry} 

\usepackage{amsmath,amssymb,amsthm,mathrsfs,appendix}
\usepackage{xcolor}
\usepackage{graphicx}
\usepackage{tikz}
\usepackage{tikz-3dplot}
\usetikzlibrary{positioning}
\usepackage{pgfplots}
\usepackage{subcaption}
\usepackage{adjustbox}

\usetikzlibrary{shapes.geometric, arrows.meta, positioning, calc}

\usepackage{marvosym}
\usepackage{cancel}

\usepackage{latexsym}
\usepackage[nobysame]{amsrefs}[11pts]
\usepackage{bm}
\usepackage{enumerate}
\usepackage{indentfirst}
\usepackage[colorlinks,
            linkcolor=black,
            anchorcolor=black,
            citecolor=black
            ]{hyperref}  
\usepackage{float}

\newcommand{\into}{\hookrightarrow}

\newcommand{\ds}{\mathrm{d}s}
\newcommand{\dt}{\mathrm{d}t}

\DeclareMathOperator{\support}{supp}
\DeclareMathOperator{\dive}{div}

\newcommand{\II}{\mathbb{I}}

\newcommand{\NN}{\mathbb{N}}

\newcommand{\RR}{\mathbb{R}}

\newcommand{\TT}{\mathbb{T}}

\newcommand{\cA}{\mathcal{A}}
\newcommand{\cB}{\mathcal{B}}
\newcommand{\cC}{\mathcal{C}}
\newcommand{\cD}{\mathcal{D}}
\newcommand{\cE}{\mathcal{E}}
\newcommand{\cF}{\mathcal{F}}
\newcommand{\cG}{\mathcal{G}}

\newcommand{\cJ}{\mathcal{J}}
\newcommand{\cK}{\mathcal{K}}

\newcommand{\cU}{\mathcal{U}}
\newcommand{\cV}{\mathcal{V}}

\newcommand{\sD}{\mathscr{D}}

\newtheorem{Definition}{Definition}
\newtheorem{Theorem}{Theorem}[section]
\newtheorem{Lemma}{Lemma}[section]
\newtheorem{Proposition}{Proposition}
\newtheorem{Remark}{Remark}

\numberwithin{Remark}{section}
\numberwithin{Proposition}{section}
\numberwithin{Definition}{section}
\numberwithin{Lemma}{section}
\numberwithin{equation}{section}
\numberwithin{Theorem}{section}

\allowdisplaybreaks[4]
\pgfplotsset{compat=1.18}

\begin{document}
\title[Viscous Gaseous Stars]{Global Dynamics of Viscous Gaseous Stars in a Physical Vacuum 
}
\date{\today}

\author{Demin Wang}
\address[Demin Wang]{School of Mathematical Sciences, Shanghai Jiao Tong University, Shanghai 200240, P. R. China} \email{\tt de-min-wang@sjtu.edu.cn}

\author{Jiawen Zhang}
\address[Jiawen Zhang]{School of Mathematical Sciences, Shanghai Jiao Tong University, Shanghai 200240, P. R. China} \email{\tt zhangjiawen317@sjtu.edu.cn}

\author{Shengguo Zhu }
\address[Shengguo Zhu]{School of Mathematical Sciences, CMA-Shanghai, and MOE-LSC, Shanghai Jiao Tong University, Shanghai 200240, P. R. China.} \email{\tt  zhushengguo@sjtu.edu.cn}

\begin{abstract}

The study of vacuum is important in understanding compressible flows. In particular, physical vacuum, which ensures a finite, nontrivial  pressure-induced acceleration in the  normal direction on the moving  boundary, naturally arises in the study of the motion of gaseous stars. In this paper, we analyze the free boundary problem for the three-dimensional compressible Navier--Stokes--Poisson equations with degenerate viscosities for self-gravitating viscous gaseous stars. For the spherically symmetric and barotropic motion, we establish the global existence of the unique classical  solution without any restriction on the size of the initial data. Our solutions  are classical   up to the moving boundary for positive times, and preserve the physical vacuum boundary behavior of the Lane--Emden star configuration.

\end{abstract}

\date{\today}

\maketitle

\section{Introduction}\label{section1}
The motion of a self-gravitating viscous gaseous star in the universe can be described
 by the following 
 vacuum free boundary problem (\textbf{VFBP}) of the three-dimensional (3-D)  compressible Navier--Stokes--Poisson equations (\textbf{CNSP}):
\begin{equation}\label{eq:1.1-vfbp}
\begin{cases}
\rho_t+\dive(\rho \boldsymbol{u})=0 &\text{in }\Omega(t),\\[1pt]
(\rho \boldsymbol{u})_t+\dive(\rho \boldsymbol{u}\otimes \boldsymbol{u})+\nabla P+\rho\nabla\Psi =\dive \TT&\text{in }\Omega(t),\\[1pt]
\rho>0&\text{in }\Omega(t),\\[1pt]
\rho=0&\text{on }\partial\Omega(t),\\[1pt]
\cV(\partial\Omega(t))=\boldsymbol{u}\cdot\boldsymbol{n}(t)&\text{on }\partial\Omega(t),\\[1pt]
(\rho,\boldsymbol{u})|_{t=0}=(\rho_0,\boldsymbol{u}_0) &\text{in }\Omega:=\Omega(0).
\end{cases}
\end{equation}
Here, $t\geq 0$ is the time, $\boldsymbol{x}=(x_1,x_2,x_3)^{\top}\in \mathbb{R}^3$  is the Eulerian spatial coordinate, the open and bounded subset  $\Omega(t)\subset \mathbb{R}^3$ denotes the changing volume occupied by the fluid and $\Omega(0)=B_1=\{\boldsymbol{x}: \,|\boldsymbol{x}| <1\}$,
$\partial \Omega(t)$ denotes the moving vacuum boundary,  $\cV(\partial \Omega(t))$ denotes the normal velocity of $\partial \Omega(t)$, and $ \boldsymbol{n}(t) $ denotes the exterior unit normal vector to $\partial \Omega(t)$. 
Moreover, $\rho\geq 0$ denotes the mass density of the fluid, $\boldsymbol{u}=(u_1,u_2, u_3)^\top$ $\in \mathbb{R}^3$   the Eulerian velocity field, and $P$  the pressure. 
For polytropic fluids, the constitutive relation is given by  
\begin{equation}\label{pressureform}
P=A\rho^{\gamma},
\end{equation}
where $A>0$ is  the entropy  constant and  $\gamma>1$ is the adiabatic exponent.
$\Psi$ denotes the gravitational potential: 
\begin{equation}\label{gravitational potential explicit expression}
    \Psi(t,\boldsymbol{x})=
\displaystyle -G\int_{\Omega(t)}\frac{\rho(t,\boldsymbol{\hat{x}})}{|\boldsymbol{x}-\boldsymbol{\hat{x}}|}\,\mathrm{d}\boldsymbol{\hat{x}}
\end{equation}
with the gravitational constant $G>0$. Since
$\rho> 0$ in $\Omega(t)$ and 
$\rho\equiv 0$ in $\mathbb{R}^3 \setminus \Omega(t)$, $\Psi$ satisfies 
\begin{equation}\label{gravitational potential}
\Delta \Psi=4\pi G\rho \quad \text{in $\mathbb{R}^3$},\qquad\quad \lim_{|\boldsymbol{x}|\rightarrow \infty} \Psi(t,\boldsymbol{x})=0.
\end{equation}
Particularly, $\Psi(t,\boldsymbol{x})$ is defined in the whole space $\mathbb{R}^3$ for each $t>0$.
$\mathbb{T}$ denotes the viscous stress tensor, given by 
\begin{equation}\label{eq:1.1t}
\mathbb{T}=2\mu(\rho)D(\boldsymbol{u})+\lambda(\rho)\dive\boldsymbol{u}\,\mathbb{I}_3,
\end{equation} 
where $D(\boldsymbol{u})=\frac{1}{2}(\nabla \boldsymbol{u}+(\nabla \boldsymbol{u})^\top)$ is the deformation tensor, $\mathbb{I}_3$ is the $3\times 3$ identity matrix,
\begin{equation}\label{fandan}
\mu(\rho)=a_1 \rho^\delta,\qquad \lambda(\rho)=a_2\rho^\delta,
\end{equation}
for some  constant $\delta> 0$. Here, $\mu(\rho)$ is the shear viscosity coefficient, $\lambda(\rho)+\frac{2}{3}\mu(\rho)$ is the bulk viscosity coefficient,  $a_1$ and $a_2$ are constants satisfying
\begin{equation}\label{kelaoxiusi}
a_1>0,\qquad\quad 2a_1+3 a_2\geq 0.
\end{equation}

Equation $\eqref{eq:1.1-vfbp}_3$ asserts that there is no vacuum inside the fluid, $\eqref{eq:1.1-vfbp}_4$ is the vacuum boundary condition stating that $\rho$ vanishes along the moving boundary $\partial \Omega(t)$, $\eqref{eq:1.1-vfbp}_5$ states that the vacuum boundary $\partial \Omega(t)$ is moving with the speed equal to the normal component of  the fluid velocity, and $\eqref{eq:1.1-vfbp}_6$ provides the initial conditions for the density, velocity, and domain. We refer the readers to \cites{Chan, cox} for more details of the related background on \textbf{CNSP}.

With the local sound speed given by $c=\sqrt{P'(\rho)}$, satisfaction of the condition 
\begin{equation}\label{PVcondition}
-\infty<\frac{\partial c^2_0}{\partial\boldsymbol{n}}<0 \qquad \text{on $\partial \Omega$}
\end{equation}
defines a physical vacuum boundary (see \cites{LiuTP,taiping2,  coutand3,Jang-M2} and the references therein), where $c_0=c|_{t=0}$. Such a kind of vacuum state has been extensively studied in the context of \textbf{VFBP} for compressible flows, which naturally arises in the study of the motion of gaseous stars or shallow water. In fact,  this notion of physical vacuum can be realized by the stationary solutions of the compressible Euler--Poisson equations or \textbf{CNSP} for gaseous stars, \textit{i.e.}, the Lane--Emden solution (see \cites{Chan, Lin,Makino}). For barotropic polytropic fluids,  \eqref{PVcondition} is equivalent to:
\begin{equation}\label{PVdistance}
-\infty<\frac{\partial \rho^{\gamma-1}_0}{\partial\boldsymbol{n}}<0 \quad \text{on} \   \partial \Omega,
\end{equation}
which ensures that the pressure will induce a finite, non-zero   acceleration in the outward normal direction on the moving  boundary.

In this paper, for the physically important case that $(\mu,\lambda)$ depend on $\rho$  in a power law  ({\it i.e.},  $\delta \in (\frac{13}{18},1)$ in \eqref{fandan}), when $\gamma\in (\frac{4}{3},6\delta-3)$, we establish the global existence of the unique classical solution of  \eqref{eq:1.1-vfbp} with large data for the spherically symmetric motion in a physical vacuum, and our   solution is classical up to the moving boundary for positive times. 

\subsection{Background of the problem}
The study on the dynamics of vacuum is crucial for understanding compressible flows.
Actually, the presence of vacuum leads to a degeneracy in the time evolution of hydrodynamic equations, which poses substantial challenges to the analysis of well-posedness theory; see \cites{fu3,lions, MPI, taiping1} and the references therein. It is noteworthy that, from a physics perspective, the concept of fluid velocity loses meaning in the absence of fluid. Moreover, the underlying assumption in the derivation of hydrodynamic equations from physical principles is that the fluid is nondilute and can be described as a continuum, which breaks down in vacuum regions, rendering hydrodynamic equations inapplicable for describing the time evolution of thermodynamic states there. Based on this consideration, the free boundary problem has garnered considerable attention for tracking vacuum problems of compressible flows. As  in  \eqref{eq:1.1-vfbp}, such a formulation of  vacuum problems  requires that the hydrodynamic equations  hold only on the domain
\begin{equation*}
\{(t,\boldsymbol{x}):\rho(t,\boldsymbol{x})>0\},
\end{equation*}
together with an evolution equation for the boundary $\partial \Omega(t)$ which separates the fluids and vacuum. A key feature of \textbf{VFBP} is that $\partial \Omega(t)$ propagates with finite speed when $\rho_0$ is of compact support. For the physical relevance of this phenomenon, we refer the reader to the survey by Nishida \cite{nishida} and the references therein.

Due to the high degeneracy in the time evolution of hydrodynamic equations caused by the behavior \eqref{PVcondition}  near the moving boundary,  the well-posedness of classical solutions to the corresponding free boundary problem becomes extremely challenging for both the inviscid and viscous compressible flow. For \textbf{VFBP} of the barotropic compressible Euler equations, important progress has been made in establishing the well-posedness of smooth solutions satisfying \eqref{PVcondition}. The local well-posedness theory was developed by Coutand--Shkoller \cite{coutand3} and Jang--Masmoudi \cite{Jang-M2}. Had\v{z}i\'{c}--Jang \cite{Jang-Hadzic} constructed global unique 3-D solutions for adiabatic exponents $\gamma \in (1,\frac{5}{3}]$, provided that the initial data are sufficiently close to the expanding compactly supported affine motions constructed by Sideris \cite{sideris} and satisfy \eqref{PVcondition}. Later,  by a different approach, Shkoller-Sideris \cite{sideris2} proved the stability of affine motions and global existence of classical solutions for all $\gamma>1$. On the other hand, for the \textbf{VFBP} of 3-D barotropic compressible Euler--Poisson equations, the local well-posedness theory of classical solutions satisfying \eqref{PVcondition}   was established by Gu--Lei \cite{GL1}, and the global well-posedness of classical solutions was obtained by Had\v{z}i\'{c}--Jang \cite{Jang-Hadzic2}, which stay close to Sideris affine solutions of the Euler equations. Some other important developments can be found in \cites{Rei,LiuTP,taiping2,coutand2, LinZeng, Oli1,Yan,Jang-Hadzic3,Strauss} and the references therein.

For the corresponding \textbf{VFBP} of viscous compressible flows in a physical vacuum (as defined in \eqref{PVcondition}), under the assumption that the viscosity coefficients are all constants, \textit{i.e.}, $\delta=0$ in \eqref{fandan}, Jang \cite{Jang} established the local well-posedness of strong solutions to  \textbf{CNSP} in the spherically symmetric and barotropic motion.
When the viscosity coefficient obeys a power law of $\rho$ ({\it i.e.}, $\rho^\delta$ with some exponent $\delta\in (0,1)$ in \eqref{fandan}), for \textbf{CNSP} in a  physical vacuum, Wang--Zhang--Zhu \cite{jiawenlocal} proved the local existence  of the unique spherically symmetric  classical solution, and 
Luo--Xin--Zeng \cite{LXZ3} established the global existence of spherically symmetric strong solutions under some smallness assumptions. On the other hand, recently, in Chen--Zhang--Zhu \cite{CZZ2}, the global well-posedness of classical solutions in a physical vacuum to  {\rm\textbf{VFBP}} of the viscous Saint--Venant system has been established in the spherically symmetric motion. Some other important developments can be found in \cites{OZ,Makino2,LWX2} and the references therein.

Despite these key advances in the well-posedness theory of \textbf{VFBP} for compressible fluids, for the physically important case in which $\delta\in [0,1)$ in \eqref{fandan}, for both the compressible Navier--Stokes equations and \textbf{CNSP}, the corresponding global well-posedness of classical solutions with large data and physical vacuum in multi-dimensional space remains open. Nevertheless, to gain deeper insight into the dynamics of viscous gaseous stars near their moving boundary or the center, \textit{e.g.}, whether the moving boundary of the star always expands with finite speed in positive time, or whether a collapse will occur at the center of the star in finite time, it is highly desirable to establish the global regularity of solutions with general smooth data for the corresponding \textbf{VFBP} of \textbf{CNSP}.

\subsection{Lagrangian reformulation of \textbf{VFBP} in spherical coordinates}\label{sec-1.2} 
In the current paper,  when the viscosity coefficients $(\mu(\rho), \lambda(\rho))$ given in \eqref{fandan} satisfy 
\begin{equation}\label{bdrelation}
\mu(\rho)=a_1\rho^\delta,\quad \lambda(\rho)=2a_1 (\delta-1)\rho^\delta \qquad \text{with} \ \ 
\delta\in (\frac{13}{18},1),
\end{equation}
we establish the  global existence of  spherically symmetric classical solutions taking the form 
\begin{equation}\label{ss-ass}
(\rho,\boldsymbol{u})(t,\boldsymbol{x}) = (\rho(t,|\boldsymbol{x}|), u(t,|\boldsymbol{x}|)\frac{\boldsymbol{x}}{|\boldsymbol{x}|}),
\end{equation}
of \textbf{VFBP} \eqref{eq:1.1-vfbp} with the initial data:
\begin{equation}\label{eq:IC}
(\rho,\boldsymbol{u})(0,\boldsymbol{x}) =(\rho_0,\boldsymbol{u}_0)(\boldsymbol{x})= (\rho_0(|\boldsymbol{x}|), u_0(|\boldsymbol{x}|)\frac{\boldsymbol{x}}{|\boldsymbol{x}|}).
\end{equation}
Our results hold for physical adiabatic exponents 
\begin{equation*}
\gamma\in (\frac{4}{3},6\delta-3)
\end{equation*}
in three spatial dimensions,
without any restriction on the size of the initial data. The  initial density  $\rho_0$ we considered  satisfies the following condition:
\begin{equation}\label{distanceeuler}
\rho_0^{\gamma-1}(\boldsymbol{x})\in H^3(\Omega), \quad \cK_1(1-|\boldsymbol{x}|) \leq \rho_0^{\gamma-1}(\boldsymbol{x})\leq \cK_2(1-|\boldsymbol{x}|)  \qquad \text{for all $\boldsymbol{x}\in \overline\Omega$},
\end{equation}
for some constants $\cK_2>\cK_1>0$. It is worth noting that  \eqref{distanceeuler} implies that $\rho_0$  satisfies the  physical vacuum boundary condition for spherically symmetric flows, {\it i.e.},
\begin{equation}\label{PVconditionr}
\rho^{\gamma-1}_0\sim 1-|\boldsymbol{x}| \qquad \text{as $|\boldsymbol{x}|$ approaches the vacuum boundary $|\boldsymbol{x}|=1$}.
\end{equation}

For  the  spherically symmetric flow, we first reformulate problem \eqref{eq:1.1-vfbp} into the following form in $I(t)=[0,R(t))$ as the radial projection of  $\Omega(t)$ 
with $R(0)=1$ and $I=[0,1)$:
\begin{equation}\label{shallow-SSR-euler}
\begin{cases}
\displaystyle 
\rho_t+u\rho_x+ \rho \big(u_x+\frac{2 u}{x}\big)=0&\text{in } I(t),\\[1pt]
\displaystyle
\rho u_t+\rho u u_x+P_x+\rho \Psi_x =2a_1\delta\Big(\rho^\delta\big(u_x+\frac{2 u}{x}\big)\Big)_x-4 a_1 \frac{(\rho^\delta)_x u}{x}&\text{in } I(t),\\[2pt]
\rho>0&\text{in } I(t),\\[2pt]
\rho=0&\text{on } \{x=R(t)\},\\[2pt]
R'(t)=u(t,R(t))&\text{on } \{x=R(t)\},\\[2pt]
(\rho,u)|_{t=0}=(\rho_0,u_0)&\text{in } I(0):=I, 
\end{cases}
\end{equation}
where $x=|\boldsymbol{x}|$ and $\Psi_x$ defined in the whole space is given by
\begin{equation}\label{psi_x express int}
 \Psi_x(t,x)=
\begin{cases}
\displaystyle\frac{4\pi G }{x^2}\int_0^x\rho(t,\hat{x})\hat{x}^2\,\mathrm{d}\hat{x} \qquad &\text{for } x\in \bar{I}(t),\\[8pt]
\displaystyle
\frac{4\pi G}{x^2}\int_0^{R(t)}\rho(t,\hat{x})\hat{x}^2\,\mathrm{d}\hat{x} \qquad &\text{for } x\in [0,\infty)\backslash\bar{I}(t).
\end{cases}
\end{equation}
Moreover, $u|_{x=0}=0$, which can be derived from the continuity of $\boldsymbol{x}\mapsto \boldsymbol{u}(t,\boldsymbol{x})$ at $\boldsymbol{x}=\boldsymbol{0}$.

Second, problem \eqref{shallow-SSR-euler}, formulated in Eulerian coordinates on the moving interval $I(t)$, can be transformed to a problem on the fixed interval $I$ by introducing the Lagrangian coordinates. To this end, denote by $x=\eta(t,r)$ the position of the fluid particle $x\in I(t)$ at $t\geq 0$ so that
\begin{equation}\label{flowmap-r-la}
\eta_t(t,r)=u(t,\eta(t,r)),\qquad \eta(0,r)=r,
\end{equation}
and $(t,r)$ are the Lagrangian coordinates. Moreover,  $\eta(t,0)=0$ for $t\in [0,T]$ due to  $u(t,0)=0$.
Then, by introducing the Lagrangian density, velocity, and gravitational potential,
\begin{equation}\label{varrho-U}
\varrho(t,r)= \rho(t, \eta(t,r)),\qquad U(t,r)= u(t, \eta(t,r)),\qquad \Phi(t,r)=\Psi(t,\eta(t,r)),
\end{equation}
we can write problem \eqref{shallow-SSR-euler} as the following initial-boundary value problem ({\rm\bf IBVP}) in the fixed domain $I$ in Lagrangian coordinates $(t,r)$:
\begin{equation}\label{eq:VFBP-La}
\begin{cases}
\displaystyle \varrho_t + \varrho\big(\frac{U_r}{\eta_r}+\frac{2U}{\eta}\big)=0 & \text{in } (0, T] \times I,\\[5pt]
\displaystyle \eta_r \varrho U_t +A  (\varrho^{\gamma})_r+\varrho\Phi_r=2a_1\delta\Big(\varrho^\delta\big(\frac{U_r}{\eta_r}+ \frac{2U}{\eta}\big)\Big)_r - 4a_1 \frac{(\varrho^\delta)_r U}{\eta}& \text{in }(0, T] \times I,\\[5pt]
\eta_t = U & \text{in } (0, T] \times I,\\[5pt]
\varrho>0 & \text{in } (0, T] \times I,\\[5pt]
\varrho|_{r=1}=0 & \text{on } (0, T],\\[5pt] 
(\varrho, U, \eta)(0,r)= (\rho_0(r), u_0(r),r) & \text{for $r\in I$},
\end{cases} 
\end{equation}
where $\delta\in (\frac{13}{18},1)$, $\gamma\in (\frac{4}{3},6\delta-3)$, and $\Phi_r$ is given by
\begin{equation}\label{phi_r express int 1}
\Phi_r(t,r)=\frac{4\pi G \eta_r}{\eta^2}\int_0^r\varrho(t,\hat{r})\eta^2\eta_r\,\mathrm{d}\hat{r} \qquad \text{for } r\in \bar{I}.
\end{equation}
It follows from the facts that  $u|_{x=0}=0$ and $\eta|_{r=0}=0$ that $U(t,0)= u(t, \eta(t,0))=0$.

Moreover, it follows from Lemma \ref{lemma-initial} in Appendix \ref{appb} that condition \eqref{distanceeuler}, which is initially satisfied by $\rho_0(\boldsymbol{x})$ in 3-D Eulerian coordinates, can be rewritten in spherical coordinates as a condition satisfied by $\rho_0(r)$: for some  constants $\cK_2>\cK_1>0$ and $\beta=\gamma-1$,
\begin{equation}\label{distance-la}
\begin{aligned}
&r\big(\rho_0^\beta,(\rho_0^\beta)_r,(\rho_0^\beta)_{rr},\frac{(\rho_0^\beta)_r}{r},(\rho_0^\beta)_{rrr},(\frac{(\rho_0^\beta)_r}{r})_r\big)\in L^2(I),\\
&\cK_1(1-r)\leq \rho_0(r)^\beta\leq \cK_2(1-r) \qquad \text{for all $r\in I$}.
\end{aligned}
\end{equation}

In fact, the corresponding study of \eqref{eq:VFBP-La} is extremely difficult, since the structure of momentum equation $\eqref{eq:VFBP-La}_2$ is full of nonlinearity, degeneracy, and singularity,   \textit{i.e.},
\begin{equation}\label{cosingu}
\underbrace{\eta_r\varrho U_t}_{\circledast} +A (\varrho^{\gamma})_r+\varrho \Phi_r=\underbrace{2a_1\delta \big(\varrho^\delta \frac{U_r}{\eta_r}\big)_r }_{\Diamond} + \underbrace{4a_1(\delta-1)(\varrho^\delta)_r\frac{ U}{\eta}+ 4a_1\delta \varrho^\delta \big(\frac{ U}{\eta}\big)_r }_{\blacktriangle},
\end{equation}
where $\blacktriangle$ denotes the coordinate singularity, $\circledast$  the degenerate time evolution, and $\Diamond$  the degenerate spatial dissipation. 
For establishing the global well-posedness of classical solutions with general data, our analysis encounters the following  major obstacles:
\begin{itemize}
\item \textit{Possible cavitation or implosion within the fluids};
\vspace{4pt}
\item \textit{Possible degeneracy of the coordinate transformation between the Eulerian and Lagrangian coordinates.} 
\end{itemize}
Dealing with these two obstacles is \ challenging due to the combined difficulties of 
\begin{itemize}
\item [\rm (i)] Coordinate singularity $\blacktriangle$ at the origin, manifested by the singular factor $\tfrac{1}{\eta}$ in \eqref{cosingu};
\vspace{3pt}
\item [\rm (ii)] Strong degeneracy in both time  evolution $\circledast$ and spatial dissipation $\Diamond$ near the vacuum;
\item [\rm (iii)] Lack of uniform positive lower and upper bounds of $(\eta_r,\frac{\eta}{r})$;
\vspace{3pt}
\item [\rm (iv)]  Strong coupling between the Navier--Stokes equations  and  Poisson equation;
\vspace{3pt}
\item [\rm (v)]  Strong nonlinear dependence of the viscosity coefficients on $\rho$ (\textit{i.e.}, $\delta\in (\frac{13}{18},1)$ in \eqref{fandan}).
\end{itemize}

Thus, new ideas are required on the large data problems 
of \textbf{VFBP} \eqref{eq:1.1-vfbp} with physical vacuum. 
Fortunately, by exploiting the intrinsic degenerate-singular structure of \textbf{CNSP} in  Lagrangian coordinates $(t,r)$,  and assuming that
\eqref{distance-la} holds, we can prove the global existence  of the unique  classical solution 
of \textbf{VFBP} \eqref{eq:1.1-vfbp} with physical vacuum and general data in the spherically symmetric motion.

\subsection{Main Results}

To clearly describe the main theorem,  for any function space 
$X$, integer $k\geq 1$ and functions $(\varphi,g_1,\cdots\!,g_k)$,
we introduce the following conventions: 
\begin{equation*}
\|\varphi(g_1,\cdots\!,g_k)\|_{X}:=\sum_{i=1}^k\|\varphi g_i\|_X,\qquad|\varphi(g_1,\cdots\!,g_k)|:=\sum_{i=1}^k|\varphi g_i|,
\end{equation*}
and   the definition of the Eulerian derivative $D_\eta$ for spherically symmetric functions $f=f(r)$:
\begin{equation}\label{Deta}
D_\eta f=\frac{f_r}{\eta_r},\qquad  D^k_\eta f=D_\eta ( D^{k-1}_\eta f).\end{equation}
Here and throughout, the above summation conventions are used only when $g_1,\cdots,g_k$ are scalar-valued functions. For vector-valued quantities, including $\mathscr D_\eta^k f$ and $\mathscr D_r^k f$ $(k= 1,\ldots,4)$ defined below, the notation $|\cdot|$ denotes the Euclidean norm.

It follows from  $\eqref{eq:VFBP-La}_1$ and $\eqref{eq:VFBP-La}_3$  that
\begin{equation}\label{eq:eta}
\varrho(t,r) = \frac{r^2\rho_0(r)}{\eta^2\eta_r}.
\end{equation}
Then problem \eqref{eq:VFBP-La}, combined with \eqref{Deta}, can be written as the following  \textbf{IBVP} for $(U, \eta)$ in $[0,T]\times \bar I$:
\begin{equation}\label{eq:VFBP-La-eta}
\begin{cases}
\displaystyle \varrho U_t +A D_\eta(\varrho^{\gamma})+\varrho D_\eta\Phi=2a_1\delta D_\eta\Big(\varrho^\delta\big(D_\eta U+ \frac{2U}{\eta}\big)\Big) - 4a_1  \frac{ U D_\eta(\varrho^\delta) }{\eta},\\[4pt]
\eta_t = U,\\[4pt]
(U, \eta)(0,r)= (u_0(r), r) \qquad \text{for $r\in I$},
\end{cases}
\end{equation}
where density $\varrho$ is given  by \eqref{eq:eta} and $D_\eta\Phi$ is given by 
\begin{equation}\label{D Phi expression}
D_\eta\Phi(t,r)=\frac{4\pi G}{\eta^2}\int_0^r\hat r^2\rho_0\,\mathrm{d}\hat{r} \qquad \text{for } r\in \bar{I}.
\end{equation}

For simplicity,   we   define classical solutions of  problem \eqref{eq:VFBP-La-eta} as follows:
\begin{Definition}\label{definition-lag}
Let $T>0$. A vector function $(U,\eta)(t,r)$ is called 
a classical solution of {\rm\bf IBVP} \eqref{eq:VFBP-La-eta} 
in $[0,T]\times \bar I$ if the following properties hold{\rm:}
\begin{enumerate}
\item[{\rm (i)}] $(U,\eta)(t,r)$ satisfies equations 
$\eqref{eq:VFBP-La-eta}_1${\rm--}$\eqref{eq:VFBP-La-eta}_2$ pointwise in $(0,T]\times \bar I$, and  takes the initial data $\eqref{eq:VFBP-La-eta}_3$ continuously{\rm ;}
\item[{\rm (ii)}] $\eta_r(t,r)$ and $\frac{\eta}{r}(t,r)$ are strictly positive in $[0,T]\times \bar I${\rm:}
\begin{equation*}
\inf_{[0,T]\times \bar I} \ \eta_r >0, \qquad \inf_{[0,T]\times \bar I} \ \frac{\eta}{r}>0;
\end{equation*}
\item[{\rm (iii)}] $(U,\eta)(t,r)$ satisfies the following regularity properties{\rm:}
\begin{equation*}
\begin{aligned}
&\big(U,U_r,\frac{U}{r}\big)\in C([0,T];C(\bar I)),\qquad \big(U_{rr},(\frac{U}{r})_r,U_t\big)\in C((0,T];C(\bar I)),\\
&\big(\eta,\eta_r,\frac{\eta}{r}\big)\in C^1([0,T];C(\bar I)),\qquad \,\big(\eta_{rr},(\frac{\eta}{r})_r\big)\in
C^1((0,T];C(\bar I)).
\end{aligned}
\end{equation*}
At $r=1$, the product
\begin{equation*}
D_\eta(\varrho^\delta)
\big(D_\eta U-\frac{2(1-\delta)}{\delta}\frac{U}{\eta}\big)
\end{equation*}
is understood through its continuous extension to
$(0,T]\times[\frac12,1]$, whose existence is required.
\end{enumerate}
\end{Definition}

Next, to clearly state our main results, we need to define the following nonlinear weighted energy functional and related parameters:
\begin{itemize}
\item A universal parameter $\varepsilon_0>0$ throughout the paper, which is small enough and satisfies
\begin{equation}\label{varepsilon0}
0<\varepsilon_0 < \min\Big\{\frac{3\gamma-4}{2(\gamma-1)},\frac{1-\delta}{\gamma-1},\frac{1}{100}\Big\}.
\end{equation}

\item The total energy:
\begin{equation}\label{E-1}
\cE(t,f)=\cE_{\mathrm{in}}(t,f)+\cE_{\mathrm{ex}}(t,f),
\end{equation}
where
\begin{equation}\label{E-2}
\begin{aligned}
\qquad \, \mathcal{E}_{\mathrm{in}}(t,f)&:=\|\zeta r(f,\mathscr{D}_\eta f,f_t,\mathscr{D}_\eta f_t)(t)\|_{L^2(I)}^2 +\|\zeta r(\mathscr{D}_\eta^2 f,\mathscr{D}_\eta^3 f)(t)\|_{L^2(I)}^2,\\
\qquad  \, \mathcal{E}_{\mathrm{ex}}(t,f)&:=\big\|\rho_0^\frac{1}{2}(f,f_t)(t)\big\|_{L^2(\frac{1}{2},1)}^2+\big\|\rho_0^\frac{\delta}{2}(D_\eta f,D_\eta f_t)(t)\big\|_{L^2(\frac{1}{2},1)}^2\\
&\quad \ +\big\|\rho_0^{(\frac{3}{2}-\varepsilon_0)(\gamma-1)}(D_\eta^2 f,D_\eta^3 f)(t)\big\|_{L^2(\frac{1}{2},1)}^2,
\end{aligned}
\end{equation}
the operator $\mathscr{D}_\eta$ is defined as follows:
\begin{equation}\label{huad1}
\qquad \begin{aligned}
\mathscr{D}_\eta f&:=(D_\eta f,\frac{f}{\eta}),\qquad \mathscr{D}_\eta^2f:=(D_\eta^2 f,D_\eta(\frac{f}{\eta})),\\
\mathscr{D}_\eta^3f&:=(D_\eta^3 f,\frac{D_\eta^2 f}{\eta},D_\eta^2(\frac{f}{\eta}),\frac{1}{\eta}D_\eta(\frac{f}{\eta})),
\end{aligned}
\end{equation}
and $\zeta=\zeta(r)\in C^\infty[0,1]$ denotes a decreasing cut-off function satisfying
\begin{equation}\label{zeta}
\zeta\in [0,1],\qquad \zeta(r)=1 \ \ \text{for  $r\in \big[0,\frac{1}{2}\big]$},\qquad \zeta(r)=0 \ \ \text{for $r\in \big[\frac{5}{8},1\big]$}.
\end{equation}

\item The total dissipation:
\begin{equation}\label{D-1}
\cD(t,f)=\cD_{\mathrm{in}}(t,f)+\cD_{\mathrm{ex}}(t,f),
\end{equation}
where 
\begin{equation}\label{D-2}
\begin{aligned}
\qquad\cD_{\mathrm{in}}(t,f)&:=\|\zeta r(f_{tt},\mathscr{D}_\eta^2 f_{t},\mathscr{D}_\eta^4 f)(t)\|_{L^2(I)}^2,\\
\qquad\cD_{\mathrm{ex}}(t,f)&:=\big\|\rho_0^\frac{1}{2}f_{tt}(t)\big\|_{L^2(\frac{1}{2},1)}^2+\big\|\rho_0^{(\frac{3}{2}-\varepsilon_0)(\gamma-1)}(D_\eta^2 f_t,D_\eta^4 f)(t)\big\|_{L^2(\frac{1}{2},1)}^2,\\
\mathscr{D}_\eta^4f&:=(D_\eta^4 f,D_\eta(\frac{D_\eta^2 f}{\eta}),D_\eta^3(\frac{f}{\eta}),D_\eta(\frac{1}{\eta}D_\eta(\frac{f}{\eta}))).
\end{aligned}
\end{equation}
\end{itemize}
Moreover, $\mathscr{D}_r^k$ ($k=1,\cdots,4$) is defined in the same way as $\mathscr{D}_\eta^k$ ($k=1,\cdots,4$), except with  $r$ in place of $\eta(r)$.

Now we are ready to  state our main result on the global existence  of the unique classical solution with general data  of \textbf{CNSP}  in Lagrangian coordinates.  
\begin{Theorem}\label{Theorem1.1} 
Let $A>0$, $a_1>0$, $G>0$, and   $(\delta,\gamma)$ satisfy
\begin{equation}\label{gamma-hold}
\delta\in (\frac{13}{18},1),\quad \gamma\in (\frac{4}{3},6\delta-3). 
\end{equation}
If $\rho_0(r)$ satisfies \eqref{distance-la} and $u_0(r)$ satisfies
\begin{equation}\label{a2}
\cE(0,U)<\infty,
\end{equation}
then, for any $T>0$, problem  \eqref{eq:VFBP-La-eta} admits a unique classical solution $(U,\eta)(t,r)$ in $[0,T]\times \bar I$ such that 
\begin{equation}\label{b1}
\begin{aligned}
&\sup_{t\in[0,T]}\big(\cE(t,U)+t\cD(t,U)\big)+\int_0^T\cD (t,U)\,\mathrm{d}t\leq C(T),\\
&(\eta_r,\frac{\eta}{r})(t,r)\in [C^{-1}(T),C(T)] \qquad \text{for all $(t,r)\in [0,T]\times \bar I$},
\end{aligned}
\end{equation}
where $C(T)>1$ is a constant  depending  only on  $(a_1,A,\delta,\gamma,G,\varepsilon_0,\rho_0,u_0,\cK_1, \cK_2,T)$. Moreover, such a classical solution admits the following boundary condition{\rm:}
\begin{equation}\label{N111}
\Big(D_\eta U-\frac{2(1-\delta)}{\delta}\frac{U}{\eta}\Big)\Big|_{r=1}=0\qquad \text{on $(0,T]$}.
\end{equation}
\end{Theorem}

We make some remarks on the results obtained in Theorem \ref{Theorem1.1}.

\begin{Remark}\label{moregamma}
Notice that stationary solutions $(\varrho,U)=(\rho_0,0)$ of problem \eqref{eq:VFBP-La-eta} {\rm(}steady gaseous spheres{\rm)} satisfy the following{\rm:}
\begin{equation*}
A(\rho_0^{\gamma})_r+\frac{4 \pi G \rho_0}{r^2} \int_0^r  \hat r^2\rho_0 \mathrm{d}\hat r=0,
\end{equation*}
which can be transformed into the famous Lane--Emden equation. These solutions are also stationary solutions of the Euler--Poisson system.  
It is interesting to note that when $\delta\in (\frac{13}{18},\frac{5}{6})$, our results hold for all $\frac{4}{3}<\gamma<6\delta-3<2$, which corresponds to the regime where Lane--Emden stars are linearly stable
{\rm(}see \cite{Lin}{\rm)}. Since the initial profile $\rho_0$ we considered in \eqref{distance-la} is bounded independently of $\gamma$, the global solution obtained  in {\rm Theorem \ref{Theorem1.1}}  captures the physical vacuum boundary behavior of the Lane--Emden star configurations for all $\gamma\in (\frac{4}{3},6\delta-3)$ with each fixed $\delta\in (\frac{13}{18},1)$. 
\end{Remark}

\begin{Remark}\label{initialexample}
We give an example of the initial data required in {\rm Theorem \ref{Theorem1.1}}. More precisely, we consider the initial data $(\rho_0,u_0)$ satisfying
\begin{equation*} 
\rho_0(r)=(1-r^2)^\frac{1}{\gamma-1}, \qquad  u_0(r)\in C_\mathrm{c}^\infty(0,1).
\end{equation*}
Then we can check that $\rho_0$ satisfies assumption \eqref{distance-la}. For $u_0$, it suffices to prove that the initial values $(U_t,U_{tr})(0,r)$ satisfy \eqref{a2}. Indeed, we first see from $\eqref{eq:VFBP-La-eta}_1$ and a direct calculation that $(U_t,U_{tr})(0,r)$ satisfy the following compatibility conditions{\rm:}
\begin{equation*} 
\begin{aligned}
U_t(0,r)&=\frac{2a_1\delta}{\rho_0}\Big(\rho_0^{\delta}\big((u_0)_r+\frac{2(\delta-1)}{\delta}\frac{u_0}{r}\big)\Big)_r+4a_1\rho_0^{\delta-1} (\frac{u_0}{r})_r\\
&\quad -\frac{A\gamma}{\gamma-1}(\rho_0^{\gamma-1})_r-\frac{4\pi G}{r^2}\int_0^r \hat r^2\rho_0\,\mathrm{d}\hat r,\\
U_{tr}(0,r)&=2a_1\delta\Big(\frac{1}{\rho_0} \big(\rho_0^{\delta} (u_0)_r\big)_r+\frac{2(\delta-1)}{\delta}\frac{1}{\rho_0}\big(\rho_0^{\delta}\frac{u_0}{r}\big)_r \Big)_r+4a_1\Big(\rho_0^{\delta-1} (\frac{u_0}{r})_r\Big)_r\\
&\quad -\frac{A\gamma}{\gamma-1}(\rho_0^{\gamma-1})_{rr}- 4\pi G \rho_0+ \frac{8\pi G}{r^3}\int_0^r \hat r^2\rho_0\,\mathrm{d}\hat r . 
\end{aligned}
\end{equation*}
Then it is direct to verify that 
\begin{equation*}
\zeta r U_t(0,r), \ \zeta r \sD_r U_{t}(0,r)\in L^2(I),\qquad  \rho_0^\frac{1}{2}U_t(0,r), \ \rho_0^\frac{\delta}{2}U_{tr}(0,r)\in L^2(\frac{1}{2},1),
\end{equation*}    
due to the regularities of $\rho_0$ and the fact that $u_0$ is compactly supported in $(0,1)$.
\end{Remark}

\begin{Remark}\label{rmk1.4}
We briefly explain how the boundary condition 
$(D_\eta U-\frac{2(1-\delta)}{\delta}\frac{U}{\eta})|_{r=1}=0$ in \eqref{N111} can be derived. First, it follows from \eqref{flowmap-r-la}, \eqref{distance-la}, $\eqref{b1}_2$, and {\rm Definition \ref{definition-lag}} that 
\begin{equation}\label{eq117}
\begin{gathered}
(\varrho^{\gamma-1}, \, D_\eta(\varrho^{\gamma-1}),\, U,\ D_{\eta}U, \, D_{\eta}^2U, \, D_{\eta}\big(\frac{U}{\eta}\big), \, U_t)\in  C((0,T]\times [\frac{1}{2},1]).
\end{gathered}
\end{equation}
Next, we reformulate $\eqref{eq:VFBP-La-eta}_1$ by multiplying it by $\varrho^{\gamma-1-\delta}${\rm:}
\begin{equation}\label{reform boundary condition}
\begin{aligned}
        &\varrho^{\gamma-\delta}U_t+\frac{A\gamma}{\beta}\varrho^{\gamma-\delta}D_{\eta}(\varrho^{\gamma-1})+\varrho^{\gamma-\delta}D_\eta\Phi\\
        &=2a_1\delta\varrho^{\gamma-1} D_\eta^2U+4a_1\delta\varrho^{\gamma-1} D_\eta(\frac{U}{\eta})+\frac{2a_1\delta^2}{\gamma-1} D_\eta(\varrho^{\gamma-1})\Big(D_\eta U-\frac{2(1-\delta)}{\delta}\frac{U}{\eta}\Big).
\end{aligned}
\end{equation}
Then, taking the limit $r\to 1$ in \eqref{reform boundary condition} and using \eqref{D Phi expression}, \eqref{eq117}, and the established lower bounds of  $(\eta,\eta_r)$ near the boundary, we obtain 
\begin{equation*}
D_\eta(\varrho^{\gamma-1})\Big(D_\eta U-\frac{2(1-\delta)}{\delta}\frac{U}{\eta}\Big)\Big|_{r=1}=0.
\end{equation*}
Since $D_\eta(\varrho^{\gamma-1})|_{r=1}\neq 0$, owing to $\rho_0^{\gamma-1} \sim 1-r$ and \eqref{eq:eta}, 
it follows immediately that \eqref{N111} holds. We emphasize that the boundary condition \eqref{N111} plays a crucial role in establishing the uniform lower and upper bounds of $(\eta_r,\frac{\eta}{r})$ in {\rm \S \ref{Section-etarlower}}--{\rm \S \ref{Section-etarupper}}. 
\end{Remark}

\begin{Remark}
For  {\rm\textbf{VFBP}} \eqref{eq:1.1-vfbp}, it follows from \eqref{distanceeuler} and {\rm Theorem \ref{Theorem1.1}}  that  the usual
stress-free boundary condition holds automatically{\rm:} for $t\in (0,T]$ and $\boldsymbol{x}\in \partial\Omega(t)$,
\begin{equation*}
(\mathbb{T}-P\II_3)\cdot \boldsymbol{n} = \big(2a_1\rho^\delta D(\boldsymbol{u})+a_2\rho^\delta \dive\boldsymbol{u}\,\mathbb{I}_3-A\rho^\gamma\II_3\big)\cdot \boldsymbol{n}=\boldsymbol{0}.
\end{equation*}
\end{Remark}

\begin{Remark}\label{generalcase}We give some comments on assumption \eqref{bdrelation} for viscosity coefficients. To handle the nonlinearity of viscosity coefficients at the center of the star, in {\rm\S \ref{Section-etarlower}}, we need to give some flow-map weighted estimates on the density in a  neighborhood of the origin, \text{i.e.}, 
\begin{equation*}
\|(\zeta\eta^2\eta_r)^\frac{1}{2}  D_\eta\varrho^{\delta-\frac{1}{2}} (t)\|_{L^\infty([0,T];L^2(I))}<\infty\qquad \text{for any $T>0$},
\end{equation*}
and assumption \eqref{bdrelation} is required in the derivation of this estimate. Such an estimate was motivated by the well-known BD entropy, which was first proposed by Bresch--Desjardins \cite{bd6} for degenerate compressible viscous flows in multi-dimensional space.
\end{Remark}

The rest of this  paper is organized as follows: 
In \S \ref{Section-notation}, we outline the main strategy for proving Theorem \ref{Theorem1.1}.
\S\ref{Section-etarlower}--\S\ref{Section-globalestimates}
are devoted to deriving the desired global uniform estimates for solutions in the purpose-built function spaces.
With these estimates in hand, we obtain the  global existence  of the unique classical solution in \S\ref{Section-global} by continuation arguments. At last,  some basic lemmas and useful coordinate transformations for spherically symmetric functions are collected in Appendices \ref{appendix A}--\ref{appb}.

\section{Notations and Main Strategies}\label{Section-notation}

In this section, we first present some notations in \S\ref{section-notaions}, which will be  used throughout this paper.  In \S\ref{subsection-strategy}, we show the main strategies of  our analysis.

\subsection{Notations}\label{section-notaions}
The following notations will be frequently used in this paper.
\subsubsection{Notations on coordinates and operators}\label{211}
\begin{itemize}
\item $\boldsymbol{x}\in \RR^3$ denotes the 3-D Eulerian spatial coordinates and  $r\in [0,1)$  the Lagrangian  radial coordinate.

\item For any function $f$ defined on a measurable subset of $\mathbb{R}^l$ ($l\geq 1$), if the independent variables of $f$ are $\boldsymbol{z}=(z_1,\cdots\!,z_l)^\top$, then 
\begin{equation*}
\begin{aligned}
&\qquad \partial_{\boldsymbol{z}}^{\boldsymbol{\varsigma}} f=\partial_{z_1}^{\varsigma_1}\cdots\partial_{z_l}^{\varsigma_l} f=f_{\underbrace{\text{\tiny$z_1\cdots z_1$}}_{\text{$\varsigma_1$-times}}\cdots\underbrace{\text{\tiny$z_l\cdots z_l$}}_{\text{$\varsigma_l$-times}}}= \frac{\partial^{\varsigma_1+\cdots+ \varsigma_l}}{\partial z_1^{\varsigma_1}\cdots \partial z_l^{\varsigma_l}}f \qquad \text{for }{\boldsymbol{\varsigma}}=(\varsigma_1,\cdots\!,\varsigma_l)\in \mathbb{N}^l.
\end{aligned}
\end{equation*}
In particular, $\nabla_{\boldsymbol{z}} f=(\partial_{z_1} f,\cdots\!,\partial_{z_l} f)^\top$, and for the derivatives with respect to the
variable $\boldsymbol{x}=(x_1,x_2,x_3)^\top\in \mathbb{R}^3$, 
we use the notation: 
$(\partial_i^{\varsigma_i},\partial^{\boldsymbol{\varsigma}},\nabla)=(\partial_{x_i}^{\varsigma_i}, \partial_{\boldsymbol{x}}^{\boldsymbol{\varsigma}},\nabla_{\boldsymbol{x}})$.

\smallskip
\item For any function $f=f(t,r)$, define the operators $D^k_\eta$ in \eqref{Deta} for any  integer $k\geq 1$, 
and   $\mathscr{D}^k_\eta$ ($k=1,2,3,4$) in \eqref{huad1} and $\eqref{D-2}_2$. In the identities below, $\mathscr D_\eta^k f$ is regarded as a vector-valued quantity, and $|\mathscr D_\eta^k f|$ denotes its Euclidean norm. Moreover, we have
\begin{equation*}
\begin{aligned}
&D_\eta(\mathscr{D}_\eta f)=\mathscr{D}_\eta^2 f,\qquad D_\eta(\mathscr{D}_\eta^3 f) =\mathscr{D}_\eta^4 f,\\
&|D_\eta f|^2+\Big|\frac{f}{\eta}\Big|^2=|\mathscr{D}_\eta f|^2,\qquad |D_\eta^2 f|^2+\Big|D_\eta(\frac{f}{\eta})\Big|^2=|\mathscr{D}_\eta^2 f|^2,\\
&|D_\eta(\mathscr{D}_\eta^2 f)|^2+\Big|\frac{\mathscr{D}_\eta^2 f}{\eta}\Big|^2=|\mathscr{D}_\eta^3 f|^2,\qquad |D_\eta^2(\mathscr{D}_\eta^2 f)|^2+\Big|D_\eta\big(\frac{\mathscr{D}_\eta^2 f}{\eta}\big)\Big|^2=|\mathscr{D}_\eta^4 f|^2.
\end{aligned}
\end{equation*} 
\end{itemize}

\subsubsection{Notations on function spaces}
\begin{itemize}

\item For any open set $\mathrm{Q}\subset \mathbb{R}^q$ ($q\in \NN^*$), $C^\ell(\overline{\mathrm{Q}})$ $(C(\overline{\mathrm{Q}})=C^0(\overline{\mathrm{Q}}))$ denotes the space of all functions $f(\boldsymbol{z})\in C^\ell(\mathrm{Q})$ such that $\nabla_{\boldsymbol{z}}^j f$ $(0\leq j\leq \ell)$ admits a unique continuous extension to $\overline{\mathrm{Q}}$, which is equipped with the norm:
\begin{equation*}
\|f\|_{C^\ell(\overline{\mathrm{Q}})}:= \max_{0\leq j\leq \ell}\|\nabla_{\boldsymbol{z}}^j f\|_{L^\infty(\mathrm{Q})}.
\end{equation*} 
If $\mathrm{Q}$ is an open interval $(a,b)$, then we simply write $C^\ell[a,b]=C^\ell ([a,b])$.
\smallskip
\item For any function space $X(I)$, unless otherwise specified, we denote:
\begin{equation*}
\begin{aligned}
&X=X(I),\quad |f|_p=\|f\|_{L^p},\quad \|f\|_{k,p}=\|f\|_{W^{k,p}},\quad \|f\|_k=\|f\|_{H^k},\\[3pt]
&L^p_{\mathrm{loc}}:=\big\{f:\, f\in L^p(K) \,\, \text{for any open interval $K$ such that $\bar K\subset I\backslash \{0\}$}\big\},\\[3pt]
&H^k_{\mathrm{loc}}:=\big\{f: \, \partial_r^jf\in L^2_{\mathrm{loc}} \,\, \text{for any $0\leq j\leq k$}\big\},\qquad \|f\|_{X_t(Y)}=\|f\|_{X([0,T];Y(I))}.
\end{aligned}
\end{equation*}

\item Let $J\subset I$ and $0\leq \mathrm{w}=\mathrm{w}(r)$ be defined on $J$. Define the weighted function space as 
\begin{equation*}
H^{k}_\mathrm{w}(J):=\big\{f:\, \sqrt{\mathrm{w}}\partial_r^j f\in L^2(J)\,\, \text{for any $0\leq j\leq k$}\big\},\qquad H^{k}_\mathrm{w}=H^{k}_\mathrm{w}(I)\quad \text{if $J=I$}.
\end{equation*}

\end{itemize}

\subsubsection{Other notations}\label{othernotation}
\begin{itemize}

\item $|C_1-C_2|\ll 1$ denotes that $C_1$ is sufficiently close to $C_2$.
\smallskip
\item $E\sim F$ denotes $C^{-1}_*E\leq F\leq C_*E$ for some constant $C_*\geq 1$, where the form of $C_*$ may be different at each occurrence.

\smallskip
\item $\zeta_{a}=\zeta_{a}(r)\in C^\infty[0,1]$ ($a\in (0,1)$) denotes a cut-off function satisfying
\begin{equation*}
\zeta_{a}\in [0,1],\qquad (\zeta_{a})_r\leq 0, \qquad  \zeta_{a}=1 \ \ \text{on $[0,a]$},\qquad \zeta_{a}=0 \ \ \text{on $\big[\frac{1+3a}{4},1\big]$},
\end{equation*}
and $\zeta_{a}^\sharp=\zeta_{a}^\sharp(r):=1-\zeta_{a}(r)$. Then $\zeta_{a}\leq \zeta_{\tilde a}$ and $\zeta_{a}^\sharp\geq \zeta^\sharp_{\tilde a}$ for $0<a<\tilde a<1$, and
\begin{equation*}
\begin{aligned}
&\support\, (\zeta_{a})_r\cup \support\, (\zeta_{a}^\sharp)_r\subset \big[a,\frac{1+3a}{4}\big],\quad \big|\big((\zeta_{a}^\frac{1}{2})_r,(\zeta_{a})_r,((\zeta_{a}^\sharp)^\frac{1}{2})_r,(\zeta_{a}^\sharp)_r\big)\big|\leq C(a),
\end{aligned}
\end{equation*}
where $C(a)>0$ is a constant depending only on $a$. In particular, if $a=\frac{1}{2}$, define \begin{equation*}
\zeta=\zeta(r):=\zeta_{\frac{1}{2}}(r),\qquad \zeta^\sharp=\zeta^\sharp(r):=1-\zeta(r).
\end{equation*}
\item $\chi_{a}=\chi_{a}(r)$ denotes the characteristic function on $[0,a]$ $(a\in (0,1))$, {\it i.e.}, $\chi_{a}=1$ on $[0,a]$ and $\chi_{a}=0$ on $(a,1]$, and $\chi_{a}^\sharp=1-\chi_{a}$. 
Then
\begin{equation*}
\begin{aligned}
&\chi_{a}\leq \zeta_{a} \quad \text{for $a\in (0,1)$},\qquad \chi_{a} \geq  \zeta_{\frac{4a-1}{3}} \quad \text{for $a\in \big(\frac{1}{4},1\big)$};\\
&\chi_{a}^\sharp\geq  \zeta_{a}^\sharp \quad \text{for $a\in (0,1)$},\qquad \chi_{a}^\sharp\leq \zeta_{\frac{4a-1}{3}}^\sharp \quad \text{for $a\in \big(\frac{1}{4},1\big)$}.
\end{aligned}
\end{equation*}
In particular, if $a=\frac{1}{2}$, define 
\begin{equation*}
\chi=\chi(r):=\chi_{\frac{1}{2}}(r),\qquad \chi^\sharp=\chi^\sharp(r):=1-\chi(r).
\end{equation*}

\end{itemize}

\subsection{Main strategies}\label{subsection-strategy} 
Now we sketch the main strategy of our analysis. 
\subsubsection{Selection of the suitable energy functionals} \label{subsec-2.0}

In the study on the spherically symmetric motion in a physical vacuum of viscous gaseous stars, we need to deal with the combined difficulties of the physical vacuum singularity on the moving boundary, the coordinate singularity at the center, and the nonlinearity of the physical viscosity coefficients, which makes it intricate to provide an effective propagation mechanism for the regularity of the fluid velocity. Then the first key point in our analysis is to introduce a proper weighted energy functional. In other words, what kind of function space should we construct the solution in? Here, we select the energy functionals in the neighborhoods of the origin and the moving boundary separately,  and then combine them appropriately to obtain uniform estimates for the velocity.

First, the interior energy functionals $(\cE_{\mathrm{in}},\cD_{\mathrm{in}})$ are chosen based on the  3-D Lagrangian spatial coordinates $(t,\boldsymbol{y})$. If we let  $\boldsymbol{U}(t,\boldsymbol{y}):=U(t,r)\frac{\boldsymbol{y}}{r}$, then $\cE_{\mathrm{in}}(t,U)+t\cD_{\mathrm{in}}(t,U)\in L^\infty(0,T)$ can be read as the following $H^k$-norms of $\boldsymbol{U}$ due to  {\rm Lemma \ref{lemma-initial}} in {\rm Appendix \ref{appb}}: 
\begin{equation}\label{regularity-La-M}
\sum_{l=0}^1 \|\zeta\partial_t^{1-l}\boldsymbol{U} (t)\|_{H^{2l+1}(\Omega)} +\sqrt{t} \sum_{l=0}^2 \|\zeta\partial_t^{2-l}\boldsymbol{U}(t)\|_{H^{2l}(\Omega)}\in L^\infty(0,T).
\end{equation}

Second,  for the exterior energy functionals $(\cE_{\mathrm{ex}},\cD_{\mathrm{ex}})$, on one hand, the weight functions $(\rho_0,\rho_0^\delta)$ come naturally from the tangential estimates and the original structure of equation $\eqref{eq:VFBP-La-eta}_1$. 
On the other hand, to ensure that the solution is classical in the region \(I^\sharp := [\frac12,1)\) and that the velocity \(U\) has spatial \(C^2\)-regularity up to the vacuum boundary for positive times, we choose the weight function \(\rho_0^{(3-2\varepsilon_0)(\gamma-1)}\) based on the following embedding, which follows from the Hardy and Sobolev inequalities:
\begin{equation*}
H^4_{\rho_0^{2\alpha}}(I^\sharp)\hookrightarrow W^{3,1}(I^\sharp)\hookrightarrow C^2(\bar I^\sharp) \qquad\text{for some $\alpha<\frac{3}{2}(\gamma-1)$ and $\Big|\alpha-\frac{3}{2}(\gamma-1)\Big|\ll 1$}.
\end{equation*}
Therefore, the form of exterior energy functionals $(\cE_{\mathrm{ex}},\cD_{\mathrm{ex}})$ is  determined.

\subsubsection{Global uniform upper and lower bounds of $(\eta_r,\frac{\eta}{r})$}
To deal with the moving boundary, 
we transformed the  \textbf{VFBP} \eqref{eq:1.1-vfbp} into the \textbf{IBVP} \eqref{eq:VFBP-La} in a fixed domain by introducing the Lagrangian coordinates.
To determine the mathematical structure of the momentum equations in Lagrangian coordinates, we must ensure that the coordinate transformation between the Eulerian and Lagrangian coordinates is nondegenerate. 
Thus, the major task before establishing the uniform estimates is to obtain the uniform upper and lower bounds of $(\eta_r,\frac{\eta}{r})$. The procedure of the proof can be roughly summarized in Figure \ref{fig:1}.

The main ingredient of our analysis consists of different types of weighted estimates adapted to the exterior region (near the moving vacuum boundary) and the interior region (near the origin).  On one hand, near the vacuum boundary, the mass identity and the weighted $L^p$-estimate for $U$ yield quantitative bounds for the tails in the integrated momentum equation. An integrating-factor argument then gives the upper bounds of $(\eta_r,\frac{\eta}{r})$ in a fixed neighborhood of the boundary. On the other hand, in the interior, we handle the coordinate singularity at the symmetry center by employing the $(\eta,\eta_r)$-weighted estimates for $\varrho$, which yield the interior lower bound of $(\eta_r,\frac{\eta}{r})$, as well as $\eta_r$-weighted estimates for both $(\varrho,U)$, which yield the interior upper bound of $(\eta_r,\frac{\eta}{r})$. The boundary estimate is combined with the interior estimate on a fixed compact interval to obtain the desired global bounds.

\begin{figure}[H]
\centering
\tikzset{
    box/.style = {rectangle, draw, text width=5.7cm, align=center, minimum height=0.8cm, rounded corners, thick, font=\small},
    arrow/.style = {thick, -{Stealth[length=2.5mm, width=2mm]}},
}

\begin{tikzpicture}[scale=0.8, transform shape, node distance=1.2cm]

\node[box] (A1) {\S\ref{subsec-3.1} Upper bound of $\eta^2\varrho^\delta$};
\node[box, right=3.2cm of A1] (A2) {\S\ref{subsec-3.4} Global upper bound of $\varrho$};

\node[box, below=of A1] (B1) {\S\ref{subsec-3.1} Weighted geometric bounds};
\node[box, below=of A2] (B2) {\S\ref{subsec-3.4} Interior and exterior lower bounds of $(\eta_r,\frac{\eta}{r})$};

\draw[arrow] (A1) -- (B1);
\draw[arrow] (A2) -- (B2);

\node[box, below=1cm of B1.south, anchor=north] (C) 
    at ($(B1.south)!0.5!(B2.south)$) {\S\ref{subsec-3.4} Global lower bound of $(\eta_r,\frac{\eta}{r})$};

\draw[arrow] (B1.south) -- (C.north west);
\draw[arrow] (B2.south) -- (C.north east);

\node[box, below right=1.2cm and -1.2cm of C] (E2) {\S\ref{Section-effectivevelocity} Interior estimates of the effective velocity};
\node[box, below left=1.2cm and -1.2cm of C] (E1) {\S\ref{subsec-4.2} Quantitative boundary-tail estimates};

\node[box, below=of E2] (F2) {\S\ref{subsec-4.3} Interior upper bound of $(\eta_r,\frac{\eta}{r})$};
\node[box, below=of E1] (F1) {\S\ref{subsec-4.2} Upper bounds of $(\eta_r,\frac{\eta}{r})$ near the vacuum boundary};

\draw[arrow] (C.south) -- ++(0, -0.3) -| (E1.north);
\draw[arrow] (C.south) -- ++(0, -0.3) -| (E2.north);

\draw[arrow] (E1) -- (F1);
\draw[arrow] (E2) -- (F2);

\node[box, below=1cm of F1.south, anchor=north] (G) 
    at ($(F1.south)!0.5!(F2.south)$) {\S\ref{subsec-4.3} Global upper bound of $(\eta_r,\frac{\eta}{r})$};

\draw[arrow] (F1.south) -- (G.north west);
\draw[arrow] (F2.south) -- (G.north east);

\end{tikzpicture}
\caption{Procedure of the proof for global upper and lower bounds of $(\eta_r,\frac{\eta}{r})$.}
\label{fig:1}
\end{figure}
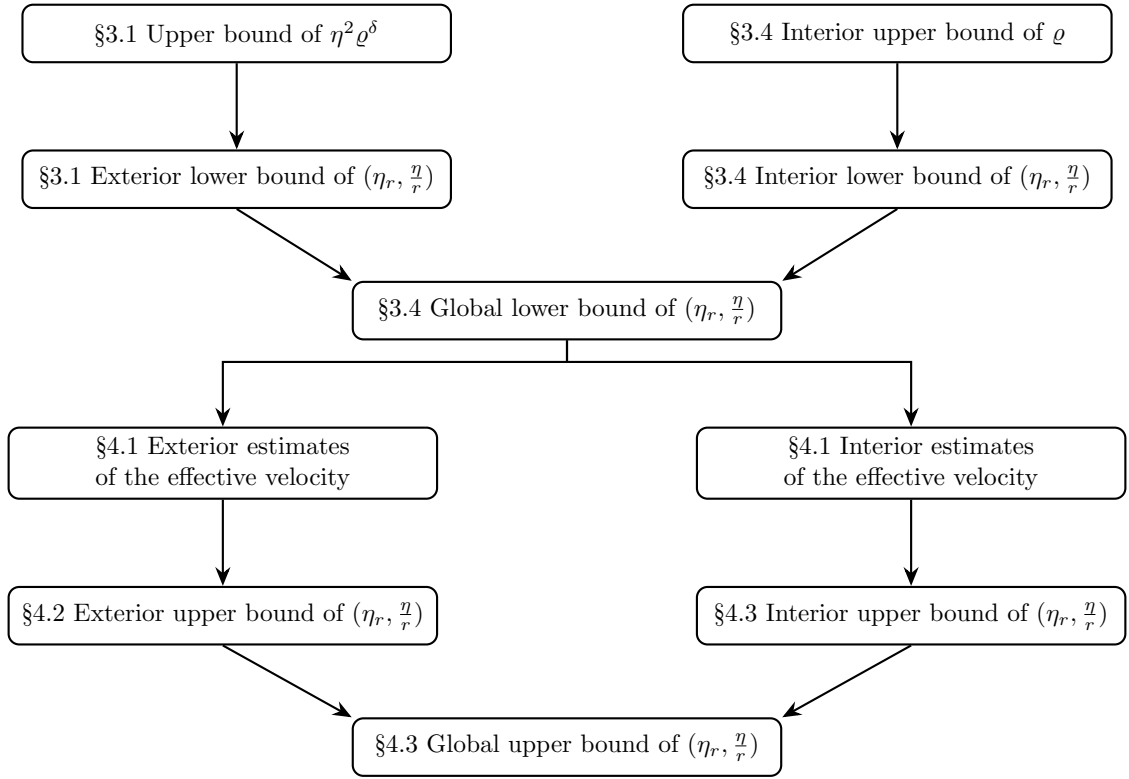

\section{Global Uniform Lower Bounds of \texorpdfstring{$(\eta_r,\frac{\eta}{r})$}{}}\label{Section-etarlower}

The purpose of this section is to establish the global uniform lower bound of  $(\eta_r,\frac{\eta}{r})$. First, we give the desired local existence of the unique classical solution. In order to carry out the continuation argument in \S \ref{Section-global}, we consider here a more general version of {\rm\textbf{IBVP}} problem compared with \eqref{eq:VFBP-La-eta}:
\begin{equation}\label{eq:VFBP-La-eta-tau}
\begin{cases}
\displaystyle \varrho U_t +A D_\eta(\varrho^{\gamma})+\varrho D_\eta\Phi=2a_1\delta D_\eta\Big(\varrho^\delta\big(D_\eta U+ \frac{2U}{\eta}\big)\Big) - 4a_1  \frac{ U D_\eta(\varrho^\delta) }{\eta},\\[4pt]
\eta_t = U,\\[4pt]
(U, \eta)(0,r)= (u_0(r),\eta_0(r)) \qquad \text{for $r\in I$},
\end{cases}
\end{equation}
where density $\varrho$ is given  by \eqref{eq:eta} and $D_\eta\Phi$ is given by \eqref{D Phi expression}. In particular, if $\eta_0(r)=r$, then problem \eqref{eq:VFBP-La-eta-tau} reduces to \eqref{eq:VFBP-La-eta}. The local existence  of the unique classical solution $(U,\eta)$ of {\rm\textbf{IBVP}} \eqref{eq:VFBP-La-eta-tau} is stated as follows.

\begin{Theorem}\label{local-Theorem1.1}\cite{jiawenlocal}*{Theorem 1.2} 
Assume that $A>0$, $a_1>0$, $G>0$, and \eqref{distance-la} and \eqref{gamma-hold} hold. Assume that $(u_0,\eta_0)$ satisfy
\begin{equation}\label{generaleta}
((\eta_0)_r,\frac{\eta_0}{r})(r)\in [\sigma_*,\sigma^*] \quad  \text{for $r\in \bar I$} ,\qquad \cE(0,U)+\mathcal{F}(\eta_0)\leq \mathfrak{M}_0, 
\end{equation}
for some constants  $\sigma^*>\sigma*>0$ and $\mathfrak{M}_0>0$, where $\mathcal{F}(\varphi)$ is defined by
\begin{equation*}
\mathcal{F}(\varphi):=\sum_{0\leq k\leq 4}|\zeta r  \sD_r^k \varphi|_2^2+\big|\chi^\sharp(\rho_0^\frac{1}{2} \varphi,\rho_0^\frac{\delta}{2}\varphi_r)\big|_2^2+\sum_{2\leq l\leq 4}\big|\chi^\sharp\rho_0^{(\gamma-1)(\frac{3}{2}-\varepsilon_0)} \partial_{r}^l \varphi\big|_2^2.
\end{equation*}
Then there exists $T_*>0$, which depends only on $(\sigma_*,\sigma^*,\mathfrak{M}_0,a_1,A,\gamma,\delta,\varepsilon_0,\rho_0,\cK_1,\cK_2,G)$, such that {\rm\textbf{IBVP}} \eqref{eq:VFBP-La-eta-tau} admits a unique classical solution $(U,\eta)$ in $[0,T_*]\times \bar I $, satisfying the boundary condition \eqref{N111} for any $t
\in (0,T_*]$ and 
\begin{equation}\label{b1-lo}
\begin{aligned}
&\cE(t,U)+t \cD(t,U)\in L^\infty(0,T_*),\qquad  \cD (t,U)\in L^1(0,T_*),\\
& (\eta_r,\frac{\eta}{r})(t,r)\in \big[\frac{\sigma_*}{2},\frac{3\sigma^*}{2}\big] \qquad \quad \ \ \,\text{for $(t,r)\in [0,T_*]\times \bar I$}.
\end{aligned}
\end{equation}
\end{Theorem}
This result follows from \cite{jiawenlocal}*{Theorem 1.2} by taking $a_2=2a_1(\delta-1)$. 
By an elementary inequality, the energy functional used here is equivalent to that in \cite{jiawenlocal}, with comparison constants independent of $(U,\eta)$, while the dissipation functionals coincide. Thus, \cite{jiawenlocal}*{Theorem 1.2} yields the stated result. In the current  paper,  \eqref{generaleta} naturally holds for $\eta_0(r)=r$, which means that the initial data $(\eta_0, \rho_0,u_0)$ considered in Theorem \ref{Theorem1.1} satisfy all the conditions required  in Theorem \ref{local-Theorem1.1}. Thus, our main task is to establish global uniform estimates for $(U,\eta)$ in the function spaces $(\cE,\cD)(t,U)$, and then extend the local solutions to global ones by standard continuation arguments.

Throughout \S\ref{Section-etarlower}--\S\ref{Section-globalestimates}, we always assume the following:
\begin{itemize}
\item $(\beta,\delta,\gamma)$ satisfy $\beta=\gamma-1$ and \eqref{gamma-hold};
\item $(U,\eta)(t,r)$ is a classical solution  of \textbf{IBVP} \eqref{eq:VFBP-La-eta} in $[0,T_*]\times \bar I $ for some $T_*>0$, which is constructed in Theorem \ref{local-Theorem1.1}; let $\overline{T}_*\geq T_*$ denote the maximal existence time of such a classical solution, and let $T\in (0,\overline{T}_*)$ be arbitrary;
\item $C_0\in (1,\infty)$ is a generic constant depending only on $(a_1,A,\gamma,\delta,\varepsilon_0,\rho_0,u_0,\cK_1,\cK_2,G)$ and $C(l_1,\cdots\!,l_k)\in (1,\infty)$ is a generic constant depending on $C_0$ and parameters $(l_1,\cdots\!,l_k)$; these constants may be different at each occurrence, \textit{i.e.},
\begin{equation*}
C_0=C_0(a_1,A,\gamma,\delta,\varepsilon_0,\rho_0,u_0,\cK_1,\cK_2,G),\qquad C(l_1,\cdots\!,l_k)=C(l_1,\cdots\!,l_k,C_0).
\end{equation*}
In particular, $C(l_1,\cdots\!,l_k)$ is continuous and strictly increasing with respect to  $G$. Moreover, if $C(\cdots\!,T)$ depends on time $T>0$, we assume $C(\cdots\!,T)$ is continuous and strictly increasing for $T\in(0,\infty)$.
\end{itemize}
 Now we are ready to introduce the so-called effective velocity.
\begin{Definition}\label{def-v}
We say that $V$ is the effective velocity if 
\begin{equation}\label{v-expression}
V=U+2a_{1}\delta\varrho^{\delta-2}D_\eta\varrho=U+\frac{2a_{1}\delta}{\delta-1}D_\eta(\varrho^{\delta-1}).
\end{equation}
Besides, we define the initial value of $V$ as $v_0=V|_{t=0}=u_0+\frac{2a_{1}\delta}{\delta-1}(\rho_0^{\delta-1})_r.$
\end{Definition}

Then we have
\begin{Proposition}\label{lemma-V expression}
The effective velocity $V$ satisfies the equation{\rm :}
\begin{equation}\label{eq:v}
\begin{aligned}
&V_t+ A\gamma \varrho^{\gamma-2} D_\eta\varrho+D_\eta \Phi=V_t+\frac{A\gamma}{2a_{1}\delta}\varrho^{\gamma-\delta}(V-U)+D_\eta \Phi=0,
\end{aligned}
\end{equation}
and takes the form{\rm:}
\begin{align}
V(t,r)&=v_0(r)\exp\Big(-\frac{A\gamma}{2a_{1}\delta}\int_0^t \varrho^{\gamma-\delta}(s,r)\,\mathrm{d}s\Big) \nonumber\\
&\quad+\frac{A\gamma}{2a_{1}\delta}\int_0^t (\varrho^{\gamma-\delta} U)(\tau,r)\exp\Big(-\frac{A\gamma}{2a_{1}\delta}\int_\tau^t \varrho^{\gamma-\delta}(s,r)\,\mathrm{d}s\Big) \,\mathrm{d}\tau\label{V-solution}\\
&\quad- \int_0^t D_\eta\Phi(\tau,r)  \exp\Big(-\frac{A\gamma}{2a_{1}\delta}\int_\tau^t \varrho^{\gamma-\delta}(s,r)\,\mathrm{d}s\Big)\,\mathrm{d}\tau.\nonumber
\end{align}
\end{Proposition}  
\begin{proof}
To derive \eqref{eq:v}, we obtain from \eqref{eq:VFBP-La} and Definition \ref{def-v} that
\begin{equation*}
\begin{aligned}
V_t&=U_t+\frac{2a_{1}\delta}{\delta-1}(D_\eta \varrho^{\delta-1})_t=U_t-2a_{1}\delta D_\eta\Big(\varrho^{\delta-1} \big(D_\eta U+\frac{2U}{\eta}\big)\Big)-\frac{2a_{1}\delta}{\delta-1} D_\eta U D_\eta \varrho^{\delta-1} \\
&=U_t-\frac{2a_{1}\delta}{\varrho} D_\eta\Big(\varrho^\delta\big(D_\eta U+\frac{2U}{\eta}\big)\Big)+4a_{1}\delta\frac{ \varrho^{\delta-2}D_\eta\varrho U }{\eta}\\
&=-A\gamma \varrho^{\gamma-2}D_\eta\varrho- D_\eta \Phi=-\frac{A\gamma}{2a_{1}\delta}\varrho^{\gamma-\delta}(V-U)-D_\eta\Phi.
\end{aligned}
\end{equation*}
Then \eqref{V-solution} can be directly obtained by solving ODE \eqref{eq:v}. 
\end{proof}

\subsection{Some basic estimates}\label{subsec-3.1}
First, we have the fundamental energy estimate:
\begin{Lemma}\label{lemma-basic energy}
For any $t\in [0,T]$,  
\begin{equation*}
\begin{aligned}
&\,\big|(r^2\rho_0)^\frac{1}{2} U(t)\big|_2^2+\big|\eta^2\eta_r\varrho^\gamma(t)\big|_1+\int_0^t \big|(\eta^2\eta_r\varrho^{\delta})^\frac{1}{2}\mathscr{D}_\eta U\big|_2^2\,\ds\leq C_0.
\end{aligned}
\end{equation*}
\end{Lemma}
\begin{proof}
We divide the proof into four steps.

\smallskip
\textbf{1.} Multiplying $\eqref{eq:VFBP-La-eta}_1$ by $2\eta^2\eta_rU$, together with \eqref{gravitational potential},
$\eqref{eq:VFBP-La}_1$, and \eqref{eq:eta}, gives
\begin{align}
&\Big(r^2\rho_0U^2+\frac{2A}{\gamma-1}\eta^2\eta_r\varrho^\gamma\Big)_t
+4a_{1} \eta^2\eta_r\varrho^\delta\underline{\Big(\delta|D_\eta U|^2+4(\delta-1)D_\eta U\frac{U}{\eta}+ 2(2\delta -1)\frac{U^2}{\eta^2}\Big)}_{:=I_2}\nonumber\\
&=\Big(\underline{\frac{1}{4\pi G}\Big(\frac{\eta^2}{\eta_r}\Phi_r^2+\big(\frac{\eta^3\Phi^2_r}{\eta_r^2}\big)_r\Big)}_{:=I_1}\Big)_t\label{lem4.1 equation}\\
&\quad +\Big(4a_{1}\delta \eta^2\eta_r\varrho^\delta\frac{U D_\eta U}{\eta_r}-2A\varrho^{\gamma}\eta^2 U+8a_{1}(\delta-1)\eta U^2\varrho^\delta\Big)_r.\nonumber
\end{align}
Here $I_1$ satisfies
\begin{equation}\label{lemma4.4 I_7 fuzhu}
    -(I_1)_t=2\eta^2\eta_r\varrho UD_\eta\Phi=2r^2\rho_0UD_\eta\Phi.
\end{equation}

\smallskip
\textbf{2.} We first give the $L^1$-estimate of $I_1$ for $t\in [0,T]$. Using the fact that $\rho(x)$ is supported in $[0,R(t)]$, \eqref{psi_x express int}, and Lemma \ref{lemma-initial}, we have, for any $t\in [0,T]$,
\begin{equation}\label{est-I1}
\begin{aligned}
\int_0^1 I_1\,\mathrm{d}r&=\frac{1}{4\pi G}\int_0^1 |D_\eta\Phi|^2\eta^2\eta_r\,\mathrm{d}r+\frac{1}{4\pi G}\big(\frac{\eta^3\Phi^2_r}{\eta_r^2}\big)(t,1)\\
&=\frac{1}{4\pi G}\int_{0}^{R(t)}x^2\Psi_x^2\,\mathrm{d}x+\frac{1}{4\pi G}(x^3\Psi_x^2)(t,R(t))\\
&=\frac{1}{4\pi G}\int_{0}^{R(t)}x^2\Psi_x^2\,\mathrm{d}x+\frac{1}{4\pi G}\int_{R(t)}^{\infty}\big(x^2\Psi_x^2-2x\Psi_x(x^2\Psi_x)_x\big)\,\mathrm{d}x\\
&=\frac{1}{4\pi G}\int_{0}^{\infty}\!x^2\Psi_x^2\,\mathrm{d}x-2\int_{R(t)}^{\infty}\!x^3\rho\Psi_x\,\mathrm{d}x\\
&=\frac{1}{4\pi G}\int_{0}^{\infty}\!x^2\Psi_x^2\,\mathrm{d}x
= \frac{1}{16\pi^2G}\|\nabla\Psi\|_{L^2(\RR^3)}^2.
\end{aligned}
\end{equation}
Here, $\Psi$ is well-defined in the whole space $\RR^3$. To estimate $\Psi$ in Euler coordinates, we multiply \eqref{gravitational potential} by $\Psi$ and integrate the resulting equality over $\RR^3$. This, together with the fact that $\rho(\boldsymbol{x})$ is supported in $\Omega(t)$, the H\"older inequality and Lemma \ref{sobolev-embedding}, gives that, for any $t\in [0,T]$,
\begin{equation*}
\begin{aligned}
\frac{1}{16\pi^2G}\|\nabla\Psi\|_{L^2(\RR^3)}^2&=-\frac{1}{4\pi} \int_{\Omega(t)}\rho\Psi\,\mathrm{d}\boldsymbol{x}\leq C_0\|\rho\|_{L^\frac{6}{5}(\Omega(t))}\|\Psi\|_{L^6(\mathbb{R}^3)}\\
&\leq C_0\|\rho\|_{L^\frac{6}{5}(\Omega(t))}\|\nabla\Psi\|_{L^2(\mathbb{R}^3)}\leq C_0\|\rho\|_{L^1(\Omega(t))}^{\frac{5\gamma-6}{6\gamma-6}}\|\rho\|^{\frac{\gamma}{6\gamma-6}}_{L^\gamma(\Omega(t))}\|\nabla\Psi\|_{L^2(\mathbb{R}^3)},
\end{aligned}
\end{equation*}
which thus leads to
\begin{equation*}
\frac{1}{16\pi^2G}\|\nabla\Psi\|_{L^2(\RR^3)}^2\leq C_0\|\rho\|_{L^1(\Omega(t))}^{\frac{5\gamma-6}{3\gamma-3}}\|\rho\|^{\frac{\gamma}{3\gamma-3}}_{L^\gamma(\Omega(t))}.
\end{equation*}
Substituting the above into \eqref{est-I1}, we then obtain from \eqref{eq:eta}, Lemma \ref{lemma-initial}, the fact that $\gamma>\frac{4}{3}$, and the Young inequality that, for any $t\in [0,T]$,
\begin{equation}\label{Poisson_Pressure}
\begin{aligned}
\int_0^1I_1(t)\,\mathrm{d}r&\leq C_0|r^2\rho_0|_{1}^{\frac{5\gamma-6}{3\gamma-3}}|\eta^2\eta_r\varrho^\gamma|^{\frac{1}{3\gamma-3}}_{1}\leq \frac{A}{\gamma-1}|\eta^2\eta_r\varrho^\gamma(t)|_1+C_0.
\end{aligned}
\end{equation}

\smallskip
\textbf{3.} For $I_2$, let 
$(X,Y)=(D_\eta U,\frac{U}{\eta})=\mathscr{D}_\eta U$. Then $I_2$ becomes a binary form:
\begin{equation*}
\begin{aligned}
I_2=\delta X^2+4(\delta-1)XY+ 2(2\delta -1)Y^2,
\end{aligned}
\end{equation*}
and the discriminant $\mathfrak{D}$ of $I_{2}$ satisfies
\begin{equation*}
\mathfrak{D}=16(\delta-1)^2-8\delta (2\delta -1)<0  \qquad \text{whenever} \ \ \delta>\frac{2}{3}.
\end{equation*}
Hence, there exists a constant $c^*_\delta>0$, depending only on $\delta$, such that 
\begin{equation}\label{positive Viscousity quadratic form}
I_{2}\geq c^*_\delta(X^2+Y^2)=c^*_\delta |\mathscr{D}_\eta U|^2.
\end{equation}

\smallskip
\textbf{4.} Now, integrating \eqref{lem4.1 equation} over $[0,T]\times I$, we obtain from \eqref{positive Viscousity quadratic form} that, for all $t\in [0,T]$,
\begin{equation*}
\begin{aligned}
&\,\big|(r^2\rho_0)^\frac{1}{2} U(t)\big|_2^2+\frac{2A}{\gamma-1}\big|\eta^2\eta_r\varrho^\gamma(t)\big|_1+4a_{1} c_\delta^*\int_0^t\big|(\eta^2\eta_r\varrho^{\delta})^\frac{1}{2}\mathscr{D}_\eta U\big|_2^2\,\mathrm{d}s\\
&\leq \big|(r^2\rho_0)^\frac{1}{2} u_0\big|_2^2+\frac{2A}{\gamma-1}|r^2\rho_0^\gamma|_1+\int_0^1I_1(t)\,\mathrm{d}r-\int_0^1I_1(0)\,\mathrm{d}r,
\end{aligned}
\end{equation*}
which, along with \eqref{Poisson_Pressure}, leads to the desired result.
\end{proof}

Next, we have the following $\eta$-weighted estimates for $\varrho$\,:
\begin{Lemma}\label{lemma-far depth}
There exists a constant $C(T)>0$ such that, for any $t\in [0,T]$,
\begin{equation*} 
|\eta\eta_r\varrho^\delta(t)|_1+|(\eta^2 \varrho^\delta) (t)|_\infty+\int_0^t |\eta\eta_r\varrho^\gamma|_1\,\mathrm{d}s\leq C(T).
\end{equation*}
\end{Lemma}
\begin{proof} We divide the proof into four steps.

\smallskip
\textbf{1.} Multiplying $\eqref{eq:VFBP-La-eta}_1$ by $\eta^2\eta_r$ and integrating the resulting equality with respect to $r$ over $[\tilde{r},1]$ for some $\tilde{r}\in [0,1)$, together with $\varrho|_{r=1}=0$, $\eqref{eq:VFBP-La}_1$, and \eqref{eq:eta}, gives
\begin{equation}\label{integ}
\begin{aligned}
&\,\frac{\mathrm{d}}{\mathrm{d}t}\Big(2a_{1}(\eta^2 \varrho^\delta)(t,\tilde{r})+4a_{1} \int_{\tilde{r}}^1\eta\eta_r\varrho^\delta\,\mathrm{d}r\Big)+2A \int_{\tilde{r}}^1 \eta\eta_r \varrho^\gamma \,\mathrm{d}r+A(\eta^2 \varrho^\gamma)(t,\tilde{r})\\
&=\frac{\mathrm{d}}{\mathrm{d}t}\int_{\tilde{r}}^1 r^2\rho_0 U \,\mathrm{d}r+\int_{\tilde{r}}^1\eta^2\eta_r\varrho D_\eta\Phi\,\mathrm{d}r.
\end{aligned}    
\end{equation} 

Integrating the above over $[0,\tau]$, for any $\tau\in [0,T]$, since $\tilde{r}\in [0,1)$ is arbitrary, gives
\begin{equation}\label{lem4.2_diffequation2}
\begin{aligned}
&\,\sup_{t\in[0,T]}\Big(2a_{1}(\eta^2 \varrho^\delta)(t,\tilde{r})+4a_{1} \int_{\tilde{r}}^1\eta\eta_r\varrho^\delta\,\mathrm{d}r\Big)\\
&\quad+2A \int_0^T\int_{\tilde{r}}^1 \eta\eta_r \varrho^\gamma \,\mathrm{d}r\,\mathrm{d}t+A\int_0^T(\eta^2 \varrho^\gamma)(t,\tilde{r})\,\mathrm{d}t\\
&\leq C_0\Big(|r^2\rho_0^\delta|_\infty+ |r\rho_0^\delta|_1+\underline{\sup_{t\in[0,T]}|r^2\rho_0 U|_1}_{:=I_3}\Big)+C_0\int_0^T\underline{|\eta^2\eta_r\varrho D_\eta\Phi|_1}_{:=I_4}\,\mathrm{d}t.
\end{aligned}    
\end{equation} 

\smallskip
\textbf{2. Estimates for $I_3$--$I_4$.}  
The estimate for $I_3$ follows from Lemma \ref{lemma-basic energy} and the H\"older inequality:
\begin{equation}\label{I3}
I_3\leq |r^2\rho_0|_1^\frac{1}{2}\big(\sup_{t\in[0,T]}|(r^2\rho_0)^\frac{1}{2} U|_2\big)\leq C_0.
\end{equation}
To estimate $I_4$, we first use the Eulerian--Lagrangian change of variables and denote the following functions:
\begin{equation}\label{poisson def E--L}
\begin{aligned}
    m(t,x)&:=\int_0^x \tilde x^2\rho(t,\tilde x)\,\mathrm{d}\tilde x,\qquad \mathcal{Q}(t):= \int_0^\infty \tilde x\rho^\gamma(t,\tilde x)\,\mathrm{d}\tilde x= |\eta\eta_r\varrho^\gamma|_1,\\
    M_0&:=\int_0^\infty \tilde x^2\rho(t,\tilde x)\,\mathrm{d}\tilde x=\int_0^1 \tilde r^2\rho_0(\tilde r)\,\mathrm{d}\tilde r\leq C_0.
\end{aligned}
\end{equation}
Hence we obtain from  the H\"older inequality that
\begin{equation}\label{lem3.2 revise m}
    \begin{aligned}
       m(t,x)&=\int_0^x(\tilde x^\frac{1}{\gamma}\rho)\tilde x^{2-\frac{1}{\gamma}}\,\mathrm{d}\tilde x\leq C_0\mathcal{Q}^\frac{1}{\gamma}(t)\Big(\int_0^ x \tilde x^{\frac{2\gamma-1}{\gamma-1}}  \,\mathrm{d}\tilde x\Big)^{\frac{\gamma-1}{\gamma}}\leq C_0\mathcal{Q}^\frac{1}{\gamma}(t) x^{\frac{3\gamma-2}{\gamma}}.
    \end{aligned}
\end{equation}

Now, we can estimate $I_4$ in Eulerian coordinates. It follows from the fact that $m(t,x)\leq M_0$, \eqref{lem3.2 revise m}, and the H\"older inequality that, for any $L>0$ to be chosen below, 
\begin{equation}\label{lem3.2 revise I4mid}
\begin{aligned}
I_4&=4\pi G \big(\int_0^L \rho m (t,\tilde x)\,\mathrm{d}\tilde x+\int_L^\infty \rho m(t,\tilde x)\,\mathrm{d}\tilde x\big)\\
&\leq C\mathcal{Q}^\frac{1}{\gamma}(t)\int_0^L\rho\tilde x^{3-\frac{2}{\gamma}}   \,\mathrm{d}\tilde x+CM_0\int_L^\infty\rho \frac{\tilde x^2}{L^2}\,\mathrm{d}\tilde x \leq C\mathcal{Q}^\frac{2}{\gamma}(t)L^{\frac{4(\gamma-1)}{\gamma}}+C\frac{M_0^2}{L^2}.
\end{aligned}
\end{equation}
Next, taking $L=\big(\frac{M_0^\gamma}{\mathcal{Q}} \big)^{\frac{1}{3\gamma-2}}$ in \eqref{lem3.2 revise I4mid}, together with the fact that $3\gamma-2>2$ and the Young inequality, yields that, for any $\varepsilon>0$,
\begin{equation}\label{lem3.2 revise I4}
\begin{aligned}
       I_4&\leq C M_0^{\frac{4(\gamma-1)}{3\gamma-2}}\mathcal{Q}^\frac{2}{3\gamma-2}\leq \varepsilon \mathcal{Q}(t)+C_\varepsilon. 
\end{aligned}
\end{equation}

\smallskip
\textbf{3. The estimate at $\tilde r=0$.} Choosing $\tilde r=0$ in \eqref{lem4.2_diffequation2}, and then substituting \eqref{lem3.2 revise I4}(let  $\varepsilon=\frac{A}{2C_0}$) into the resulting inequality, together with the Eulerian–Lagrangian change of variables, the fact that $\eta(t,0)=0$, and \eqref{I3}, yields 
\begin{equation}\label{lem3.2 revise second half}
\begin{aligned}
&\sup_{t\in[0,T]}\int_{0}^1\eta\eta_r\varrho^\delta\,\mathrm{d}r+ A\int_0^T\mathcal{Q}(t)\,\mathrm{d}t+\int_0^T I_4 \,\mathrm{d}t\leq C(T).
\end{aligned}    
\end{equation} 

\smallskip
\textbf{4. The estimate for arbitrary $\tilde r$.} 
Finally, dropping the other nonnegative terms on the left side of \eqref{lem4.2_diffequation2}, along with \eqref{I3} and \eqref{lem3.2 revise second half}, yields that, for arbitrary $\tilde r\in [0,1)$,
\begin{equation}
    \begin{aligned}
        \sup_{t\in[0,T]}(\eta^2 \varrho^\delta)(t,\tilde{r})\leq C(T),
    \end{aligned}
\end{equation}
which implies that
\begin{equation}\label{lem3.2 revise first half}
    \sup_{t\in[0,T]}|\eta^2 \varrho^\delta|_\infty\leq C(T).
\end{equation}
Thus, \eqref{lem3.2 revise first half}, combined with \eqref{lem3.2 revise second half}, completes the proof. 
\end{proof}

Now, we establish the lower bound for $\eta_r$ on any compact subintervals of $(0,1)$ and for $\frac{\eta}{r}$ in the exterior part.

\begin{Lemma}\label{lemma-low jacobi 1}
There exists a constant $C(T)>1$ such that, for all $(t,r)\in[0,T]\times\bar I$,
\begin{equation}\label{1''}
\eta^{2(1-\frac{1}{\delta})}\eta_r\geq C(T)^{-1}r^2\rho_0,\qquad
\eta^{3-\frac{2}{\delta}}(t,r)\geq C(T)^{-1}\int_0^r \hat r^2\rho_0(\hat r)\,\mathrm{d}\hat r.
\end{equation}
Consequently, for any $a\in(0,1)$,
\begin{equation*}
\begin{aligned}
&\frac{\eta(t,r)}{r}\geq C(a,T)^{-1},\quad \varrho(t,r)\leq C(a,T)
\qquad\text{for }(t,r)\in[0,T]\times[a,1],\\
&\eta_r(t,r)\geq C(a,T)^{-1}\qquad\text{for $(t,r)\in [0,T]\times\big[a,\frac{1+3a}{4}\big]$}.
\end{aligned}
\end{equation*}
\end{Lemma}

\begin{proof}
Inequality $\eqref{1''}_1$ follows from Lemma \ref{lemma-far depth} and the identity
\begin{equation*}
\eta^2\varrho^\delta=\eta^{2(1-\delta)}\eta_r^{-\delta}r^{2\delta}\rho_0^\delta.
\end{equation*}
Integrating $\eqref{1''}_1$ from $0$ to $r\in (0,1]$, and using $\eta(t,0)=0$ and $3\delta-2>0$, gives $\eqref{1''}_2$. Since $\rho_0>0$ in $[0,1)$, this proves the lower bound for $\frac{\eta}{r}$ on $[a,1]$, which, along with $|\eta^2\varrho^\delta|_\infty\leq C(T)$ in Lemma \ref{lemma-far depth}, leads to the upper bound of $\varrho$ on $[a,1]$. Finally, to show the lower bound of $\eta_r$ on $\big[a,\frac{1+3a}{4}\big]$, we first see that, for any $a\in (0,1)$,
\begin{equation*}
|\eta^2\varrho|_\infty\leq |\eta^2\varrho^\delta|_\infty|\varrho|_\infty^{1-\delta}\leq C(a,T)\qquad\text{for any $(t,r)\in[0,T]\times[a,1]$}.
\end{equation*}
Then, by \eqref{eq:eta} and the fact that $r^2\rho_0\geq C(a)^{-1}$ on $[a,\frac{1+3a}{4}]$, we obtain the stated lower bound for $\eta_r$. 
\end{proof}

\subsection{Flow-map-weighted estimates for the density}
\begin{Lemma}\label{near-BD}
For any $a\in (0,1)$, there exists a constant $C(a,T)>0$ such that, for all $t\in [0,T]$,
\begin{equation*}
\big|(\zeta_{a}r^2\rho_0)^\frac{1}{2}V(t)\big|_2+\big|(\zeta_{a}\eta^2\eta_r)^\frac{1}{2}  D_\eta(\varrho^{\delta-\frac{1}{2}}) (t)\big|_2+\int_0^t\big|(\zeta_{a}\eta^2\eta_r)^\frac{1}{3}\varrho^{\gamma+\delta-1}\big|_3\,\mathrm{d}s\leq C(a,T),
\end{equation*}
where the cut-off function $\zeta_a$ is defined in {\rm\S \ref{othernotation}}.
\end{Lemma}
\begin{proof}
We divide the proof into four steps.

\smallskip
\textbf{1.} Multiplying \eqref{eq:v} by $\zeta_{a}r^2 \rho_0V$ and integrating over $I$, along with \eqref{v-expression}, gives that
\begin{equation}\label{dt-G1}
\begin{aligned}
&\,\frac{1}{2}\frac{\mathrm{d}}{\dt}\int_0^1 \zeta_{a} r^2\rho_0V^2\,\mathrm{d}r+\frac{8Aa_{1} \gamma\delta}{(\gamma+\delta-1)^2} \int_0^1 \zeta_{a}\eta^{2}\eta_r|D_\eta\varrho^{\frac{\gamma+\delta-1}{2}}|^2 \,\mathrm{d}r\\
&=-A \int_0^1 \zeta_{a}\eta^2(\varrho^\gamma)_r U\,\mathrm{d}r-\int_0^1\zeta_ar^2\rho_0VD_\eta \Phi\,\mathrm{d}r\\
&=A \int_0^1 \zeta_{a} \varrho^\gamma \eta_r D_\eta(\eta^2 U)\,\mathrm{d}r+A\int_0^1 (\zeta_{a})_r  \eta^2 \varrho^\gamma U \,\mathrm{d}r -\int_0^1\zeta_ar^2\rho_0VD_\eta \Phi\,\mathrm{d}r:=\sum_{i=5}^7 I_{i}. 
\end{aligned}
\end{equation}

\smallskip
\textbf{2. Estimates for $I_5$--$I_6$.} For $I_{5}$, it follows from \eqref{eq:eta} that
\begin{equation}\label{G1,1}
I_{5}=A \int_0^1 \zeta_{a}\frac{(r^{2}\rho_0)^{\gamma}}{(\eta^{2}\eta_r)^\gamma} (\eta^{2} \eta_r)_t\,\mathrm{d}r=-\frac{A}{\gamma-1} \frac{\mathrm{d}}{\mathrm{d}t}\int_0^1\zeta_{a}\eta^{2} \eta_r \varrho^\gamma \,\mathrm{d}r.
\end{equation}
 
To estimate $I_{6}$, we obtain from \eqref{eq:eta}, the facts that $\mathrm{supp}(\zeta_a)_r\subset[a,\frac{1+3a}{4}]$ and $\rho_0^\beta\sim 1-r$, Lemmas \ref{lemma-basic energy} and \ref{lemma-low jacobi 1}, and the H\"older inequality that
\begin{equation}\label{G1,2}
\begin{aligned}
I_{6}&= A\int_{0}^{1}  (\zeta_{a})_r\frac{(r^{2}\rho_0)^{\gamma}U}{\eta^{2(\gamma-1)}\eta_r^\gamma} \,\mathrm{d}r \leq C(a,T) \int_{a}^{\frac{1+3a}{4}} (r^{2}\rho_0)^{\gamma}|U|\,\mathrm{d}r\\
& \leq C(a,T) \int_{a}^{\frac{1+3a}{4}} (r^2\rho_0)^\frac{1}{2} |U| \,\mathrm{d}r\leq C(a,T) \Big(\int_0^1 r^{2} \rho_0 U^2 \,\mathrm{d}r\Big)^\frac{1}{2}\leq C(a,T).  
\end{aligned}
\end{equation}

\smallskip
\textbf{3. Estimate for $I_7$.} We first show the following auxiliary inequality:
\begin{equation}\label{G1,3,fuzhu}
    \Big(\int_0^1 (\zeta_{a})^{3}\eta^2\eta_r\varrho^{3\gamma+3\delta-3}\,\mathrm{d}r\Big)^{\frac{1}{3}}
    \leq C(a,T) +C_0\int_0^1\zeta_a\eta^2\eta_r\big|D_\eta \varrho^{\frac{\gamma+\delta-1}{2}}\big|^2\,\mathrm{d}r.
\end{equation}
Indeed, a direct calculation, combined with the Young inequality, gives
\begin{equation}\label{G1,3,fuzhu,1}
\begin{aligned}
        &\int_0^1\zeta_a\eta^2\eta_r\big|D_\eta \varrho^{\frac{\gamma+\delta-1}{2}}\big|^2\,\mathrm{d}r=\int_0^1 \eta^2\eta_r\big|\sqrt{\zeta_a}   D_\eta \varrho^{\frac{\gamma+\delta-1}{2}}\big|^2\,\mathrm{d}r\\
        &=\int_0^1\eta^2\eta_r\big|D_\eta(\sqrt{\zeta_a}\varrho^{\frac{\gamma+\delta-1}{2}})-  (D_\eta \sqrt{\zeta_a})\varrho^{\frac{\gamma+\delta-1}{2}} \big|^2 \,\mathrm{d}r\\
        &\geq \frac{1}{2}\int_0^1\eta^2\eta_r\big|D_\eta(\sqrt{\zeta_a}\varrho^{\frac{\gamma+\delta-1}{2}}) \big|^2 \,\mathrm{d}r-C_0\int_0^1\eta^2\eta_r|D_\eta\sqrt{\zeta_a}|^2\varrho^{\gamma+\delta-1}\,\mathrm{d}r:=\sum_{i=1}^2I_{7,i}.
\end{aligned}
\end{equation}

For $I_{7,1}$, define $\tilde{\zeta}_a(t,x):=\zeta_a(\eta_*(t,x))$, where $\eta_*(t,\cdot): I(t)  \to I$ denotes the inverse flow map\footnote{Since $\eta$ is a classical solution of problem \eqref{eq:VFBP-La-eta} in $[0,T]\times \bar I$ satisfying $\eqref{b1-lo}_2$, $\eta(t,\cdot)$ is bijective in $ \bar I$ for each fixed $t\in [0,T]$, and hence its inverse $\eta_*$ is well-defined.} of $\eta$. Set $\tilde{\zeta}_a(t,\boldsymbol{x}):=\tilde{\zeta}_a(t,|\boldsymbol{x}|)$ in $\Omega(t)$, and extend it by zero outside $\Omega(t)$. Then we see from the coordinate transformations and Lemmas \ref{sobolev-embedding} and \ref{lemma-initial} that
\begin{equation}\label{G1,3,fuzhu,2}
\begin{aligned}
I_{7,1}&=\frac{1}{8\pi}\int_{\Omega(t)}\big|\nabla((\tilde{\zeta}_a)^\frac{1}{2}\rho^{\frac{\gamma+\delta-1}{2}})\big|^2\,\mathrm{d}\boldsymbol{x}=\frac{1}{8\pi}\int_{\mathbb{R}^3}\big|\nabla((\tilde{\zeta}_a)^\frac{1}{2}\rho^{\frac{\gamma+\delta-1}{2}})\big|^2\,\mathrm{d}\boldsymbol{x}\\
&\geq \frac{1}{C_0}\Big(\int_{\mathbb{R}^3}\tilde{\zeta}_a^3\rho^{3\gamma+3\delta-3}\,\mathrm{d}\boldsymbol{x} \Big)^{\frac{1}{3}}
=\frac{1}{C_0}\Big(\int_{0}^{1}(\zeta_a)^3\eta^2\eta_r\varrho^{3\gamma+3\delta-3}\,\mathrm{d}r \Big)^{\frac{1}{3}}.
\end{aligned}
\end{equation}

For $I_{7,2}$, it follows from the fact that $4-2\gamma-2\delta<0$, \eqref{eq:eta}, and Lemma \ref{lemma-low jacobi 1} that
\begin{equation}\label{G1,3,fuzhu,3}
\begin{aligned}
I_{7,2}&=-C_0\int_a^{\frac{1+3a}{4}} |(\sqrt{\zeta_a})_r|^2\eta^{4-2\gamma-2\delta}\eta_r^{-\gamma-\delta}(r^2\rho_0)^{\gamma+\delta-1} \,\mathrm{d}r \geq -C(a,T).
\end{aligned}
\end{equation}
Therefore, \eqref{G1,3,fuzhu,1}, combined with \eqref{G1,3,fuzhu,2}--\eqref{G1,3,fuzhu,3}, leads to \eqref{G1,3,fuzhu}.

Now, we estimate $I_7$. It follows from \eqref{v-expression} and \eqref{lemma4.4 I_7 fuzhu} that  
\begin{equation}\label{G1,3,1}
\begin{aligned}
I_7&=-\int_0^1\zeta_ar^2\rho_0UD_\eta \Phi \,\mathrm{d}r-\frac{2a_{1}\delta}{\delta-1}\int_0^1 \zeta_ar^2\rho_0 D_\eta(\varrho^{\delta-1})  D_\eta \Phi\,\mathrm{d}r\\
&= \frac{1}{2}\frac{\mathrm{d}}{\mathrm{d}t}\int_0^1\zeta_aI_1 \,\mathrm{d}r \ \underline{-2a_{1}\int_0^1 \zeta_a \eta^2\eta_rD_\eta\Phi(D_\eta\varrho^\delta)\,\mathrm{d}r}_{:=I_{7,3}}.
\end{aligned}
\end{equation}

For $I_{7,3}$, let
$
\vartheta=\frac{\delta}{3\gamma+3\delta-4}\in (0,\frac{1}{3})$.
Then, based on facts that $(\zeta_a)_r\leq 0$, $0\leq \zeta_a^{1-3\vartheta}\leq 1$, and
\begin{equation*}
D_\eta \Phi\geq0,\qquad D_\eta(\eta^2 D_\eta\Phi)= 4\pi G \eta^2\varrho,  
\end{equation*}
due to \eqref{D Phi expression}, we obtain from \eqref{eq:eta}, \eqref{G1,3,fuzhu}, and the H\"older and Young inequalities that      
\begin{align}
        I_{7,3}
        &=2a_{1}\int_0^1\zeta_a \eta_r\varrho^\delta D_\eta(\eta^2 D_\eta \Phi)\,\mathrm{d}r+\underline{2a_{1}\int_0^1(\zeta_a)_r\eta^2\varrho^\delta D_\eta \Phi\,\mathrm{d}r}_{ \ \leq 0}\nonumber\\
        &\leq  8\pi Ga_{1}\int_0^1 \zeta_a \eta^2 \eta_r\varrho^{\delta+1}\,\mathrm{d}r =8\pi Ga_{1} \int_0^1 \zeta_a^{1-3\vartheta}(r^2\rho_0)^{1-\vartheta}  (\zeta_a^3\eta^2\eta_r\varrho^{3\gamma+3\delta-3})^\vartheta\,\mathrm{d}r \label{G1,3,3}\\
        &\leq  C(a)\Big(\int_0^1\zeta_a^{3}\eta^2\eta_r\varrho^{3\gamma+3\delta-3}\,\mathrm{d}r\Big)^{\vartheta} \leq  C(a,T)\Big(1+ \int_0^1\zeta_a\eta^2\eta_r\big|D_\eta \varrho^{\frac{\gamma+\delta-1}{2}}\big|^2\,\mathrm{d}r\Big)^{3\vartheta}\nonumber\\
        &\leq C(a,T)+\frac{4Aa_{1} \gamma\delta}{(\gamma+\delta-1)^2}\int_0^1\zeta_a\eta^2\eta_r\big|D_\eta \varrho^{\frac{\gamma+\delta-1}{2}}\big|^2\,\mathrm{d}r.\nonumber
\end{align}

\smallskip
\textbf{4.} Substituting \eqref{G1,1}--\eqref{G1,2} and \eqref{G1,3,1}-\eqref{G1,3,3} into \eqref{dt-G1} and integrating the resulting inequality over $[0,t]$, along with \eqref{Poisson_Pressure} and Lemma \ref{lemma-basic energy}, yields
\begin{equation}\label{L2-V}
\begin{aligned}
&\,\big|(\zeta_{a} r^2\rho_0)^\frac{1}{2}V(t)\big|_2^2+\big|\zeta_{a}\eta^{2}\eta_r \varrho^\gamma(t)\big|_1+\int_0^t \big|(\zeta_{a}\eta^2\eta_r)^\frac{1}{2}  D_\eta (\varrho^\frac{\gamma+\delta-1}{2})\big|_2^2\,\mathrm{d}s\\
&\leq C_0\Big(|(\zeta_{a} r^2\rho_0)^\frac{1}{2}v_0 |_2^2+\big|\zeta_{a}r^2 \rho_0^\gamma\big|_1 +\sup_{s\in[0,t]}\int_0^1\zeta_aI_1(s)\,\mathrm{d}r\Big)+C(a,T) \leq   C(a,T).
\end{aligned}
\end{equation}
Here, for the boundedness of the initial data, it suffices to note that
$\rho_0^\beta\sim 1-r$ and
\begin{equation*}
|(\zeta_{a} r^2\rho_0)^\frac{1}{2}(\rho_0^{\delta-1})_r|_2\leq C(a)|r(\rho_0^\beta)_r|_2\leq C(a).
\end{equation*}
Since \eqref{G1,3,fuzhu} and \eqref{L2-V} hold for any cut-off functions $\zeta_a$ with $a \in (0,1)$, we also have
\begin{equation*}
\begin{aligned}
&\int_0^t\big|(\zeta_{a}\eta^2\eta_r)^\frac{1}{3}\varrho^{\gamma+\delta-1}\big|_3\,\mathrm{d}s\leq \int_0^t\big|(\zeta_{\frac{1+3a}{4}}^3\eta^2\eta_r)^\frac{1}{3}\varrho^{\gamma+\delta-1}\big|_3\,\mathrm{d}s\\
&\leq C(a,T) +C_0\int_0^t\int_0^1\zeta_{\frac{1+3a}{4}}\eta^2\eta_r\big|D_\eta \varrho^{\frac{\gamma+\delta-1}{2}}\big|^2\,\mathrm{d}r\mathrm{d}s \leq C(a,T).
\end{aligned}
\end{equation*}

Finally, it follows from \eqref{v-expression}, \eqref{L2-V}, and Lemma \ref{lemma-basic energy} that, for all $t\in [0,T]$,
\begin{equation*}
\big|(\zeta_{a}\eta^2\eta_r)^\frac{1}{2}  D_\eta(\varrho^{\delta-\frac{1}{2}} )\big|_2 \leq C(a,T),
\end{equation*}
which leads to the desired estimates.
\end{proof}

\begin{Lemma}\label{lemma-near depth}
Let parameters $(l_1,l_2)$ satisfy
\begin{equation}\label{lemma4.5 index}
l_1\in [2\delta-1,1],   \qquad   l_2\in \big[\frac{1}{2\delta-1},\frac{2}{\delta}\big].
\end{equation}
Then, for any $a\in (0,1)$ and $(l_1,l_2)$ defined above, there exist positive constants $C(l_1,a,T)$ and $C(l_2,a,T)$ such that, for all $t\in [0,T]$,
\begin{equation}\label{fff}
|\zeta_{a} \eta^{\frac{l_1-(2\delta-1)}{1-\delta}}\eta_r \varrho^{l_1}(t)|_1\leq C(l_1,a,T),\qquad 
|\zeta_{a}\eta^{l_2} \varrho(t)|_\infty\leq C(l_2,a,T).
\end{equation}
\end{Lemma}
\begin{proof} 
First, we can obtain from integration by parts, the facts that $\mathrm{supp}(\zeta_a)_r\subset[a,\frac{1+3a}{4}]$ and $\delta>\frac{2}{3}$, Lemmas \ref{lemma-far depth}--\ref{near-BD}, and the H\"older inequality that, for all $t\in [0,T]$ and $a\in (0,1)$,
\begin{equation*}
\begin{aligned}
|\zeta_{a}\eta_r \varrho^{2\delta-1}|_1 &=\int_0^1 \zeta_{a} \varrho^{2\delta-1}\,\mathrm{d}\eta=-\int_0^1 (\zeta_{a})_r\eta \varrho^{2\delta-1}\,\mathrm{d}r-2\int_0^1\zeta_{a}\eta \varrho^{\delta-\frac{1}{2}}(\varrho^{\delta-\frac{1}{2}})_r\,\mathrm{d}r\\
&\leq C(a)\int_a^\frac{1+3a}{4} \eta^{\frac{2-3\delta}{\delta}}(\eta^2 \varrho^\delta)^\frac{2\delta-1}{\delta}\,\mathrm{d}r+2\big|(\zeta_{a}\eta^2\eta_r)^\frac{1}{2} D_\eta \varrho^{\delta-\frac{1}{2}}\big|_2|\zeta_{a}\eta_r \varrho^{2\delta-1}|_1^\frac{1}{2}\\
&\leq C(a,T)\big(|\eta^2 \varrho^\delta|_\infty^\frac{2\delta-1}{\delta}+ |\zeta_{a}\eta_r\varrho^{2\delta-1}|_1^\frac{1}{2}\big) 
\leq C(a,T)\big(1+ |\zeta_{a}\eta_r \varrho^{2\delta-1}|_1^\frac{1}{2}\big),
\end{aligned}
\end{equation*}
which, along with the Young inequality, yields
\begin{equation*} 
|\zeta_{a}\eta_r \varrho^{2\delta-1}|_1\leq C(a,T)\qquad \text{for all $t\in[0,T]$}.
\end{equation*}

This, together with \eqref{eq:eta} and the H\"older inequality, yields that, for all $l_1\in[2\delta-1,1]$ and $k_1:=\frac{1-l_1}{2(1-\delta)}\in [0,1]$, 
\begin{equation*}
\begin{aligned}
    |\zeta_{a} \eta^{\frac{l_1-(2\delta-1)}{1-\delta}}\eta_r \varrho^{l_1}(t)|_1
&\leq|\zeta_{a}\eta_r \varrho^{2\delta-1}|_1^{k_1}|\zeta_{a}\eta^2\eta_r \varrho|_1^{1-k_1}\leq |\zeta_{a}\eta_r \varrho^{2\delta-1}|_1^{k_1}|r^2\rho_0|_1^{1-k_1}\leq C(l_1,a,T), 
\end{aligned}
\end{equation*}
which implies $\eqref{fff}_1$.

Next, it follows from $\eqref{fff}_1$, $\mathrm{supp}(\zeta_a)_r\subset[a,\frac{1+3a}{4}]$, Lemmas \ref{lemma-far depth}--\ref{near-BD}, and \ref{sobolev-embedding} that
\begin{equation}\label{lemma4.5 infty}
\begin{aligned}
|\zeta_{a}\eta \varrho^{2\delta-1}|_\infty &\leq C_0|(\zeta_{a}\eta \varrho^{2\delta-1})_r|_1 \leq C_0\big|((\zeta_{a})_r\eta \varrho^{2\delta-1},\,\zeta_{a} \eta_r\varrho^{2\delta-1},\,\zeta_{a}\eta\varrho^{\delta-\frac{1}{2}}(\varrho^{\delta-\frac{1}{2}})_r)\big|_1\\
&\leq  C(a)\int_a^\frac{1+3a}{4}  \eta^{\frac{2-3\delta}{\delta}}(\eta^2 \varrho^\delta)^\frac{2\delta-1}{\delta}\,\mathrm{d}r+C_0|\zeta_{a} \eta_r \varrho^{2\delta-1}|_1\\
&\quad +C_0\big|(\zeta_{a}\eta^2\eta_r)^\frac{1}{2} D_\eta(\varrho^{\delta-\frac{1}{2}})\big|_2 | \zeta_{a}\eta_r \varrho^{2\delta-1}|_1^\frac{1}{2} \\
&\leq C(a,T) \big(|\eta^2\varrho^\delta|_\infty^{\frac{2\delta-1}{\delta}}+| \zeta_{a}\eta_r \varrho^{2\delta-1}|_1+| \zeta_{a}\eta_r \varrho^{2\delta-1}|_1^\frac{1}{2}\big)\leq C(a,T).
\end{aligned}
\end{equation}
Since Lemma \ref{lemma-far depth} and \eqref{lemma4.5 infty} hold for any cut-off function $\zeta_a$ with $a\in(0,1)$, we have
\begin{equation*}
|\zeta_{a}\eta^\frac{1}{2\delta-1} \varrho|_\infty \leq C(a,T),\qquad  |\zeta_a\eta^\frac{2}{\delta} \varrho|_\infty \leq C(a,T),
\end{equation*}
which, along with the H\"older inequality, lead to $\eqref{fff}_2$. 
\end{proof}

\subsection{\texorpdfstring{$L^p$}{}-estimates for \texorpdfstring{$(r^2\rho_0)^\frac{1}{p}(U,V)$}{}} 
Our goal of this subsection is to establish the $L^p$-energy estimates  for $(U,V)$, which can be stated as follows:
\begin{Lemma}\label{lp-uv}  
Let  $\gamma\in (\frac{4}{3},6\delta-3)$ and $\delta\in (\frac{13}{18},1)$, and define
\begin{equation}\label{p+}
p_{*}=p_{*}(\delta):=\frac{\delta(2\delta-1)+\sqrt{\delta(2\delta-1)(3\delta-2)}}{(1-\delta)^2}.
\end{equation}
Then, for any $p\in [2,p_{*})$, there exists a constant $C(p,T)>0$ such that, for all $t\in[0,T]$,
\begin{equation}\label{zong-jiehe}
\big|(r^2\rho_0)^\frac{1}{p}U(t)\big|_p+\big|(\zeta r^2\rho_0)^\frac{1}{p}V(t)\big|_p+\int_0^t \big|(\eta^2\eta_r\varrho^\delta)^\frac{1}{2}|U|^\frac{p-2}{2}\mathscr{D}_\eta U\big|_2^2\,\mathrm{d}s\leq C(p,T),
\end{equation}
where the cut-off function $\zeta$ is defined in {\rm\S \ref{othernotation}}.
\end{Lemma}
\begin{Remark}\label{remark-p+}
A calculation shows that $p_{*}(\delta)$ is strictly increasing with respect to $\delta$, and
\begin{equation}
\lim_{\delta\to \frac{2}{3}^+}p_{*}(\delta)=2,\qquad \lim_{\delta\to 1^-}p_{*}(\delta)=\infty,\qquad \text{and} \ \ \ \ p_{*}(\frac{13}{18})>7.
\end{equation}
\end{Remark}

The proof of Lemma \ref{lp-uv} will be divided into the following two cases:
\begin{align}
\text{Case (I):}& \ \  \gamma\in \big(\frac{4}{3},\frac{3p_{*} +1}{p_{*}}\delta-\frac{p_{*} +1}{p_{*}}\big), \qquad\qquad\qquad\quad\quad\,\delta\in \big(\frac{7p_{*} +3}{9p_{*} +3},1\big);\label{dede-gaga1}\\
\text{Case (II):}& \ \  \gamma\in \big(\frac{4}{3},6\delta-3\big)\cap \big[\frac{3p_{*} +1}{p_{*}}\delta-\frac{p_{*} +1}{p_{*}},\infty\big), \qquad \delta\in (\frac{13}{18},1).\label{dede-gaga2} 
\end{align}

\subsubsection{Proof for Case {\rm(I)}}\label{subsub-2}

We first give the $L^p$-energy estimates for $U$.
\begin{Lemma}\label{lemma-u Lp-n3}
Let $(\gamma,\delta)$ satisfy \eqref{dede-gaga1}. Then for any $p\in[2,p_{*})$, there exists a constant $C(p,T)>0$ such that, for all $t\in [0,T]$,
\begin{equation*}
\big|(r^2\rho_0)^\frac{1}{p}U(t)\big|_p+\int_0^t\big|(\eta^2\eta_r\varrho^\delta)^\frac{1}{2}|U|^\frac{p-2}{2}\mathscr{D}_\eta U\big|_2^2\,\ds\leq C(p,T).
\end{equation*}
\end{Lemma}
\begin{proof}
We divide the proof into four steps.

\smallskip
\textbf{1.} Let $p\in  [2,p_{*})$ with $p_{*}$ being defined in \eqref{p+}. Multiplying $\eqref{eq:VFBP-La-eta}_1$ by $\eta^2\eta_r |U|^{p-2}U$ gives
\begin{equation}\label{lemma4.7 eq}
\begin{aligned}
&\,\frac{1}{p}(r^2\rho_0|U|^p)_t+(2a_{1}\eta^2\eta_r\varrho^\delta)\cdot I_8\\
&=\Big(2a_{1}\delta\eta^2\varrho^\delta|U|^{p-2}U\big(D_\eta U+2\frac{U}{\eta}\big)   -4a_{1}\eta\varrho^\delta|U|^{p}-\frac{Ar^{2\gamma}\rho_0^\gamma}{\eta^{2(\gamma-1)}\eta_r^\gamma} |U|^{p-2}U \Big)_r\\
&\quad+\frac{A(r^{2}\rho_0)^\gamma}{(\eta^{2}\eta_r)^{\gamma-1}}|U|^{p-2}\big((p-1)D_\eta U+\frac{2U}{\eta}\big)-r^2\rho_0 D_\eta\Phi|U|^{p-2}U.
\end{aligned}
\end{equation}
Here, $I_8$ is a binary form:
\begin{equation*}
I_8:=(p-1)\delta X^2-2p(1-\delta)XY+(4\delta-2)Y^2 \qquad \text{with $(X,Y)=|U|^{\frac{p-2}{2}}\mathscr{D}_\eta U$}.
\end{equation*}
The discriminant $\mathfrak{D}=\mathfrak{D}(p)$ of $I_8$ satisfies $\mathfrak{D}(p_{*})=0$ and
\begin{equation*}
\mathfrak{D}=4((1-\delta)^2p^2-2\delta(2\delta-1)p+4\delta^2-2\delta)<0 \qquad \text{whenever $p\in [2,p_{*})$}, \end{equation*}
which implies that there exists a constant $c_{\delta,p}^{*}>0$, depending only on $\delta$ and $p$, such that
\begin{equation*} 
I_8\geq  c_{\delta,p}^{*}(X^2+Y^2)= c_{\delta,p}^{*}|U|^{p-2}|\mathscr{D}_\eta U|^2.
\end{equation*}

Hence, integrating \eqref{lemma4.7 eq} over $I$, along with the H\"older and Young inequalities, yields
\begin{equation}\label{dt-G2,n3}
\begin{aligned}
&\,\frac{1}{p}\frac{\mathrm{d}}{\mathrm{d}t}\big|(r^2\rho_0)^\frac{1}{p}U\big|_p^p+2a_{1} c_{\delta,p}^{*}\big|(\eta^2\eta_r\varrho^\delta)^\frac{1}{2}|U|^\frac{p-2}{2}\mathscr{D}_\eta U\big|_2^2\\
&\leq \underline{\int_0^1 r^2\rho_0 D_\eta\Phi|U|^{p-1}\,\mathrm{d}r}_{:=I_9}+\int_0^1 \frac{A(r^{2}\rho_0)^\gamma}{(\eta^{2}\eta_r)^{\gamma-1}}|U|^{p-2}\big((p-1)D_\eta U+\frac{2U}{\eta}\big)\,\mathrm{d}r\\
&\leq I_9+C(p)\underline{\Big|\frac{(r^2\rho_0)^{\gamma-\frac{\delta}{2}}}{(\eta^2\eta_r)^{\gamma-\frac{\delta+1}{2}}}|U|^\frac{p-2}{2}\Big|_2^2}_{:=I_{10}}  +\frac{a_{1}}{2} c_{\delta,p}^{*}\big|(\eta^2\eta_r\varrho^\delta)^\frac{1}{2}|U|^\frac{p-2}{2}\mathscr{D}_\eta U\big|_2^2.
\end{aligned}   
\end{equation}

\smallskip
\textbf{2. Estimate for $I_9$.} We first show the following auxiliary inequality:
\begin{equation}\label{lemma4.7 fuzhu}
|(\eta^2\eta_r)^\frac{1}{3}\varrho|_3 \leq C(T)+ C_0\big|(\zeta\eta^2\eta_r)^\frac{1}{3}\varrho^{\gamma+\delta-1}\big|_3.
\end{equation}
Indeed, set $\vartheta=\frac{2}{3\gamma+3\delta-4}\in (0,1)$. Then \eqref{lemma4.7 fuzhu} follows from \eqref{eq:eta}, Lemma \ref{lemma-low jacobi 1}, and the H\"older and Young inequalities: 
\begin{align*}
|(\eta^2\eta_r)^\frac{1}{3}\varrho|_3&\leq |\chi(\eta^2\eta_r)^\frac{1}{3}\varrho|_3+|\chi^\sharp(\eta^2\eta_r)^\frac{1}{3}\varrho|_3\\
&\leq \big|(\zeta\eta^2\eta_r)^\frac{1}{3}\varrho^{\gamma+\delta-1}\big|_3^{\vartheta}|\eta^2\eta_r\varrho|_1^\frac{(1-\vartheta)}{3}+\Bigl(\int_{\frac{1}{2}}^1 r^2\rho_0\varrho^2\,\mathrm{d}r\Bigr)^{\frac{1}{3}}\\
&\leq C_0\big|(\zeta\eta^2\eta_r)^\frac{1}{3}\varrho^{\gamma+\delta-1}\big|_3+C(T).
\end{align*}

Now, we estimate $I_9$. It follows from the H\"older and Young inequalities that 
\begin{equation}\label{lemma4.7 I9 fuzhu}
I_9\leq \underline{\big|(\eta^2\eta_r\varrho)^\frac{1}{p} D_\eta\Phi\big|_p}_{:=I_{9,1}}|(r^2\rho_0)^{\frac{1}{p}}U|_p^{p-1}\leq C(p)I_{9,1}\cdot\big(1+|(r^2\rho_0)^{\frac{1}{p}}U|_p^{p}\big).
\end{equation}
For $I_{9,1}$, we first obtain from \eqref{psi_x express int} and the H\"older inequality that, for $x\in \bar{I}(t)$ in Eulerian coordinates,
\begin{equation}\label{nabla Psi inte-expression}
\begin{aligned}
|\Psi_x|&=\frac{4\pi G }{x^2}\int_0^x \tilde x^2\rho(t,\tilde x)\,\mathrm{d}\tilde x \leq  \frac{4\pi G }{x^2}\Big(\int_0^x \tilde x^2\rho^3(t,\tilde x)\,\mathrm{d}\tilde x\Big)^\frac{1}{3}\Big(\int_0^x \tilde x^2\,\mathrm{d}\tilde x\Big)^\frac{2}{3}\\
&\leq C\|x^\frac{2}{3}\rho(t,x)\|_{L^3(\bar I(t))}.
\end{aligned}
\end{equation}
Then, by \eqref{eq:eta}, \eqref{lemma4.7 fuzhu}, \eqref{nabla Psi inte-expression}, and the Eulerian–Lagrangian change of variables, we have, for any $t\in [0,T]$,
\begin{equation}\label{lemma4.7 I9,1}
\begin{aligned}
I_{9,1}&=\big\|x^2\rho |\Psi_x|^p \big\|_{L^1(\bar I(t))}^{\frac{1}{p}}\leq C(p)\|x^2\rho\|_{L^1(\bar I(t))}^\frac{1}{p} \|\Psi_x\|_{L^\infty(\bar I(t))}\\
&\leq C(p)\|x^2\rho\|_{L^1(\bar I(t))}^\frac{1}{p}\|x^\frac{2}{3}\rho(t,x)\|_{L^3(\bar I(t))} \\
&\leq C(p)|r^2\rho_0|_1^\frac{1}{p}|(\eta^2\eta_r)^\frac{1}{3}\varrho|_3\leq C(p,T)\big(1+\big|(\zeta\eta^2\eta_r)^\frac{1}{3}\varrho^{\gamma+\delta-1}\big|_3\big).
\end{aligned}
\end{equation}
Substituting \eqref{lemma4.7 I9,1} into \eqref{lemma4.7 I9 fuzhu} leads to, for all $p\in[2,p_{*})$,
\begin{equation}\label{lemma4.7 I9}
I_9\leq  C(p,T)\big(1+ \big|(\zeta\eta^2\eta_r)^\frac{1}{3}\varrho^{\gamma+\delta-1}\big|_3\big)\big(1+|(r^2\rho_0)^{\frac{1}{p}}U|_p^{p}\big).
\end{equation}

\smallskip
\textbf{3. Estimate for $I_{10}$.} We first obtain from \eqref{eq:eta}, the fact that $\rho_0^\beta\sim 1-r$, Lemma \ref{lemma-low jacobi 1}, and the H\"older and Young inequalities that 
\begin{equation}\label{G2:n=3}
\begin{aligned}
I_{10} &\leq \underline{\int_0^\frac{1}{2} \frac{(r^2\rho_0)^{2\gamma-\delta}}{(\eta^2\eta_r)^{2\gamma-\delta-1}}|U|^{p-2}\,\mathrm{d}r}_{:=I_{10,1}}\!\!+\Big(\int_\frac{1}{2}^1\frac{(r^2\rho_0)^{\frac{p}{2}(2\gamma-\delta-1)+1}}{(\eta^2\eta_r)^{\frac{p}{2}(2\gamma-\delta-1)}}\,\mathrm{d}r\Big)^\frac{2}{p}\big|(r^2\rho_0)^\frac{1}{p}U\big|_p^{p-2}\\
&\leq I_{10,1}+\Bigl(\int_{\frac{1}{2}}^1r^2\rho_0\varrho^{\frac{p}{2}(2\gamma-\delta-1)}\,\mathrm{d}r\Bigr)^{\frac{2}{p}}\big|(r^2\rho_0)^\frac{1}{p}U\big|_p^{p-2}\\
&\leq I_{10,1}+\big|(r^2\rho_0)^\frac{1}{p}U\big|_p^{p} +C(p,T).
\end{aligned}
\end{equation}

For $I_{10,1}$, it follows from \eqref{eq:eta}, Lemma \ref{lemma-near depth}, and the H\"older and Young inequalities that, for any $\vartheta\in [0,1]$,
\begin{equation}\label{g2,1-xiao}
\begin{aligned}
I_{10,1}&=\int_0^\frac{1}{2} \frac{(r^2\rho_0)^{\frac{2\upsilon}{p}}}{\eta^{\frac{4\upsilon-4-2\vartheta(p-2)}{p}}\eta_r^{\frac{2\upsilon-2}{p}}}\Big((\eta_r\varrho^\delta)^{\frac{p-2}{p}}|U|^{p-2}\Big)^{\vartheta}\Big((r^2\rho_0)^{\frac{p-2}{p}}|U|^{p-2}\Big)^{1-\vartheta}\,\mathrm{d}r\\
&\leq \Big(\int_0^\frac{1}{2} \eta^{2+\vartheta(p-2)}\eta_r\varrho^{\upsilon}\,\mathrm{d}r\Big)^\frac{2}{p}\big|(\eta_r\varrho^\delta)^\frac{1}{p}U\big|_p^{\vartheta(p-2)}\big|(r^2\rho_0)^\frac{1}{p}U\big|_p^{(1-\vartheta)(p-2)}\\
&\leq \underline{\big|\zeta \eta^\frac{2+\vartheta(p-2)}{\upsilon-2\delta+1} \varrho \big|_\infty^{\frac{2(\upsilon-2\delta+1)}{p}}}_{:=\tilde{I}_{10,1}} \big|\zeta \eta_r\varrho^{2\delta-1}\big|_1^\frac{2}{p} \big|(\eta_r\varrho^\delta)^\frac{1}{p}U\big|_p^{\vartheta(p-2)}\big|(r^2\rho_0)^\frac{1}{p}U\big|_p^{(1-\vartheta)(p-2)},
\end{aligned}
\end{equation}
where $\upsilon$ is defined by
\begin{equation*}
\upsilon:=p\gamma+\frac{(p-2)(1-\delta)\vartheta-p\delta-p+2}{2}.
\end{equation*}

In order to apply Lemma \ref{lemma-near depth} to $\tilde{I}_{10,1}$ in \eqref{g2,1-xiao}, we need to choose suitable $\vartheta\in [0,1]$ and $p_0\in[2,p_{*})$ such that, for any $p\in [p_0,p_{*})$ and $(\gamma,\delta)$ satisfying \eqref{dede-gaga1}, 
\begin{equation}\label{dengjia0}
\frac{2+\vartheta(p-2)}{\upsilon-2\delta+1}\in \big[\frac{1}{2\delta-1},\frac{2}{\delta}\big],
\end{equation}
which is equivalent to
\begin{equation*}
\underline{f}(p;\gamma,\delta):=\frac{(2\gamma-1)p+8-(p+12)\delta}{(5\delta-3)(p-2)}\leq \vartheta\leq \frac{(2\gamma-1)p+4-(p+6)\delta}{(2\delta-1)(p-2)}.
\end{equation*}
Consequently, it suffices to choose suitable $p_0\in [2,p_{*})$ such that, for $(\gamma,\delta)$ satisfying \eqref{dede-gaga1}, 
\begin{equation}\label{dengjia01}
\underline{f}(p;\gamma,\delta)\in [0,1] \qquad \text{ for all } p\in [p_0,p_*).
\end{equation}
Then for each fixed $p\in [p_0,p_*)$, we set $\vartheta=\vartheta(p;\gamma,\delta)=\underline{f}(p;\gamma,\delta)$ to fulfill \eqref{dengjia0}. 

To verify \eqref{dengjia01}, we first consider $p=p_0$. Since $\gamma>\frac{4}{3}$ and $\delta<1$,
\begin{equation}\label{dengjia02}
\begin{aligned}
&\underline{f}(p_0;\gamma,\delta)\geq 0\iff (2\gamma-1)p_0+8\geq (p_0+12)\delta \impliedby p_0\geq 6;\\
&\underline{f}(p_0;\gamma,\delta)\leq 1\iff \frac{4}{3}<\gamma\leq \frac{3p_0+1}{p_0}\delta-\frac{p_0+1}{p_0},\quad \frac{7p_0+3}{9p_0+3}<\delta<1.
\end{aligned}
\end{equation}
We note from Remark \ref{remark-p+} and \eqref{dede-gaga1} that
\begin{equation*}
\begin{aligned}
&3\delta-1-\frac{1-\delta}{p_0}<3\delta-1-\frac{1-\delta}{p_{*}},
\quad  \frac{7p_0+3}{9p_0+3}>\frac{7p_{*}+3}{9p_{*}+3},\quad p_{*}(\frac{7p_{*}+3}{9p_{*}+3})>p_{*}(\frac{13}{18})>7>6.
\end{aligned}
\end{equation*}
Thus, there indeed exists a $p_0\in [6,p_{*})\subset[2,p_{*})$ sufficiently close to $p_{*}$ such that  the right-hand sides of $\eqref{dengjia02}_1$--$\eqref{dengjia02}_2$ hold. This completes the proof of \eqref{dengjia01}. 
Then, following the monotonicity of $p_0$ in  $\eqref{dengjia02}_2$, we obtain that \eqref{dengjia01} holds for any $p\in [p_0,p_*)$.

Now, based on the above discussion, \eqref{g2,1-xiao}, Lemma \ref{lemma-near depth}, and the Young inequality, we obtain that, for any $\varepsilon\in (0,1)$, $(\gamma,\delta)$ satisfying \eqref{dede-gaga1}, and $p\in[p_0,p_{*})$, 
\begin{equation}\label{g2,1-xiao-de}
I_{10,1}
\leq C_0\big|(r^2\rho_0)^\frac{1}{p}U\big|_p^{p}+\varepsilon\big|(\eta_r\varrho^\delta)^\frac{1}{p}U  \big|_p^{p}+C(\varepsilon,p,T).
\end{equation}

Substituting \eqref{g2,1-xiao-de} into \eqref{G2:n=3} gives that, for all $\varepsilon\in (0,1)$ and $p\in[p_0,p_{*})$,
\begin{equation}\label{g3above}
I_{10} \leq C_0\big|(r^2\rho_0)^\frac{1}{p}U\big|_p^{p}+\varepsilon\big|(\eta^2\eta_r\varrho^\delta)^\frac{1}{2}|U|^\frac{p-2}{2}\mathscr{D}_\eta U\big|_2^{2}+C(\varepsilon,p,T). 
\end{equation}

\smallskip
\textbf{4.} Combining \eqref{dt-G2,n3} with \eqref{lemma4.7 I9} and \eqref{g3above}, and then choosing a suitable small $\varepsilon\in(0,1)$, we arrive at
\begin{equation*} 
\begin{aligned}
&\frac{\mathrm{d}}{\mathrm{d}t}\big|(r^2\rho_0)^\frac{1}{p}U\big|_p^p+ a_{1} c_{\delta,p}^{*}\big|(\eta^2\eta_r\varrho^\delta)^\frac{1}{2}|U|^\frac{p-2}{2}\mathscr{D}_\eta U\big|_2^2 \\
&\leq C(p,T)\big(1+ |(\zeta\eta^2\eta_r)^\frac{1}{3}\varrho^{\gamma+\delta-1}|_3\big)\big(1+\big|(r^2\rho_0)^{\frac{1}{p}}U\big|_p^{p}\big).
\end{aligned}
\end{equation*}
This, combined with the Gr\"onwall inequality and Lemma \ref{near-BD}, gives that, for any $p\in [p_0,p_{*})$,
\begin{equation*}
\big|(r^2\rho_0)^\frac{1}{p}U(t)\big|_p+\int_0^t\big|(\eta^2\eta_r\varrho^\delta)^\frac{1}{2}|U|^\frac{p-2}{2}\mathscr{D}_\eta U\big|_2^2\,\mathrm{d}s\leq C(p,T).
\end{equation*}
Finally, it follows from Lemma \ref{lemma-basic energy} and the H\"older inequality that the above holds for all $p\in [2,p_{*})$.
\end{proof}

Based on Lemma \ref{lemma-u Lp-n3}, we obtain the following interior $L^p$-energy estimates for $V$.
\begin{Lemma}\label{lemma-v-n=3}
Let $(\gamma,\delta)$ satisfy \eqref{dede-gaga1}. Then for any $p\in [2,p_{*})$, there exists a constant $C(p,T)>0$ such that
\begin{equation*}
\big|(\zeta r^2\rho_0)^\frac{1}{p}V(t)\big|_p\leq C(p,T)\qquad \text{for all } \ t\in[0,T] .
\end{equation*}
\end{Lemma}
\begin{proof}
We divide the proof into four steps.

\smallskip
\textbf{1.}  We first let $p\in [3,p_{*})$. Multiplying \eqref{eq:v} by $\zeta r^2\rho_0 |V|^{p-2}V$
and integrating  over $I$, together with the Young inequality, yields
\begin{equation}\label{dt-v}
\begin{aligned}
&\,\frac{1}{p}\frac{\mathrm{d}}{\mathrm{d}t}\big|(\zeta r^2\rho_0)^\frac{1}{p}V\big|_p^p+\frac{A\gamma}{2a_{1}\delta}\big|(\zeta \eta^2\eta_r\varrho^{\gamma-\delta+1})^\frac{1}{p}  V\big|_p^p\\
&\leq   |\zeta\eta^2\eta_r\varrho D_\eta\Phi|V|^{p-1}|_1 +\frac{A\gamma}{2a_{1}\delta} \int_0^1\zeta\eta^2\eta_r\varrho^{\gamma-\delta+1}|U||V|^{p-1}\,\mathrm{d}r :=\sum_{i=11}^{12}I_i.
\end{aligned}
\end{equation}

\smallskip
\textbf{2.} We estimate $I_{11}$. Similarly to the estimate of $I_{9,1}$ in Lemma \ref{lemma-u Lp-n3}, it follows from \eqref{lemma4.7 fuzhu}, \eqref{lemma4.7 I9,1}, and the H\"older and Young inequalities that 
\begin{equation}\label{lemma4.8 I11}
\begin{aligned}
I_{11}&\leq |(\zeta\eta^2\eta_r\varrho)^\frac{1}{p}D_\eta \Phi|_p|(\zeta r^2\rho_0)^\frac{1}{p}V|_p^{p-1}\\
&\leq  C(p,T)\big(1+\big|(\zeta\eta^2\eta_r)^\frac{1}{3}\varrho^{\gamma+\delta-1}\big|_3\big)\big(1+\big|(\zeta r^2\rho_0)^\frac{1}{p}V\big|_p^{p}\big).
\end{aligned}
\end{equation}

\smallskip
\textbf{3.} For $I_{12}$, we obtain from the H\"older and Young inequalities that, for all $\vartheta\in (0,1)$, 
\begin{align}
I_{12}
&\leq \frac{A\gamma}{2a_{1}\delta}|(\zeta\eta^2\eta_r\varrho^\varsigma)^\frac{1}{p}U|_p
|(\zeta r^2\rho_0)^\frac{1}{p}V|_p^{\vartheta(p-1)}
\big|(\zeta\eta^2\eta_r\varrho^{\gamma-\delta+1})^\frac{1}{p}V\big|_p^{(1-\vartheta)(p-1)}\nonumber\\
&\leq C(p) |(\zeta\eta^2\eta_r\varrho^\varsigma)^\frac{1}{p}U|_p^p +C(p)|(\zeta r^2\rho_0)^\frac{1}{p}V|_p^p+\frac{A\gamma}{4a_{1}\delta}\big|(\zeta\eta^2\eta_r\varrho^{\gamma-\delta+1})^\frac{1}{p}V\big|_p^p\label{lemma4.8 I12}\\
&\leq C(p)\Big(\underline{\big|\zeta_{\frac{5}{8}}\eta^\frac{2}{\varsigma-\delta}\varrho\big|^{\varsigma-\delta}_\infty}_{:=I_{12,1}}\big|\eta_r\varrho^\delta|U|^p\big|_1 +|(\zeta r^2\rho_0)^\frac{1}{p}V|_p^p\Big)+\frac{A\gamma}{4a_{1}\delta}\big|(\zeta\eta^2\eta_r\varrho^{\gamma-\delta+1})^\frac{1}{p}V\big|_p^p,\nonumber
\end{align}
where $\varsigma=(\gamma-\delta)(\vartheta p+1-\vartheta)+1$. 

Next, to apply Lemma \ref{lemma-near depth} to $I_{12,1}$, we claim that,  for each $(\gamma,\delta)$ satisfying \eqref{dede-gaga1} and $p\in [3,p_{*})$, there exists a $\vartheta=\vartheta_0$, depending only on $(\gamma,\delta,p)$, such that
\begin{equation}\label{equiv001}
\frac{2}{\varsigma-\delta}\in\big[\frac{1}{2\delta-1}, \frac{2}{\delta}\big],
\end{equation}
or, equivalently,
\begin{equation}\label{equiv01}
\underline{g}(p;\gamma,\delta):=\frac{(3\delta-1)-\gamma}{(\gamma-\delta)(p-1)}\leq \vartheta\leq \frac{(6\delta-3)-\gamma}{(\gamma-\delta)(p-1)}.
\end{equation}
Indeed, it suffices to show 
\begin{equation}\label{equiv02}
\underline{g}(p;\gamma,\delta)\in (0,1) \qquad \text{for any $(\gamma,\delta)$ satisfying \eqref{dede-gaga1} and $p\in [3,p_*)$}.
\end{equation}
Then we can set $\vartheta_0=\underline{g}(p;\gamma,\delta)$ to fulfill \eqref{equiv01}.

Via a direct calculation, we have
\begin{equation*}
\eqref{equiv02}\iff \frac{(p+2)\delta-1}{p}<\gamma<3\delta-1\impliedby  \frac{(p+2)\delta-1}{p}\leq \frac{4}{3}\impliedby  \frac{p+1}{p}\leq \frac{4}{3}.  
\end{equation*}
Clearly, the right-hand side of the above holds for any $p\in [3,p_{*})$.

Therefore, for any $(\gamma,\delta)$ satisfying \eqref{dede-gaga1} and $p\in [3,p_{*})$, we can fix $\vartheta=\vartheta_0=\vartheta_0(\gamma,\delta,p)\in (0,1)$ such that \eqref{equiv001} holds, and then we can apply Lemma \ref{lemma-near depth} to $I_{12,1}$ to obtain
\begin{equation}\label{g4}
I_{12,1}\leq C(p,T) \qquad \text{for all $t\in[0,T]$}.
\end{equation}

\smallskip
\textbf{4.} It follows from  \eqref{dt-v}--\eqref{lemma4.8 I12} and \eqref{g4} that 
\begin{equation*}
\begin{aligned}
&\,\frac{1}{p}\frac{\mathrm{d}}{\mathrm{d}t}\big|(\zeta r^2\rho_0)^\frac{1}{p}V\big|_p^p+\frac{A\gamma}{4a_{1}\delta}\big|(\zeta \eta^2\eta_r\varrho^{\gamma-\delta+1})^\frac{1}{p}  V\big|_p^p\\
&\leq  C(p,T)\big(1+\big|(\eta^2\eta_r\varrho^\delta)^\frac{1}{2}|U|^\frac{p-2}{2}\mathscr{D}_\eta U\big|_2^2+\big|(\zeta\eta^2\eta_r)^\frac{1}{3}\varrho^{\gamma+\delta-1}\big|_3\big)\big(1+\big|(\zeta r^2\rho_0)^\frac{1}{p}V\big|_p^{p}\big),
\end{aligned}
\end{equation*}
which, along with Lemmas \ref{near-BD} and \ref{lemma-u Lp-n3}, yields that, for all $t\in[0,T]$ and $p\in [3,p_{*})$, \begin{equation}\label{vabove3}
\big|(\zeta r^2\rho_0)^\frac{1}{p}V(t)\big|_p\leq C(p,T)\big(\big|(\zeta r^2\rho_0)^\frac{1}{p}v_0\big|_p+1\big)\leq C(p,T).
\end{equation}
Here, for the boundedness of the initial data, it suffices to note that $\rho_0^\beta\sim 1-r$ and
\begin{equation*}
    \big|(\zeta r^2\rho_0)^\frac{1}{p}(\rho_0^{\delta-1})_r\big|_p\leq C_0|r^\frac{2}{p}(\rho_0^\beta)_r|_p\leq C(p).
\end{equation*}

Finally, it follows from Lemma \ref{near-BD} and the H\"older inequality that \eqref{vabove3} holds for all $p\in[2,p_{*})$, which leads to the desired estimates of this lemma.
\end{proof}

\begin{proof}[Proof of Lemma \ref{lp-uv} for Case {\rm(I)}]
Combining Lemmas \ref{lemma-u Lp-n3}--\ref{lemma-v-n=3}, we get the desired estimates of Lemma \ref{lp-uv} under Case (I).     
\end{proof}

\subsubsection{Proof for Case {\rm (II)}}\label{subsub-3}
We first consider the $L^p$-energy estimates for $U$. 
\begin{Lemma}\label{lemma-u Lp}
Let $(\gamma,\delta)$ satisfy \eqref{dede-gaga2}. Then, for any $p\in[2,p_{*})$ and $\varepsilon\in (0,1)$, there exists a constant $C(\varepsilon,p,T)>0$ such that, for all $t\in [0,T]$,
\begin{equation}\label{eq,u-lp}
\begin{aligned}
&\frac{\mathrm{d}}{\mathrm{d}t}\big|(r^2\rho_0)^\frac{1}{p}U\big|_p^p+2a_{1}c_{\delta,p}^{*} \big|(\eta^2\eta_r\varrho^\delta)^\frac{1}{2}|U|^\frac{p-2}{2}\mathscr{D}_\eta U\big|_2^2\\
&\leq C(\varepsilon, p,T)\big(1+ \big|(\zeta\eta^2\eta_r)^\frac{1}{3}\varrho^{\gamma+\delta-1}\big|_3\big)\big(1+\big|(r^2\rho_0)^{\frac{1}{p}}U\big|_p^{p}\big) +\varepsilon\big|(\zeta\eta^2\eta_r\varrho^{\gamma-\delta+1})^\frac{1}{p}V\big|_p^p.
\end{aligned}
\end{equation}
\end{Lemma}
\begin{proof}
We divide the proof into three steps.

\smallskip
\textbf{1.}  First, repeating the same calculations as in Steps 1--2 of the proof for Lemma \ref{lemma-u Lp-n3} gives
\begin{equation}\label{dt-G2}
\begin{aligned}
&\,\frac{\mathrm{d}}{\mathrm{d}t}\big|(r^2\rho_0)^\frac{1}{p}U\big|_p^p+\frac{3p}{2}a_{1} c_{\delta,p}^{*}\big|(\eta^2\eta_r\varrho^\delta)^\frac{1}{2}|U|^\frac{p-2}{2}\mathscr{D}_\eta U\big|_2^2\\
&\leq C(p,T)\big(1+ \big|(\zeta\eta^2\eta_r)^\frac{1}{3}\varrho^{\gamma+\delta-1}\big|_3\big)\big(1+\big|(r^2\rho_0)^{\frac{1}{p}}U\big|_p^{p}\big) +C(p)I_{10},
\end{aligned}
\end{equation}
where $I_{10}$ is defined in \eqref{dt-G2,n3} of the proof for Lemma \ref{lemma-u Lp-n3}.

\smallskip
\textbf{2.} Now, we estimate $I_{10}$. Repeating  calculations \eqref{G2:n=3} and $\eqref{g2,1-xiao}_1$--$\eqref{g2,1-xiao}_2$(with $\vartheta=1$) in Step 3 of the proof for Lemma \ref{lemma-u Lp-n3}, we have, for any $\varepsilon\in(0,1)$,
\begin{equation}\label{g2,1-da}
\begin{aligned}
I_{10}&\leq  \Big(\int_0^\frac{1}{2} \eta^{p}\eta_r\varrho^{p\gamma-p\delta+\delta}\,\mathrm{d}r\Big)^\frac{2}{p}\big|(\eta_r\varrho^\delta)^\frac{1}{p}U\big|_p^{p-2}+\big|(r^2\rho_0)^\frac{1}{p}U\big|_p^p+C(p,T)\\
&\leq C(\varepsilon,p)\underline{\int_0^1 \zeta\eta^{p}\eta_r\varrho^{p\gamma-p\delta+\delta}\,\mathrm{d}r}_{:=\mathcal{I}_{(\mathfrak{a}_0,\mathfrak{b}_0)}}+  
\varepsilon\big|(\eta_r\varrho^\delta)^\frac{1}{p}U\big|_p^{p}+\big|(r^2\rho_0)^\frac{1}{p}U\big|_p^p+C(p,T).   
\end{aligned}
\end{equation}
Then we choose 
\begin{equation*}
\begin{aligned}
(\mathfrak{a}_0,\mathfrak{b}_0)&=(p,p\gamma-p\delta+\delta),\ \
(\mathfrak{a}_1,\mathfrak{b}_1)=\big(\frac{p}{p-1}\mathfrak{a}_0+\frac{p-2}{p-1},\frac{p}{p-1}\mathfrak{b}_0-\frac{\gamma}{p-1}+1-\delta\big), 
\end{aligned}
\end{equation*}
and denote
\begin{equation*}
\mathcal{I}_{(k,\ell)}:=\int_0^1\zeta \eta^{k}\eta_r\varrho^{\ell}\,\mathrm{d}r.
\end{equation*}
Moreover, to treat the possible case $\gamma-(p-1)(1-\delta)=0$ uniformly, let $\varsigma=\frac{\gamma-(p-1)(1-\delta)}{p}$ and
\begin{equation*}
    \mathscr{Q}(z):=\int_1^z\tilde z^{\varsigma-1}\,\mathrm{d}\tilde z=\left\{
    \begin{aligned}
        &\frac{z^\varsigma-1}{\varsigma}\qquad & \varsigma \neq 0,\\
        &\log z\qquad & \varsigma=0,
    \end{aligned}
    \right . \qquad \text{ for } z>0,
\end{equation*}
which, together with \eqref{v-expression}, implies that
\begin{equation}\label{revise 3.67 1}
    D_\eta \mathscr{Q}(\varrho)=\varrho^{\varsigma-1}D_\eta\varrho=\frac{1}{2a_1\delta}\varrho^{\varsigma+1-\delta}(V-U)=\frac{1}{2a_1\delta}\varrho^{\frac{\gamma+1-\delta}{p}}(V-U).
\end{equation}

Since $2\mathfrak{b}_0\geq \delta (\mathfrak{a}_0+1)$, it follows from \eqref{eq:eta}, the fact that $\rho_0^\beta\sim 1-r$, Lemmas \ref{lemma-far depth}--\ref{lemma-low jacobi 1}, integration by parts, and the H\"older and Young inequalities that, for all $\varepsilon_0\in(0,1)$,
\begin{align}
\mathcal{I}_{(\mathfrak{a}_0,\mathfrak{b}_0)}&= -\frac{1}{\mathfrak{a}_0+1}\int_0^1 \zeta_r \eta^{\mathfrak{a}_0+1} \varrho^{\mathfrak{b}_0}\,\mathrm{d}r\nonumber-\frac{\mathfrak{b}_0}{\mathfrak{a}_0+1}\int_0^1 \zeta\eta^{\mathfrak{a}_0+1} \varrho^{\mathfrak{b}_0-\varsigma} \left(\mathscr{Q}(\varrho)\right)_r\,\mathrm{d}r\nonumber\\
&\leq C(\mathfrak{a}_0)\Big(\int_\frac{1}{2}^\frac{5}{8}\eta^{\mathfrak{a}_0+1-\frac{2}{\delta}\mathfrak{b}_0} \,\mathrm{d}r\Big)\big|\eta^2\varrho^\delta\big|_\infty^{\frac{\mathfrak{b}_0}{\delta}}    +\varepsilon_0 \big|(\zeta\eta^2\eta_r)^\frac{1}{p}D_\eta(\mathscr{Q}(\varrho))\big|_p^p\label{a0}\\
&\quad+C(p,\mathfrak{a}_0,\mathfrak{b}_0) \varepsilon_0^{-\frac{1}{p-1}}\mathcal{I}_{(\mathfrak{a}_1,\mathfrak{b}_1)}\nonumber\\
&\leq  C(\mathfrak{a}_0,\mathfrak{b}_0,T)+\varepsilon_0\big|(\zeta\eta^2\eta_r)^\frac{1}{p}D_\eta(\mathscr{Q}(\varrho))\big|_p^p+ C(p,\mathfrak{a}_0,\mathfrak{b}_0) \varepsilon_0^{-\frac{1}{p-1}}\mathcal{I}_{(\mathfrak{a}_1,\mathfrak{b}_1)}.\nonumber
\end{align}

Next, define two sequences $(\{\mathfrak{a}_j\}_{j\in\NN},\{\mathfrak{b}_j\}_{j\in\NN})$ as follows:
\begin{equation}\label{ditui}
\begin{aligned}
\mathfrak{a}_{j+1}&=\frac{p}{p-1}\mathfrak{a}_j+\frac{p-2}{p-1} \qquad &&\text{with } \mathfrak{a}_0=p,\\
\mathfrak{b}_{j+1}&=\frac{p}{p-1}\mathfrak{b}_j-\frac{\gamma}{p-1}+1-\delta \qquad  &&\text{with } \mathfrak{b}_0=p\gamma-p\delta+\delta.
\end{aligned}
\end{equation}
Clearly, we can solve for $(\mathfrak{a}_j,\mathfrak{b}_j)$ from \eqref{ditui} that, for $j\in\mathbb{N}$,
\begin{equation}\label{lem3.9 aj bj}
\mathfrak{a}_{j}=2(p-1)\big(\frac{p}{p-1}\big)^j-(p-2),\quad \mathfrak{b}_{j}=(\gamma-2\delta+1)(p-1)\big(\frac{p}{p-1}\big)^j+\gamma-(p-1)(1-\delta),
\end{equation}
and check that 
\begin{equation*}
2\mathfrak{b}_{j+1}\geq \delta(\mathfrak{a}_{j+1}+1) \ \ \text{ for } \left\{
    \begin{aligned}
        &\ j=0,\ \quad \gamma\in\big[3\delta-1+\frac{\delta-1}{p_*} ,3\delta-1\big], \text{ and } p\in[2,p_{*}), \\
        &\ j\in \NN,\quad \ \gamma\in\big(3\delta-1,6\delta-3\big),\  \text{ and } p\in[2,p_{*}). 
    \end{aligned}
    \right .
\end{equation*}
We then estimate $I_{10}$ separately. 

For $\gamma\in\big[3\delta-1+\frac{\delta-1}{p_*},3\delta-1\big]$, it follows from  \eqref{a0} that, for all $\varepsilon_0\in (0,1)$,
\begin{equation}\label{lemma4.9 middle gamma 1}
\begin{aligned}
\mathcal{I}_{(\mathfrak{a}_0,\mathfrak{b}_0)} &\leq  C(\mathfrak{a}_0,\mathfrak{b}_0,T)+\varepsilon_0\big|(\zeta\eta^2\eta_r)^\frac{1}{p}D_\eta(\mathscr{Q}(\varrho))\big|_p^p+ C(p,\mathfrak{a}_0,\mathfrak{b}_0) \varepsilon_0^{-\frac{1}{p-1}}\mathcal{I}_{(\mathfrak{a}_1,\mathfrak{b}_1)},\\
    \mathcal{I}_{(\mathfrak{a}_1,\mathfrak{b}_1)}
    &\leq |\zeta_{\frac{5}{8}}\eta^{\frac{p+2}{(\gamma-\delta)(p+1)+2-2\delta}}\varrho|_{\infty}^{(\gamma-\delta)(p+1)+2-2\delta}|\zeta\eta_r\varrho^{2\delta-1}|_1  \leq C(p,T).
\end{aligned}
\end{equation}
Here,  a direct calculation shows that for each $\frac{4}{3}<\gamma\in\big[3\delta-1+\frac{\delta-1}{p_*},3\delta-1\big]$ and $p\in [2,p_{*})$,
\begin{equation*}
 \frac{1}{2\delta-1}\leq \frac{p+2}{\gamma(p+1)-\delta p+2-3\delta} \leq \frac{2}{\delta},
\end{equation*}
which, along with Lemma \ref{lemma-near depth} and  \eqref{lemma4.9 middle gamma 1}, yields that 
\begin{equation}\label{lemma4.9 middle gamma 3}
\begin{aligned}
&\mathcal{I}_{(\mathfrak{a}_0,\mathfrak{b}_0)} \leq  C(\varepsilon_0,\mathfrak{a}_0,\mathfrak{b}_0,p,T)+\varepsilon_0\big|(\zeta\eta^2\eta_r)^\frac{1}{p}D_\eta(\mathscr{Q}(\varrho))\big|_p^p.
\end{aligned}    
\end{equation}

Then, for $\gamma\in\big(3\delta-1,6\delta-3\big)$, following the same argument as in \eqref{a0} thus implies that, for all $\varepsilon_j\in (0,1)$ and $j\in\NN$,
\begin{equation*}
\mathcal{I}_{(\mathfrak{a}_j,\mathfrak{b}_j)} \leq  C(\mathfrak{a}_j,\mathfrak{b}_j,T)+\varepsilon_j\big|(\zeta\eta^2\eta_r)^\frac{1}{p}D_\eta(\mathscr{Q}(\varrho))\big|_p^p+ C(p,\mathfrak{a}_j,\mathfrak{b}_j) \varepsilon_j^{-\frac{1}{p-1}}\mathcal{I}_{(\mathfrak{a}_{j+1},\mathfrak{b}_{j+1})},
\end{equation*}
which, along with \eqref{a0}, yields, for all $j\in \mathbb{N}_+$,   
\begin{align}
\mathcal{I}_{(\mathfrak{a}_0,\mathfrak{b}_0)}&\leq C(\mathfrak{a}_0,\mathfrak{b}_0,T)+\sum_{k=1}^j C(\mathfrak{a}_k,\mathfrak{b}_k,T)\prod_{\ell=0}^{k-1}C(p,\mathfrak{a}_\ell,\mathfrak{b}_\ell) \varepsilon_\ell^{-\frac{1}{p-1}} \nonumber\\
&\quad +\Big(\varepsilon_0+\sum_{k=1}^j \varepsilon_k\prod_{\ell=0}^{k-1}C(p,\mathfrak{a}_\ell,\mathfrak{b}_\ell) \varepsilon_\ell^{-\frac{1}{p-1}}\Big)\big|(\zeta\eta^2\eta_r)^\frac{1}{p}D_\eta(\mathscr{Q}(\varrho))\big|_p^p  \nonumber\\
&\quad+\Big(\prod_{\ell=0}^{j}C(p,\mathfrak{a}_\ell,\mathfrak{b}_\ell) \varepsilon_\ell^{-\frac{1}{p-1}}\Big)\mathcal{I}_{(\mathfrak{a}_{j+1},\mathfrak{b}_{j+1})} \label{aj}\\
&\leq  C(\mathfrak{a}_0,\mathfrak{b}_0,T)+\sum_{k=1}^j C(\mathfrak{a}_k,\mathfrak{b}_k,T)\prod_{\ell=0}^{k-1}C(p,\mathfrak{a}_\ell,\mathfrak{b}_\ell) \varepsilon_\ell^{-\frac{1}{p-1}}\nonumber\\
&\quad +\Big(\varepsilon_0+\sum_{k=1}^j \varepsilon_k\prod_{\ell=0}^{k-1}C(p,\mathfrak{a}_\ell,\mathfrak{b}_\ell) \varepsilon_\ell^{-\frac{1}{p-1}}\Big)\big|(\zeta\eta^2\eta_r)^\frac{1}{p}D_\eta(\mathscr{Q}(\varrho))\big|_p^p \nonumber\\
&\quad +\Big(\prod_{\ell=0}^{j}C(p,\mathfrak{a}_\ell,\mathfrak{b}_\ell) \varepsilon_\ell^{-\frac{1}{p-1}}\Big)\big|\zeta_{\frac{5}{8}}\eta^{\frac{\mathfrak{a}_{j+1}}{\mathfrak{b}_{j+1}-(2\delta-1)}}\varrho\big|_\infty^{\mathfrak{b}_{j+1}-(2\delta-1)}\big|\zeta_{\frac{5}{8}} \eta_r \varrho^{2\delta-1}\big|_1.\nonumber
\end{align}
Here, we need to check that Lemma \ref{lemma-near depth} is applicable to the last term of the right-hand side of the above. To this end, for each $\gamma\in (3\delta-1,6\delta-3)$, $\delta\in(\frac{13}{18},1)$, and $p\in[2,p_{*})$, we need to show that
\begin{equation}\label{lemma4.9 fanwei}
\exists\, j\in \mathbb{N}, \text{ which depends only on $(p,\gamma,\delta)$, such that } \ \
\frac{\mathfrak{a}_{j+1}}{\mathfrak{b}_{j+1}-(2\delta-1)}\in [\frac{1}{2\delta-1},\frac{2}{\delta}].
\end{equation}
After a direct calculation, this is equivalent to showing that, for each $\gamma\in (3\delta-1,6\delta-3)$, $\delta\in(\frac{13}{18},1)$, and $p\in[2,p_{*})$,
\begin{equation}\label{baohan}
\gamma\in (3\delta-1,6\delta-3)\subset \bigcup_{j\in \mathbb{N}}\big[f_*(j;p,\delta),f^*(j;p,\delta)\big],
\end{equation}
where
\begin{equation*}
\begin{aligned}
f_*(j;p,\delta)&:=3\delta-1-\frac{(\frac{3\delta}{2}-1)p+1-\delta}{(p-1)(\frac{p}{p-1})^{j+1}+1},\\
f^*(j;p,\delta)&:=6\delta-3-\frac{p(3\delta-2)+1-\delta}{(p-1)(\frac{p}{p-1})^{j+1}+1}.
\end{aligned}    
\end{equation*}
Clearly, both $f_*(j;p,\delta)$ and $f^*(j;p,\delta)$ are increasing with respect to $j$ as $j\to\infty$, and 
\begin{equation*}
\begin{aligned}
f_*(0;p,\delta)<3\delta-1,&\qquad &f_*(j;p,\delta)&<f^*(j;p,\delta),\\
\lim_{j\to\infty}f_*(j;p,\delta)=3\delta-1,&\qquad &\lim_{j\to\infty}f^*(j;p,\delta)&=6\delta-3.
\end{aligned}
\end{equation*}
Consequently, to obtain \eqref{baohan}, it suffices to show that, for any $j\in \mathbb{N}$,
\begin{equation*}
[f_*(j;p,\delta),f^*(j;p,\delta)]\cap [f_*(j+1;p,\delta),f^*(j+1;p,\delta)]\neq\varnothing,
\end{equation*}
or, equivalently, to show that $f_*(j+1;p)\leq f^*(j;p)$ for any $j\in\mathbb{N}$, {\it i.e.},
\begin{equation}\label{baohan2}
3\delta-1-\frac{(\frac{3\delta}{2}-1)p+1-\delta}{(p-1)(\frac{p}{p-1})^{j+2}+1}\leq
6\delta-3-\frac{p(3\delta-2)+1-\delta}{(p-1)(\frac{p}{p-1})^{j+1}+1}
\qquad \text{for any $j\in\mathbb{N}$}.   
\end{equation}
Let 
\begin{equation*}
    X=(\frac{p}{p-1})^{j+1}>1\implies (p-1)(\frac{p}{p-1})^{j+1}=(p-1)X,\quad (p-1)(\frac{p}{p-1})^{j+2}=pX.
\end{equation*}
Then \eqref{baohan2} holds if and only if
\begin{equation*}
\begin{aligned}
&(3\delta-2)p(p-1)X^2-\Big(\frac{(3\delta-2)(p-1)(p-2)}{2} +1-\delta\Big)X-(3\delta-2)\big(\frac{p-2}{2}\big):=F(X) \geq 0.  
\end{aligned}
\end{equation*}
Indeed, for $\gamma\in (3\delta-1,6\delta-3)$, $\delta\in(\frac{13}{18},1)$, and $p\in[2,p_{*})$, a direct calculation gives
\begin{equation*}
F'(X) \geq F'(1)\geq13\delta-9>0 \qquad \text{for all } X\geq 1.
\end{equation*}
Since $F(X)$ is strictly increasing on $[1,\infty)$, we obtain 
\begin{equation*}
\begin{aligned}
F(X)\geq F(1)=  (3\delta-2)\frac{p^2}{2}-1+\delta\geq 7\delta-5>0,
\end{aligned}
\end{equation*}
which thus implies claim \eqref{baohan2}. Therefore, for each $\gamma\in (3\delta-1,6\delta-3)$,  $\delta\in(\frac{13}{18},1)$, and  $p\in[2,p_{*})$,  we can set $j=j_0$ depending only on $(p,\gamma,\delta)$ in \eqref{aj} such that  \eqref{lemma4.9 fanwei} holds, and   
\begin{equation}\label{aj'}
\begin{aligned}
\mathcal{I}_{(\mathfrak{a}_0,\mathfrak{b}_0)}&\leq \Big(C(\mathfrak{a}_0,\mathfrak{b}_0,T)+\sum_{k=1}^{j_0} C(\mathfrak{a}_k,\mathfrak{b}_k,T)\prod_{\ell=0}^{k-1}C(p,\mathfrak{a}_\ell,\mathfrak{b}_\ell) \varepsilon_\ell^{-\frac{1}{p-1}} \Big)\\
&\quad +\Big(\varepsilon_0+\sum_{k=1}^{j_0} \varepsilon_k\prod_{\ell=0}^{k-1}C(p,\mathfrak{a}_\ell,\mathfrak{b}_\ell) \varepsilon_\ell^{-\frac{1}{p-1}}\Big)\big|(\zeta\eta^2\eta_r)^\frac{1}{p}D_\eta(\mathscr{Q}(\varrho))\big|_p^p\\
&\quad + \Big(\prod_{\ell=0}^{j_0} C(p,\mathfrak{a}_\ell,\mathfrak{b}_\ell) \varepsilon_\ell^{-\frac{1}{p-1}}\Big)C(\mathfrak{a}_{j_0+1},\mathfrak{b}_{j_0+1},T),
\end{aligned}    
\end{equation}
where one has used the estimates obtained in  Lemma \ref{lemma-near depth}.
Now, let $\tilde\varepsilon\in (0,1)$, and set 
\begin{equation*}
\varepsilon_0=\tilde\varepsilon,\qquad\,\,\,\, \varepsilon_k=\frac{\tilde\varepsilon}{j_0}\prod_{\ell=0}^{k-1}\frac{\varepsilon_\ell^\frac{1}{p-1}}{C(p,\mathfrak{a}_\ell,\mathfrak{b}_\ell)} \quad\, \text{for } 1\leq k\leq j_0.
\end{equation*}
Then we obtain from \eqref{eq:eta}, \eqref{v-expression}, \eqref{revise 3.67 1},  \eqref{lemma4.9 middle gamma 3}, \eqref{aj'}, Lemma \ref{lemma-near depth}, and the H\"older and Young inequalities that, for $\gamma\in\big[3\delta-1+\frac{\delta-1}{p_*},6\delta-3\big)$,
\begin{align}
\mathcal{I}_{(\mathfrak{a}_0,\mathfrak{b}_0)}&\leq C(\tilde\varepsilon,p,T)  +\big(\tilde\varepsilon+\sum_{k=1}^{j_0}\frac{\tilde\varepsilon}{j_0}\big)\big|(\zeta\eta^2\eta_r)^\frac{1}{p}D_\eta(\mathscr{Q}(\varrho))\big|_p^p \nonumber\\[-2pt]
&\leq C(\tilde\varepsilon,p,T)  +C(p)\tilde\varepsilon \big|(\zeta\eta^2\eta_r)^\frac{1}{p}\varrho^\frac{\gamma-\delta+1}{p}(U,V)\big|_p^p \nonumber\\
&\leq C(\tilde\varepsilon,p,T)+C(p)\tilde\varepsilon \big|(\zeta\eta^2\eta_r\varrho^{\gamma-\delta+1})^\frac{1}{p}V\big|_p^p \label{aj''}\\
&\quad+C(p)\tilde\varepsilon  \big|\zeta_{\frac{5}{8}}\eta^\frac{2\vartheta}{\gamma-\delta+(1-\delta)\vartheta}\varrho\big|_\infty^{\gamma-\delta+(1-\delta)\vartheta} \big|(\eta_r\varrho^\delta)^\frac{1}{p}U\big|_p^{p\vartheta}
\big|(\eta^2\eta_r\varrho)^\frac{1}{p}U\big|_p^{p(1-\vartheta)} \nonumber\\
&\leq C(\tilde\varepsilon,p,T)  +C(p,T)\tilde\varepsilon \big|(\eta_r\varrho^{\delta})^\frac{1}{p}U\big|_p^p \nonumber\\
&\quad+C(p)\tilde\varepsilon\big|(\zeta\eta^2\eta_r\varrho^{\gamma-\delta+1})^\frac{1}{p}V\big|_p^p+C(p)\big|(r^2\rho_0)^\frac{1}{p}U\big|_p^{p}.\nonumber    
\end{align}
Here $0<\vartheta=\frac{\gamma-\delta}{5\delta-3}< 1$ for all $\gamma<6\delta-3$. Then, for any $\varepsilon\in (0,1)$, setting $\tilde\varepsilon$ such that
$0<\tilde\varepsilon<\min\big\{\varepsilon,\frac{\varepsilon}{C(\varepsilon,p)},\frac{\varepsilon}{C(\varepsilon,p,T)}\big\}<1$,
we thus get from \eqref{aj''} and \eqref{g2,1-da} that
\begin{equation}\label{g2,1-da'}
I_{10} \leq 2\varepsilon\big|(\eta_r\varrho^\delta)^\frac{1}{p}U\big|_p^{p}+\varepsilon \big|(\zeta\eta^2\eta_r\varrho^{\gamma-\delta+1})^\frac{1}{p}V\big|_p^p+C(p)\big|(r^2\rho_0)^\frac{1}{p}U\big|_p^{p}+C(\varepsilon,p,T).
\end{equation}

\smallskip
\textbf{3.} Collecting \eqref{dt-G2} and \eqref{g2,1-da'}, and then setting $\varepsilon$ sufficiently small, yields \eqref{eq,u-lp}.
\end{proof}

Based on Lemma \ref{lemma-u Lp}, we obtain the following interior $L^p$-energy estimates for $V$.

\begin{Lemma}\label{lemma-v-lp}
Let $\gamma\in [3\delta-1+\frac{\delta-1}{p_*},6\delta-3)$. Then, for any $p\in[2,p_{*})$ and $\varepsilon\in(0,1)$, there exists a constant $C(p,T)>0$ such that, for all $t\in[0,T]$,
\begin{equation}\label{eq,v-lp}
\begin{aligned}
&\,\frac{\mathrm{d}}{\mathrm{d}t}\big|(\zeta r^2\rho_0)^\frac{1}{p}V\big|_p^p+\frac{pA\gamma}{4a_{1}\delta}\big|(\zeta \eta^2\eta_r\varrho^{\gamma-\delta+1})^\frac{1}{p}  V\big|_p^p \\
       &\leq C(p,T) \big(1+ \big|(\zeta\eta^2\eta_r)^\frac{1}{3}\varrho^{\gamma+\delta-1}\big|_3\big)\big(1+\big|(\zeta r^2\rho_0)^\frac{1}{p}V\big|_p^{p}\big)\\
       &\quad +C(\varepsilon, p,T) \big|(r^2\rho_0)^\frac{1}{p}U \big|_p^{p} +\varepsilon\big|(\eta_r\varrho^\delta)^\frac{1}{p}U\big|_p^p. 
\end{aligned}
\end{equation}

\end{Lemma}
\begin{proof}
Let $\gamma$ satisfy \eqref{dede-gaga2} and $p\in[2,p_{*})$. Multiplying \eqref{eq:v} by $\zeta r^2\rho_0 |V|^{p-2}V$  
and integrating over $I$, we obtain from Lemma \ref{lemma-near depth} and the H\"older and Young inequalities that, for $\vartheta\in   (0,1]$,
\begin{align}
&\,\frac{1}{p}\frac{\mathrm{d}}{\mathrm{d}t}\big|(\zeta r^2\rho_0)^\frac{1}{p}V\big|_p^p+\frac{A\gamma}{2a_{1}\delta}\big|(\zeta \eta^2\eta_r\varrho^{\gamma-\delta+1})^\frac{1}{p}  V\big|_p^p \nonumber\\
&\leq \underline{|\zeta\eta^2\eta_r\varrho D_\eta\Phi |V|^{p-1}|_1}_{:=I_{13}} +\frac{A\gamma}{2a_{1}\delta} \int_0^1 \zeta \eta^2 \eta_r\varrho^{\gamma-\delta+1}|V|^{p-1}|U|\,\mathrm{d}r\nonumber\\
&\leq I_{13}+C(p)\big|(\zeta \eta^2\eta_r\varrho^{\gamma-\delta+1})^\frac{1}{p}  U\big|_p^p+\frac{A\gamma}{4a_{1}\delta}\big|(\zeta \eta^2\eta_r\varrho^{\gamma-\delta+1})^\frac{1}{p}  V\big|_p^p \label{lemma4.10 1}\\
&\leq I_{13}+\frac{A\gamma}{4a_{1}\delta}\big|(\zeta \eta^2\eta_r\varrho^{\gamma-\delta+1})^\frac{1}{p}  V\big|_p^p \nonumber\\
&\quad +C(p)\underline{\big|\zeta_{\frac{5}{8}}\eta^\frac{2-2\vartheta}{\gamma-2\delta+1-\vartheta+\delta\vartheta}\varrho\big|_\infty^{\gamma-2\delta+1-\vartheta+\delta\vartheta}}_{:=I_{14}} 
\big|(\eta^2\eta_r\varrho)^\frac{1}{p}U \big|_p^{p\vartheta}
\big|\big(\eta_r\varrho^\delta\big)^\frac{1}{p}U\big|_p^{p(1-\vartheta)} \nonumber\\
&\leq I_{13}+C(p,T)\varepsilon^{1-\frac{1}{\vartheta}} I_{14}^{\frac{1}{\vartheta}}\ \big|(\eta^2\eta_r\varrho)^\frac{1}{p}U \big|_p^{p}+\varepsilon\big|\big(\eta_r\varrho^\delta\big)^\frac{1}{p}U\big|_p^p+\frac{A\gamma}{4a_{1}\delta}\big|(\zeta \eta^2\eta_r\varrho^{\gamma-\delta+1})^\frac{1}{p}  V\big|_p^p.\nonumber
\end{align}
For $I_{13}$, similarly to the estimates of $I_{9,1}$ in Lemma \ref{lemma-u Lp-n3}, it follows from \eqref{lemma4.7 I9,1} and the H\"older and Young inequalities that 
\begin{equation}\label{lemma4.10 I13}
\begin{aligned}
        I_{13}&\leq |(r^2\rho_0)^{\frac{1}{p}}D_\eta\Phi|_p  |(\zeta r^2\rho_0)^\frac{1}{p}V|_p^{p-1}\\
        &\leq C(p,T)\big(1+ \big|(\zeta\eta^2\eta_r)^\frac{1}{3}\varrho^{\gamma+\delta-1}\big|_3\big)\big(1+\big|(\zeta r^2\rho_0)^\frac{1}{p}V\big|_p^{p}\big).
\end{aligned}
\end{equation}
For $I_{14}$, 
set $\vartheta=\frac{6\delta-3-\gamma}{5\delta-3}\in (0,1]$; then 
$\frac{2-2\vartheta}{\gamma-2\delta+1-\vartheta+\delta\vartheta}=\frac{1}{2\delta-1}$.
 By   Lemma \ref{lemma-near depth}, we have 
\begin{equation}\label{lemma4.10 I14}
    I_{14}\leq C(T).
\end{equation}
Finally, substituting \eqref{lemma4.10 I13}--\eqref{lemma4.10 I14} into \eqref{lemma4.10 1} leads to \eqref{eq,v-lp}. 
\end{proof}

Now we can finish the proof of  Lemma \ref{lp-uv}. 
\begin{proof}[Proof of Lemma \ref{lp-uv} for Case {\rm(II)}]
Let $(\gamma,\delta)$ satisfy \eqref{dede-gaga2}, and fix $p\in[2,p_*)$. Multiplying \eqref{eq,v-lp} in Lemma \ref{lemma-v-lp} by $\frac{8a_{1}\delta \varepsilon}{pA\gamma}$, combined with Lemma \ref{lemma-u Lp}, yields that, for all $\varepsilon\in (0,1)$,
\begin{align}
&\frac{\mathrm{d}}{\mathrm{d}t}\Big(\big|(r^2\rho_0)^\frac{1}{p}U\big|_p^p+\frac{8a_{1}\delta \varepsilon}{pA\gamma}\big|(\zeta r^2\rho_0)^\frac{1}{p}V\big|_p^p\Big)+2a_1c_{\delta,p}^{*} \big|(\eta^2\eta_r\varrho^\delta)^\frac{1}{2}|U|^\frac{p-2}{2}\mathscr{D}_\eta U\big|_2^2 \nonumber\\
&\leq C(\varepsilon,p,T)\big(1+ \big|(\zeta\eta^2\eta_r)^\frac{1}{3}\varrho^{\gamma+\delta-1}\big|_3\big)\big(1+\big|(r^2\rho_0)^{\frac{1}{p}}U\big|_p^{p}+\big|(\zeta r^2\rho_0)^\frac{1}{p}V\big|_p^{p}\big)\label{jiehe}\\
&\quad+C(\varepsilon,p,T)\big|(r^2\rho_0)^\frac{1}{p}U \big|_p^{p}+\varepsilon C(p)\big|(\eta_r\varrho^\delta)^\frac{1}{p}U\big|_p^p+C(\varepsilon,p,T).\nonumber
\end{align}

Thus, we can choose $\varepsilon$ in \eqref{jiehe} sufficiently small such that 
$0<\varepsilon<\min\big\{1,\frac{a_1c_{\delta,p}^{*}}{2C(p)}\big\}$,
and obtain from Lemma \ref{near-BD} that \eqref{zong-jiehe} in Lemma \ref{lp-uv} holds. 
\end{proof}

\subsection{Global Uniform Lower Bounds of \texorpdfstring{$(\eta_r,\frac{\eta}{r})$}{}} \label{subsec-3.4}

With the help of Lemmas \ref{lemma-near depth}--\ref{lp-uv}, we are now ready to establish the global uniform upper bound for $\varrho$ in $[0,T]\times \bar I$. 
\begin{Lemma}\label{lemma-bound depth}
Let $\gamma\in (\frac{4}{3},6\delta-3)$ and $\delta\in (\frac{13}{18},1)$. Then there exists a constant $C(T)>0$ such that
\begin{equation*}
|\varrho(t)|_\infty \leq C(T) \qquad  \text{for all } t\in [0,T].
\end{equation*}
\end{Lemma}
\begin{proof}
It follows from \eqref{v-expression}, Lemmas \ref{lemma-low jacobi 1}, \ref{lemma-near depth}--\ref{lp-uv}, and \ref{sobolev-embedding}, and the H\"older and Young inequalities that, for fixed $p\in (3,p_{*})$,
\begin{equation}\label{shangjie}
\begin{aligned}
|\zeta \varrho^{\delta}|_\infty&\leq C_0\big|(\zeta_r \varrho^{\delta},\zeta(\varrho^\delta)_r)\big|_1\leq C_0\big(\big|\zeta_r\eta^{-2}\big|_1|\zeta_{\frac{5}{8}}\eta^2\varrho^\delta|_\infty+|\zeta\eta_r\varrho^{\delta-1}D_\eta \varrho|_1\big)\\
&\leq  C(T)+C_0\big|\zeta\eta^{-\frac{2}{p-1}} \eta_r \varrho \big|_1^\frac{p-1}{p} \big|(\zeta r^2\rho_0)^\frac{1}{p}(V,U)\big|_p \\
&\leq C(T)\big(1+\underline{\big|\zeta\eta^{-\frac{2}{p-1}} \eta_r \varrho \big|_1}_{:=I_{15}}\big). 
\end{aligned}     
\end{equation}

For $I_{15}$, it follows from integration by parts, \eqref{eq:eta}, \eqref{v-expression}, Lemmas \ref{lemma-low jacobi 1}, \ref{lemma-near depth}, and \ref{lp-uv}, and the H\"older and Young inequalities that, for all $t\in[0,T]$ and $p\in(3,p_{*})$,
\begin{align}
I_{15}&=\int_0^1 \zeta \eta^{-\frac{2}{p-1}}\eta_r\varrho\,\mathrm{d}r\leq C(p)\Big(\int_0^1 |\zeta_r| \eta^{\frac{p-3}{p-1}}\varrho\,\mathrm{d}r+\int_0^1 \zeta \eta^{\frac{p-3}{p-1}}\eta_r\varrho^{2-\delta} (|V|+|U|)\,\mathrm{d}r\Big)\nonumber\\
&\leq C(p)|\eta^2\varrho^\delta|_\infty^\frac{1}{\delta}\int_\frac{1}{2}^\frac{5}{8}\eta^{\frac{p-3}{p-1}-\frac{2}{\delta}} \,\mathrm{d}r\label{458}\\
&\quad+ C(p)\Big( \int_0^1 \zeta  \eta^{\frac{p}{p-1}(\frac{p-3}{p-1})-\frac{2}{p-1}}\eta_r\varrho^{\frac{p}{p-1}(2-\delta)-\frac{1}{p-1}} \, \mathrm{d}r\Big)^\frac{p-1}{p}\big|(\zeta r^2\rho_0)^\frac{1}{p}(V,U)\big|_p \nonumber\\
&\leq C(p,T)+C(p,T)\int_0^1 \zeta  \eta^{\frac{p}{p-1}(\frac{p-3}{p-1})-\frac{2}{p-1}}\eta_r\varrho^{\frac{p}{p-1}(2-\delta)-\frac{1}{p-1}} \, \mathrm{d}r.\nonumber
\end{align}    

Similarly, we choose 
\begin{equation*}
\begin{aligned}
(\mathfrak{a}_0,\mathfrak{b}_0)&=(-\frac{2}{p-1},1),\ \ 
(\mathfrak{a}_1,\mathfrak{b}_1)=\Big(\frac{p}{p-1}(\mathfrak{a}_0+1)-\frac{2}{p-1},\frac{p}{p-1}(\mathfrak{b}_0+1-\delta)-\frac{1}{p-1}\Big), 
\end{aligned}
\end{equation*}
and denote
$\mathcal{I}_{(k,\ell)}=\int_0^1\zeta \eta^{k}\eta_r\varrho^{\ell}\,\mathrm{d}r$.
Next, define two sequences $(\{\mathfrak{a}_j\}_{j\in\NN},\{\mathfrak{b}_j\}_{j\in\NN})$ as follows:
\begin{equation}\label{lemma4.11 ditui}
\begin{aligned}
\mathfrak{a}_{j+1}=&\frac{p}{p-1}\mathfrak{a}_j+\frac{p-2}{p-1} \qquad \text{with } \mathfrak{a}_0=-\frac{2}{p-1},\\
\mathfrak{b}_{j+1}=&\frac{p}{p-1}\mathfrak{b}_j+\frac{(1-\delta)p-1}{p-1} \qquad  \text{with } \mathfrak{b}_0=1.
\end{aligned}
\end{equation}
Solving \eqref{lemma4.11 ditui} for $(\mathfrak{a}_j,\mathfrak{b}_j)$ gives, for $j\in\mathbb{N}$,
\begin{equation*}
\mathfrak{a}_{j}=(p-2-\frac{2}{p-1})\big(\frac{p}{p-1}\big)^j-(p-2), \ \ \mathfrak{b}_{j}=(1-\delta)p\big(\frac{p}{p-1}\big)^j+1-(1-\delta)p,
\end{equation*}
and a direct calculation shows that
\begin{equation*}
    2\mathfrak{b}_{j+1}\geq \delta(\mathfrak{a}_{j+1}+1),\qquad \forall\ j\in \NN,  \quad \text{ for all }\ p\in \big[2,1+\frac{2\delta}{3\delta-2}\big].
\end{equation*}
Following the same argument as in \eqref{458} thus implies that, for all $j\in \NN$, 
\begin{equation}\label{lemma4.11 00}
\begin{aligned}
        \mathcal{I}_{(\mathfrak{a}_0,\mathfrak{b}_0)}&\leq C(p,T)(1+\mathcal{I}_{(\mathfrak{a}_{j+1},\mathfrak{b}_{j+1})})\\
        &\leq C(p,T)\big(1+|\zeta_{\frac{5}{8}}\eta^{\frac{\mathfrak{a}_{j+1}}{\mathfrak{b}_{j+1}-2\delta+1}}\varrho|_\infty^{\mathfrak{b}_{j+1}-2\delta+1}|\zeta\eta_r\varrho^{2\delta-1}|_1\big).
\end{aligned}
\end{equation}

Here, in order to apply Lemma  \ref{lemma-near depth} in \eqref{lemma4.11 00}, we need to check that for each $\delta\in (\frac{13}{18},1)$, there exist a sufficiently large $j\in \NN$ and a $p\in \big[2,1+\frac{2\delta}{3\delta-2}\big]$ such that
\begin{equation}\label{lemma4.11 pxz}
    \frac{1}{2\delta-1}\leq \frac{\mathfrak{a}_j}{\mathfrak{b}_j-2\delta+1}\leq \frac{2}{\delta},
\end{equation}
which is equivalent to the existence of some $p\in[2,p_*)$ such that 
\begin{equation}\label{lemma4.11 pfanwei}
    3<\frac{7\delta-4}{3\delta-2}< p<\min\big\{p_{*}, \frac{5\delta-2}{3\delta-2}\big\}.
\end{equation}
A direct calculation shows that
\begin{equation*}
\begin{aligned}
3<\frac{7\delta-4}{3\delta-2}<\frac{5\delta-2}{3\delta-2}, \qquad \text{ and } \quad\frac{7\delta-4}{3\delta-2}<\frac{19}{3}<7<p_*.
\end{aligned}
\end{equation*}
Hence, there exists such $p$ and 
then we obtain from Lemma \ref{lemma-near depth} that
\begin{equation}\label{lemma4.11 pjg}
     I_{15}\leq C(T)(1+\mathcal{I}_{(\mathfrak{a}_0,\mathfrak{b}_0)})\leq C(T).
\end{equation}

Thus, collecting \eqref{shangjie} and \eqref{lemma4.11 pjg} gives that, for all $t\in[0,T]$,
\begin{equation*} 
|\zeta\varrho^{\delta}|_\infty\leq C(T)(1+I_{15})\leq C(T),
\end{equation*}
which, together with Lemmas \ref{lemma-far depth}--\ref{lemma-low jacobi 1}, yields that, for all $t\in[0,T]$,
\begin{equation*}
|\varrho|_\infty^\delta\leq |\zeta\varrho^\delta|_\infty+|\zeta^\sharp\varrho^\delta|_\infty\leq C(T)+|\zeta^\sharp\eta^{-2}|_\infty|\eta^2\varrho^\delta|_\infty\leq C(T),
\end{equation*}
where the cut-off function $\zeta^\sharp$ is defined in \S \ref{othernotation}. 
\end{proof}

Now we can establish the global uniform lower bounds for $(\eta_r,\frac{\eta}{r})$.  
\begin{Lemma}\label{lemma-lower bound jacobi}
There exists a constant $C(T)>1$ such that 
\begin{equation*}
\frac{\eta(t,r)}{r}\geq C(T)^{-1},\quad \eta_r(t,r) \geq C(T)^{-1} 
\qquad\,\, \text{for all $(t,r)\in [0,T]\times\bar I$}.
\end{equation*}
\end{Lemma}
\begin{proof}
We divide the proof into three steps.

\smallskip
\textbf{1. Lower bound of $\frac{\eta}{r}$.} It follows from Lemma \ref{lemma-bound depth} and \eqref{eq:eta} that
\begin{equation*} 
(\eta^2\eta_r)(t,r)=\frac{r^2\rho_0(r)}{\varrho(t,r)}\geq C(T)^{-1}r^2\rho_0(r) \qquad \text{for $0<r<1$}.
\end{equation*}
Since $\eta(t,0)=0$ and $\rho_0^\beta\sim1-r$, integrating the above over $[0,r]$ yields
\begin{equation}\label{lower-eta/r}
\frac{\eta^3(t,r)}{r^3}\geq\frac{3}{C(T)r^3}\int_0^r\hat r^2\rho_0(\hat r)\,\mathrm{d}\hat r\geq C(T)^{-1} \qquad \text{for $0<r\leq1$}.
\end{equation}
Taking $r\to0$ in \eqref{lower-eta/r}, we have
\begin{equation}\label{=0}
\frac{\eta(t,r)}{r}\Big|_{r=0}=\eta_r(t,0)\geq C(T)^{-1} \qquad \text{for $t\in[0,T]$}.
\end{equation}

\smallskip
\textbf{2. Interior lower bound of $\eta_r$.}
First,  \eqref{phi_r express int 1}, \eqref{eq:eta}, and \eqref{lower-eta/r} imply
\begin{equation}\label{non-phi}
0\leq D_\eta\Phi(t,r)=\frac{4\pi G}{\eta^2(t,r)}\int_0^r\hat r^2\rho_0(\hat r)\,\mathrm{d}\hat r\leq C(T)r.
\end{equation}

Next, multiplying $\eqref{eq:VFBP-La-eta}_1$ by $\eta^2\eta_r$, together with $\eqref{eq:VFBP-La}_1$ and \eqref{eq:eta}, gives
\begin{equation}\label{momo}
\bigl(2a_1(\eta^2\varrho^\delta)_r+r^2\rho_0U-4a_1\eta\eta_r\varrho^\delta\bigr)_t+r^2\rho_0D_\eta\Phi
=2A\eta\eta_r\varrho^\gamma-A(\eta^2\varrho^\gamma)_r.
\end{equation}
Integrating over $[0,r]\times[0,t]$, using $\eta(t,0)=0$, Lemma \ref{lemma-bound depth}, and \eqref{non-phi}, we obtain
\begin{equation*}
\begin{aligned}
2a_1(\eta^2\varrho^\delta)(t,r)
&\leq2a_1r^2\rho_0^\delta(r)+\int_0^r\hat r^2\rho_0(\hat r)(|U(t,\hat r)|+|u_0(\hat r)|)\,\mathrm{d}\hat r\\
&\quad+4a_1\int_0^r(\eta\eta_r\varrho^\delta)(t,\hat r)\,\mathrm{d}\hat r
+2A\int_0^t\int_0^r\eta\eta_r\varrho^\gamma\,\mathrm{d}\hat r\,\mathrm{d}s\\
&\leq C_0 r^2 + I_{16}+ I_{17}+ I_{18}.
\end{aligned}
\end{equation*}

For $I_{16}$, by Lemma \ref{lp-uv} with $p=3$ and the H\"older inequality, we have
\begin{equation*}
I_{16}\leq \sup_{s\in[0,t]}\bigl|(r^2\rho_0)^{\frac{1}{3}}U(s)\bigr|_3
\Bigl(\int_0^r\hat r^2\rho_0\,\mathrm{d}\hat r\Bigr)^{\frac{2}{3}}\leq C(T)r^2.
\end{equation*}
For $I_{17}$--$I_{18}$, since $\delta\in(\frac{1}{2},1)$, it follows from \eqref{eq:eta}, \eqref{lower-eta/r}, and Lemma \ref{lemma-bound depth} that
\begin{equation*}
\begin{aligned}
I_{17}&=\int_0^r\hat r^{2\delta}\rho_0^\delta\eta^{1-2\delta}\eta_r^{1-\delta}\,\mathrm{d}\hat r
\leq C(T)\int_0^r\hat r\eta_r^{1-\delta}\,\mathrm{d}\hat r\\
&\leq C(T)\Bigl(\int_0^r\hat r^{\frac{1}{\delta}}\,\mathrm{d}\hat r\Bigr)^\delta
\Bigl(\int_0^r\eta_r\,\mathrm{d}\hat r\Bigr)^{1-\delta}
=C(T)r^{1+\delta}\eta^{1-\delta}(t,r),\\
I_{18}&=\int_0^t\int_0^r\frac{\hat r^2\rho_0\varrho^{\gamma-1}}{\eta}\,\mathrm{d}\hat r\,\mathrm{d}s\leq C(T)r^2.
\end{aligned}  
\end{equation*}

Combining these estimates and using \eqref{lower-eta/r}, we obtain that, for all $(t,r)\in [0,T]\times \bar I$,
\begin{equation*}
(\eta^2\varrho^\delta)(t,r)\leq C(T)\bigl(r^2+r^{1+\delta}\eta^{1-\delta}(t,r)\bigr)
\leq C(T)r^{2\delta}\eta^{2(1-\delta)}(t,r).
\end{equation*}
Thus, by \eqref{eq:eta},
\begin{equation*}
\eta_r^\delta(t,r)=\frac{r^{2\delta}\rho_0^\delta\eta^{2(1-\delta)}}{\eta^2\varrho^\delta}
\geq C(T)^{-1}\rho_0^\delta(r)\geq C(T)^{-1} \qquad \text{for $0<r\leq\frac{3}{4}$}.
\end{equation*}
Together with \eqref{=0}, this gives the required lower bound on $[0,\frac{3}{4}]$.

\smallskip
\textbf{3. Exterior lower bound of $\eta_r$.} For $r\in[\frac{1}{2},1)$, integrating \eqref{momo} over $[r,1]\times[0,t]$ and retaining the nonnegative terms on the left-hand side gives
\begin{equation*}
\begin{aligned}
2a_1(\eta^2\varrho^\delta)(t,r)
&\leq2a_1r^2\rho_0^\delta(r)+4a_1\int_r^1\hat r\rho_0^\delta(\hat r)\,\mathrm{d}\hat r\\
&\quad+\int_r^1\hat r^2\rho_0(\hat r)(|U(t,\hat r)|+|u_0(\hat r)|)\,\mathrm{d}\hat r
+\int_0^t\int_r^1\hat r^2\rho_0D_\eta\Phi\,\mathrm{d}\hat r\,\mathrm{d}s\\
&\leq C_0\rho_0^\delta+I_{19}+I_{20}.
\end{aligned}
\end{equation*}
Then it follows from \eqref{phi_r express int 1}, \eqref{non-phi}, Lemma \ref{lp-uv} with $p=4$, and $\rho_0^\beta\sim1-r$ that
\begin{equation*}
\begin{aligned}
I_{19}&\leq \sup_{s\in[0,t]}\bigl|(r^2\rho_0)^{\frac{1}{4}}U(s)\bigr|_4
\Bigl(\int_r^1\hat r^2\rho_0\,\mathrm{d}\hat r\Bigr)^{\frac{3}{4}}
\leq C(T)\rho_0(r)^{\frac{3(\beta+1)}{4}}\leq C(T)\rho_0^\delta,\\
I_{20}&\leq C(T)\int_0^t\int_r^1\hat r^3\rho_0 \,\mathrm{d}\hat r\,\mathrm{d}s \leq C(T)\rho_0^{\beta+1}\leq C(T)\rho_0^\delta,
\end{aligned}
\end{equation*}
where we have used the facts that $\beta+1=\gamma>\frac{4}{3}$ and $\delta<1$. 

Consequently, the above estimates imply $(\eta^2\varrho^\delta)(t,r)\leq C(T)\rho_0^\delta(r)$, which, along with \eqref{eq:eta} and \eqref{lower-eta/r}, yields that, for all $(t,r)\in [0,T]\times [\frac{1}{2},1)$,
\begin{equation*}
\eta_r^\delta(t,r)\geq C(T)^{-1}r^{2\delta}\eta^{2(1-\delta)}(t,r)\geq C(T)^{-1}.
\end{equation*}
Finally, the boundary condition \eqref{N111} and $\eta(0,r)=r$ give
\begin{equation}\label{-1pre upp jacobi}
\frac{(\eta_r)_t}{\eta_r}(t,1)-2(\frac{1}{\delta}-1)\frac{\eta_t}{\eta}(t,1)=0\implies (\eta^{2(1-\frac{1}{\delta})}\eta_r)(t,1)=1, 
\end{equation}
which leads to $\eta_r(t,1)=\eta(t,1)^{\frac{2(1-\delta)}{\delta}}\geq C(T)^{-1}$. This completes the proof. 
\end{proof}

\section{Global  Uniform Upper Bounds of \texorpdfstring{$(\eta_r,\frac{\eta}{r})$}{}}\label{Section-etarupper}

This section is devoted to establishing the global uniform upper bounds of $(\eta_r,\frac{\eta}{r})$.  

\subsection{Global Uniform Estimates of the Effective Velocity}\label{Section-effectivevelocity}
In the first part, we establish the  uniform estimates for $V$. We set $\beta=\gamma-1$ and let $p_*$ denote the constant defined  in \eqref{p+}.  

\subsubsection{Boundedness of the effective velocity away from the origin}
\begin{Lemma}\label{lemma-refine u Lp}
For any $p\in (0,p_{*})$, $\iota\in \big(\frac{\beta+1}{\beta}\frac{1}{p_{*}}-\frac{1}{p},\infty\big)$, and $a\in (0,1)$, there exists a positive constant $C(p,\iota,a,T)$ such that
\begin{equation*}
\big|\chi_{a}^\sharp \rho_0^{\iota\beta} U(t)\big|_p\leq C(p,\iota,a,T) \qquad \text{for all $t\in [0,T]$},
\end{equation*}
where  $\chi_{a}^\sharp$ denotes the characteristic function on $(a,1]$ for $a\in (0,1)$
{\rm  (}see {\rm \S \ref{othernotation})}.

\end{Lemma}
\begin{proof}
Let $p\in (0,p_{*})$, $\iota\in \big(\frac{\beta+1}{\beta}\frac{1}{p_{*}}-\frac{1}{p},\infty\big)$, and  $a\in (0,1)$. Let $\varepsilon$ be a fixed constant satisfying 
\begin{equation*}
\max\Big\{0,(\iota p+1)\beta-1,(\iota p+1)\beta-\frac{p}{2},\frac{(\iota p+1)\beta^2}{\beta+1}\Big\}<\varepsilon<(\iota p+1)\beta-\frac{p}{p_{*}}.
\end{equation*}
Note that $\varepsilon$ is well-defined under the above constraint and satisfies
\begin{equation*}
0<(\iota p+1)\beta-\varepsilon<1,\qquad 2<\frac{p}{(\iota p+1)\beta-\varepsilon}<p_{*}, \qquad \text{and} \ \ \frac{-\beta+\varepsilon}{1+\varepsilon-(\iota p+1)\beta}>-\beta.
\end{equation*}
Hence, it follows from Lemma \ref{lp-uv}, the fact that $\rho_0^\beta\sim 1-r$, and the H\"older inequality that
\begin{align*}
|\chi_{a}^\sharp \rho_0^{\iota\beta} U|_p^p&=\int_a^1 \rho_0^{\iota\beta p} |U|^p\,\mathrm{d}r=\int_a^1 \rho_0^{-\beta+\varepsilon} (\rho_0^{(\iota p+1)\beta-\varepsilon} |U|^p)\,\mathrm{d}r\\
&\leq \Big(\int_a^1 \rho_0^\frac{-\beta+\varepsilon}{1+\varepsilon-(\iota p+1)\beta}\,\mathrm{d}r\Big)^{1+\varepsilon-(\iota p+1)\beta}\Big(\int_a^1 \rho_0 |U|^\frac{p}{(\iota p+1)\beta-\epsilon}\,\mathrm{d}r\Big)^{(\iota p+1)\beta-\varepsilon}
\leq C(p,\iota,a,T).
\end{align*}

\end{proof}

\begin{Lemma}\label{lemma-v Lp ex}
For any $p\in [2,p_{*})$, $\iota\in \big(\frac{p-1}{p}+\frac{1-\delta}{\beta},\infty\big)$, and $a\in (0,1)$, there exist  positive constants $C(p,\iota,a,T)$ and $C(T)$ such that 
\begin{equation*}
\big|\chi_{a}^\sharp\rho_0^{\iota\beta}V(t)\big|_p\leq C(p,\iota,a,T), \quad |D_\eta\Phi|_\infty\leq C(T) \qquad \text{for all }t\in[0,T].
\end{equation*}
\end{Lemma}
\begin{proof}
Let $p\in [2,p_{*})$, $\iota\in (\frac{p-1}{p}+\frac{1-\delta}{\beta},\infty)$, and $a\in (0,1)$. First, multiplying \eqref{V-solution} by $\chi_{a}^\sharp\rho_0^{\iota\beta}$ and taking the $L^p$-norm of the resulting equality, we obtain from \eqref{eq:eta} and the Minkowski integral inequality that, for all $t\in[0,T]$,
\begin{equation*}
|\chi_{a}^\sharp\rho_0^{\iota\beta}V|_p\leq |\chi_{a}^\sharp\rho_0^{\iota\beta}v_0|_p+\frac{A\gamma}{2a_{1}\delta}\int_0^t \Big|\chi_{a}^\sharp\frac{r^{2(\gamma-\delta)}\rho_0^{\iota\beta+\gamma-\delta}U}{(\eta^2\eta_r)^{\gamma-\delta}}\Big|_p\,\mathrm{d}s+\int_0^t \big| \chi_{a}^\sharp \rho_0^{\iota\beta}D_\eta\Phi \big|_p \,\mathrm{d}s,
\end{equation*}
which, along with Lemmas \ref{lemma-lower bound jacobi} and \ref{lemma-refine u Lp}, 
implies 
\begin{equation}\label{lemma6.2 eq}
\begin{aligned}
    |\chi_{a}^\sharp\rho_0^{\iota\beta}V|_p&\leq |\chi_{a}^\sharp\rho_0^{\iota\beta}v_0|_p+C(p,a,T)\int_0^t \big|\chi_a^\sharp\rho_0^{\iota\beta+\gamma-\delta}U\big|_p+\big| \chi_{a}^\sharp \rho_0^{\iota\beta}D_\eta\Phi \big|_\infty \,\mathrm{d}s:=\sum_{i=1}^{2}J_{i}.\end{aligned}
\end{equation}
Here, for the initial data $J_{1}$, since $\rho_0^\beta\sim 1-r$ and $\iota-1+\frac{\delta-1}{\beta}>-\frac{1}{p}$, we have
\begin{equation}\label{lemma6.2 I16}
\begin{aligned}
J_{1}&\leq C(p)\big(\big|\chi_{a}^\sharp\rho_0^{\iota\beta}u_0\big|_p+\big|\chi_{a}^\sharp\rho_0^{(\iota-1)\beta+\delta-1}(\rho_0^\beta)_r\big|_p\big)\\
&\leq C(p)\big(|\rho_0^{\beta}|_\infty^\iota |u_0|_\infty+\big|(1-r)^{\iota-1+\frac{\delta-1}{\beta}}\big|_p|(\rho_0^\beta)_r|_\infty\big)\leq C(p,\iota).
\end{aligned}    
\end{equation}
For $J_{2}$, it follows from \eqref{D Phi expression} and Lemmas \ref{lemma-lower bound jacobi} and \ref{lemma-refine u Lp} that
\begin{align*}
|D_\eta\Phi|&=\frac{4\pi G}{\eta^2}\int_0^r\hat{r}^2\rho_0\,\mathrm{d}\hat{r}\leq C(T)\frac{1}{r^2}\int_0^r\hat{r}^2\rho_0\,\mathrm{d}\hat{r}\leq C(T),\\
    J_{2}&\leq  C(p,a,T)+\int_0^t |\rho_0^{\beta}|^\iota_\infty  \big| \chi_{a}^\sharp   D_\eta\Phi \big|_\infty \,\mathrm{d}s\leq C(p,a,T),
    \end{align*}
which, along with 
\eqref{lemma6.2 eq}--\eqref{lemma6.2 I16}, yields the desired estimates. 
\end{proof}

\begin{Lemma}\label{lemma-v Linfty ex}
For any $a\in (0,1)$, there exists a constant $C(a,T)>0$ such that 
\begin{equation*}   
\big|\chi_{a}^\sharp\rho_0^{\beta+1-\delta}V(t)\big|_\infty\leq C(a,T) \qquad \text{for all $t\in [0,T]$}.
\end{equation*}
\end{Lemma}
\begin{proof}
We divide the proof into the following two steps.

\smallskip
\textbf{1.} Let $k\in \NN^*$ be a fixed constant such that 
  $ \max\big\{2,\frac{1}{\gamma-\delta}\big\}<  k< \frac{p_{*}}{2}$. Indeed, such a $k$ exists since $p_*>7$ and $\gamma-\delta>\frac{1}{3}$.
It follows from  \eqref{v-expression} and Lemma \ref{sobolev-embedding} that
\begin{equation}\label{pre-g4-g7}
\begin{aligned}
\big|\chi_{a}^\sharp \rho_0^{\beta+1-\delta}\varrho^{\gamma-\delta}U\big|_\infty^k
&\leq C(a)\big|\chi_a^\sharp(\rho_0^{k(\beta+1-\delta)}\varrho^{(\gamma-\delta)k}|U|^k)_r\big|_1\\
&\leq C(a)\big|\chi_a^\sharp\rho_0^{(k-1)\beta+k(1-\delta)}(\rho_0^\beta)_r\varrho^{(\gamma-\delta)k}U^k\big|_1\\
&\quad+C(a)\big|\chi_a^\sharp\rho_0^{k(\beta+1-\delta)}\varrho^{(\gamma-\delta)k+1-\delta}\eta_rU^k(V,U)\big|_1\\
&\quad +C(a)\big|\chi_a^\sharp\rho_0^{k(\beta+1-\delta)}\varrho^{(\gamma-\delta)k}\eta_rU^{k-1}D_\eta U\big|_1:=\sum_{i=3}^{5} J_i.\end{aligned}    
\end{equation}

For $J_3$--$J_5$, by \eqref{eq:eta}, Lemmas \ref{lemma-bound depth}--\ref{lemma-lower bound jacobi} and \ref{lemma-refine u Lp}--\ref{lemma-v Lp ex}, and the H\"older inequality, we have
\begin{align}
J_{3}&\leq C_0|\varrho|_\infty^{(\gamma-\delta)k}|(\rho_0^\beta)_r|_\infty\big|\chi_{a}^\sharp\rho_0^{(1-\frac{1}{k}+\frac{1-\delta}{\beta})\beta}U\big|_k^k\leq C(a,T),\nonumber\\ 
J_4&\leq C_0\big|\frac{r^2}{\eta^2}\big|_\infty|\varrho|_\infty^{(\gamma-\delta)k-\delta}\big(|\chi_{a}^\sharp \rho_0^{\gamma-\delta}V|_2\big|\chi_{a}^\sharp\rho_0^\frac{k(\gamma-\delta)+\delta-\beta}{k}U\big|_{2k}^k+\big|\chi_{a}^\sharp \rho_0^\frac{k\beta+1}{k+1}U\big|_{k+1}^{k+1}\big)\leq C(a,T),\nonumber\\
J_5&\leq C(a) 
\Big|\frac{r^{3-\delta}}{\eta^{3-\delta}\sqrt{\eta_r}^{1-\delta}}\Big|_\infty|\varrho|_\infty^{(\gamma-\delta)k-1}\big|\chi_{a}^\sharp \rho_0^{k(\beta+1-\delta)+1-\frac{\delta}{2}} U^{k-1} \big|_2\big|(\eta^2\eta_r\varrho^\delta)^\frac{1}{2}D_\eta U\big|_2 \label{G4-G7}\\
&\leq C(a,T)\big|\chi_{a}^\sharp\rho_0^{\frac{2k(\beta+1-\delta)+2-\delta}{2k-2}}U \big|_{2(k-1)}^{k-1}\big|(\eta^2\eta_r\varrho^\delta)^\frac{1}{2}D_\eta U\big|_2\leq C(a,T)\big|(\eta^2\eta_r\varrho^\delta)^\frac{1}{2}D_\eta U\big|_2.\nonumber
\end{align}    
Substituting \eqref{G4-G7} into \eqref{pre-g4-g7} leads to
\begin{equation}\label{abb}
\big|\chi_{a}^\sharp \rho_0^{\beta+1-\delta}\varrho^{\gamma-\delta}U\big|_\infty^k\leq C(a,T)\big(\big|(\eta^2\eta_r\varrho^\delta)^\frac{1}{2}D_\eta U\big|_2+1\big).    
\end{equation}

\smallskip
\textbf{2.} Next, multiplying \eqref{V-solution} by $\chi_{a}^\sharp\rho_0^{\beta+1-\delta}$ and taking the $L^\infty$-norm of the resulting equality, we obtain from \eqref{abb}, Lemmas \ref{lemma-basic energy} and \ref{lemma-v Lp ex}, and the Young inequality that
\begin{equation}\label{v-infty-integral}
\begin{aligned}
|\chi_{a}^\sharp \rho_0^{\beta+1-\delta}V|_\infty&\leq |\chi_{a}^\sharp \rho_0^{\beta+1-\delta}v_0|_\infty+\int_0^t \frac{A\gamma}{2a_{1}\delta}\big|\chi_{a}^\sharp \rho_0^{\beta+1-\delta}\varrho^{\gamma-\delta}U\big|_\infty+\big|\chi_{a}^\sharp \rho_0^{\beta+1-\delta}D_\eta\Phi\big|_\infty\,\mathrm{d}s\\
&\leq C(T)+C(a,T)\int_0^t \big(\big|(\eta^2\eta_r\varrho^\delta)^\frac{1}{2}D_\eta U\big|_2^\frac{1}{k}+1\big)\,\mathrm{d}s\\
&\leq C(a,T)+C(a,T)\int_0^t \big|(\eta^2\eta_r\varrho^\delta)^\frac{1}{2}D_\eta U\big|_2^2\,\mathrm{d}s \leq C(a,T),
\end{aligned}
\end{equation}
where the initial data can be controlled by
\begin{equation*}
|\chi_{a}^\sharp \rho_0^{\beta+1-\delta}v_0|_\infty\leq C_0\big|(\rho_0^{\beta+1-\delta}u_0,(\rho_0^\beta)_r)\big|_\infty \leq C_0\big(|\rho_0^{\beta}|_\infty^\frac{\beta-\delta+1}{\beta} |u_0|_\infty+ |(\rho_0^\beta)_r|_\infty\big)\leq C_0.    
\end{equation*}
\end{proof}

\subsubsection{Boundedness of the effective velocity near the symmetric center}

\begin{Lemma}\label{new-u-lp}
For all $p\in [2,p_{*})$, there exists a constant $C(p,T)>0$ such that, for all $t\in [0,T]$,
\begin{equation}\label{gujil44}
\big|(\zeta^2\eta_r\varrho)^\frac{1}{p}U(t)\big|_p^p+\int_0^t \big|(\zeta^2\eta_r\varrho^\delta)^\frac{1}{2}|U|^\frac{p-2}{2}\mathscr{D}_\eta U\big|_2^2\,\mathrm{d}s \leq C(p,T)\big(\sup_{s\in[0,t]}|\zeta V|_\infty^2+1\big).
\end{equation}
\end{Lemma}
\begin{proof}We divide the proof into two steps.

\smallskip
\textbf{1.} Multiplying  $\eqref{eq:VFBP-La-eta}_1$ by $\zeta^2\eta_r|U|^{p-2}U$ gives
\begin{align*}
&\,\frac{1}{p}(\zeta^2\eta_r \varrho|U|^p)_t+2a_{1}(p-1)\zeta^2\eta_r\varrho^\delta J_6\\
&=A(p-1)\zeta^2\eta_r\varrho^\gamma|U|^{p-2}D_\eta U-\zeta^2\eta_r\varrho D_\eta\Phi|U|^{p-2}U-\frac{2}{p}\frac{\zeta^2\eta_r\varrho V|U|^p}{\eta}\\
&\quad -2\zeta \zeta_r\Big( 2a_{1}\delta \varrho^\delta|U|^{p-2}U D_\eta U-4a_{1}(1-\delta)  \frac{\varrho^\delta|U|^p}{\eta}-A\varrho^\gamma|U|^{p-2}U+\frac{4a_{1}}{p}\frac{\varrho^\delta|U|^p}{\eta}\Big)\\
&\quad+\Big(2a_{1}\delta \zeta^2\varrho^\delta|U|^{p-2}U D_\eta U -4a_{1}(1-\delta) \frac{\zeta^2\varrho^\delta|U|^p}{\eta}- A\zeta^2\varrho^\gamma|U|^{p-2}U+\frac{4a_{1}}{p}\frac{\zeta^2\varrho^\delta|U|^p}{\eta} \Big)_r,
\end{align*}
where  $J_6$ is a binary form:
\begin{equation*}
    J_6:=\delta X^2-2(1-\delta)XY+\frac{2}{p} Y^2 \qquad \text{ with }\ (X,Y)=|U|^{\frac{p-2}{2}}\mathscr{D}_\eta U.
\end{equation*}
A direct calculation shows that the discriminant $\mathfrak{D}=\frac{4}{p}\big((1-\delta)^2p-2\delta\big)$ of $J_{6}$ is strictly negative,
which implies that there exists a constant $c_{\delta,p}^*>0$, depending on $(\delta,p)$, such that
\begin{equation}\label{lemma6.4 I22}
    J_6\geq c^*_{\delta,p}(X^2+Y^2)=c^*_{\delta,p}|U|^{p-2}|\mathscr{D}_\eta U|^2.
\end{equation}
Then integrating the above over $I$ leads to
\begin{equation}\label{G8-G10}
\begin{aligned}
&\,\frac{1}{p}\frac{\mathrm{d}}{\mathrm{d}t}\big|(\zeta^2\eta_r \varrho)^\frac{1}{p} U\big|_p^p  +2a_{1}c_{\delta,p}^* \big|(\zeta^2\eta_r\varrho^{\delta })^\frac{1}{2}|U|^\frac{p-2}{2}\mathscr{D}_\eta U\big|_2^2\\
&\leq A(p-1) \int_0^1\zeta^2\varrho^\gamma|U|^{p-2}U_r\,\mathrm{d}r- \frac{2}{p} \int_0^1 \frac{\zeta^2\eta_r \varrho V|U|^p}{\eta}\,\mathrm{d}r-\int_0^1\zeta^2\eta_r\varrho D_\eta\Phi|U|^{p-2}U\,\mathrm{d}r\\
&\quad- 2 \int_0^1 \zeta\zeta_r \Big( 2a_{1}\delta \varrho^\delta|U|^{p-2}U D_\eta U+\frac{4a_{1} }{p} \frac{\varrho^\delta|U|^p}{\eta}\\
&\qquad \qquad \qquad \ \ -4a_{1}(1-\delta)\frac{\varrho^\delta|U|^p}{\eta}-A\varrho^\gamma|U|^{p-2}U\Big)\,\mathrm{d}r:=\sum_{i=7}^{10} J_i.
\end{aligned}
\end{equation}

\smallskip
\textbf{2. Estimates for $J_{7}$--$J_{10}$.} For $J_{7}$, set $0<\varsigma=\frac{1+\delta}{2}-\frac{2(1-\delta)}{p}<1$. It follows from Lemmas \ref{lemma-near depth} and \ref{lemma-bound depth}, and the H\"older and Young inequalities that
\begin{equation}\label{G8}
\begin{aligned}
J_{7}&
\leq C(p)|\varrho|_\infty^{\gamma-\varsigma}|\zeta\eta_r\varrho^{2\delta-1}|_1^\frac{1}{p}\big|(\zeta^2\eta_r\varrho)^\frac{1}{p} U\big|_p^\frac{p-2}{2}\big|(\zeta^2\eta_r\varrho^\delta)^\frac{1}{2}|U|^\frac{p-2}{2}D_\eta U\big|_2\\
&\leq C(p,T)+\big|(\zeta^2\eta_r\varrho)^\frac{1}{p} U\big|_{p}^{p}+\frac{a_1c_{\delta,p}^*}{4}\big|(\zeta^2\eta_r\varrho^\delta)^\frac{1}{2}|U|^\frac{p-2}{2}D_\eta U\big|_2^2.
\end{aligned}
\end{equation}

For $J_{8}$, it follows from \eqref{eq:eta}, Lemma \ref{lemma-bound depth}, and the H\"older and Young inequalities that
\begin{equation}\label{4100}
\begin{aligned}
J_{8}& \leq C(p)|\varrho^{1-\delta}|_\infty|\zeta V|_\infty \big|(\eta_r\varrho^\delta)^\frac{1}{p}U\big|_p^\frac{p}{2} \Big|\big(\frac{\zeta^2\eta_r\varrho^\delta}{\eta^2}\big)^\frac{1}{p}U\Big|_p^\frac{p}{2}\\
&\leq C(p,T)|\zeta V|_\infty^2 \big|(\eta_r\varrho^\delta)^\frac{1}{p}U\big|_p^p+\frac{a_{1}c_{\delta,p}^*}{4} \Big|\big(\frac{\zeta^2\eta_r\varrho^\delta}{\eta^2}\big)^\frac{1}{p}U\Big|_p^p.\nonumber
\end{aligned}
\end{equation}

Similarly, for $J_{9}$-$J_{10}$, we obtain from \eqref{eq:eta}, Lemmas  \ref{lemma-bound depth}--\ref{lemma-lower bound jacobi}, and \ref{lemma-v Lp ex} that
\begin{equation}\label{G10}
    \begin{aligned}
        J_{9} & \leq |\varrho|_\infty^{\frac{(p-1)(1-\delta)}{p}}|D_\eta\Phi|_{\infty}|\zeta^2\eta_r\varrho|_1^\frac{1}{p}
        |(\zeta^2\eta_r\varrho^\delta)^\frac{1}{p}U|_p^{p-1}\\
        &\leq C(T)\big|\frac{r^2\rho_0}{\eta^2}\big|_\infty+C(T)\big|(\zeta^2\eta_r\varrho^\delta)^\frac{1}{p}U\big|_p^{p}\leq C(T)(1+\big|(\eta_r\varrho^\delta)^\frac{1}{p}U\big|_p^{p}),\\[-0.5mm]
        J_{10}&\leq C(p)\int_\frac{1}{2}^\frac{5}{8} \Big(\varrho^\delta|U|^{p-1} |D_\eta U|+ \frac{\eta_r\varrho^\delta|U|^{p}}{\eta\eta_r}+\varrho^{\gamma-1}\big(\frac{r^2\rho_0}{\eta^2\eta_r}\big)|U|^{p-1}\Big)\,\mathrm{d}r\\
        &\leq C(p,T)\big(\big|(\eta_r\varrho^\delta)^\frac{1}{2}|U|^\frac{p}{2}\big|_2\big|(\eta^2\eta_r\varrho^\delta)^\frac{1}{2}|U|^{\frac{p-2}{2}}D_\eta U\big|_2\\
        &\qquad \qquad \ \ + \big|(\eta_r\varrho^\delta)^\frac{1}{p}U\big|_p^p +|\varrho|_\infty^{\gamma-1}\big|(r^2\rho_0)^\frac{1}{p}U\big|_{p}^{p-1}\big)\\
        &\leq C(p,T)+C(p,T)\big|(\eta^2\eta_r\varrho^\delta)^\frac{1}{2}|U|^{\frac{p-2}{2}}\mathscr{D}_\eta U\big|_2^2.   
        \end{aligned}
        \end{equation}
Consequently, collecting  \eqref{G8-G10}--\eqref{G10}, we have 
\begin{equation*}
\begin{aligned}
&\frac{1}{p}\frac{\mathrm{d}}{\mathrm{d}t}\big|(\zeta^2\eta_r \varrho)^\frac{1}{p} U\big|_p^p  +a_{1}c_{\delta,p}^* \big|(\zeta^2\eta_r\varrho^{\delta })^\frac{1}{2}|U|^\frac{p-2}{2}\mathscr{D}_\eta U\big|_2^2\\
&\leq C(p,T)\Big(1+\big|(\zeta^2\eta_r\varrho)^\frac{1}{p} U\big|_{p}^{p}+(|\zeta V|_\infty^2+1) \big|(\eta_r\varrho^\delta)^\frac{1}{p}U\big|_p^p+\big|(\eta^2\eta_r\varrho^\delta)^\frac{1}{2}|U|^{\frac{p-2}{2}} \sD_\eta U\big|_2^2\Big),
\end{aligned}
\end{equation*}
which, along with Lemma \ref{lp-uv} and the Gr\"onwall inequality, yields the desired estimates.
\
\end{proof}

Next, we can derive the $L^1([0,T];L^\infty)$-estimate of $\zeta\varrho^{\gamma-\delta} U$.
\begin{Lemma}\label{cru3}
For any $\varepsilon\in(0,1)$, there exists a constant $C(\varepsilon,T)>0$ such that 
\begin{equation}\label{lemma4.5}
\int_0^t |\zeta\varrho^{\gamma-\delta}U|_{\infty}\, \mathrm{d}s\leq C(\varepsilon,T) +\varepsilon\sup_{s\in[0,t]}|\zeta V|_\infty \qquad \text{for all $t\in [0,T]$}.
\end{equation}
\end{Lemma}
\begin{proof}
Let $k\in \NN^*$ be a fixed number such that 
$\max\big\{3,\frac{1}{\gamma-\delta}\big\}<k<p_{*}-1$.

First, it follows from Lemma \ref{sobolev-embedding} that
\begin{equation}\label{6.10}
\begin{aligned}
|\zeta \varrho^{\gamma-\delta}U|_\infty^{k}&\leq C_0\big(\big|\zeta^{k-1}\zeta_r\varrho^{k(\gamma-\delta)} U^{k}\big|_1+ \big|\zeta^k\varrho^{k(\gamma-\delta)+1-\delta}\eta_rU^k(V,U)\big|_1\big)\\
&\quad+ C_0\big|\zeta^k\varrho^{k(\gamma-\delta)}\eta_r U^{k-1}D_\eta U\big|_1:=\sum_{i=11}^{13}J_i.
\end{aligned}    
\end{equation}

Then, using \eqref{v-expression}, Lemmas \ref{lemma-bound depth}--\ref{lemma-lower bound jacobi}, and  the H\"older inequality, we have
\begin{align}
J_{11}&\leq C_0|\eta_r^{-1}|_\infty|\zeta|_\infty^{k-3}|\varrho|_\infty^{k(\gamma-\delta)-1}\big|(\zeta^2\varrho \eta_r)^\frac{1}{k}U\big|_k^{k}\leq C(T)\big|(\zeta^2\varrho \eta_r)^\frac{1}{k}U\big|_k^{k},\nonumber\\
J_{12}&\leq C_0|\varrho|_\infty^{k(\gamma-\delta)-\delta}|\zeta|_\infty^{k-3}\big(|\zeta V|_\infty\big|(\zeta^2\eta_r\varrho)^\frac{1}{k}U\big|_k^k+ |\zeta|_\infty\big|(\zeta^2\eta_r\varrho)^\frac{1}{k+1}U\big|_{k+1}^{k+1}\big)\nonumber\\
&\leq C(T)\big(|\zeta V|_\infty\big|(\zeta^2\eta_r\varrho)^\frac{1}{k}U\big|_k^k+ \big|(\zeta^2\eta_r\varrho)^\frac{1}{k+1}U\big|_{k+1}^{k+1}\big),\label{wuqiong-G12-G14}\\
J_{13}&\leq C_0|\varrho|_\infty^{k(\gamma-\delta)-\frac{1+\delta}{2}}|\zeta|_\infty^{k-2}\big|(\zeta^2\eta_r\varrho)^{\frac{1}{k}}U\big|_{k}^\frac{k}{2}\big|(\zeta^2\eta_r\varrho^\delta   )^\frac{1}{2}|U|^\frac{k-2}{2}D_\eta U\big|_2 \nonumber\\
&\leq C(T)\big|(\zeta^2\eta_r\varrho)^{\frac{1}{k}}U\big|_{k}^\frac{k}{2}\big|(\zeta^2\eta_r\varrho^\delta)^\frac{1}{2}|U|^\frac{k-2}{2}D_\eta U\big|_2,  \nonumber
\end{align}
which, along with  \eqref{6.10}--\eqref{wuqiong-G12-G14}, Lemma \ref{new-u-lp},  and the Young inequality, gives
\begin{equation*}
|\zeta \varrho^{\gamma-\delta}U|_\infty\leq C(T)\Big(1+\sup_{s\in[0,t]}|\zeta V|_\infty^\frac{3}{k} +\big(1+\sup_{s\in[0,t]}|\zeta V|_\infty^\frac{1}{k}\big)\big|(\zeta^2\eta_r\varrho^\delta)^\frac{1}{2}|U|^\frac{k-2}{2}D_\eta U\big|_2^\frac{1}{k}\Big).  
\end{equation*}

Finally, integrating the above over $[0,t]$, then we obtain from the fact that $k>3$, Lemma \ref{new-u-lp}, and the H\"older and Young inequalities that, for all $\varepsilon\in (0,1)$, \eqref{lemma4.5} holds.
\end{proof}

\begin{Lemma}\label{lemma-v Linfty in}
For any $a\in(0,1)$, there exists a constant $C(a,T)>0$ such that
\begin{equation*}
|\zeta_{a} V(t)|_\infty\leq C(a,T) \qquad \text{for all $t\in [0,T]$}.
\end{equation*}
\end{Lemma}
\begin{proof}
First, it follows from \eqref{V-solution} and Lemmas \ref{lemma-v Lp ex} and     \ref{cru3} that, for any $\varepsilon\in(0,1)$,
\begin{equation}\label{daiding}
\begin{aligned}
     \sup_{s\in[0,t]}|\zeta V|_\infty &\leq |\zeta v_0|_\infty +\frac{A\gamma}{2a_{1}\delta}\int_0^t \big|\zeta \varrho^{\gamma-\delta}U\big|_\infty\,\mathrm{d}s+\int_0^t |\zeta D_\eta\Phi|_{\infty}\,\mathrm{d}s\\
     &\leq C(\varepsilon,T) +\frac{A\gamma \varepsilon}{2a_{1}\delta}\sup_{s\in[0,t]}|\zeta V|_\infty.
\end{aligned}
\end{equation}
Here, the bound of $|\zeta v_0|_\infty$ follows from \eqref{v-expression} and the fact that $\rho_0^\beta\sim 1-r$, {\it i.e.},
\begin{equation}
|\zeta v_0|_\infty\leq |\zeta u_0|_\infty+\frac{2a_{1}\delta}{1-\delta} |\zeta(\rho_0^{\delta-1})_r|_\infty\leq |\zeta u_0|_\infty+C_0|\zeta\rho_0^{\delta-1-\beta}|_\infty|(\rho_0^\beta)_r|_\infty\leq C_0.
\end{equation}

Then setting $\varepsilon=\min\big\{\frac{1}{2},\frac{a_{1}\delta}{A\gamma}\big\}$ in \eqref{daiding},
we obtain that, for all $t\in [0,T]$,
\begin{equation*}
|\zeta V(t)|_\infty\leq C(T)\implies |\zeta_{a} V(t) |_\infty\leq C(T) \qquad \text{for any } a\in\big(0,\frac{1}{2}\big],
\end{equation*}
  which, along with the fact that $\rho_0^\beta\sim 1-r$ and Lemma \ref{lemma-v Linfty ex}, yields that, for $a\in (\frac{1}{2},1)$,\begin{equation*}
|\zeta_{a} V(t)|_\infty\leq |\zeta V(t)|_\infty+\|V(t)\|_{L^\infty(\frac{1}{2},\frac{1+3a}{4})}\leq C(T)+C(a) |\chi^\sharp\rho_0^{\beta+1-\delta} V(t)|_\infty\leq C(a,T).
\end{equation*}

\end{proof}

\subsection{Uniform upper bounds of \texorpdfstring{$(\eta_r,\frac{\eta}{r})$}{} in the exterior domain}\label{subsec-4.2}

We first establish the uniform upper bounds in a fixed neighborhood of the vacuum boundary.

\begin{Lemma}\label{lemma-upper jacobi-ex}
There exist constants $a_T\in(\frac{1}{2},1)$ and $C(T)>0$, depending only on the initial data, the fixed parameters, and $T$, such that
\begin{equation*}
\frac{\eta(t,r)}{r}+\eta_r(t,r)\leq C(T)\qquad\text{for $(t,r)\in[0,T]\times[a_T,1]$}.
\end{equation*}
\end{Lemma}
\begin{proof}
First, for $\frac{1}{2}\leq r<1$, \eqref{eq:eta}, Lemma \ref{lemma-lower bound jacobi}, and the H\"older inequality give
\begin{equation}\label{cc1}
\begin{aligned}
\int_r^1\eta\eta_r\varrho^\delta\,\mathrm{d}\hat r
&\leq\Bigl(\int_r^1\hat r^2\rho_0\,\mathrm{d}\hat r\Bigr)^\delta
\Bigl(\int_r^1\eta^{\frac{1-2\delta}{1-\delta}}\eta_r\,\mathrm{d}\hat r\Bigr)^{1-\delta}\\
&\leq C_0\Bigl(\int_r^1\hat r^2\rho_0\,\mathrm{d}\hat r\Bigr)^\delta\big(\eta(t,r)^\frac{2-3\delta}{1-\delta}-\eta(t,1)^\frac{2-3\delta}{1-\delta}\big)^{1-\delta}
\leq C(T)\rho_0(r)^{\delta(\beta+1)}.
\end{aligned}   
\end{equation}
Here, we have used the fact that $\delta>\frac{2}{3}$ and the lower bound of $\eta(t,r)$ in Lemma \ref{lemma-lower bound jacobi}. Moreover, recall the estimate $I_{19}$ in the proof of Lemma \ref{lemma-lower bound jacobi}:
\begin{equation}\label{cc2}
\sup_{t\in[0,T]}\Bigl|\int_r^1\hat r^2\rho_0U\,\mathrm{d}\hat r\Bigr|
\leq \sup_{t\in[0,T]}\bigl|(r^2\rho_0)^{\frac{1}{4}}U\bigr|_4
\Bigl(\int_r^1\hat r^2\rho_0\,\mathrm{d}\hat r\Bigr)^{\frac{3}{4}}
\leq C(T)\rho_0(r)^{\frac{3(\beta+1)}{4}}.
\end{equation}

Recall the identity \eqref{integ} in the proof of Lemma \ref{lemma-far depth}. Since $D_\eta\Phi\geq0$ and $\gamma>\delta$, we obtain from Lemma \ref{lemma-bound depth} that
\begin{equation*}
\begin{aligned}
&\int_r^1 \eta^2\eta_r \varrho D_\eta\Phi\,\mathrm{d}\hat r\geq 0,\\
&\Big|A\eta^2\varrho^\gamma+2A\int_r^1\eta\eta_r\varrho^\gamma\,\mathrm{d}\hat r \Big|\leq C(T)\Bigl(2a_1\eta^2\varrho^\delta+4a_1\int_r^1\eta\eta_r\varrho^\delta\,\mathrm{d}\hat r\Bigr).
\end{aligned}
\end{equation*}
Thus, we can multiply \eqref{integ} by the factor $e^{C(T)t}$ to get
\begin{equation*}
\frac{\mathrm{d}}{\mathrm{d}t}\Bigl(e^{C(T)t}\Bigl(2a_1\eta^2\varrho^\delta+4a_1\int_r^1\eta\eta_r\varrho^\delta\,\mathrm{d}\hat r\Bigr)\Bigr)
\geq e^{C(T)t}\frac{\mathrm{d}}{\mathrm{d}t}\int_r^1\hat r^2\rho_0U\,\mathrm{d}\hat r.
\end{equation*}
Integrating over $[0,t]$ and integrating the right-hand side by parts in time, we obtain
\begin{equation*}
\begin{aligned}
&2a_1(\eta^2\varrho^\delta)(t,r)+\underline{4a_1\int_r^1(\eta\eta_r\varrho^\delta)(t,\hat r)\,\mathrm{d}\hat r}_{\,\geq 0}\\
&\geq2a_1e^{-C(T)t}r^2\rho_0(r)^\delta+\int_r^1\hat r^2\rho_0U(t,\hat r)\,\mathrm{d}\hat r
-e^{-C(T)t}\int_r^1\hat r^2\rho_0u_0\,\mathrm{d}\hat r\\
&\quad-C(T)\int_0^t e^{-C(T)(t-s)}\Bigl(\int_r^1\hat r^2\rho_0U(s,\hat r)\,\mathrm{d}\hat r\Bigr)\,\mathrm{d}s.
\end{aligned}
\end{equation*}
Consequently, by \eqref{cc1}--\eqref{cc2},
\begin{equation*}
(\eta^2\varrho^\delta)(t,r)\geq\rho_0(r)^\delta
\Bigl(C(T)^{-1}-C(T)\bigl(\rho_0(r)^{\delta\beta}+\rho_0(r)^{\frac{3(\beta+1)}{4}-\delta}\bigr)\Bigr).
\end{equation*}
Since $\beta+1=\gamma >\frac{4}{3}>\delta>0$ and $\rho_0^\beta\sim 1-r$, we can therefore choose a fixed $a_T\in(\frac{1}{2},1)$, depending only on $C(T)$, such that
\begin{equation*}
(\eta^2\varrho^\delta)(t,r)\geq C(T)^{-1}\rho_0(r)^\delta
\qquad\text{for $(t,r)\in[0,T]\times[a_T,1)$}.
\end{equation*}
By \eqref{eq:eta}, this gives
\begin{equation}\label{1pre upp jacobi}
(\eta^{\frac{2(\delta-1)}{\delta}}\eta_r)(t,r)
=\frac{r^2\rho_0(r)}{(\eta^2\varrho^\delta)^{\frac{1}{\delta}}}\leq C(T)
\qquad\text{for $a_T\leq r<1$}.
\end{equation}

To bound $\eta(t,a_T)$, we use the monotonicity of $\eta$, the identity $\eta_t=U$, and Lemma \ref{lemma-basic energy}:
\begin{equation}\label{2pre upp jacobi}
\eta(t,a_T)\leq\frac{\int_{a_T}^1\hat r^2\rho_0(\hat r)\eta(t,\hat r)\,\mathrm{d}\hat r}{\int_{a_T}^1\hat r^2\rho_0(\hat r)\,\mathrm{d}\hat r}
\leq C(T)\Bigl(1+\int_0^t\bigl|(r^2\rho_0)^{\frac{1}{2}}U(s)\bigr|_2\,\mathrm{d}s\Bigr)\leq C(T).
\end{equation}
Since $3\delta-2>0$, integrating \eqref{1pre upp jacobi} over $[a_T,r]$ and using \eqref{2pre upp jacobi} yields
\begin{equation}\label{3pre upp jacobi}
\eta(t,r)^{\frac{3\delta-2}{\delta}}\leq\eta(t,a_T)^{\frac{3\delta-2}{\delta}}+C(T)(r-a_T)\leq C(T)
\qquad\text{for $a_T\leq r\leq1$}.
\end{equation}
Finally, \eqref{1pre upp jacobi} and \eqref{3pre upp jacobi} imply
\begin{equation*}
\eta_r(t,r)
\leq C(T)\eta(t,r)^{\frac{2(1-\delta)}{\delta}}\leq C(T)
\qquad\text{for $a_T\leq r<1$}.
\end{equation*}
The estimate of $\eta_r(t,1)$ follows from  \eqref{-1pre upp jacobi} and \eqref{3pre upp jacobi}. This completes the proof.
\end{proof}

\subsection{Uniform upper bounds of \texorpdfstring{$(\eta_r,\frac{\eta}{r})$}{} near the symmetric center}\label{subsec-4.3}

The following weighted estimate for the gravitational force is used in the interior velocity estimates.

\begin{Lemma}\label{Lemma etar 1/2 D phi L2}
There exists a constant $C(T)>0$ such that
\begin{equation}\label{result eta r 1/2 D phi L2}
|\sqrt{\eta_r}D_\eta\Phi|_2\leq C(T).
\end{equation}
\end{Lemma}
\begin{proof}
It follows from the  integration by parts, \eqref{D Phi expression}, and Lemmas \ref{lemma-lower bound jacobi} and \ref{lemma-v Lp ex} that
\begin{equation*}
    \begin{aligned}
        |\sqrt{\eta_r}D_\eta\Phi|_2^2&\leq C\int_0^1\frac{\eta_r}{\eta^4}M^2(r)\,\mathrm{d}r\leq C\Big|\frac{M}{\eta^2}\Big|_\infty^\frac{3}{2}+C\int_0^1\frac{Mr^2}{\eta^3}\,\mathrm{d}r\\
        &\leq C(T)+C(T)|D_\eta\Phi|_\infty\int_0^1\big|\frac{r}{\eta}\big|_\infty\,\mathrm{d}r
        \leq C(T),
    \end{aligned}
\end{equation*}
where $M(r)$ satisfies
    $M(r):=\int_0^r\hat{r}^2\rho_0\,\mathrm{d}\hat{r}\leq C_0$ and $ D_\eta\Phi=4\pi G\frac{M}{\eta^2}$. 
\end{proof}

\begin{Lemma}\label{lemma zeta rho 1-delta u L 2 L infty}
For any $a\in(0,1)$, there exists a constant $C(a,T)>0$ such that 
\begin{equation*}
    \int_0^t|\zeta_a\varrho^{1-\delta}U|_\infty^2\,\mathrm{d}s\leq C(a,T) \qquad \text{ for all }\, t\in [0,T].
\end{equation*}
\end{Lemma}
\begin{proof}
We check that, for each $\delta\in (\frac{13}{18},1)$, there exists some $q\in [3,p_*)$ such that
\begin{equation*}
    (q+2)-(q+4)\delta>0, \qquad q-(q+1)\delta>0, \qquad (q+1)-(q+2)\delta>0.
\end{equation*}
It follows from \eqref{v-expression}, Lemmas \ref{lemma-bound depth}, \ref{lemma-v Linfty in}, and \ref{sobolev-embedding}, and the H\"older and Young inequalities that
\begin{align*}
        &\,|\zeta_a\varrho^{1-\delta}U|_\infty^q\leq |(\zeta_a^q\varrho^{q(1-\delta)}|U|^q)_r|_1\\
        &\leq C_0\big(|\zeta_a^q\varrho^{(q+1)(1-\delta)}\eta_r|U|^q(V-U)|_1+|\zeta_a^q\eta_r\varrho^{q(1-\delta)}|U|^{q-2}UD_\eta U|_1+|\zeta_a^{q-1}(\zeta_a)_r\varrho^{q(1-\delta)}|U|^q|_1 \big)\\
        &\leq C_0|\zeta_aV|_\infty|\varrho|^{(q+1)-(q+2)\delta}_\infty       |(\zeta_a^2\eta_r\varrho^{\delta})^\frac{1}{q}U|_q^q +C_0|\zeta_a\varrho^{1-\delta}U|_\infty |\varrho|_\infty^{q-(q+1)\delta} |(\zeta_a^2\eta_r\varrho^{\delta})^\frac{1}{q}U|_q^q \\
        &\quad +C_0|\zeta_a\varrho^{1-\delta}U|_\infty^\frac{q-2}{2}|\varrho|^\frac{(q+2)-(q+4)\delta}{2}_\infty |(\zeta_a^2\eta_r\varrho^{\delta})^\frac{1}{q}U|_q^\frac{q}{2}  |(\zeta_a^2\eta_r\varrho^\delta)^\frac{1}{2}D_\eta U|_2+|\varrho|_\infty^{q(1-\delta)}|(r^2\rho_0)^\frac{1}{q}U|_q^q\\
        &\leq C(a,T)\big(1+|(\zeta_a^2\eta_r\varrho^{\delta})^\frac{1}{q}U|_q^q+|(\zeta_a^2\eta_r\varrho^{\delta})^\frac{1}{q}U|_q^\frac{q^2}{q-1}+|(\zeta_a^2\eta_r\varrho^\delta)^\frac{1}{2}D_\eta U|_2^\frac{2q}{3}\big)+\frac{1}{20}|\zeta_a\varrho^{1-\delta}U|_\infty^q,
\end{align*}
which,   along with Lemmas  \ref{lp-uv} and \ref{new-u-lp}, and the H\"older and Young inequalities, yields 
    \begin{align*}
        &\,\int_0^t|\zeta_a\varrho^{1-\delta}U|_\infty^2\,\mathrm{d}s\\
        &\leq C(a,T)\Big(1+\int_0^t |(\zeta_a^2\eta_r\varrho^{\delta})^\frac{1}{q}U|_q^q\,\mathrm{d}s+\int_0^t |(\zeta_a^2\eta_r\varrho^\delta)^\frac{1}{2}D_\eta U|_2^2\,\mathrm{d}s\Big)\leq C(a,T).
    \end{align*} 
\end{proof}

\begin{Lemma}\label{lemma7.6 fuzhu}
Let $\alpha=\frac{5\delta-3}{2}$. For any $a\in(0,1)$, there exists a constant $C(a,T)$ such that, for all $t\in [0,T]$ and $\ell=2,\cdots\!,5$,
\begin{equation}\label{lemma416guji}
|(\zeta_a^2\eta_r\varrho^{1-\alpha})^\frac{1}{\ell}U|_\ell^\ell+ \int_0^t\big|\zeta_a\sqrt{\eta_r}\varrho^\frac{\delta-\alpha}{2}|U|^\frac{\ell-2}{2}\mathscr{D}_\eta U\big|_2^2\,\mathrm{d}s\leq  C(a,\ell,T).
 \end{equation}
\end{Lemma}
\begin{proof}
Let $\alpha=\frac{5\delta-3}{2}$.  
We divide the proof into three steps.

\textbf{1. $L^2$-estimate for $\zeta_a\sqrt{\eta_r}\varrho^{\frac{1-\alpha}{2}}U$.} First, multiplying $\eqref{eq:VFBP-La-eta}_1$ by $2\zeta_{a}^2\eta_r\varrho^{-\alpha} U$ with $a\in(0,1)$, along with \eqref{v-expression}, gives
\begin{align*}
&\,\big(\zeta_a^2\eta_r\varrho^{1-\alpha}|U|^2\big)_t+4a_{1}\delta\zeta_a^2\eta_r\varrho^{\delta-\alpha}\Big(|D_\eta U|^2-2\frac{1-\delta}{\delta-\alpha} D_\eta U\frac{U}{\eta}+\frac{1-\alpha}{\delta-\alpha}\frac{U^2}{\eta^2}\Big)\\
&=\Big(4a_{1}\delta\zeta_a^2\varrho^{\delta-\alpha}UD_\eta U+4a_{1}\delta\frac{2\delta-\alpha-1}{\delta-\alpha}\frac{\zeta_a^2\varrho^{\delta-\alpha} U^2}{\eta}+\frac{-\alpha}{3}\zeta_a^2\varrho^{1-\alpha}U^3 \Big)_r\\
&\quad -2\zeta_a(\zeta_a)_r\Big(4a_{1}\delta\varrho^{\delta-\alpha}UD_\eta U+4a_{1}\delta\frac{2\delta-\alpha-1}{\delta-\alpha}\frac{\varrho^{\delta-\alpha}U^2}{\eta}+\frac{-\alpha}{3}\varrho^{1-\alpha}U^3 \Big)\\
&\quad -\frac{A\gamma}{a_{1}\delta}\zeta_a^2\eta_r\varrho^{\gamma-\delta+1-\alpha}(V-U)U-2\zeta_a^2\eta_r\varrho^{1-\alpha}D_\eta\Phi U\\
&\quad +2\zeta_a^2\eta_r\varrho^{1-\alpha}VU(\alpha D_\eta U-(1-\alpha)\frac{U}{\eta})+\frac{\alpha(1-\alpha)\zeta_a^2\varrho^{2-\alpha-\delta}\eta_r(V-U)U^3}{6a_{1}\delta}.
\end{align*}
Then, integrating the above over $I$, we obtain
\begin{align}
&\,\frac{\mathrm{d}}{\mathrm{d}t}\big|\zeta_a\sqrt{\eta_r}\varrho^\frac{1-\alpha}{2}U\big|_2^2+ 4a_{1}\delta\int_0^1 \zeta_a^2\eta_r\varrho^{\delta-\alpha} \cdot J_{19}\,\mathrm{d}r \nonumber\\
&=-\frac{A\gamma}{a_{1}\delta}\int_0^1\zeta_a^2\eta_r\varrho^{\gamma-\delta+1-\alpha}(V-U)U\,\mathrm{d}r-\int_0^12\zeta_a^2\eta_r\varrho^{1-\alpha}D_\eta\Phi U\,\mathrm{d}r \nonumber\\
&\quad-2\int_0^1 \zeta_a(\zeta_a)_r\Big(4a_{1}\delta\varrho^{\delta-\alpha}UD_\eta U+4a_{1}\delta\frac{2\delta-\alpha-1}{\delta-\alpha}\frac{\varrho^{\delta-\alpha}U^2}{\eta}+\frac{-\alpha}{3}\varrho^{1-\alpha}U^3\Big)\,\mathrm{d}r \label{lemma7.6 fuzhu eq}\\
&\quad +\int_0^1 2\zeta_a^2\eta_r\varrho^{1-\alpha}VU(\alpha D_\eta U-(1-\alpha)\frac{U}{\eta})\,\mathrm{d}r   \nonumber\\
&\quad+\frac{\alpha(1-\alpha)}{6a_{1}\delta}\int_0^1\zeta_a^2\eta_r\varrho^{2-\delta-\alpha}(V-U)U^3\,\mathrm{d}r:=\sum_{i=20}^{24}J_i.\nonumber
\end{align}
Here, $J_{19}$ is a binary form:
\begin{equation*}
J_{19}=X^2-2\frac{1-\delta}{\delta-\alpha} XY+\frac{1-\alpha}{\delta-\alpha}Y^2   \qquad \text{with } (X,Y)=\mathscr{D}_\eta U.
\end{equation*}
A direct calculation shows that the discriminant $\mathfrak{D}=\frac{4(1-\delta)^2}{(\delta-\alpha)^2}-\frac{4(1-\alpha)}{\delta-\alpha}$ of $J_{19}$ is strictly negative,
which implies that there exists a constant $c^{*}_{\delta,\alpha}>0$, depending only on $\delta$ and $\alpha$, such that
\begin{equation}
J_{19}\geq  c^{*}_{\delta,\alpha}(X^2+Y^2)= c^{*}_{\delta,\alpha}|\mathscr{D}_\eta U\big|^2.
\end{equation}

We continue to estimate $J_{20}$--$J_{24}$. It follows from Lemmas \ref{lp-uv}, \ref{lemma-bound depth}--\ref{lemma-lower bound jacobi}, \ref{lemma-v Linfty in}, \ref{Lemma etar 1/2 D phi L2}  and the H\"older and Young inequalities that
\begin{align}
J_{20}&\leq C_0|\varrho|_\infty^{\gamma-\delta}|\zeta_a\sqrt{\eta_r}\varrho^\frac{1-\alpha}{2}U|_2^2+C_0\big|\zeta_a\sqrt{\eta_r}\varrho^{\gamma-\delta+\frac{1-\alpha}{2}} V\big|_2|\zeta_a\sqrt{\eta_r}\varrho^\frac{1-\alpha}{2}U|_{2}\nonumber\\
&\leq C(T)\big(|\zeta_a\sqrt{\eta_r}\varrho^\frac{1-\alpha}{2}U|_2^2+\big|\zeta_a\sqrt{\eta_r}\varrho^{\gamma-\delta+\frac{1-\alpha}{2}} V\big|_2^2\big),\nonumber\\
J_{21}&\leq C_0|\varrho|^{\frac{1-\alpha}{2}}_\infty|\zeta_a\sqrt{\eta_r}\varrho^{\frac{1-\alpha}{2}}U|_2|\sqrt{\eta_r}D_\eta\Phi|_2\leq C(T)\big(|\zeta_a\sqrt{\eta_r}\varrho^{\frac{1-\alpha}{2}}U|_2^2+1\big), \nonumber\\
J_{22}&\leq C_0|\varrho|_\infty^\frac{\delta-\alpha}{2}|(r^2\rho_0)^\frac{1}{2}U|_2|\zeta_a\sqrt{\eta_r}\varrho^{\frac{\delta-\alpha}{2}}\mathscr{D}_\eta U|_2+C_0|\varrho|_\infty^{1-\alpha}|(r^2\rho_0)^{\frac{1}{3}}U|_3^3  \nonumber\\
&\leq a_{1}\delta c^{*}_{\delta,\alpha}|\zeta_a\sqrt{\eta_r}\varrho^{\frac{\delta-\alpha}{2}}\mathscr{D}_\eta U|_2^2+C(T),\label{lemma7.6 fuzhu r1}\\
J_{23}&\leq C_0|\varrho|_\infty^\frac{1-\delta}{2}|\zeta_{\frac{1+3a}{4}} V|_\infty|\zeta_a\sqrt{\eta_r}\varrho^\frac{1-\alpha}{2}U|_2|\zeta_a\sqrt{\eta_r}\varrho^\frac{\delta-\alpha}{2}\mathscr{D}_\eta U|_2\nonumber\\
&\leq C(a,T)|\zeta_a\sqrt{\eta_r}\varrho^\frac{1-\alpha}{2}U|_2^2 +a_{1}  \delta   c^{*}_{\delta,\alpha}|\zeta_a\sqrt{\eta_r}\varrho^\frac{\delta-\alpha}{2}\mathscr{D}_\eta U|_2^2,\nonumber\\
J_{24}&\leq C_0\int_0^1\zeta_a^2\eta_r\varrho^{2-\delta-\alpha}VU^3\,\mathrm{d}r\leq C_0|\zeta_{\frac{1+3a}{4}} V|_\infty|\zeta_{\frac{1+3a}{4}}\varrho^{1-\delta}U|_\infty |\zeta_a\sqrt{\eta_r}\varrho^\frac{1-\alpha}{2}U|_2^2\nonumber\\
&\leq C(a,T)(1+|\zeta_{\frac{1+3a}{4}}\varrho^{1-\delta}U|_\infty^2)|\zeta_a\sqrt{\eta_r}\varrho^\frac{1-\alpha}{2}U|_2^2,\nonumber
\end{align}
where in $J_{24}$, we have used the facts that $0<\alpha<1$ and $\eta_r>0$.

Therefore, collecting \eqref{lemma7.6 fuzhu eq}--\eqref{lemma7.6 fuzhu r1} yields
\begin{equation}\label{lemma7.6 fuzhu eq1}
\begin{aligned}
&\,\frac{\mathrm{d}}{\mathrm{d}t}|\zeta_a\sqrt{\eta_r}\varrho^\frac{1-\alpha}{2}U|_2^2+2a_{1}\delta c^{*}_{\delta,\alpha}\big|\zeta_a\sqrt{\eta_r}\varrho^\frac{\delta-\alpha}{2}\mathscr{D}_\eta U\big|_2^2\\
&\leq C(a,T)(1+|\zeta_{\frac{1+3a}{4}}\varrho^{1-\delta}U|_\infty^2) |\zeta_a\sqrt{\eta_r}\varrho^\frac{1-\alpha}{2}U|_2^2+C(T)(\big|\zeta_a\sqrt{\eta_r}\varrho^{\gamma-\delta+\frac{1-\alpha}{2}}V\big|_{2}^{2}+1).
\end{aligned}
\end{equation}

\textbf{2. Estimate for $V$.} Multiplying \eqref{eq:v} by $\zeta_a^2\eta_r\varrho^{2(\gamma-\delta)+(1-\alpha)}V$, along with  $\eqref{eq:VFBP-La}_1$, yields
\begin{align*}
&\,\frac{1}{2}\big(\zeta_a^2\eta_r\varrho^{2(\gamma-\delta)+(1-\alpha)}V^{2}\big)_t+\frac{A\gamma}{2a_{1}\delta}\zeta_a^2\eta_r\varrho^{3(\gamma-\delta)+(1-\alpha)} V^2\\
&=\big(\frac{\alpha}{2}-(\gamma-\delta)\big)\zeta_a^2\eta_r\varrho^{2(\gamma-\delta)+(1-\alpha)} V^2 D_\eta U-\big(2(\gamma-\delta)+(1-\alpha)\big)\zeta_a^2\eta_r\varrho^{2(\gamma-\delta)+(1-\alpha)} V^2\frac{U}{\eta}\\
&\quad +\frac{A\gamma}{2a_{1}\delta}\zeta_a^2\eta_r\varrho^{3(\gamma-\delta)+(1-\alpha)}  VU-\zeta_a^2\eta_r\varrho^{2(\gamma-\delta)+(1-\alpha)}D_\eta\Phi V  .
\end{align*}
Integrating the above over $I$, we obtain from Lemmas \ref{lemma-bound depth}, \ref{Lemma etar 1/2 D phi L2}, and \ref{lemma-v Linfty in} that
\begin{align}
&\,\frac{1}{2}\frac{\mathrm{d}}{\mathrm{d}t}\big|\zeta_a\sqrt{\eta_r}\varrho^{\gamma-\delta+\frac{1}{2}(1-\alpha)}V\big|_2^2+ \frac{A\gamma}{2a_{1}\delta} \big|\zeta_a\sqrt{\eta_r}\varrho^{\frac{3}{2}(\gamma-\delta)+\frac{1}{2}(1-\alpha)}V\big|_2^2\nonumber\\
&\leq  C_0\big|\zeta_a\sqrt{\eta_r}\varrho^{\gamma-\delta+\frac{1}{2}(1-\alpha)}V\big|_2|\varrho|_\infty^\frac{2\gamma+1-3\delta}{2}|\zeta_{\frac{1+3a}{4}}V|_\infty\big|\zeta_a\sqrt{\eta_r}\varrho^\frac{\delta-\alpha}{2}\mathscr{D}_\eta U\big|_2 \label{lemma7.6 fuzhu eq2}\\
&\quad +C_0|\zeta_a\sqrt{\eta_r}\varrho^{\gamma-\delta+\frac{1}{2}(1-\alpha)}V|_2\big(|\varrho|_\infty^{(\gamma-\delta)+\frac{(1-\alpha)}{2}}|\sqrt{\eta_r}D_\eta\Phi|_2+ |\varrho|_\infty^{2\gamma-2\delta}|\zeta_a\sqrt{\eta_r}\varrho^\frac{1-\alpha}{2}U|_2\big) \nonumber\\
&\leq  C(a,T)\big(\big|\zeta_a\sqrt{\eta_r}\varrho^{\gamma-\delta+\frac{1}{2}(1-\alpha)}V\big|_2^2+|\zeta_a\sqrt{\eta_r}\varrho^\frac{1-\alpha}{2}U|_2^2\big) +a_{1}\delta c^{*}_{\delta,\alpha}\big|\zeta_a\sqrt{\eta_r}\varrho^\frac{\delta-\alpha}{2}\mathscr{D}_\eta U\big|_2^2. \nonumber
\end{align}

Combining \eqref{lemma7.6 fuzhu eq2} with \eqref{lemma7.6 fuzhu eq1} gives
\begin{align*}
&\,\frac{\mathrm{d}}{\mathrm{d}t}\big(|\zeta_a\sqrt{\eta_r}\varrho^\frac{1-\alpha}{2}U|_2^2+\frac{1}{2}\big|\zeta_a\sqrt{\eta_r}\varrho^{\gamma-\delta+\frac{1}{2}(1-\alpha)}V\big|_2^2\big)+a_{1}\delta c^{*}_{\delta,\alpha}\big|\zeta_a\sqrt{\eta_r}\varrho^\frac{\delta-\alpha}{2}\mathscr{D}_\eta U\big|_2^2\\
&\leq C(a,T)\big(1+(1+|\zeta_{\frac{1+3a}{4}}\varrho^{1-\delta}U|_\infty^2) |\zeta_a\sqrt{\eta_r}\varrho^\frac{1-\alpha}{2}U|_2^2+\big|\zeta_a\sqrt{\eta_r}\varrho^{\gamma-\delta+\frac{1-\alpha}{2}}V\big|_{2}^{2}\big),
\end{align*}
which, along with the Gr\"onwall inequality and Lemma \ref{lemma zeta rho 1-delta u L 2 L infty}, yields that, for all $t\in [0,T]$,
\begin{equation}\label{lamme7.6 fuzhu r2}
|\zeta_a\sqrt{\eta_r}\varrho^\frac{1-\alpha}{2}U|_2^2+\big|\zeta_a\sqrt{\eta_r}\varrho^{\gamma-\delta+\frac{1}{2}(1-\alpha)}V\big|_2^2+ \int_0^t\big|\zeta_a\sqrt{\eta_r}\varrho^\frac{\delta-\alpha}{2}\mathscr{D}_\eta U\big|_2^2\,\mathrm{d}s\leq  C(a,T).
\end{equation}

\textbf{3. $L^5$-estimates for $(\zeta_a^2\eta_r\varrho^{1-\alpha})^\frac{1}{5}U$.} Multiplying $\eqref{eq:VFBP-La-eta}_1$ by $\ell\zeta_a^2\eta_r\varrho^{-\alpha}|U|^{\ell-2}U$ with $\ell=3,4,5$, together with \eqref{v-expression}, gives
\begin{align*}
&\,\big(\zeta_a^2\eta_r\varrho^{1-\alpha}|U|^\ell\big)_t+2\ell(\ell-1)a_{1}\delta\zeta_a^2\eta_r\varrho^{\delta-\alpha}\underline{|U|^{\ell-2}\Big(|D_\eta U|^2-2\frac{1-\delta}{\delta-\alpha}D_\eta U\frac{U}{\eta}+\frac{2}{\ell}\frac{1-\alpha}{\delta-\alpha}\frac{U^2}{\eta^2}\Big)}_{:=J_{25}}\\
&=\Big(2a_{1}\delta\ell\zeta_a^2\varrho^{\delta-\alpha}|U|^{\ell-2}UD_\eta U+4a_{1}\delta\frac{\ell\delta-\ell+1-\alpha}{\delta-\alpha}\zeta_a^2\varrho^{\delta-\alpha}\frac{|U|^{\ell}}{\eta}-\frac{(\ell-1)\alpha}{\ell+1}\zeta_a^2\varrho^{1-\alpha}U|U|^\ell\Big)_r\\
&\quad -2\zeta_a(\zeta_a)_r \Big(2a_{1}\delta\ell\varrho^{\delta-\alpha}|U|^{\ell-2}UD_\eta U+4a_{1}\delta\frac{\ell\delta-\ell+1-\alpha}{\delta-\alpha}\varrho^{\delta-\alpha}\frac{|U|^{\ell}}{\eta}-\frac{(\ell-1)\alpha}{\ell+1}\varrho^{1-\alpha}U|U|^\ell\Big)      \\
&\quad -\frac{A\gamma\ell}{2a_{1}\delta}\zeta_a^2\eta_r\varrho^{\gamma+1-\delta-\alpha}(V-U)|U|^{\ell-2}U-\ell\zeta_a^2\eta_r\varrho^{1-\alpha}D_\eta\Phi|U|^{\ell-2}U\\
&\quad +\zeta_a^2\eta_r\varrho^{1-\alpha}VU|U|^{\ell-2}\big(\alpha\ell D_\eta U+2(\alpha-1)\frac{U}{\eta}\big) +\frac{(\ell-1)\alpha(1-\alpha)}{2(\ell+1)a_{1}\delta}\zeta_a^2\eta_r\varrho^{2-\delta-\alpha}(V-U)U|U|^{\ell}.
\end{align*}
Here, $J_{25}$ is a binary form:
\begin{equation*}
J_{25}=X^2-2\frac{1-\delta}{\delta-\alpha}XY+\frac{2}{\ell}\frac{1-\alpha}{\delta-\alpha}Y^2   \qquad \text{with } (X,Y)=|U|^\frac{\ell-2}{2}\mathscr{D}_\eta U.
\end{equation*}
A  direct calculation shows that   the discriminant $\mathfrak{D}=\frac{4(1-\delta)^2}{(\delta-\alpha)^2}-\frac{8}{\ell}\frac{1-\alpha}{\delta-\alpha}$ of $J_{25}$ is strictly negative, 
which implies that there exists a constant $c^{*}_{\delta,\ell}>0$, depending only on $(\delta,\ell)$, such that
\begin{equation}
J_{25}\geq  c^{*}_{\delta,\ell}(X^2+Y^2)= c^{*}_{\delta,\ell}|U|^{\ell-2}|\mathscr{D}_\eta U\big|^2.
\end{equation}

Then, integrating the above over $I$, we obtain
\begin{align}
&\,\frac{\mathrm{d}}{\mathrm{d}t}|(\zeta_a^2\eta_r\varrho^{1-\alpha})^\frac{1}{\ell}U|_\ell^\ell+2\ell(\ell-1)a_{1}\delta\int_0^1\zeta_a^2\eta_r\varrho^{\delta-\alpha}J_{25}\,\mathrm{d}r\nonumber\\
&=-\frac{A\gamma\ell}{2a_{1}\delta}\int_0^1\zeta_a^2\eta_r\varrho^{\gamma-\delta+1-\alpha}(V-U)|U|^{\ell-2}U\,\mathrm{d}r-\ell\int_0^1 \zeta_a^2\eta_r\varrho^{1-\alpha}D_\eta\Phi|U|^{\ell-2}U\,\mathrm{d}r \label{6.3}\\
&\quad -2\int_0^1\zeta_a(\zeta_a)_r\Big(2a_{1}\delta\ell\varrho^{\delta-\alpha}|U|^{\ell-2}UD_\eta U+4a_{1}\delta\frac{\ell\delta-\ell+1-\alpha}{\delta-\alpha}\varrho^{\delta-\alpha}\frac{|U|^{\ell}}{\eta}-\frac{(\ell-1)\alpha}{\ell+1}\varrho^{1-\alpha}U|U|^\ell\Big)\,\mathrm{d}r  \nonumber\\
&\quad +\int_0^1\zeta_a^2\eta_r\varrho^{1-\alpha}VU|U|^{\ell-2}\big(\alpha\ell D_\eta U+2(\alpha-1)\frac{U}{\eta}\big)\,\mathrm{d}r\nonumber\\
&\quad+\frac{(\ell-1)\alpha(1-\alpha)}{2(\ell+1)a_{1}\delta}\int_0^1\zeta_a^2\eta_r\varrho^{2-\delta-\alpha}(V-U)U|U|^{\ell}\,\mathrm{d}r:=\sum_{i=26}^{30}J_i.\nonumber
\end{align}
For $J_{26}$--$J_{30}$, it follows from the fact that $0\leq \alpha<1$, Lemmas \ref{lp-uv}, \ref{lemma-bound depth},  \ref{lemma-v Lp ex}, and  \ref{lemma-v Linfty in}, and  the H\"older and Young inequalities that
\begin{align}
J_{26}&\leq C_0|\varrho|_\infty^{\gamma-\delta}|(\zeta_a^2\eta_r\varrho^{1-\alpha})^\frac{1}{\ell}U|_\ell^\ell+C_0|\varrho|_\infty^{\gamma-\delta}|\zeta_{\frac{1+3a}{4}}V|_\infty\big|(\zeta_a^2\eta_r\varrho^{1-\alpha})^\frac{1}{\ell-1}U\big|_{\ell-1}^{\ell-1}\nonumber\\
&\leq C(a,T)\big(|(\zeta_a^2\eta_r\varrho^{1-\alpha})^\frac{1}{\ell}U|_\ell^\ell+|(\zeta_a^2\eta_r\varrho^{1-\alpha})^\frac{1}{\ell-1}U|_{\ell-1}^{\ell-1}\big),\nonumber\\
J_{27}&\leq C_0|D_\eta\Phi|_\infty|(\zeta_a^2\eta_r\varrho^{1-\alpha})^\frac{1}{\ell-1}U|_{\ell-1}^{\ell-1}\leq C(T) |(\zeta_a^2\eta_r\varrho^{1-\alpha})^\frac{1}{\ell-1}U|_{\ell-1}^{\ell-1},\nonumber\\
J_{28}&\leq C(\ell)|\zeta_a\sqrt{\eta_r}\varrho^{\frac{\delta-\alpha}{2}}|U|^\frac{\ell-2}{2}\mathscr{D}_\eta U|_2|(r^2\rho_0)^\frac{1}{\ell}U|_\ell^\frac{\ell}{2}+C_0|(r^2\rho_0)^\frac{1}{\ell+1}U|_{\ell+1}^{\ell+1} \nonumber \\
&\leq C(\ell,T)+\frac{a_{1} \delta c^{*}_{\delta,\ell}}{8}\big|\zeta_a\sqrt{\eta_r}\varrho^\frac{\delta-\alpha}{2}|U|^\frac{\ell-2}{2}\mathscr{D}_\eta U\big|_2^2,\label{6.8}\\
J_{29}&\leq C(\ell)|\varrho|_\infty^\frac{1-\delta}{2}|\zeta_{\frac{1+3a}{4}} V|_\infty|(\zeta_a^2\eta_r\varrho^{1-\alpha})^\frac{1}{\ell}U|_\ell^\frac{\ell}{2}\big|\zeta_a\sqrt{\eta_r}\varrho^\frac{\delta-\alpha}{2}|U|^\frac{\ell-2}{2}\mathscr{D}_\eta U\big|_2\nonumber\\
&\leq C(a,\ell,T)|(\zeta_a^2\eta_r\varrho^{1-\alpha})^\frac{1}{\ell}U|_\ell^\ell+\frac{a_{1} \delta c^{*}_{\delta,\ell}}{8}\big|\zeta_a\sqrt{\eta_r}\varrho^\frac{\delta-\alpha}{2}|U|^\frac{\ell-2}{2}\mathscr{D}_\eta U\big|_2^2,\nonumber\\
J_{30}&\leq C(\ell)|\zeta_a^2\eta_r\varrho^{2-\delta-\alpha} V|U|^{\ell+1}|_1\leq C(\ell)|\zeta_{\frac{1+3a}{4}} V|_\infty|\zeta_{\frac{1+3a}{4}}\varrho^{1-\delta}U|_\infty |(\zeta_a^2\eta_r\varrho^{1-\alpha})^\frac{1}{\ell}U|_\ell^\ell \nonumber\\
&\leq C(a,\ell,T)(1+|\zeta_{\frac{1+3a}{4}}\varrho^{1-\delta}U|_\infty^2) |(\zeta_a^2\eta_r\varrho^{1-\alpha})^\frac{1}{\ell}U|_\ell^\ell.\nonumber
\end{align}

Therefore, collecting \eqref{6.3}--\eqref{6.8} yields
\begin{equation*}
\begin{aligned}
&\,\frac{\mathrm{d}}{\mathrm{d}t}|(\zeta_a^2\eta_r\varrho^{1-\alpha})^\frac{1}{\ell}U|_\ell^\ell+a_{1}\delta c^{*}_{\delta,\ell}\big|\zeta_a\sqrt{\eta_r}\varrho^\frac{\delta-\alpha}{2}|U|^\frac{\ell-2}{2}\mathscr{D}_\eta U\big|_2^2\\
&\leq C(a,\ell,T)\big(1+(1+|\zeta_{\frac{1+3a}{4}}\varrho^{1-\delta}U|_\infty^2) |(\zeta_a^2\eta_r\varrho^{1-\alpha})^\frac{1}{\ell}U|_\ell^\ell+|(\zeta_a^2\eta_r\varrho^{1-\alpha})^\frac{1}{\ell-1}U|_{\ell-1}^{\ell-1}\big).
\end{aligned}
\end{equation*}
Applying the Gr\"onwall inequality, we can iteratively obtain from \eqref{lamme7.6 fuzhu r2} and Lemma \ref{lemma zeta rho 1-delta u L 2 L infty} that, for $\ell=2,3,4,5$ and any $t\in [0,T]$, \eqref{lemma416guji} holds.
\end{proof}

Based on Lemma \ref{lemma7.6 fuzhu}, we have the following $L^2([0,T];L^\infty(I))$-estimate for $\zeta_a\varrho^\frac{1-\delta}{2}U$.

\begin{Lemma}\label{lemma zeta rho 1-delta/2 u L 2 L infty}
For any $a\in(0,1)$, there exists a constant $C(a,T)>0$ such that
\begin{equation*}
    \int_0^t |\zeta_a\varrho^\frac{1-\delta}{2}U|_\infty^2\,\mathrm{d}s\leq C(a,T) \qquad \text{for all }\, t\in[0,T].
\end{equation*}
\end{Lemma}
\begin{proof}
First, let $\alpha=\frac{5\delta-3}{2}$. 
Next, it follows from the fact that $2\alpha-5\delta+3=0$, \eqref{v-expression}, Lemmas \ref{lp-uv}, \ref{lemma-bound depth}, \ref{lemma-v Linfty in}, \ref{lemma7.6 fuzhu}, and \ref{sobolev-embedding}, and the H\"older and Young inequalities that
\begin{align*}
&\,|\zeta_a\varrho^\frac{1-\delta}{2}U|_\infty^5\leq \big|(\zeta_a^5\varrho^\frac{5(1-\delta)}{2}|U|^5)_r \big|_1\\
&\leq C_0\big(|(\zeta_a^5)_r\varrho^\frac{5(1-\delta)}{2}U^5|_1+|\zeta_a^5\varrho^{\frac{7(1-\delta)}{2}}\eta_r(V-U)U^5|_1+|\zeta_a^5\varrho^\frac{5(1-\delta)}{2}\eta_r |U|^{3}U D_\eta U|_1  \big)\\
&\leq C_0|\varrho|^{\frac{5(1-\delta)}{2}}_\infty|(r^2\rho_0)^\frac{1}{5}U|_5^5+C_0(|\varrho|_\infty^{1-\delta}|\zeta_aV|_\infty+|\zeta_a\varrho^{1-\delta}U|_\infty)|(\zeta_a^2\eta_r\varrho^{1-\alpha})^{\frac{1}{5}}U|_5^5\\
&\quad+C_0|\varrho|_\infty^\frac{1-\delta}{2}|(\zeta_a^2\eta_r\varrho^{1-\alpha})^\frac{1}{5}U|_5^\frac{5}{2}|\zeta_a\sqrt{\eta_r}\varrho^\frac{\delta-\alpha}{2}|U|^{\frac{3}{2}}D_\eta U|_2\\
&\leq C(a,T)\big(1+|\zeta_a\sqrt{\eta_r}\varrho^\frac{\delta-\alpha}{2}|U|^{\frac{3}{2}}D_\eta U|_2+|\zeta_a\varrho^{1-\delta}U|_\infty\big).
\end{align*}
Consequently, we obtain from the Young inequality that
\begin{equation*}
    |\zeta_a\varrho^\frac{1-\delta}{2}U|_\infty^2\leq C(a,T)\big(1+|\zeta_a\sqrt{\eta_r}\varrho^\frac{\delta-\alpha}{2}|U|^{\frac{3}{2}}D_\eta U|_2^2+|\zeta_a\varrho^{1-\delta}U|_\infty^2\big).
\end{equation*}
Integrating the above over $[0,t]$, along with Lemmas \ref{lemma zeta rho 1-delta u L 2 L infty} and \ref{lemma7.6 fuzhu}, completes the proof.
\end{proof}

\begin{Lemma}\label{lemma-u infty}
For any $a\in(0,1)$, there exists a constant $C(a,T)$ such that, for all $t\in [0,T]$,
\begin{equation*}
|\zeta_{a}\sqrt{\eta_r}U(t)|_2^2+\int_0^t\big(\big|\zeta_{a}\sqrt{\eta_r}\varrho^\frac{\delta-1}{2}\mathscr{D}_\eta U\big|_2^2+|\zeta_{a} U|_\infty^2\big)\,\mathrm{d}s\leq C(a,T).
 \end{equation*}
\end{Lemma}
\begin{proof}
We divide the proof into four steps.

\textbf{1.} First, multiplying $\eqref{eq:VFBP-La-eta}_1$ by $2\zeta_{a}^2\eta_r\varrho^{-1} U$ with $a\in(0,1)$ and integrating over $I$ give
\begin{align}
&\,\frac{\mathrm{d}}{\dt}|\zeta_{a}\sqrt{\eta_r} U|_2^2 + 4a_{1}\delta\big|\zeta_{a} \sqrt{\eta_r}\varrho^\frac{\delta-1}{2} \mathscr{D}_\eta U\big|_2^2 \nonumber\\
&\leq -\frac{A\gamma}{a_{1}\delta}\int_0^1\zeta_{a}^2\eta_r\varrho^{\gamma-\delta}(V-U)U\,\mathrm{d}r-2\int_0^1\zeta_a^2\eta_rD_\eta \Phi U\,\mathrm{d}r \label{dt-g17-g19}  \\
&\quad+2\int_0^1\zeta_{a}^2\eta_rVU D_\eta U\,\mathrm{d}r-2(1-\delta)\int_0^1\zeta_{a}^2\eta_r(V-U) \frac{U^2}{\eta}\,\mathrm{d}r\nonumber\\
&\quad- 2\int_0^1\zeta_{a}(\zeta_{a})_r\Big(4a_{1}\delta\varrho^{\delta-1} UD_\eta U+4a_{1}\delta \frac{\varrho^{\delta-1}U^2}{\eta}-\frac{1}{3}U^3\Big)\,\mathrm{d}r:=\sum_{i=31}^{35}J_i. \nonumber
\end{align}

\textbf{2.} Now, we estimate $J_{31}$--$J_{35}$. It follows from \eqref{eq:eta}, Lemmas \ref{lemma-far depth}, \ref{lp-uv}, \ref{lemma-bound depth}, \ref{lemma-lower bound jacobi}, \ref{lemma-v Linfty in}, and \ref{Lemma etar 1/2 D phi L2}, the fact that $\support(\zeta_a)_r\subset[a,\frac{1+3a}{4}]$, and the H\"older and Young inequalities that
\begin{align}
J_{31}&\leq C_0|\zeta_{a}\sqrt{\eta_r} U|_2\big(\big|\zeta_{a}\sqrt{\eta_r}\varrho^{\gamma-\delta} V\big|_2+|\varrho|_\infty^{\gamma-\delta} |\zeta_{a}\sqrt{\eta_r} U|_2\big)\leq C(T)\big|\zeta_{a}\sqrt{\eta_r} (\varrho^{\gamma-\delta} V,U)\big|_2^2,\nonumber\\
J_{32}&\leq C_0|\zeta_a\sqrt{\eta_r}U|_2|\sqrt{\eta_r}D_\eta\Phi|_2\leq C(T)+ |\zeta_a\sqrt{\eta_r}U|_2^2, \nonumber\\
J_{33}&\leq C_0|\varrho|_\infty^{\frac{1-\delta}{2}} \big|\zeta_{\frac{1+3a}{4}} V\big|_\infty  |\zeta_{a}\sqrt{\eta_r} U|_2|\zeta_{a}\sqrt{\eta_r}\varrho^\frac{\delta-1}{2}D_\eta U|_2\nonumber\\
&\leq C(a,T)|\zeta_{a}\sqrt{\eta_r} U|_2^2+a_{1}\delta|\zeta_{a} \sqrt{\eta_r}\varrho^\frac{\delta-1}{2}D_\eta U|_2^2,\label{g17-g19}\\
J_{34}&\leq C_0 \big(|\varrho|_\infty^\frac{1-\delta}{2}|\zeta_{\frac{1+3a}{4}} V\big|_\infty+|\zeta_{\frac{1+3a}{4}}\varrho^{\frac{1-\delta}{2}}U|_\infty\big)|\zeta_{a}\sqrt{\eta_r} U|_2\big|\zeta_{a}\sqrt{\eta_r}\varrho^\frac{\delta-1}{2}\frac{U}{\eta}\big|_2\nonumber\\
&\leq  C(a,T)\big(1+|\zeta_{\frac{1+3a}{4}}\varrho^{\frac{1-\delta}{2}}U|_\infty^2  \big)|\zeta_{a}\sqrt{\eta_r} U|_2^2+a_{1}\delta\big|\zeta_{a}\sqrt{\eta_r}\varrho^\frac{\delta-1}{2}\frac{U}{\eta}\big|_2^2,  \nonumber\\
J_{35}&\leq C(a,T)|(r^2\rho_0)^\frac{1}{2}U|_2 \big(|\eta^2\varrho^\delta|_\infty^{\frac{1}{2}}\big|\zeta_{a}\sqrt{\eta_r}\varrho^\frac{\delta-1}{2}\mathscr{D}_\eta U\big|_2+\big|(r^2\rho_0)^\frac{1}{4}U\big|_4^2\big)\nonumber\\
&\leq C(a,T)+a_{1}\delta \big|\zeta_{a}\sqrt{\eta_r}\varrho^\frac{\delta-1}{2}\mathscr{D}_\eta U\big|_2^2.\nonumber
\end{align}    
Therefore, we obtain from \eqref{dt-g17-g19}--\eqref{g17-g19} that
\begin{equation}\label{dtu}
\begin{aligned}
&\frac{\mathrm{d}}{\dt}\big|\zeta_{a}\sqrt{\eta_r} U\big|_2^2 + a_{1}\delta\big|\zeta_{a}\sqrt{\eta_r}\varrho^{\frac{\delta-1}{2}}\mathscr{D}_\eta U\big|_2^2 \\
&\leq C(a,T)\big((1+|\zeta_{\frac{1+3a}{4}}\varrho^{\frac{1-\delta}{2}}U|_\infty^2)|\zeta_a\sqrt{\eta_r}U|_2^2+|\zeta_{a}\sqrt{\eta_r} \varrho^{\gamma-\delta} V|_2^2+1\big).  
\end{aligned}
\end{equation}

\textbf{3.} Multiplying \eqref{eq:v} by $2\zeta_{a}^2\eta_r\varrho^{2\gamma-2\delta} V$,  and integrating  over $I$,
along  with $\eqref{eq:VFBP-La}_1$, and Lemmas \ref{lemma-bound depth} and \ref{lemma-v Linfty in}, yields
\begin{align}
\frac{\mathrm{d}}{\dt}\big|\zeta_{a}\sqrt{\eta_r}\varrho^{\gamma-\delta} V\big|_2^2&\leq C_0|\varrho|_\infty^{\gamma-\delta+\frac{1-\delta}{2}} \big|\zeta_{\frac{1+3a}{4}}V\big|_\infty \big|\zeta_{a}\sqrt{\eta_r}\varrho^{\gamma-\delta} V\big|_2\big|\zeta_{a}\sqrt{\eta_r}\varrho^\frac{\delta-1}{2}\mathscr{D}_\eta U\big|_2\nonumber\\
&\quad + C_0\big(\big|\zeta_{a}\sqrt{\eta_r}\varrho^{\gamma-\delta} V\big|_2+ |\varrho|_\infty^{\gamma-\delta} \big|\zeta_{a}\sqrt{\eta_r}U\big|_2\big)|\varrho|_\infty^{\gamma-\delta}\big|\zeta_{a}\sqrt{\eta_r}\varrho^{\gamma-\delta} V\big|_2\label{dtv}\\
&\quad +C_0|\varrho|^{\gamma-\delta}_\infty |\sqrt{\eta_r}D_\eta\Phi|_2 \big|\zeta_{a}\sqrt{\eta_r}\varrho^{\gamma-\delta} V\big|_2    \nonumber \\
&\leq  C(a,T)\big|\zeta_{a}\sqrt{\eta_r} (U,\varrho^{\gamma-\delta}V)\big|_2^2+\frac{a_{1}\delta}{8}\big|\zeta_{a}\sqrt{\eta_r}\varrho^\frac{\delta-1}{2}\mathscr{D}_\eta U\big|_2^2.\nonumber
\end{align}

\textbf{4.} Therefore, combining \eqref{dtu}--\eqref{dtv} leads to
\begin{equation*}
\begin{aligned}
    &\,\frac{\mathrm{d}}{\mathrm{d}t}\big|\zeta_{a}\sqrt{\eta_r}\big(U,\varrho^{\gamma-\delta} V\big)\big|_2^2+ \frac{a_{1}\delta}{8}\big|\zeta_{a}\sqrt{\eta_r}\varrho^\frac{\delta-1}{2}\mathscr{D}_\eta U\big|_2^2\\
    &\leq C(a,T)\big((1+|\zeta_{\frac{1+3a}{4}}\varrho^{\frac{1-\delta}{2}}U|_\infty^2)|\zeta_a\sqrt{\eta_r}U|_2^2+|\zeta_{a}\sqrt{\eta_r} \varrho^{\gamma-\delta} V|_2^2+1\big),
\end{aligned}
\end{equation*}
which, along with the Gr\"onwall inequality, yields that, for all $a\in(0,1)$ and $t\in [0,T]$,
\begin{equation*}
\begin{aligned}
\big|\zeta_{a}\sqrt{\eta_r}(U,\varrho^{\gamma-\delta} V)(t)\big|_2^2+ \int_0^t\big|\zeta_{a}\sqrt{\eta_r}\varrho^\frac{\delta-1}{2}\mathscr{D}_\eta U\big|_2^2\,\mathrm{d}s
&\leq C(a,T)\big(\big|\zeta_{a}(u_0,\rho_0^{\gamma-\delta}v_0)\big|_2^2+1\big) \leq C(a,T).
\end{aligned}    
\end{equation*}
Here, to derive the $L^2$-bound of $\zeta_{a}(u_0,\rho_0^{\gamma-\delta}v_0)$, it suffices to check that
\begin{equation*}
|\zeta_{a}\rho_0^{\gamma-\delta}(\rho_0^{\delta-1})_r|_2\leq |\zeta_{a}\rho_0^{\gamma-1-\beta}(\rho_0^\beta)_r|_2\leq C(a)|(\rho_0^\beta)_r|_\infty\leq C(a).
\end{equation*}

Finally, it follows from the above, Lemmas  \ref{lemma-bound depth}--\ref{lemma-lower bound jacobi}, and \ref{sobolev-embedding}, and the H\"older and Young inequalities that, for all $a\in(0,1)$ and  $t\in[0,T]$,
\begin{align*}
\int_0^t |\zeta_{a} U|_\infty^2\,\mathrm{d}s 
&\leq C(a)\int_0^t \big(\big|\zeta_{\frac{1+3a}{4}}\sqrt{\eta_r}U\big|_2^2+ |\varrho|_\infty^\frac{1-\delta}{2}|\zeta_{a}\sqrt{\eta_r}U|_2|\zeta_{a}\sqrt{\eta_r}\varrho^{\frac{\delta-1}{2}}D_\eta U|_2\big)\,\mathrm{d} s \leq C(a,T).
\end{align*}

This completes the proof of Lemma \ref{lemma-u infty}.
\end{proof}

We now establish the global uniform upper bounds of $(\eta_r,\frac{\eta}{r})$ in $[0,T]\times [0,a]$ for $a\in (0,1)$.
\begin{Lemma}\label{lemma-upper jacobi near}
For any $a\in (0,1)$, there exists a constant $C(a,T)>0$ such that
\begin{equation*}
\chi_a\frac{\eta(t,r)}{r}+\chi_a\eta_r(t,r)\leq C(a,T) \qquad \text{for all $(t,r)\in [0,T]\times \bar I$},
\end{equation*}
where  $\chi_a$ denotes the characteristic function on $[0,a]$ 
{\rm  (}see {\rm \S \ref{othernotation})}.

\end{Lemma}
\begin{proof}
We divide the proof into three steps.

\textbf{1.} First, multiplying \eqref{eq:v} by $2\chi_{a} \eta_r V$ ($a\in(0,1)$) and integrating the resulting equality over $I$, together with Lemmas \ref{lemma-bound depth}, \ref{lemma-v Linfty in}, and \ref{lemma-u infty}, and the H\"older inequality, gives
\begin{equation*}
\begin{aligned}
&\,\frac{\mathrm{d}}{\dt}|\chi_{a}\sqrt{\eta_r}V|_2^2+\frac{A\gamma}{a_{1}\delta}\big|\chi_{a}\sqrt{\eta_r}\varrho^\frac{\gamma-\delta}{2}V\big|_2^2=\int_0^a V^2U_r-2\eta_rVD_\eta\Phi\,\mathrm{d}r+\frac{A\gamma}{a_{1}\delta}\int_0^a \eta_r \varrho^{\gamma-\delta} UV\,\mathrm{d}r\\
&\leq C_0|\chi_a\sqrt{\eta_r}V|_2\big(|\zeta_{a} V|_\infty|\varrho|_\infty^\frac{1-\delta}{2} |\zeta_{a} \sqrt{\eta_r}\varrho^\frac{\delta-1}{2} D_\eta U|_2+|\sqrt{\eta_r}D_\eta \Phi|_2+\frac{A\gamma}{a_{1}\delta}|\varrho|_\infty^{\gamma-\delta}|\zeta_{a}\sqrt{\eta_r}U|_2\big),
\end{aligned}
\end{equation*}
which, along with the Young inequality, leads to
\begin{equation}\label{fangzhe1}
\frac{\mathrm{d}}{\dt}|\chi_{a}\sqrt{\eta_r}V|_2^2 \leq C(a,T)\big(|\chi_{a}\sqrt{\eta_r}V|_2^2+|\zeta_{a} \sqrt{\eta_r}\varrho^\frac{\delta-1}{2} D_\eta U|_2^2+1\big).
\end{equation}
Then applying the Gr\"onwall inequality to the above and using Lemma \ref{lemma-u infty}, we obtain that, for any $t\in [0,T]$ and $a\in (0,1)$,
\begin{equation}\label{435}
|\chi_{a}\sqrt{\eta_r}V(t)|_2\leq C(a,T)(|\chi_{a}v_0|_2+1)\leq C(a,T).
\end{equation}
Here, to derive the $L^2$-bound of $\chi_{a}v_0$, it suffices to note that $\rho_0^\beta\sim 1-r$ and
\begin{equation*}
|\chi_{a}(\rho_0^{\delta-1})_r|_2\leq C(a)|(\rho_0^\beta)_r|_2\leq C(a).
\end{equation*}

\textbf{2. Uniform boundedness of $\varrho^{\delta-1}$ near the origin.} First, it follows from \eqref{v-expression}, \eqref{435}, Lemmas \ref{lemma-lower bound jacobi}, \ref{lemma-u infty}, and \ref{sobolev-embedding}, and the H\"older and Young inequalities that
\begin{align}
|\chi_{a}\varrho^{\delta-1}|_\infty^2&\leq C(a)|\chi_{a}\varrho^{\delta-1}|_2^2+ C(a)|\chi_{a}\varrho^{\delta-1}(\varrho^{\delta-1})_r|_1 \nonumber\\
&\leq C(a)|\chi_{a}\varrho^{\delta-1}|_2^2+C(a)|\chi_{a}\sqrt{\eta_r}\varrho^{\delta-1}|_2|\chi_{a}\sqrt{\eta_r}D_\eta(\varrho^{\delta-1})|_2\label{log h-pre}\\
&\leq C(a)(|\eta_r^{-1}|_\infty+1) |\chi_{a}\sqrt{\eta_r}\varrho^{\delta-1}|_2^2+\big|\sqrt{\eta_r}(\chi_{a}V,\zeta_{a}U)\big|_2^2\nonumber\\
&\leq C(a,T)(|\chi_{a}\sqrt{\eta_r}\varrho^{\delta-1}|_2^2+1).\nonumber
\end{align}

Next, multiplying $\eqref{eq:VFBP-La}_1$ by $2(\delta-1)\chi_{a}\eta_r\varrho^{2\delta-3}$ and integrating the resulting equality over $I$, we obtain from the above, and the H\"older and Young inequalities that
\begin{align*}
\frac{\mathrm{d}}{\dt}|\chi_{a}\sqrt{\eta_r}\varrho^{\delta-1}|_2^2&=(3-2\delta)\int_0^a \eta_r (\varrho^{\delta-1})^2 D_\eta U\,\mathrm{d}r+(4-4\delta)\int_0^a \eta_r(\varrho^{\delta-1})^2\frac{U}{\eta}\,\mathrm{d}r\\
&\leq C_0|\varrho|_\infty^\frac{1-\delta}{2} |\chi_{a} \varrho^{\delta-1}|_\infty |\chi_{a}\sqrt{\eta_r}\varrho^{\delta-1}|_2\big|\zeta_{a}\sqrt{\eta_r}\varrho^\frac{\delta-1}{2}\mathscr{D}_\eta U\big|_2 \\
&\leq C(a,T)(|\chi_{a}\sqrt{\eta_r}\varrho^{\delta-1}|_2^2+1)\big|\zeta_{a}\sqrt{\eta_r}\varrho^\frac{\delta-1}{2}\mathscr{D}_\eta U\big|_2,    
\end{align*}
which, along with Lemma \ref{lemma-u infty}, the fact that $\rho_0^\beta\sim 1-r$, and the Gr\"onwall and H\"older inequalities, yields that, for all $t\in [0,T]$ and $a\in (0,1)$,
\begin{equation}\label{log h-L2}
|\chi_{a}\sqrt{\eta_r} \varrho^{\delta-1} (t)|_2\leq C(a,T)(|\chi_{a} \rho_0^{\delta-1}|_2+1)\leq C(a,T). 
\end{equation}
This, together with \eqref{log h-pre}, also yields that, for all $t\in [0,T]$ and $a\in (0,1)$,
\begin{equation}\label{log h-Linfty}
|\chi_{a} \varrho^{\delta-1}(t)|_\infty\leq C(a,T).
\end{equation}

\textbf{3. Upper bounds for $(\eta_r,\frac{\eta}{r})$ in $[0,T]\times [0,a]$.} Note that \eqref{log h-Linfty}, together with \eqref{eq:eta}, implies that, for all $(t,r)\in [0,T]\times [0,a]$ and $a\in (0,1)$,
\begin{equation}\label{439}
\frac{r^2\rho_0(r)}{(\eta^2\eta_r)(t,r)}\geq C(a,T)^{-1} \implies (\eta^2\eta_r)(t,r) \leq C(a,T)r^2.
\end{equation}
Therefore, it follows from \eqref{439} and Lemma \ref{lemma-lower bound jacobi} that, for all $(t,r)\in [0,T]\times [0,a]$,
\begin{equation*}
\frac{\eta(t,r)}{r} \leq \Big(\frac{C(a,T)}{\eta_r(t,r)}\Big)^\frac{1}{2} \leq C(a,T),\qquad \eta_r(t,r)\leq  \frac{C(a,T)r^2}{\eta^2(t,r)}\leq C(a,T).
\end{equation*}

This completes the proof of Lemma \ref{lemma-upper jacobi near}.
\end{proof}

Choosing $a_T$ as in Lemma \ref{lemma-upper jacobi-ex} and applying Lemma \ref{lemma-upper jacobi near} on $[0,a_T]$, we obtain, for all $(t,r)\in[0,T]\times\bar I$,
\begin{equation*}
\frac{\eta(t,r)}{r}+\eta_r(t,r)=\chi_{a_T}\bigl(\frac{\eta(t,r)}{r}+\eta_r(t,r)\bigr)+\chi_{a_T}^\sharp\bigl(\frac{\eta(t,r)}{r}+\eta_r(t,r)\bigr)\leq C(T),
\end{equation*} 
which thus provides the global uniform upper bounds of $(\eta_r,\frac{\eta}{r})$ in $[0,T]\times \bar I$:
\begin{Lemma}\label{lemma-upper jacobi}
There exists a constant $C(T)>0$ such that
\begin{equation*}
\frac{\eta(t,r)}{r}+\eta_r(t,r)\leq C(T) \qquad \text{for all $(t,r)\in[0,T]\times \bar I$}.
\end{equation*}
\end{Lemma}

Finally, it follows from \eqref{eq:eta} and Lemmas \ref{lemma-lower bound jacobi} and \ref{lemma-upper jacobi} that
 there is no vacuum formation inside the fluids in finite time: 
\begin{Lemma}\label{non-vac}
There exists a constant $C(T)>1$ such that
\begin{equation*}
C(T)^{-1}\leq \frac{\varrho(t,r)}{\rho_0(r)}\leq C(T) \qquad\,\, \text{for all $(t,r)\in[0,T]\times \bar I$}.
\end{equation*}
\end{Lemma}

\section{Global  Uniform Weighted Energy Estimates on the Velocity}\label{Section-globalestimates}
The purpose of this section is to establish the global uniform estimates on the velocity. We first give the following lemma, which will be frequently used in the later analysis.
\begin{Lemma}\label{im-1}
For any $a\in (0,1)$, the following estimates hold{\rm :}
\begin{align*}
&\big|\zeta_a r\mathscr{D}_\eta f\big|_2\sim\Big|\zeta_a r\big(D_\eta f+ \frac{2f}{\eta}\big)\Big|_2,\\
&\big|\zeta_a r\mathscr{D}_\eta^k f\big|_2\sim \Big|\zeta_a r 
\mathscr{D}_\eta^{k-2} \Big(D_\eta\big(D_\eta f+ \frac{2f}{\eta}\big)\Big) \Big|_2 \qquad\text{for $k=2,3,4$},
\end{align*}
where $F_1\sim F_2$ denotes that there exists a constant $C(T)\geq 1$ depending only on $(C_0,T)$ such that $C(T)^{-1}F_1\leq F_2\leq C(T)F_1$.
\end{Lemma}
\begin{proof}
For simplicity, we only give the proof for $a=\frac{1}{2}$ and show that 
\begin{equation*}
\big|\zeta r\mathscr{D}_\eta^4 f\big|_2\leq C(T) \Big|\zeta r 
\mathscr{D}_\eta^2 \Big(D_\eta\big(D_\eta f+ \frac{2f}{\eta}\big)\Big) \Big|_2.
\end{equation*}
Besides, the following fact will be used frequently later:
\begin{equation*}
D_\eta\zeta\leq 0 \qquad \text{for all $(t,r)\in [0,T]\times\bar I$}. 
\end{equation*}

First, a direct calculation gives
\begin{equation}\label{L6}
\begin{aligned}
&\,\Big|\zeta (\eta^2\eta_r)^\frac{1}{2} D_\eta^3\big(D_\eta f+ \frac{2f}{\eta}\big)\Big|_2^2\\
&=\int_0^1 \zeta^2 \eta^2 \eta_r \big(|D_\eta^4 f|^2 +4 \big|D_\eta^3\big(\frac{f}{\eta}\big)\big|^2\big) \,\mathrm{d}r+\underline{4\int_0^1 \zeta^2\eta^2 \eta_r D_\eta^4f D_\eta^3\big(\frac{f}{\eta}\big)\,\mathrm{d}r}_{:=L_{1}}.
\end{aligned}
\end{equation}
Then, for $L_{1}$, note that the following identity holds:
\begin{equation*}
D_\eta^4 f=\eta D_\eta^4\big(\frac{f}{\eta}\big)+ 4 D_\eta^3\big(\frac{f}{\eta}\big).
\end{equation*}
Hence, this, together with integration by parts, yields  
\begin{equation*}
\begin{aligned}
L_{1}&=16\int_0^1 \zeta^2 \eta^2 \eta_r \big|D_\eta^3\big(\frac{f}{\eta}\big)\big|^2 \,\mathrm{d}r +4\int_0^1 \zeta^2 \eta^{3} \eta_r D_\eta^4\big(\frac{f}{\eta}\big) D_\eta^3\big(\frac{f}{\eta}\big)\,\mathrm{d}r\\
&=10\int_0^1 \zeta^2 \eta^2 \eta_r \big|D_\eta^3\big(\frac{f}{\eta}\big)\big|^2 \,\mathrm{d}r -4\int_0^1 \zeta D_\eta \zeta \eta^{3} \eta_r  \big|D_\eta^3 \big(\frac{f}{\eta}\big)\big|^2\, \mathrm{d}r\geq 0,    
\end{aligned}
\end{equation*}
which, along with \eqref{L6}, gives
\begin{equation}\label{L6'}
\Big|\zeta (\eta^2\eta_r)^\frac{1}{2} D_\eta^3\big(D_\eta f+ \frac{2f}{\eta}\big)\Big|_2^2\geq  \int_0^1 \zeta^2 \eta^2 \eta_r |D_\eta^2(\mathscr{D}_\eta^2 f)|^2  \,\mathrm{d}r.
\end{equation}

Next, notice that
\begin{equation}\label{idenen}  
D_\eta\big(\frac{D_\eta^2 f}{\eta}\big)=D_\eta^3\big(\frac{f}{\eta}\big)+2D_\eta\Big(\frac{1}{\eta}D_\eta\big(\frac{f}{\eta}\big)\Big)=\eta D_\eta^2\Big(\frac{1}{\eta}D_\eta\big(\frac{f}{\eta}\big)\Big)+4D_\eta\Big(\frac{1}{\eta}D_\eta\big(\frac{f}{\eta}\big)\Big).
\end{equation}
Hence, this, together with a direct calculation, implies
\begin{align}
&\,\Big|\zeta (\eta^2\eta_r)^\frac{1}{2} D_\eta\Big(\frac{1}{\eta}D_\eta\big(D_\eta f+ \frac{2 f}{\eta}\big)\Big)\Big|_2^2\nonumber\\
&=\Big|\zeta(\eta^2\eta_r)^\frac{1}{2}\Big(\eta D_\eta^2\big(\frac{1}{\eta}D_\eta\big(\frac{f}{\eta}\big)\big)+6D_\eta\big(\frac{1}{\eta}D_\eta\big(\frac{f}{\eta}\big)\big)\Big)\Big|_2^2\label{9.6}\\
&=\int_0^1 \zeta^2 \eta^{2} \eta_r \Big(\Big|\eta D_\eta^2\Big(\frac{1}{\eta}D_\eta\big(\frac{f}{\eta}\big)\Big)\Big|^2 +36 \Big|D_\eta\Big(\frac{1}{\eta}D_\eta\big(\frac{f}{\eta}\big)\Big)\Big|^2\Big)\,\mathrm{d}r\nonumber\\
&\quad +\underline{12\int_0^1 \zeta^2\eta^{3}\eta_r D_\eta^2\Big(\frac{1}{\eta}D_\eta\big(\frac{f}{\eta}\big)\Big)D_\eta\Big(\frac{1}{\eta}D_\eta\big(\frac{f}{\eta}\big)\Big)\,\mathrm{d}r}_{:=L_{2}},\nonumber
\end{align}
where $L_{2}$ can be handled by integration by parts:
\begin{equation*}
\begin{aligned}
L_{2}&=-18\int_0^1 \zeta^2\eta^2\eta_r\Big|D_\eta\Big(\frac{1}{\eta}D_\eta\big(\frac{f}{\eta}\big)\Big)\Big|^2\,\mathrm{d}r-12\int_0^1 \zeta D_\eta\zeta \eta^{3} \eta_r \Big|D_\eta\Big(\frac{1}{\eta}D_\eta\big(\frac{f}{\eta}\big)\Big)\Big|^2 \,\mathrm{d}r\\
&\geq -18 \int_0^1 \zeta^2\eta^2\eta_r\Big|D_\eta\Big(\frac{1}{\eta}D_\eta\big(\frac{f}{\eta}\big)\Big)\Big|^2\,\mathrm{d}r.    
\end{aligned}
\end{equation*}
Therefore, \eqref{9.6}, combined with the above, implies 
\begin{equation}\label{l6"}
\Big|\zeta (\eta^2\eta_r)^\frac{1}{2} D_\eta\Big(\frac{1}{\eta}D_\eta\big(D_\eta f+ \frac{2 f}{\eta}\big)\Big)\Big|_2^2\geq 18\int_0^1 \zeta^2 \eta^{2} \eta_r   \Big|D_\eta\Big(\frac{1}{\eta}D_\eta\big(\frac{f}{\eta}\big)\Big)\Big|^2 \,\mathrm{d}r.
\end{equation}

Finally, it follows from the above and \eqref{L6'}--\eqref{idenen} that
\begin{equation}\label{l6""}
\int_0^1 \zeta^2 \eta^{2} \eta_r \Big|D_\eta\big(\frac{D_\eta^2 f}{\eta}\big)\Big|^2 \,\mathrm{d}r\leq C_0\Big|\zeta (\eta^2\eta_r)^\frac{1}{2} \sD_\eta^2 \Big(D_\eta \big(D_\eta f+ \frac{2f}{\eta}\big)\Big)\Big|_2^2.
\end{equation}
Hence, by \eqref{L6'}, \eqref{l6"}--\eqref{l6""}, and Lemmas \ref{lemma-lower bound jacobi} and \ref{lemma-upper jacobi}, we obtain the desired estimates.
\end{proof}

\subsection{Tangential estimates of the velocity}
We first give time-spatial estimates for the velocity. For simplicity, denote 
$L_*=D_\eta U-\frac{2(1-\delta)}{\delta}\frac{ U}{\eta}$.

\begin{Lemma}\label{lemma-time-space}
 For any $a\in (0,1)$ and  any $\iota$ satisfying
\begin{equation}\label{iota-}
\iota\in\Big(\frac{\beta+1}{p_{*}}-\frac{\beta}{2},\delta+\frac{\beta}{2}\Big),
\end{equation}
there exists a constant $C(\iota,a,T)>0$ such that, for all $t\in[0,T]$,
\begin{equation*}
\big|\chi^\sharp_{a}\rho_0^{\iota-\frac{\beta}{2}}L_*\big|_\infty \leq C(\iota,a,T)\big(1+\big|\chi^\sharp_{a}\rho_0^\iota(D_\eta U,\rho_0^{1-\delta}U_t)\big|_2\big).
\end{equation*}
Besides, for any such $\iota$, $a\in (0,1)$, and $\sigma>0$, there exists  $C(\sigma,\iota,a,T)>0$ such that 
\begin{equation*}
\begin{aligned}
\big|\chi^\sharp_{a}\rho_0^{\iota-\beta+\sigma}L_*\big|_2\leq C(\sigma,\iota,a,T)\big(1+\big|\chi^\sharp_{a}\rho_0^\iota(D_\eta U,\rho_0^{1-\delta}U_t)\big|_2\big).
\end{aligned}
\end{equation*}
\end{Lemma}

\begin{proof}
Integrating $\eqref{eq:VFBP-La-eta}_1\times\eta_r$ over $[r,1]$ ($r\in(a,1)$) gives
\begin{equation}\label{zhongyao}
\begin{aligned}
L_*&=\frac{A}{2 a_{1}\delta}\varrho^{\gamma-\delta}-\frac{1}{2 a_{1}\delta \varrho^{\delta}}\int_r^1 \frac{\hat r^2\rho_0}{\eta^2}D_\eta \Phi\,\mathrm{d}\hat r\\
&\quad+\frac{2}{\delta\varrho^\delta}\int_r^1 \eta_r\varrho^\delta \Big(\frac{D_\eta U}{\eta}-\frac{U}{\eta^2}\Big)\mathrm{d}\hat r-\frac{1}{2 a_{1}\delta \varrho^{\delta}}\int_r^1\frac{\hat r^2}{\eta^2}\rho_0U_t\,\mathrm{d}\hat r.
\end{aligned}
\end{equation}
Then due to  $\rho_0^\beta\sim 1-r$ and Lemmas \ref{lemma-lower bound jacobi}, \ref{lemma-v Lp ex}, and \ref{lemma-upper jacobi}--\ref{non-vac}, we have
\begin{equation}\label{floww}
\begin{aligned}
|L_{*}|&\leq C(T)\big(1+|D_\eta \Phi|_\infty\big)\rho_0^{\gamma-\delta}
+C(a,T)\rho_0^{-\delta}\int_r^1\big(\rho_0^\delta (|D_\eta U|+|U|)+\rho_0|U_t|\big)\mathrm{d}\hat r\\
&\leq C(T)\rho_0^{\gamma-\delta} +C(a,T)\rho_0^{-\delta}\Big(\int_r^1 \rho_0^{2\delta-2\iota}\,\mathrm{d}\hat r\Big)^\frac{1}{2}\Big(\int_a^1 \rho_0^{2\iota}\big(|D_\eta U|^2+|U|^2+\rho_0^{2-2\delta}|U_t|^2\big)\mathrm{d}\hat r\Big)^\frac{1}{2}\\
&\leq C(T)\rho_0^{\gamma-\delta}+C(\iota,a,T)\rho_0^{\frac{\beta}{2}-\iota}\big|\chi^\sharp_{a}\rho_0^\iota(D_\eta U,\rho_0^{1-\delta}U_t)\big|_2,
\end{aligned}
\end{equation}
for all $r\in [a,1)$ and $\iota$ satisfying \eqref{iota-}, which, along with Lemma \ref{lemma-refine u Lp}, yields  
\begin{equation*}
\big|\chi^\sharp_{a}\rho_0^{\iota-\frac{\beta}{2}}L_*\big|_\infty \leq C(\iota,a,T)\big(1+\big|\chi^\sharp_{a}\rho_0^\iota(D_\eta U,\rho_0^{1-\delta}U_t)\big|_2\big).
\end{equation*}

The $L^2$-estimate can be treated similarly; we omit the proof for brevity.
\end{proof}

\subsubsection{Zeroth-order and first-order  tangential estimates}

First, the zeroth-order tangential estimate for 
$U$ directly follows from  Lemmas \ref{lemma-basic energy},   \ref{lemma-lower bound jacobi}, \ref{lemma-refine u Lp}, and \ref{lemma-u infty}.
\begin{Lemma}\label{lemma-u-0order}
There exists a constant $C(T)>0$ such that
\begin{equation*}
\big|(r^2\rho_0)^\frac{1}{2}U(t)\big|_2^2+ \int_0^t\big|(r^2\rho_0^\delta)^\frac{1}{2}\mathscr{D}_\eta U\big|_2^2\,\mathrm{d}s\leq C(T) \qquad \text{for all $t\in [0,T]$}.
\end{equation*}
Moreover, for any $p\in (0,p_{*})$, $\iota\in (\frac{\beta+1}{\beta}\frac{1}{p_{*}}-\frac{1}{p},\infty)$, and $a\in (0,1)$, there exist two positive constants $C(p,\iota,a,T)$ and $C(a,T)$ such that, for all $t\in [0,T]$,
\begin{equation*}
\big|\chi_{a}^\sharp \rho_0^{\iota\beta} U(t)\big|_p\leq C(p,\iota,a,T),\qquad
|\zeta_{a}U(t)|_2^2+\int_0^t\big(\big|\zeta_{a} \mathscr{D}_\eta U\big|_2^2+|\zeta_{a} U|_\infty^2\big)\,\mathrm{d}s\leq C(a,T).
\end{equation*}
\end{Lemma}

Next we give the desired first-order tangential estimates for the velocity.
\begin{Lemma}\label{lemma-u-D1}
There exists a constant $C(T)>0$ such that 
\begin{equation*}
\big|(r^2\rho_0^\delta)^\frac{1}{2}\mathscr{D}_\eta U(t)\big|_2^2+ \int_0^t\big|(r^2\rho_0)^\frac{1}{2}U_t\big|_2^2\,\mathrm{d}s\leq C(T) \qquad \text{for all $t\in [0,T]$}.
\end{equation*}
\end{Lemma}

\begin{proof}
We divide the proof into three steps. For simplicity, denote 
$
L_*=D_\eta U-\frac{2(1-\delta)}{\delta}\frac{ U}{\eta}
$.

\textbf{1.} We give some auxiliary estimates. First, $\eqref{eq:VFBP-La-eta}_1$, combined with \eqref{v-expression}, yields 
\begin{equation}\label{ell-0}
\begin{aligned}
D_\eta\big(D_\eta U+ \frac{2U}{\eta}\big)&=\frac{1}{2a_{1}\delta}\varrho^{1-\delta} U_t +\frac{A\gamma}{4a_{1}^2\delta^2} \varrho^{\gamma-2\delta+1}(V-U)+\frac{1}{2a_{1}\delta}\varrho^{1-\delta}D_\eta\Phi\\
&\quad-\frac{1}{2a_{1}}\varrho^{1-\delta}(V-U)L_*.
\end{aligned}
\end{equation}
Then it follows from \eqref{ell-0}, and Lemmas \ref{lemma-bound depth}, \ref{lemma-v Lp ex}, \ref{lemma-v Linfty in}, \ref{im-1}, and \ref{lemma-u-0order} that  
\begin{equation}\label{in-ell-2}
\begin{aligned}
\big|\zeta_{a} r\mathscr{D}_\eta^2 U \big|_2&\leq C(T)|\varrho|_\infty^{1-\delta}\big(|\zeta_{a} r U_t|_2 +|r|_\infty|\varrho|_\infty^{\gamma-\delta}|\zeta_{a} (V,U)|_2+|D_\eta\Phi|_\infty\big)\\
&\quad+C(T)|\zeta_{a}(V,U)|_\infty\big|\zeta_{\frac{1+3a}{4}} r \mathscr{D}_\eta U\big|_2\\
&\leq C(a,T)\big(|\zeta_{a} r  U_t|_2+ (1+|\zeta_{a} U|_\infty)\big|\zeta_{\frac{1+3a}{4}} r  \mathscr{D}_\eta U\big|_2 +1\big).
\end{aligned}
\end{equation}

Next, multiplying \eqref{ell-0} by $\chi^\sharp \rho_0^\frac{\delta+\varepsilon}{2}$ ($\varepsilon>0$), we obtain from Lemmas \ref{lemma-lower bound jacobi}, \ref{lemma-v Lp ex}, \ref{lemma-upper jacobi}, \ref{lemma-time-space} (with $\iota=\frac{\delta}{2}$), and \ref{lemma-u-0order} that
\begin{equation}\label{ex-ell-2}
\begin{aligned}
\big|\chi^\sharp \rho_0^\frac{\delta+\varepsilon}{2} D_\eta^2 U\big|_2&\leq C(T)\big(\big|\chi^\sharp \rho_0^\frac{\delta+\varepsilon}{2}(U,D_\eta U,\rho_0^{1-\delta}U_t)\big|_2+|\rho_0^\beta|_\infty^{\frac{2-\delta+\varepsilon}{2\beta}}|D_\eta\Phi|_\infty\big)\\[2pt]
&\quad+C(T)\big|\chi^\sharp \rho_0^{\frac{\beta+\varepsilon}{2}+1-\delta}(V,U)\big|_2\big|\chi^\sharp \rho_0^{\frac{\delta-\beta}{2}}L_* \big|_\infty \\[-2pt]
&\leq C(\varepsilon,T)\big(\big|\chi^\sharp(\rho_0^\frac{\delta}{2}D_\eta U,\rho_0^{\frac{1}{2}}U_t)\big|_2 +1\big).
\end{aligned}
\end{equation}

\textbf{2.} Multiplying $\eqref{eq:VFBP-La-eta}_1$ by $\eta^2\eta_rU_t$ and integrating over I, along with \eqref{eq:eta} and \eqref{v-expression}, gives 
\begin{equation}\label{dt-j1-j4}
a_{1} \frac{\mathrm{d}}{\dt}L_{3} + \big|(r^2\rho_0)^\frac{1}{2}U_t\big|_2^2 = L_4+L_5^\flat+L_5^\sharp+L_6^\sharp,  
\end{equation}
where
\begin{align*}
L_3&:=\int_0^1 \eta^2\eta_r\varrho^\delta\Big(\delta|D_\eta U|^2- 4(1-\delta) D_\eta U\frac{U}{\eta}+ (4\delta-2)\frac{U^2}{\eta^2}\Big)\,\mathrm{d}r,\\
L_4&:= -\int_0^1\Big(r^2\rho_0D_\eta\Phi+\frac{A\gamma}{2a_{1}\delta} \frac{(r^2\rho_0)^{\gamma-\delta+1}}{(\eta^2\eta_r)^{\gamma-\delta}}  (V-U) \Big)U_t\,\mathrm{d}r,\\
L_5^\flat&:=-4a_{1}\int_0^1 \zeta\eta^2\eta_r\varrho^\delta \Big(\frac{\delta(1+\delta)}{4}(D_\eta U)^3-\frac{3\delta(1-\delta)}{2}(D_\eta U)^2\frac{U}{\eta}\\
&\qquad\qquad\qquad\qquad\quad-\frac{(6\delta-3)(1-\delta)}{2}D_\eta U\frac{U^2}{\eta^2}+\delta(2\delta-1) \frac{U^3}{\eta^3}\Big)\,\mathrm{d}r,\\[-2mm]
L_5^\sharp&:=- a_{1}\delta(1+\delta)\int_0^1 \zeta^\sharp\eta^2\eta_r\varrho^\delta  (D_\eta U)^3 \,\mathrm{d}r,\\
L_6^\sharp&:=6a_{1}\delta(1-\delta)\int_0^1 \zeta^\sharp\eta^2\eta_r\varrho^\delta \Big((D_\eta U)^2\frac{U}{\eta}+ \frac{2\delta-1}{\delta}D_\eta U\frac{U^2}{\eta^2}-\frac{2(2\delta-1)}{3(1-\delta)} \frac{U^3}{\eta^3}\Big)\,\mathrm{d}r.
\end{align*}

\textbf{2.1. Estimates for $L_3$ and $L_{4}$.} First,  $L_3$ can be handled by Lemmas \ref{lemma-lower bound jacobi} and \ref{non-vac} and a calculation similar to $I_2$ in Step 3 of the proof of Lemma \ref{lemma-basic energy}:
\begin{equation}
L_3\geq \frac{1}{C(T)}\big|(r^2\rho_0^\delta)^\frac{1}{2}\sD_\eta U\big|_2^2.
\end{equation}
Second,  we can obtain from Lemmas \ref{lemma-lower bound jacobi}, \ref{lemma-v Lp ex}, and  \ref{lemma-v Linfty in} that
\begin{equation}\label{rm-J1}
\begin{aligned}
L_{4} &\leq C_0\Big(|D_\eta\Phi|_\infty + \big(|\zeta(V,U)|_2 + |\zeta^\sharp \rho_0^{\beta+1-\delta} (V,U)|_2 \big)\Big|\frac{r^{2}}{\eta^2\eta_r}\Big|_\infty^{\gamma-\delta}\Big)\big|(r^2\rho_0)^\frac{1}{2}U_t\big|_2 \\[-2mm]
&\leq C(T)(|\zeta V|_\infty+|\chi^\sharp \rho_0^{\beta+1-\delta} V|_\infty+1)\big|(r^2\rho_0)^\frac{1}{2}U_t\big|_2 \leq C(T)+\frac{1}{8}\big|(r^2\rho_0)^\frac{1}{2}U_t\big|_2^2.
\end{aligned}    
\end{equation}

\textbf{2.2. Estimate for $L_{5}^\flat$.} 
First, Lemmas \ref{lemma-upper jacobi} and \ref{GNinequality} imply that, for any $a\in (0,1)$,
\begin{equation}\label{DU-L44}
\big|\chi_a r \mathscr{D}_\eta U\big|_4^2\leq C(a,T)\big|\zeta_a r \mathscr{D}_\eta U\big|_2\big|\zeta_a r  (\mathscr{D}_\eta U,\mathscr{D}_\eta^2 U)\big|_2.   
\end{equation}
It follows from this inequality, the fact that $\rho_0^\beta\sim 1-r$, \eqref{in-ell-2}, Lemma \ref{lemma-bound depth}, and the H\"older and Young inequalities that
\begin{equation}\label{jj1}
\begin{aligned}
L_5^\flat&\leq C(T)\big|\zeta \mathscr{D}_\eta U\big|_2\big|\chi_\frac{5}{8} r \mathscr{D}_\eta U\big|_4^2\leq C(T)\big|\zeta \mathscr{D}_\eta U\big|_2\big|\zeta_\frac{5}{8} r \mathscr{D}_\eta U\big|_2 \big|\zeta_\frac{5}{8} r(\mathscr{D}_\eta U,\mathscr{D}_\eta^2 U)\big|_2 \\
&\leq C(T)\big|\zeta \mathscr{D}_\eta U\big|_2\big|(r^2\rho_0^\delta)^\frac{1}{2} \mathscr{D}_\eta U\big|_2^2\\[2pt]
&\quad + C(T)\big|\zeta \mathscr{D}_\eta U\big|_2\big|(r^2\rho_0^\delta)^\frac{1}{2} \mathscr{D}_\eta U\big|_2\big(|\zeta_\frac{5}{8} r  U_t|_2+ (1+|\zeta_\frac{5}{8} U|_\infty)\big|\zeta_{\frac{3}{4}} r  \mathscr{D}_\eta U\big|_2 +1\big) \\[-2pt]
&\leq C(T)\big(1 + |\zeta_\frac{5}{8} U|_\infty^2 + \big|\zeta\mathscr{D}_\eta U\big|_2^2\big)\big(1 + \big|(r^2\rho_0^\delta)^\frac{1}{2}\mathscr{D}_\eta U\big|_2^2\big) + \frac{1}{8}\big|(r^2\rho_0)^\frac{1}{2}U_t\big|_2^2.    
\end{aligned}
\end{equation}

\textbf{2.3. Estimate for $L_{5}^\sharp$.} We first obtain from integration by parts and \eqref{v-expression} that
\begin{equation*} 
\begin{aligned}
L_5^\sharp&=a_{1}\delta(1+\delta) \int_0^1 D_\eta(\zeta^\sharp  \eta^2\varrho^\delta|D_\eta U|^2)\eta_r U \,\mathrm{d}r \\
&=\underline{-\frac{\delta(1+\delta)}{2}\int_0^1 \zeta^\sharp r^2\rho_0 U^2 |D_\eta U|^2 \,\mathrm{d}r}_{\,\leq 0}+L_{5,1}^\sharp +L_{5,2}^\sharp, 
\end{aligned}
\end{equation*}
where
\begin{align*}
L_{5,1}^\sharp&:= 2a_{1}\delta(1+\delta) \int_0^1 \zeta^\sharp \eta^2\eta_r\varrho^\delta  U (D_\eta U) (D_\eta^2 U) \,\mathrm{d}r,\\
L_{5,2}^\sharp&:=a_{1}\delta(1+\delta)\int_0^1 \eta^2\eta_r\Big(\big(D_\eta \zeta^\sharp + \frac{2\zeta^\sharp}{\eta}\big)\varrho^\delta +\frac{1}{2a_{1}} \zeta^\sharp \varrho V\Big) U|D_\eta U|^2 \,\mathrm{d}r.
\end{align*}

For $L_{5,1}^\sharp$, we first see from Lemmas \ref{lemma-lower bound jacobi} and \ref{lemma-upper jacobi}--\ref{non-vac} that
\begin{equation}\label{L51}
\begin{aligned}
L_{5,1}^\sharp &\leq C(T)\int_\frac{1}{2}^1 \rho_0^\delta  |U| |L_*| |D_\eta^2 U| \,\mathrm{d}r  + C(T)\int_\frac{1}{2}^1 \rho_0^\delta  |U|^2 |D_\eta^2 U| \,\mathrm{d}r:=L_{5,11}^\sharp +L_{5,12}^\sharp. 
\end{aligned}
\end{equation}
Recall from $p_* >\frac{16}{3}>\frac{4(\beta+1)}{3\beta}$ that we can let $(q,\varepsilon)$ be two fixed parameters satisfying
\begin{equation*}
q\in(2,p_*),\qquad 0<\varepsilon<\frac{3\beta p_*-4(\beta+1)}{3p_*}\iff\frac{(3q-4)\beta}{4q}-\frac{3\varepsilon}{4}>\frac{\beta+1}{p_*}-\frac{\beta}{q}.
\end{equation*}
Then, for $L^\sharp_{5,11}$, it follows from \eqref{ex-ell-2}, Lemmas \ref{lemma-time-space} (with $\iota=\frac{\delta}{2}$), \ref{lemma-u-0order}, and \ref{GNinequality}  that
\begin{align}
L_{5,11}^\sharp&\leq C(T)\big|\chi^\sharp\rho_0^{\frac{(3q-4)\beta}{4q}-\frac{3\varepsilon}{4}} U\big|_q\big|\chi^\sharp\rho_0^{\frac{\delta+\varepsilon}{2}-\frac{(q-2)\beta}{q}}\sD_\eta U\big|_\frac{q}{q-2}^\frac{1}{2}\big|\chi^\sharp \rho_0^\frac{\delta-\beta}{2}L_*\big|_\infty^\frac{1}{2} \big|\chi^\sharp\rho_0^\frac{\delta+\varepsilon}{2}D_\eta^2 U\big|_2\nonumber\\
&\leq C(T)  \big|\chi^\sharp\rho_0^\frac{\delta+\varepsilon}{2}\sD_\eta U\big|_2^\frac{1}{4}\big|\chi^\sharp\rho_0^\frac{\delta+\varepsilon}{2}(\sD_\eta U, \sD_\eta^2 U)\big|_2^\frac{1}{4} \big(1+\big|\chi^\sharp (\rho_0^\frac{\delta}{2}D_\eta U,\rho_0^\frac{1}{2}U_t)\big|_2^\frac{3}{2}\big) \label{guji519}\\
&\leq C(T)\big(1+\big|(r^2\rho_0^\delta)^\frac{1}{2}\sD_\eta U\big|_2^2\big)+\frac{1}{16}\big|(r^2\rho_0)^\frac{1}{2}U_t\big|_2^2.  \nonumber  
\end{align}

For $L^\sharp_{5,12}$, let $\varepsilon\in (0,\delta+\beta-\frac{4(\beta+1)}{p_*})$. Then $\frac{\delta-\varepsilon}{4}>\frac{\beta+1}{p_*}-\frac{\beta}{4}$.  Similarly, we have
\begin{equation*}
\begin{aligned}
L_{5,12}^\sharp&\leq C(T)\big|\chi^\sharp\rho_0^\frac{\delta-\varepsilon}{4}U\big|_4^2\big|\chi^\sharp\rho_0^\frac{\delta+\varepsilon}{2}D_\eta^2 U\big|_2
\leq  C(T)\big(1+\big|(r^2\rho_0^\delta)^\frac{1}{2}\sD_\eta U\big|_2^2\big)+\frac{1}{16}\big|(r^2\rho_0)^\frac{1}{2}U_t\big|_2^2,
\end{aligned}
\end{equation*}
which, along with  \eqref{L51}--\eqref{guji519}, yields
\begin{equation}\label{L5,1}
L_{5,1}^\sharp\leq C(T)\big(1+\big|(r^2\rho_0^\delta)^\frac{1}{2}\sD_\eta U\big|_2^2\big)+\frac{1}{8}\big|(r^2\rho_0)^\frac{1}{2}U_t\big|_2^2.
\end{equation}

To estimate $L_{5,2}^\sharp$, we first make a similar decomposition:
\begin{equation*}
\begin{aligned}
L_{5,2}^\sharp \leq C(T)\int_\frac{1}{2}^1 (\rho_0^\delta+\rho_0|V|) \big( |U| |L_*|^2+|U|^3\big) \,\mathrm{d}r.
\end{aligned}
\end{equation*}
Then letting $(q,\varepsilon)$ be fixed constants such that
\begin{equation*}
1<q<\frac{\beta p_*}{\beta+1},\qquad 0<\varepsilon<\frac{\beta q}{q-1}\Big(\frac{1}{q}-\frac{\beta+1}{\beta p_*}\Big)\iff -\frac{\varepsilon(q-1)}{q}>\frac{\beta+1}{p_*}-\frac{\beta}{q},
\end{equation*}
we obtain from \eqref{ex-ell-2}, Lemmas \ref{lemma-lower bound jacobi}, \ref{lemma-v Linfty ex}, \ref{lemma-time-space}--\ref{lemma-u-0order}, and \ref{hardy-inequality}--\ref{GNinequality},
and  $\frac{\delta}{3}+\frac{5\beta}{6}>\frac{1}{2}$ that
\begin{align}
L_{5,2}^\sharp&\leq C(T) |\chi^\sharp (\rho_0^\beta,\rho_0^{\beta+1-\delta}V)|_\infty \big|\chi^\sharp \rho_0^{-\frac{\varepsilon(q-1)}{q}} U\big|_q\big|\chi^\sharp \rho_0^\frac{\delta-\beta}{2}L_*\big|_\infty^\frac{2}{q}\big|\chi^\sharp \rho_0^\frac{\delta+\varepsilon-\beta}{2}\sD_\eta U\big|_2^\frac{2(q-1)}{q}\nonumber\\
&\quad + C(T) |\chi^\sharp (\rho_0^\beta,\rho_0^{\beta+1-\delta}V)|_\infty \big|\chi^\sharp \rho_0^{\frac{\delta-\beta}{3}} U\big|_3^3\label{L5,2}\\
& \leq C(T)\big(1+\big|\chi^\sharp \rho_0^\frac{\delta}{2}(D_\eta U,\rho_0^{1-\delta}U_t)\big|_2^\frac{2}{q}\big) \big|\chi^\sharp\rho_0^\frac{\delta+\varepsilon}{2}\sD_\eta U\big|_2^\frac{q-1}{q} \big|\chi^\sharp\rho_0^\frac{\delta+\varepsilon}{2}(\sD_\eta U,\sD_\eta^2 U)\big|_2^\frac{q-1}{q} \nonumber\\
&\quad +C(T)\big|\chi^\sharp \rho_0^{\frac{\delta}{3}+\frac{5\beta}{6}}(U,D_\eta U)\big|_2^\frac{3}{2}\big|\chi^\sharp \rho_0^{\frac{\delta}{3}+\frac{5\beta}{6}} (U,D_\eta U,D_\eta^2 U)\big|_2^\frac{3}{2}+C(T)\nonumber\\
&\leq C(T)\big(1+\big|(r^2\rho_0^\delta)^\frac{1}{2}\sD_\eta U\big|_2^4\big)+\frac{1}{8}\big|(r^2\rho_0)^\frac{1}{2}U_t\big|_2^2.\nonumber
\end{align}

Hence, combining the estimates of \eqref{L5,1}--\eqref{L5,2}, we arrive at
\begin{equation}\label{L5s}
L_5^\sharp\leq C(T)\big(1+\big|(r^2\rho_0^\delta)^\frac{1}{2}\sD_\eta U\big|_2^4\big)+\frac{1}{4}\big|(r^2\rho_0)^\frac{1}{2}U_t\big|_2^2.
\end{equation}

\textbf{2.4. Estimate for $L_6^\sharp$.} Due to Lemmas \ref{lemma-lower bound jacobi} and \ref{lemma-upper jacobi}--\ref{non-vac}, $L_6^\sharp$ can be controlled by
\begin{equation*} 
L_6^\sharp\leq C(T)\int_\frac{1}{2}^1 \rho_0^\delta \big(|D_{\eta}U|^2|U|+|D_{\eta}U| |U|^2+|U|^3\big)\,\mathrm{d}r.
\end{equation*}
Since the total degree in $(|U|,|D_\eta U|)$ of each term in the above integrand is equal to $3$ and the degree of $|D_\eta U|$ is smaller than that of $|D_\eta U|$ in $L_5^\sharp$, we can handle $L_6^\sharp$ by a simpler argument analogous to the treatment of $L_{5,2}^\sharp$ in Step 2.3. In conclusion, we can similarly obtain
\begin{equation}\label{L6s}
L_6^\sharp\leq C(T)\big(1+\big|(r^2\rho_0^\delta)^\frac{1}{2}\sD_\eta U\big|_2^4\big)+\frac{1}{8}\big|(r^2\rho_0)^\frac{1}{2}U_t\big|_2^2.
\end{equation}

\textbf{3. Close the energy estimate.} Collecting \eqref{dt-j1-j4}--\eqref{rm-J1}, \eqref{jj1}, and \eqref{L5s}--\eqref{L6s}, and then applying the Gr\"onwall inequality, we can get  the desired estimate  from Lemma \ref{lemma-u-0order}.
\end{proof}

\subsubsection{Second-order tangential estimates}For simplicity, denote 
$
L_*=D_\eta U-\frac{2(1-\delta)}{\delta}\frac{ U}{\eta}
$.
\begin{Lemma}\label{lemma-u-D2}
There exists a constant $C(T)>0$ such that 
\begin{equation}\label{tan1}
\big|(r^2\rho_0)^\frac{1}{2}U_t(t)\big|_2^2+ \int_0^t\big|(r^2\rho_0^\delta)^\frac{1}{2}\sD_\eta U_t\big|_2^2\,\mathrm{d}s\leq C(T) \qquad \text{for all $t\in [0,T]$}.
\end{equation}
Moreover, for all $t\in [0,T]$,
\begin{equation}\label{tan2}
\begin{aligned}
\mathrm{(i)}&  \ \ \big|\chi^\sharp \rho_0^\frac{\delta-\beta}{2}L_*(t)\big|_\infty \leq C(T),\quad  \big|\chi^\sharp \rho_0^\frac{\delta+\varepsilon}{2}D_\eta^2 U(t)\big|_2 \leq C(\varepsilon,T) \ \ \text{for any $\varepsilon>0$},\\ 
& \ \ \big|\chi^\sharp \rho_0^{\iota\beta} U(t)\big|_p\leq C(p,\iota,T)\  \  \text{for any $p\in [2,\infty]$  and $\iota>-\frac{1}{p}$};\\[4pt]
\mathrm{(ii)}&  \ \ \big|\zeta_{a}U(t)\big|_\infty+\big|\zeta_{a} r\sD_\eta^2 U (t)\big|_2 \leq C(a,T) \ \   \text{for any $a\in (0,1)$}. 
\end{aligned}
\end{equation}
\end{Lemma}
\begin{proof}
We divide the proof into two steps.

\smallskip
\textbf{1.} First,  since $\rho_0^\beta\sim 1-r$, it follows from Lemmas \ref{lemma-lower bound jacobi}, \ref{lemma-upper jacobi}, \ref{lemma-u-0order}--\ref{lemma-u-D1}, and \ref{sobolev-embedding}--\ref{hardy-inequality} that, for all $a\in (0,1)$,
\begin{equation}\label{U-infty-new}
\begin{aligned}
|\zeta_{a}U|_\infty^2&\leq C_0|\zeta_{a}U|_2 |((\zeta_{a})_rU,\zeta_{a} U_r)|_2\leq C(a,T)(1+ |\zeta_{a}D_\eta U|_2)\\
&\leq C(a,T)\big(1 + \big|r(\zeta_{a}D_\eta U,(\zeta_{a})_rD_\eta U,\zeta_{a}D_\eta^2 U)\big|_2\big) \leq C(a,T)\big(1 + \big|\zeta_{a}r D_\eta^2 U\big|_2\big).
\end{aligned}    
\end{equation}

Next, it follows from  Lemmas \ref{lemma-u-D1} and \ref{hardy-inequality} that
\begin{equation*}
\big|\chi_{a} r^\frac{1}{2}\mathscr{D}_\eta U\big|_4\leq C(a,T)\big|\chi_{a} r^\frac{5}{4}(\mathscr{D}_\eta U,\mathscr{D}_\eta^2 U)\big|_2 \leq C(a,T)\big(\big|\zeta_{a} r\mathscr{D}_\eta^2 U\big|_2+1\big).
\end{equation*}
Consequently, \eqref{in-ell-2}, together with the above, \eqref{U-infty-new}, Lemma \ref{lemma-u-D1}, and the Young inequality, gives that, for all $a\in (0,1)$,
\begin{equation}\label{in-ell-2-new}
|\zeta_{a}U|_\infty^2+\big|\chi_{a} r^\frac{1}{2}\mathscr{D}_\eta U\big|_4+\big|\zeta_{a} r\mathscr{D}_\eta^2 U\big|_2 \leq C(a,T)\big(\big|(r^2\rho_0)^\frac{1}{2}U_t\big|_2+1\big).
\end{equation}

Besides, it follows from \eqref{ex-ell-2} and Lemmas \ref{lemma-upper jacobi} and \ref{lemma-time-space}--\ref{lemma-u-D1} that 
\begin{equation}\label{ex-ell-2-new}
\begin{aligned}
&\big|\chi^\sharp \rho_0^\frac{\delta-\beta}{2}L_*(t)\big|_\infty \leq C(T)\big(\big|(r^2\rho_0)^\frac{1}{2}U_t\big|_2 +1\big),\\
& \big|\chi^\sharp \rho_0^\frac{\delta+\varepsilon}{2}D_\eta^2 U(t)\big|_2 \leq C(\varepsilon,T)\big(\big|(r^2\rho_0)^\frac{1}{2}U_t\big|_2 +1\big)\quad \text{for any $\varepsilon>0$},\\ 
& \big|\chi^\sharp \rho_0^{\iota\beta} U(t)\big|_p\leq C(p,\iota,T)\big(\big|(r^2\rho_0)^\frac{1}{2}U_t\big|_2 +1\big)\quad \text{for any $p\in[2,\infty]$ and $\iota>-\frac{1}{p}$},    
\end{aligned}
\end{equation}
where we have also used the following estimate due to Lemma \ref{hardy-inequality}:
\begin{equation*}
\big|\chi^\sharp \rho_0^{\iota\beta} U(t)\big|_p\leq C(p,T)\big|\chi^\sharp \rho_0^{(\iota+\frac{1}{p})\beta+\frac{3\beta}{2}}(U,D_\eta U,D_\eta^2U)(t)\big|_2.     
\end{equation*}

\smallskip
\textbf{2.}
First, multiplying $\eqref{eq:VFBP-La-eta}_1$ by $\eta^2\eta_r$, and then applying $\partial_t$ to the resulting equality, along with $\eqref{eq:VFBP-La}_1$ and \eqref{v-expression}, yields
\begin{align}
& \ r^2\rho_0U_{tt} - \Big(\underline{A\gamma\varrho^\gamma\big(D_\eta U+\frac{2U}{\eta}\big)}_{:=\mathfrak{L}_1}\Big)_r\eta^2 +\underline{\frac{A\gamma}{a_1\delta}\eta^2\eta_r\varrho^{\gamma-\delta+1}(V-U)\frac{U}{\eta}}_{:=\mathfrak{L}_2}-2r^2\rho_0D_\eta \Phi \frac{U}{\eta} \nonumber\\
&=2a_1\delta\Big(\eta^2 \varrho^\delta \big(D_\eta U_t-\frac{2(1-\delta)}{\delta} \frac{U_t}{\eta}\big)\Big)_r+4a_1\eta^2\eta_r\varrho^\delta\Big((1-\delta)D_\eta U_t- (2\delta-1) \frac{U_t}{\eta}\Big)\frac{1}{\eta}  \label{eq-2order-pre} \\
&\quad -\Big(\underline{2a_1\eta^2\varrho^\delta\big(\delta(1+\delta) |D_\eta U|^2-4\delta (1-\delta) \frac{U}{\eta}D_\eta U-2(1-\delta)(2\delta-1)\frac{U^2}{\eta^2}\big)}_{:=\mathfrak{L}_3}\Big)_r \nonumber\\
&\quad -\underline{ 4a_1\eta^2\eta_r\varrho^\delta\Big(\delta(1-\delta)  |D_\eta U|^2+2(1-\delta)(2\delta-1)\frac{U}{\eta} D_\eta U-2\delta(2\delta-1)  \frac{U^2}{\eta^2}\Big)\frac{1}{\eta}}_{:=\mathfrak{L}_4}.\nonumber
\end{align}

Next, multiplying \eqref{eq-2order-pre} by $U_t$ and integrating over $I$, we have
\begin{equation}\label{dt-J3-J5}
\frac{1}{2}\frac{\mathrm{d}}{\dt}\big|(r^2\rho_0)^\frac{1}{2}U_t\big|_2^2+L_7 = L_8+L_9+L_{10},
\end{equation}
where
\begin{align*}
L_7&:= 2a_1\int_0^1  \eta^2\eta_r \varrho^\delta \Big(\delta|D_\eta U_t|^2-4(1-\delta) \frac{U_t}{\eta} D_\eta U_t+ (4\delta-2) \frac{|U_t|^2}{\eta^2}\Big) \,\mathrm{d}r,\\
L_8&:=-\int_0^1 \eta^{2}\eta_r\mathfrak{L}_1\Big(D_\eta U_t+\frac{2U_t}{\eta}\Big)\,\mathrm{d}r,\qquad L_9:=-\int_0^1\mathfrak{L}_2 U_t\,\mathrm{d}r+2\int_0^1 r^2\rho_0 D_\eta\Phi \frac{U}{\eta}U_t\,\mathrm{d}r,\\
L_{10}&:= \int_0^1 \eta_r \mathfrak{L}_3 D_\eta U_t\,\mathrm{d}r  -  \int_0^1 \mathfrak{L}_4 U_t\,\mathrm{d}r.
\end{align*}

\smallskip
\textbf{2.1. Estimates for $L_7$--$L_9$.} First,  $L_7$ can be handled in the same way as $L_3$ in Step 2.1 of the proof of Lemma \ref{lemma-u-D1}:
\begin{equation}\label{rm-J3}
L_7\geq 2a_1 c_\delta^*\big|(\eta^2\eta_r\varrho^\delta)^{\frac{1}{2}}\sD_\eta U_t\big|_2^2\geq \frac{1}{C(T)}\big|(r^2\rho_0^\delta)^{\frac{1}{2}}\sD_\eta U_t\big|_2^2.
\end{equation}

Next, for $L_8$--$L_9$, it follows from  \eqref{in-ell-2-new}, Lemmas \ref{lemma-bound depth}--\ref{lemma-lower bound jacobi}, \ref{lemma-v Linfty ex}, \ref{lemma-v Linfty in}, and \ref{lemma-u-0order}--\ref{lemma-u-D1}, and the H\"older and Young inequalities that
\begin{equation}\label{rm-J4}
\begin{aligned}
L_8&\leq C(T)\big|(r^2\rho_0^\delta)^\frac{1}{2}\sD_\eta U\big|_2\big|(r^2\rho_0^\delta)^\frac{1}{2}\sD_\eta U_t\big|_2 \leq C(T)+\frac{a_1c_\delta^*}{8} \big|(\eta^2\eta_r\varrho^\delta)^{\frac{1}{2}}\sD_\eta U_t\big|_2^2,\\
L_9&\leq C(T)\big( |\chi^\sharp \rho_0^{\beta+1-\delta} V|_\infty\big|(r^2\rho_0)^\frac{1}{2}U\big|_2+\big|\chi^\sharp\rho_0^\frac{2\gamma-2\delta+1}{4}U\big|_4^2\big)\big|(r^2\rho_0)^\frac{1}{2}U_t\big|_2 \\
&\quad + C(T)\big(|\zeta(V,U)|_\infty+|D_\eta\Phi|_\infty\big)\big|(r^2\rho_0^\delta)^\frac{1}{2}\sD_\eta U\big|_2\big|(r^2\rho_0)^\frac{1}{2}U_t\big|_2\\
&\leq C(T)\big(\big|(r^2\rho_0)^\frac{1}{2}U_t\big|_2^4+1\big).    
\end{aligned}
\end{equation}

\smallskip
\textbf{2.2. Estimate for $L_{10}$.} First, notice that $L_{10}$ can be controlled by the integrals of the following forms, due to Lemmas \ref{lemma-lower bound jacobi} and \ref{lemma-upper jacobi}--\ref{non-vac}: 
\begin{equation}\label{rm-J5}
\begin{aligned}
L_{10}&\leq C(T)\int_0^1 \chi^\sharp r^2\rho_0^\delta |L_*||D_\eta U| |\sD_\eta U_t|\,\mathrm{d}r +C(T)\int_0^1 \chi^\sharp r^2\rho_0^\delta(|L_*|+|U|)|U| |\sD_\eta U_t|\,\mathrm{d}r\\
&\quad +C(T)\int_0^1 \chi r^2\rho_0^\delta |\sD_\eta U|^2|\sD_\eta U_t|\,\mathrm{d}r :=L_{10,1}^\sharp+L_{10,2}^\sharp +L_{10}^\flat.
\end{aligned}
\end{equation}
For simplicity, we only sketch the estimates for $L_{10,1}^\sharp$ and $L_{10}^\flat$, and $L_{10,2}^\sharp$ can be treated in an easier way similar to $L_{10,1}^\sharp$.

Then, for $L_{10,1}^\sharp$--$L_{10}^\flat$, we obtain from the fact that $\frac{3\beta}{2}>\frac{1}{2}>\frac{\delta}{2}$, \eqref{in-ell-2-new}--\eqref{ex-ell-2-new}, Lemmas \ref{lemma-lower bound jacobi}, \ref{lemma-u-0order}--\ref{lemma-u-D1}, and \ref{hardy-inequality}, and the H\"older and Young inequalities that 
\begin{align*}
L_{10,1}^\sharp&\leq C(T) \big|\chi^\sharp\rho_0^\frac{\delta-\beta}{2}L_*\big|_\infty \big|\chi^\sharp\rho_0^\frac{\beta}{2}D_\eta U\big|_2 \big|(r^2\rho_0^\delta)^\frac{1}{2}\sD_\eta U_t \big|_2\\[-2pt]
&\leq C(T) \big(\big|(r^2\rho_0)^\frac{1}{2}U_t\big|_2 +1\big) \big|\chi^\sharp\rho_0^\frac{3\beta}{2}(D_\eta U,D_\eta^2 U)\big|_2 \big|(r^2\rho_0^\delta)^\frac{1}{2}\sD_\eta U_t \big|_2\\[2pt]
&\leq C(T)\big(\big|(r^2\rho_0)^\frac{1}{2}U_t\big|_2^4+1\big)+\frac{a_1c_\delta^*}{8} \big|(\eta^2\eta_r\varrho^\delta)^{\frac{1}{2}}\sD_\eta U_t \big|_2^2,\\
L_{10}^\flat&\leq C(T)\big|\chi r^\frac{1}{2}\sD_\eta U\big|_4^2\big|(r^2\rho_0^\delta)^\frac{1}{2}\sD_\eta U_t \big|_2 \leq C(T)\big(\big|(r^2\rho_0)^\frac{1}{2}U_t\big|_2^4+1\big) +\frac{a_1c_\delta^*}{8}\big|(\eta^2\eta_r\varrho^\delta)^{\frac{1}{2}}\sD_\eta U_t \big|_2^2.
\end{align*}

Combining \eqref{rm-J5} with the above thus leads to
\begin{equation}\label{rm-J5'}
L_{10}\leq C(T)\big(\big|(r^2\rho_0)^\frac{1}{2}U_t\big|_2^4+1\big) +\frac{a_1c_\delta^*}{2}\big|(\eta^2\eta_r\varrho^\delta)^{\frac{1}{2}}\sD_\eta U_t \big|_2^2.   
\end{equation}

 Collecting \eqref{dt-J3-J5}--\eqref{rm-J4} and \eqref{rm-J5'}, and applying the Gr\"onwall inequality, we obtain from Lemma \ref{lemma-u-D1} that \eqref{tan1} holds. Clearly, \eqref{tan1}, together with \eqref{in-ell-2-new}--\eqref{ex-ell-2-new}, yields \eqref{tan2}. 
\end{proof}

\subsubsection{Third-order tangential estimates}\label{subsub914}
We first give the interior $L^2$-estimates for $(D_\eta V,\frac{V}{\eta})$.
\begin{Lemma}\label{lemma-Vr-L2}
For any $a\in(0,1)$, there exists a constant $C(a,T)>0$ such that
\begin{equation*}
\big|\zeta_a r\sD_\eta V(t)\big|_2 \leq C(a,T) \qquad \text{for all $t\in[0,T]$}.
\end{equation*}    
\end{Lemma}
\begin{proof}
It suffices to give the proof for $a=\frac{1}{2}$. First, we obtain from the fact that $\rho_0^\beta\sim 1-r$ and Lemmas \ref{lemma-upper jacobi}, \ref{lemma-u-D1}--\ref{lemma-u-D2}, and \ref{hardy-inequality} that 
\begin{equation}\label{DU-L2}
\begin{aligned}
|\zeta \sD_\eta U|_2&\leq C(T)\big|\chi_{\frac{5}{8}}r (\sD_\eta U ,\sD_\eta^2 U)\big|_2\leq C(T)\big(\big|(r^2\rho_0^\delta)^\frac{1}{2}\sD_\eta U \big|_2+1\big) \leq C(T),
\end{aligned}
\end{equation}
and obtain from \eqref{D Phi expression} and Lemmas \ref{lemma-lower bound jacobi} and \ref{lemma-upper jacobi} that
\begin{equation}\label{DD-phi}
|\zeta_\frac{5}{8} r\sD_rD_\eta \Phi|_\infty\leq C(T)\Big(\Big|\zeta_\frac{5}{8}\frac{1}{r^3}\int_0^r\hat r^2\rho_0\,\mathrm{d}\hat r\Big|_\infty+|\rho_0|_\infty\Big)\leq C(T).
\end{equation}

Next, applying $\zeta^2 r^2 \sD_r V\sD_r$ to \eqref{eq:v} and integrating  over $I$, then we obtain from \eqref{DU-L2}--\eqref{DD-phi}, Lemmas \ref{lemma-bound depth}, \ref{lemma-v Linfty in}, \ref{lemma-upper jacobi}, and \ref{lemma-u-D1}--\ref{lemma-u-D2}, and the H\"older and Young inequalities that 
\begin{equation*}
\frac{\mathrm{d}}{\dt} |\zeta r \sD_r V|_2^2 \leq C(T)\big(|\zeta r \sD_\eta U|_2 + \big|\zeta_{\frac{5}{8}}(V,U)\big|_\infty^2+ |\zeta r\sD_rD_\eta \Phi|_\infty\big)|\zeta r\sD_rV|_2 \leq C(T) |\zeta r \sD_r V|_2,
\end{equation*}
which, along with the Gr\"onwall inequality and  Lemma \ref{lemma-lower bound jacobi}, leads to the desired estimate. The $L^2$-boundedness of the initial value $\zeta r \sD_r v_0$ is due to \eqref{distance-la} and Lemma \ref{lemma-initial}.
\end{proof}

\begin{Lemma}\label{lemma-u-D3}
There exists a constant $C(T)>0$ such that
\begin{equation}\label{tan3}
\big|(r^2\rho_0^\delta)^\frac{1}{2} \sD_\eta U_t (t)\big|_2^2+\int_0^t\big|(r^2\rho_0)^\frac{1}{2}U_{tt}\big|_2^2\,\mathrm{d}s\leq C(T) \qquad\text{for all $t\in[0,T]$}.
\end{equation}
In addition, for any $t\in[0,T]$, $a\in (0,1)$, the following estimates also hold{\rm:}
\begin{equation}\label{tan4}
|(U,\sD_\eta U)(t)|_\infty\leq C(T),\qquad \big|\zeta_{a}r\sD_\eta^3 U(t)\big|_2\leq C(a,T).
\end{equation}
\end{Lemma}
\begin{proof}
We divide the proof into two steps.

\smallskip
\textbf{1.} We first give some auxiliary estimates associated with the third-order derivatives of $U$.

\smallskip
\textbf{1.1. Interior estimates.} First, it follows from \eqref{DU-L2} (clearly, \eqref{DU-L2} also holds for $\zeta\mapsto\zeta_a$)  and Lemmas \ref{lemma-upper jacobi}, \ref{lemma-u-D2}, and \ref{sobolev-embedding}--\ref{hardy-inequality} that, for all $a\in (0,1)$,
\begin{equation}\label{in-Ur-infty}
\begin{aligned}
|\zeta_{a}\sD_\eta U|_\infty^2&\leq C(a,T)|\zeta_{a}\sD_\eta U|_2\big|\big((\zeta_{a})_r \sD_\eta U,\zeta_{a}\sD_\eta^2 U\big)\big|_2 \\
&\leq C(a,T)\big(1+|\zeta_{a}\sD_\eta^2 U|_2\big)\leq C(a,T)\big(1+|\zeta_{a}rD_\eta(\sD_\eta^2 U)|_2\big).
\end{aligned}
\end{equation}

Next, applying $\sD_\eta$ to \eqref{ell-0}, along with \eqref{v-expression}, yields
\begin{equation*}
\sD_\eta\Big(D_\eta \big(D_\eta U+ \frac{2U}{\eta}\big)\Big)=L_{11},
\end{equation*}
where $L_{11}$ has the control of the following form in view of Lemma \ref{non-vac}:
\begin{equation*} 
\begin{aligned}
L_{11}&\leq C(T)\rho_0^{1-\delta}\big(|\sD_\eta U_{t}|+ \rho_0^{1-\delta}|(V,U)||U_t|+|(V,U)| |\sD_\eta^2 U|\big)+C(T)|(\sD_\eta V,\sD_\eta U)||\sD_\eta U| \\
&\quad + C(T)\rho_0^{1-\delta}|(V,U)|^2|\sD_\eta U|+C(T)\rho_0^{1-\delta}\big(\rho_0^{1-\delta}|(V,U)| |D_\eta\Phi| + |\sD_\eta (D_\eta \Phi)|\big) \\
&\quad +C(T)\rho_0^{\gamma-2\delta+1}\big(|(\sD_\eta V,\sD_\eta U)|+ \rho_0^{1-\delta}|(V,U)|^2\big).
\end{aligned}
\end{equation*}

Then  we can obtain the following inequality from \eqref{in-Ur-infty}, Lemmas \ref{lemma-bound depth}--\ref{lemma-lower bound jacobi}, \ref{lemma-v Lp ex}, \ref{lemma-v Linfty in}, \ref{lemma-upper jacobi}, and \ref{lemma-u-D1}--\ref{lemma-Vr-L2}, and the H\"older inequality:
\begin{equation}\label{in-ell-3-pre1}
\Big|\zeta_{a} r\sD_\eta\Big(D_\eta\big(D_\eta U+ \frac{2U}{\eta}\big)\Big)\Big|_2\leq  C(a,T)\big(|\zeta_{a}r\sD_\eta U_t|_2+|\zeta_{a}\sD_\eta U|_\infty +1\big),
\end{equation}
where the following estimate on $\sD_\eta (D_\eta \Phi)$ has also been used in view of \eqref{D Phi expression}: 
\begin{equation}\label{DDphi}
|\sD_\eta (D_\eta \Phi)|_\infty \leq C(T)\Big(\Big|\frac{1}{r^3}\int_0^r \hat r^2\rho_0\,\mathrm{d}\hat r\Big|_\infty+|\rho_0|_\infty\Big)\leq C(T).
\end{equation}
Thus, \eqref{in-ell-3-pre1}, combined with \eqref{in-Ur-infty}, the Young inequality, and Lemma \ref{im-1}, yields that 
\begin{equation}\label{in-ell-3}
|\zeta_{a}\sD_\eta U|_\infty^2+\big|\zeta_{a}r\sD_\eta^3 U \big|_2 \leq C(a,T)\big(\big|(r^2\rho_0^\delta)^\frac{1}{2} \sD_\eta U_t\big|_2+1\big)\qquad\text{for all } a\in (0,1) . 
\end{equation}

\textbf{1.2. Exterior estimates.} First, to show that $\chi^\sharp(U,\sD_\eta U)$ is bounded, it suffices to control the $L^\infty$-norm of $\chi^\sharp(U,D_\eta U)$ due to Lemma \ref{lemma-lower bound jacobi}. Then, since $\frac{3\beta}{2}>\frac{1}{2}>\frac{\delta}{2}$, it follows from Lemmas \ref{lemma-upper jacobi}, \ref{lemma-time-space} (with $\iota=\frac{\beta}{2}$), \ref{lemma-u-0order}--\ref{lemma-u-D2}, and \ref{sobolev-embedding}--\ref{hardy-inequality} and the H\"older inequality that
\begin{equation}\label{Ur-infty-high}
\begin{aligned}
&|\chi^\sharp (U,D_\eta U)|_\infty  \leq C(T)\Big(|\chi^\sharp (U,D_\eta U)|_2+|\chi^\sharp L_*|_\infty\Big) \\ 
&\qquad\leq C(T)\big(1+\big|\chi^\sharp\rho_0^\frac{3\beta}{2}(D_\eta U,D_\eta^2 U,U_t,D_\eta U_t)\big|_2\big) \leq C(T)\big(1+\big|(r^2\rho_0^\delta)^\frac{1}{2}\sD_\eta U_t\big|_2\big).
\end{aligned}    
\end{equation}

\smallskip
\textbf{2. Third-order tangential estimates.} Now, multiplying \eqref{eq-2order-pre} by $U_{tt}$ and integrating over $I$, we obtain 
\begin{equation}\label{dt-J6-J11}
\frac{1}{2}\frac{\mathrm{d}}{\mathrm{d}t}L_7+\big|(r^2\rho_0)^\frac{1}{2}U_{tt}\big|_2^2=L_{12}+L_{13}+L_{14},
\end{equation}
where $L_7$ is defined as in \eqref{dt-J3-J5} and
\begin{align*}
&\begin{aligned}
L_{12}&:= -4a_1(1-\delta)\int_0^1 \eta^2 \eta_r \varrho^\delta\Big(-\delta D_\eta U+(1-2\delta)\frac{U}{\eta}\Big) D_\eta U_t\frac{U_t}{\eta}\,\mathrm{d}r\\
&\quad \ +2a_1(2\delta-1)\int_0^1 \eta^2\eta_r\varrho^\delta\Big((1-\delta)D_\eta U-2\delta\frac{U}{\eta}\Big)\frac{|U_t|^2}{\eta^2} \,\mathrm{d}r\\
&\quad \ -a_1\delta\int_0^1  \eta^2\eta_r \varrho^\delta\Big( (1+\delta)D_\eta U+2 (1-\delta)\frac{U}{\eta}\Big)  |D_\eta U_t|^2 \,\mathrm{d}r,    
\end{aligned}\\
&\begin{aligned}
L_{13}&:= \int_0^1 \eta_r \mathfrak{L}_3 D_\eta U_{tt}\,\mathrm{d}r  - \int_0^1 \mathfrak{L}_4 U_{tt}\,\mathrm{d}r,
\end{aligned}\\
&\begin{aligned}
L_{14}&:=\int_0^1 \big(\mathfrak{L}_1\big)_r\eta^2U_{tt}\,\mathrm{d}r-\int_0^1 \mathfrak{L}_2 U_{tt}\,\mathrm{d}r +2\int_0^1 r^2\rho_0D_\eta \Phi \frac{U}{\eta}U_{tt}\,\mathrm{d}r.
\end{aligned}
\end{align*}

\smallskip
\textbf{2.1. Estimate for $L_{12}$.} First, it follows from Lemmas \ref{lemma-upper jacobi}--\ref{non-vac}, and  \eqref{in-ell-3}--\eqref{Ur-infty-high} that
\begin{equation}\label{L12}
L_{12}\leq C(T) \int_0^1 r^2\rho_0^\delta|\sD_\eta U| |\sD_\eta U_t|^2\,\mathrm{d}r\leq C(T)\big(\big|(r^2\rho_0^\delta)^\frac{1}{2}\sD_\eta U_t\big|_2^4+1\big).
\end{equation}

\smallskip
\textbf{2.2. Estimate for $L_{13}$.} First, $L_{13}$ can be rewritten as follows in view of the chain rules with respect to $\partial_t$ (in $D_\eta U_{tt}$ and $\frac{U_{tt}}{\eta}$) and \eqref{eq:eta}:
\begin{equation}\label{L13-0}
L_{13}=\underline{\int_0^1 (\mathfrak{L}_4)_tU_t-\eta_r (\mathfrak{L}_3)_t D_\eta U_{t} \,\mathrm{d}r}_{:=L_{13,1}} +\frac{\mathrm{d}}{\mathrm{d}t} \underline{\int_0^1 \eta_r \mathfrak{L}_3 D_\eta U_{t}- \mathfrak{L}_4U_t \,\mathrm{d}r}_{:=L_{13,2}},
\end{equation}
where, by Lemmas \ref{lemma-upper jacobi}--\ref{non-vac}, $L_{13,1}$ and $L_{13,2}$ have controls of the following form:
\begin{align*}
L_{13,1}&\leq C(T)\int_0^1 r^2\rho_0^\delta\big(|\sD_\eta U|^3+|\sD_\eta U| |\sD_\eta U_t|\big)|\sD_\eta U_t|\,\mathrm{d}r \\
L_{13,2}&\leq C(T)\int_0^1 r^2\rho_0^\delta |\sD_\eta U|^2 |\sD_\eta U_t|\,\mathrm{d}r.  
\end{align*}

Next, by \eqref{in-ell-3}--\eqref{Ur-infty-high}, $L_{13,1}$ can be treated in a similar way to that of $L_{12}$:
\begin{equation}\label{L131}
L_{13,1}\leq C(T)\big(\big|(r^2\rho_0^\delta)^\frac{1}{2}\sD_\eta U_t\big|_2^4+1\big).
\end{equation}
For $L_{13,2}$, we can first deduce from Lemmas \ref{lemma-lower bound jacobi}, \ref{lemma-upper jacobi}, \ref{lemma-u-0order}--\ref{lemma-u-D2}, and \ref{hardy-inequality} that
\begin{equation*}
\big|(r^2\rho_0^\delta)^\frac{1}{4}\sD_\eta U\big|_4\leq C(T)\Big(\big|\chi r^\frac{5}{4} (\sD_\eta U,\sD_\eta^2 U)\big|_2 + \big|\chi^\sharp \rho_0^\frac{\delta+3\beta}{4}(U,D_\eta U,D_\eta^2 U)\big|_2\Big) \leq C(T),    
\end{equation*}
where we have also used \eqref{tan2} and the fact that $\frac{\delta+3\beta}{4}>\frac{\delta}{2}$. Hence, it follows that
\begin{equation}\label{l4-fuzhu}
L_{13,2}\leq C(T)\big|(r^2\rho_0^\delta)^\frac{1}{4}\sD_\eta U\big|_4^2\big|(r^2\rho_0^\delta)^\frac{1}{2}\sD_\eta U_t\big|_2\leq C(T)\big|(r^2\rho_0^\delta)^\frac{1}{2}\sD_\eta U_t\big|_2.
\end{equation}

\smallskip
\textbf{2.3. Estimate for $L_{14}$.} For $L_{14}$, we first obtain from \eqref{v-expression} and Lemmas \ref{lemma-upper jacobi}--\ref{non-vac} that
\begin{equation}\label{rm-J60}
\begin{aligned}
L_{14}& \leq C(T)\int_0^1  r^2 \big(\rho_0^{\gamma-\delta+1}|V-U||\sD_\eta U|+\rho_0^\gamma|\sD_\eta^2 U|\big) |U_{tt}|\,\mathrm{d}r\\
&\quad +C(T)\int_0^1 r^2 \rho_0 |D_\eta\Phi| |\sD_\eta U||U_{tt}|\,\mathrm{d}r:=L_{14,1}+L_{14,2}.
\end{aligned}
\end{equation}

Clearly, $L_{14,2}$ can be controlled by Lemma \ref{lemma-v Lp ex} and  \eqref{in-ell-3}--\eqref{Ur-infty-high}:
\begin{equation}
L_{14,2}\leq C(T)\big(\big|(r^2\rho_0^\delta)^\frac{1}{2}\sD_\eta U_t\big|_2^2+1\big)+\frac{1}{8}\big|(r^2\rho_0)^\frac{1}{2} U_{tt}\big|_2^2.
\end{equation}
For $L_{14,1}$, it follows from \eqref{Ur-infty-high}, Lemmas \ref{lemma-lower bound jacobi}, \ref{lemma-v Linfty ex}--\ref{lemma-v Linfty in}, and \ref{lemma-u-D1}--\ref{lemma-u-D2}  that 
\begin{align}
L_{14,1}&\leq C(T)\big(|\zeta(V,U)|_\infty \big|(r^2\rho_0^\delta)^\frac{1}{2}\sD_\eta U\big|_2 + \big|\zeta r \sD_\eta^2 U\big|_2\big)\big|(r^2\rho_0)^\frac{1}{2}U_{tt}\big|_2\nonumber\\
&\quad + C(T) |\chi^\sharp (\rho_0^{\beta+1-\delta} V,U)|_\infty \big|(r^2\rho_0^\delta)^\frac{1}{2}\sD_\eta U\big|_2\big|(r^2\rho_0)^\frac{1}{2}U_{tt}\big|_2\label{jj61}\\
&\quad +C(T)\big|\chi^\sharp\rho_0^{\gamma-\frac{1}{2}}(U,D_\eta U,D_\eta^2 U)\big|_2 \big|(r^2\rho_0)^\frac{1}{2}U_{tt}\big|_2 \nonumber\\
&\leq C(T)\big(\big|(r^2\rho_0^\delta)^\frac{1}{2}\sD_\eta U_t\big|_2^4+1\big)+\frac{1}{8}\big|(r^2\rho_0)^\frac{1}{2}U_{tt}\big|_2^2.\nonumber
\end{align}

\smallskip
\textbf{2.4. Close the energy estimate.} Collecting \eqref{dt-J6-J11}--\eqref{jj61} gives 
\begin{equation*}
\frac{1}{2}\frac{\mathrm{d}}{\dt}L_7+\frac{1}{2}\big|(r^2\rho_0)^\frac{1}{2}U_{tt}\big|_2^2 \leq \frac{\mathrm{d}}{\dt}L_{13,2}+C(T)\big(\big|(r^2\rho_0^\delta)^\frac{1}{2}\sD_\eta U_t\big|_2^4+1\big).
\end{equation*}
Applying the Gr\"onwall inequality to the above, we obtain \eqref{tan3} from \eqref{rm-J3}, \eqref{l4-fuzhu}, Lemma \ref{lemma-u-D2}, and the Young inequality.  Finally, \eqref{tan3}, together with \eqref{in-ell-3}--\eqref{Ur-infty-high}, yields \eqref{tan4}.
\end{proof}

\subsection{Interior estimates of the velocity}\label{sub92}
\begin{Lemma}\label{ell-inner}
There exists a constant $C(T)>0$ such that
\begin{equation*}
\cE_{\mathrm{in}}(t,U)+\int_0^t \cD_{\mathrm{in}}(s,U)\,\mathrm{d}s \leq C(T) \qquad\text{for all $t\in [0,T]$}.
\end{equation*}
\end{Lemma}
\begin{proof}
We divide the proof into three steps.

\smallskip
\textbf{1.} It follows from  Lemmas \ref{lemma-u-0order}--\ref{lemma-u-D2}, and \ref{lemma-u-D3} that 
\begin{equation}\label{total-Ein}
\cE_{\mathrm{in}}(t,U)\leq C(T) \qquad\text{for all $t\in [0,T]$}.
\end{equation}
Then,  according to  Lemma \ref{lemma-u-D3}, it remains to show that
\begin{equation}\label{D-txx/D-xxxx}
\int_0^t\big|\zeta r\sD_\eta^2 U_{t}\big|_2^2\,\mathrm{d}s\leq C(T),\quad \int_0^t\big|\zeta r\sD_\eta^4 U\big|_2^2\,\mathrm{d}s\leq C(T) \qquad\text{for all $t\in [0,T]$}.
\end{equation}

\smallskip
\textbf{2. Proof of $\eqref{D-txx/D-xxxx}_1$.} First, it follows from \eqref{eq:v}, Lemmas \ref{lemma-bound depth}, \ref{lemma-v Lp ex}, \ref{lemma-v Linfty in}, and \ref{lemma-u-D3} that
\begin{equation}\label{est-vt-0}
|\zeta r V_t|_2\leq |\zeta V_t|_\infty \leq C_0\big(|\varrho|_\infty^{\gamma-\delta} |\zeta (V,U)|_\infty+ |D_\eta \Phi|_\infty\big)\leq C(T) \qquad\text{for all $t\in [0,T]$}.
\end{equation}

Next, applying $\partial_t$ to \eqref{ell-0}, along with $\eqref{eq:VFBP-La}_1$ and \eqref{D Phi expression}, we have
\begin{equation}\label{ell-0-t-0}
D_\eta\big(D_\eta U_t+ \frac{2U_t}{\eta}\big)=L_{15},\end{equation}
where $L_{15}$ has the control of the following form in view of Lemmas \ref{lemma-upper jacobi}--\ref{non-vac}:
\begin{equation*}
\begin{aligned}
L_{15}&\leq  C(T)\rho_0^{1-\delta}\big(|U_{tt}|+ |\sD_\eta U| |U_t| \big)+ C(T) \rho_0^{\gamma-2\delta+1} \big(|(V_t,U_t)|+|\sD_\eta U| |(V,U)|\big)  \\
&\quad +C(T)\rho_0^{1-\delta}\big(|\sD_\eta U||(D_\eta\Phi,V_t,U_t)|+|(V,U)| (|\sD_\eta U|^2+ |\sD_\eta U_t|)\big)+|\sD_\eta U| |\sD_\eta^2 U|.
\end{aligned}
\end{equation*}
Then we obtain from \eqref{total-Ein}, \eqref{est-vt-0}, and Lemmas \ref{lemma-bound depth}, \ref{lemma-v Lp ex}, \ref{lemma-v Linfty in}, \ref{lemma-u-D2}, and \ref{lemma-u-D3} that
\begin{align}
|\zeta r L_{15}|_2&\leq C(T)\big(|\zeta r U_{tt}|_2+ |\sD_\eta U|_\infty\big|\zeta rU_t\big|_2 
+ |\zeta r (V_t,U_t)|_2+|\zeta(V,U)|_\infty|\sD_\eta U|_\infty\big) \nonumber\\
&\quad + C(T)|\chi_\frac{5}{8}(\sD_\eta U,V,U)|_\infty\big(|\zeta r (V_t,U_t,\sD_\eta U_t)|_2+ |\sD_\eta U|_\infty^2+|D_\eta\Phi|_\infty\big)\label{proof0}\\
&\quad+|\sD_\eta U|_\infty\big|\zeta r\sD_\eta^2 U\big|_2 
\leq C(T)(|\zeta r U_{tt}|_2+1),\nonumber
\end{align}
which, along with \eqref{ell-0-t-0} and Lemmas \ref{im-1} and \ref{lemma-u-D3}, leads to $\eqref{D-txx/D-xxxx}_1$.

\smallskip
\textbf{3. Proof of $\eqref{D-txx/D-xxxx}_2$.} We first show that, for any $t\in[0,T]$,
\begin{equation}\label{high-vr-vrr}
\big|\zeta r^\frac{1}{2}\sD_\eta V(t)\big|_4+\big|\zeta r\sD_\eta^2 V(t)\big|_2\leq C(T).
\end{equation}

Indeed, due to Lemmas \ref{lemma-upper jacobi}, \ref{lemma-Vr-L2}--\ref{lemma-u-D3}, and \ref{hardy-inequality}, we have
\begin{equation}\label{vr-urr-l4}
\begin{aligned}
\big|\zeta r^\frac{1}{2}\sD_\eta V\big|_4&\leq C(T) \big|r^\frac{5}{4}(\zeta \sD_\eta V,\zeta_r \sD_\eta V,\zeta \sD_\eta^2 V)\big|_2 \leq C(T)\big(\big|\zeta r\sD_\eta^2 V \big|_2+1\big),\\
\big|\chi_\frac{5}{8} r^\frac{1}{2}\sD_\eta^2 U\big|_4&\leq C(T)
\big|\chi_\frac{5}{8} r^\frac{5}{4}(\sD_\eta^2 U, D_\eta(\sD_\eta^2 U))\big|_2\leq C(T)\big(
\big|\chi_\frac{5}{8} r  \sD_\eta^3 U\big|_2+1\big)\leq C(T),
\end{aligned}    
\end{equation}
and, by \eqref{D Phi expression}, \eqref{v-expression}, and Lemma \ref{lemma-lower bound jacobi}, we obtain that, for \textit{a.e.} $(t,r)\in (0,T)\times (0,1)$,
\begin{equation}\label{vr-urr-l5}
|\sD_\eta^2(D_\eta\Phi)|\leq C(T)\Big(\Big|\frac{1}{r^4}\int_0^r \hat r^2\rho_0\,\mathrm{d}\hat r\Big|+\frac{\rho_0}{r}+\rho_0^{2-\delta}|(V,U)|\Big).
\end{equation}

Next, by \eqref{D Phi expression} and \eqref{v-expression}--\eqref{eq:v}, we have 
\begin{equation}\label{DDV0}
\Big(D_\eta\big(D_\eta V+\frac{2V}{\eta}\big)\Big)_t+\frac{A\gamma}{2a_1\delta}\varrho^{\gamma-\delta}D_\eta\big(D_\eta V+\frac{2V}{\eta}\big)=L_{16}, 
\end{equation}
where $L_{16}$ has the control of the following form in view of Lemma \ref{non-vac}:
\begin{equation*}
\begin{aligned}
L_{16}&\leq C(T) \big(|\sD_\eta U| |\sD_\eta^2 V|+ |\sD_\eta^2 U| |\sD_\eta V|+ |\sD_\eta^2(D_\eta\Phi)| +\rho_0^{\gamma-3\delta+2}|(V,U)|^3\big)\\
&\quad +C(T)\big(\rho_0^{\gamma-2\delta+1}|(V,U)| |(\sD_\eta U,\sD_\eta V)| + \rho_0^{\gamma-\delta} |\sD_\eta^2 U|\big).
\end{aligned}
\end{equation*}
Hence, solving ODE \eqref{DDV0} with respect to $D_\eta (D_\eta V+\frac{2V}{\eta})$, along with Lemma \ref{im-1}, yields
\begin{equation}\label{DDV}
\big|\zeta r\sD_\eta^2 V\big|_2\leq C(T)\Big(\big|\zeta r\sD_r^2 v_0\big|_2 +\int_0^t|\zeta rL_{16}|_2\,\mathrm{d}\tau\Big)\leq C(T)\Big(1 +\int_0^t|\zeta rL_{16}|_2\,\mathrm{d}\tau\Big),
\end{equation}
where the $L^2$-boundedness of $\zeta r\sD_r^2 v_0$ is due to \eqref{distance-la} and the fact that $\rho_0^\beta\sim 1-r$.

Hence, for $L_{16}$, using \eqref{total-Ein}, \eqref{vr-urr-l4}--\eqref{vr-urr-l5}, Lemmas \ref{lemma-bound depth}, \ref{lemma-v Linfty in}, and \ref{lemma-Vr-L2}--\ref{lemma-u-D3}, and the H\"older inequality, we have
\begin{equation*}
\begin{aligned}
|\zeta rL_{16}|_2&\leq C(T)\big(|\sD_\eta U|_\infty\big|\zeta r\sD_\eta^2 V\big|_2 +\big|\zeta r^\frac{1}{2}\sD_\eta V\big|_4\big|\zeta r^\frac{1}{2}\sD_\eta^2 U\big|_4+\big|\zeta r\sD_\eta^2(D_\eta\Phi)\big|_2\big) \\
&\quad +C(T)\big(|\chi_\frac{5}{8}(V,U)|_\infty^3+|\chi_\frac{5}{8}(V,U)|_\infty\big|\zeta r(\sD_\eta U,\sD_\eta V)\big|_2+\big|\zeta r\sD_\eta^2 U\big|_2\big)\\
&\leq  C(T)\big(\big|\zeta r\sD_\eta^2 V\big|_2+1\big),
\end{aligned}
\end{equation*}
which, together with \eqref{DDV}, the Gr\"onwall inequality and \eqref{vr-urr-l4}, implies claim \eqref{high-vr-vrr}.

Now, applying $\sD_\eta^2$ to \eqref{ell-0}, along with \eqref{v-expression}, we have
\begin{align}
\sD_\eta^2\Big(D_\eta \big(D_\eta U+ \frac{2U}{\eta}\big)\Big)&=\underline{\frac{1}{2a_{1}\delta}\sD_\eta^2(\varrho^{1-\delta} U_t)}_{:=L_{17}}-\underline{\sD_\eta^2\big(\frac{1}{2a_{1}}\varrho^{1-\delta}(V-U)L_*\big)}_{:=L_{18}}\nonumber\\
&\quad +\underline{\sD_\eta^2\Big(\frac{A\gamma}{4a_{1}^2\delta^2} \varrho^{\gamma-2\delta+1}(V-U)+\frac{1}{2a_{1}\delta}\varrho^{1-\delta}D_\eta\Phi\Big)}_{:=L_{19}}.\label{DDDU}
\end{align}
Then $L_{17}$--$L_{19}$ have controls of the following form in view of Lemma \ref{non-vac}:
\begin{align*}
L_{17}&\leq C(T)\rho_0^{1-\delta}\big(|\sD_\eta^2 U_{t}|+ \rho_0^{1-\delta} |(V,U)||\sD_\eta U_{t}|+\rho_0^{1-\delta}|(D_\eta V, D_\eta U)| |U_t|+\rho_0^{2-2\delta}|(V,U)|^2 |U_t|\big),\\
L_{18}&\leq C(T)\rho_0^{1-\delta} \big(\rho_0^{1-\delta} |(V,U)|^2 |\sD_\eta^2 U|+|(\sD_\eta V,\sD_\eta U)| |\sD_\eta^2 U| + |(V,U)| |\sD_\eta^3 U|\big)\\
&\quad +C(T)\rho_0^{1-\delta}\big(\rho_0^{2-2\delta}|(V,U)|^3+\rho_0^{1-\delta}|(V,U)| |(\sD_\eta V,\sD_\eta U)| + |(\sD_\eta^2 V,\sD_\eta^2 U)| \big) |\sD_\eta U|,\\
L_{19}&\leq C(T)\rho_0^{2-2\delta} \big(\rho_0^{1-\delta}|(V,U)|^2 + |(D_\eta V,D_\eta U)|\big)|D_\eta\Phi| \\
&\quad +  C(T)\rho_0^{1-\delta}\big(\rho_0^{1- \delta}|(V,U)| |\sD_\eta(D_\eta\Phi)| +  |\sD_\eta^2(D_\eta\Phi)|\big)\\
&\quad +C(T)\rho_0^{\gamma-2\delta+1}\big(\rho_0^{2-2\delta}|(V,U)|^3 +\rho_0^{1-\delta}|(V,U)||(\sD_\eta V,\sD_\eta U)|+ |(\sD_\eta^2 V,\sD_\eta^2 U)|\big). 
\end{align*}

For $L_{17}$, we can first obtain from a calculation similar to \eqref{vr-urr-l4} that
\begin{equation}\label{Url4}
|\chi_\frac{5}{8}r^\frac{1}{2}\sD_\eta U|_4+|\chi_\frac{5}{8}r^\frac{1}{2}U_t|_4\leq C(T) \big(\big|\chi_\frac{5}{8} r  \sD_\eta^2 U\big|_2+\big|\chi_\frac{5}{8} r \sD_\eta U_t\big|_2+1\big)\leq C(T).
\end{equation}
Then, based on \eqref{total-Ein}, \eqref{high-vr-vrr}, \eqref{Url4}, and Lemmas \ref{lemma-v Linfty in} and \ref{lemma-u-D3}, we obtain
\begin{equation}\label{4jie-00}
|\zeta r L_{17}|_2\leq  C(T)\big(\big|\zeta r \sD_\eta^2 U_t\big|_2+1\big).
\end{equation}

For $L_{18}$,  it follows from \eqref{total-Ein}, \eqref{high-vr-vrr}, $\eqref{vr-urr-l4}_2$, \eqref{Url4}, and Lemmas \ref{lemma-v Linfty in} and \ref{lemma-u-D3} that
\begin{equation}\label{4jie-000}
\begin{aligned}
|\zeta rL_{18}|_2&\leq C(T)\Big(\sum_{j=1}^3|\chi_\frac{5}{8}(V,U)|_\infty^{4-j} |\zeta r\sD_\eta^j U|_2+\big|\zeta r^\frac{1}{2} (\sD_\eta V,\sD_\eta U)\big|_4\big|\chi_\frac{5}{8} r^\frac{1}{2}\sD_\eta^2 U \big|_4\Big)\\
&\ \  +C(T)\big(|\chi_\frac{5}{8}(V,U)|_\infty\big|\zeta r (\sD_\eta V,\sD_\eta U)\big|_2 +\big|\zeta r (\sD_\eta^2 V,\sD_\eta^2 U)\big|_2\big)|\sD_\eta U|_\infty\leq C(T).
\end{aligned}
\end{equation}

Finally, since the weight function $\rho_0$ in $L_{19}$ has a higher degree and the derivative order of $U$ is lower, we can obtain the estimate $|\zeta rL_{19}|_2 \leq C(T)$ from Lemma \ref{lemma-v Lp ex}, \eqref{DDphi}, \eqref{vr-urr-l5}, and a method similar to that used for controlling $L_{18}$. 
Thus, the estimate of $L_{19}$, together with \eqref{DDDU}, \eqref{4jie-00}--\eqref{4jie-000}, and Lemma \ref{im-1}, implies
\begin{equation}\label{proof1}
\big|\zeta r\sD_\eta^4 U\big|_2 \leq  C(T)\big(\big|\zeta r \sD_\eta^2 U_t\big|_2+1\big)\qquad \text{for all $t\in[0,T]$},
\end{equation}
which, along with $\eqref{D-txx/D-xxxx}_1$, leads to $\eqref{D-txx/D-xxxx}_2$.
\end{proof}

\subsection{Exterior estimates of the velocity}\label{sub93}
We aim to derive the following exterior estimates:
\begin{Lemma}\label{ell-ex}
There exists a constant $C(T)>0$ such that
\begin{equation*}
\cE_{\mathrm{ex}}(t,U)+\int_0^t \cD_{\mathrm{ex}}(s,U)\,\mathrm{d}s\leq C(T) \qquad\text{for all $t\in[0,T]$}.
\end{equation*}
\end{Lemma}
The proof of Lemma \ref{ell-ex} will be fulfilled by \S \ref{931}--\S\ref{934}.

\subsubsection{Some preliminaries}\label{931}
In what follows, we choose the parameter $\varepsilon_0$ 
satisfying \eqref{varepsilon0}:
\begin{equation*}
\displaystyle 0<\varepsilon_0< \min\Big\{\frac{3\gamma-4}{2(\gamma-1)} ,\frac{1-\delta}{\gamma-1},\frac{1}{100}\Big\}.
\end{equation*}
Then we can obtain the following relations immediately, which will be used frequently later:
\begin{equation}
\big(\frac{3}{2}-\varepsilon_0\big)\beta>\max\Big\{\beta,\frac{1}{2}\Big\}>\frac{\delta}{2}.
\end{equation}

To further simplify our calculations, we define the quantity $\Lambda$ as 
\begin{equation*} 
\Lambda:=D_\eta (\varrho^\beta). \end{equation*}
Clearly, this, together with \eqref{eq:eta} and \eqref{v-expression}, also yields
\begin{equation}\label{v-express2}
\Lambda_t+\Lambda \Big((\beta+1) D_\eta U+2\beta\frac{U}{\eta}\Big)+\beta\varrho^\beta D_\eta\Big(D_\eta U+\frac{2U}{\eta}\Big)=0,
\quad\Lambda=\frac{\beta}{2a_1\delta}\varrho^{\gamma-\delta}(V-U).
\end{equation}

Then, by \eqref{v-express2}, Lemmas \ref{lemma-v Linfty ex}, 
\ref{lemma-v Linfty in}, \ref{non-vac}, and \ref{lemma-u-D3}, we can derive the following 
lemma:
\begin{Lemma}\label{lemma-bound-lambda}
There exists a constant $C(T)>0$ such that
\begin{equation*}
|\Lambda(t)|_\infty\leq C(T) \qquad\,\, \text{for all $t\in[0,T]$}.
\end{equation*}
Moreover, for all $(t,r)\in [0,T]\times \bar I$,
\begin{align*}
\chi^\sharp|\Lambda_t| &\leq C(T)\big(1+ \chi^\sharp\rho_0^\beta |D_\eta^2 U|\big),\\
\chi^\sharp|D_\eta\Lambda| &\leq C(T)\Big(1+\int_0^t \chi^\sharp\big(|D_\eta^2 U|+\rho_0^\beta |D_\eta^3 U|\big)\,\mathrm{d}s\Big),\\
\chi^\sharp|D_\eta^2\Lambda|& \leq  C(T)\Big(1+\big|(\rho_0^\beta)_{rrr}\big|+\int_0^t \chi^\sharp\big(|D_\eta^2 U|^2+|D_\eta^3 U|+\rho_0^\beta |D_\eta^4 U|\big)\,\mathrm{d}s\Big)\\
&\quad +C(T)\Big(\int_0^t \chi^\sharp \rho_0^\beta |D_\eta^3 U| \,\mathrm{d}s\Big)\Big(\int_0^t \chi^\sharp |D_\eta^2 U|\,\mathrm{d}s\Big).
\end{align*}
\end{Lemma}
\begin{proof}
The estimate of $\Lambda_t$ follows directly from \eqref{v-express2}. Based on Lemma \ref{lemma-u-D3} and the $L^\infty$-estimate of $\Lambda$, the estimate of $D_\eta\Lambda$ follows from the following ODE:
\begin{equation*}
\begin{aligned}
&(D_\eta\Lambda)_t + D_\eta\Lambda \Big((\beta+2) D_\eta U+2\beta\frac{U}{\eta}\Big) \\
&\qquad+\Lambda D_\eta\Big((2\beta+1) D_\eta U+4\beta\frac{U}{\eta}\Big) +\beta\varrho^\beta D_\eta^2\Big(D_\eta U+\frac{2U}{\eta}\Big)=0,
\end{aligned}
\end{equation*}
and the estimate of $D_\eta^2\Lambda$ follows from that of $D_\eta\Lambda$,   and the following ODE:
\begin{equation*}
\begin{aligned}
&(D_\eta^2\Lambda)_t + D_\eta^2\Lambda \Big((\beta+3) D_\eta U+2\beta\frac{U}{\eta}\Big)+ 3D_\eta\Lambda D_\eta \Big((\beta+1) D_\eta U+2\beta\frac{U}{\eta}\Big)\\
&\qquad+\Lambda D_\eta^2\Big((3\beta+1) D_\eta U+6\beta\frac{U}{\eta}\Big) +\beta\varrho^\beta D_\eta^3\Big(D_\eta U+\frac{2U}{\eta}\Big)=0.
\end{aligned}
\end{equation*}
\end{proof}

\subsubsection{Elliptic estimates}
The first lemma concerns the second-order elliptic estimate.
\begin{Lemma}\label{lemma-u-ell-D2}
There exists a constant $C(T)>0$ such that 
\begin{equation*}
\big|\chi^\sharp\rho_0^{(\frac{3}{2}-\varepsilon_0)\beta}D_\eta^2 U(t)\big|_2+
\big|\chi^\sharp \rho_0^{(\frac{1}{2}-\varepsilon_0)\beta} U_t(t)\big|_2+\big|\chi^\sharp \rho_0^{\beta}U_t(t)\big|_\infty\leq C(T) \qquad\text{for all $t\in [0,T]$}.
\end{equation*}
\end{Lemma}
\begin{proof}
The boundedness of the first term follows  from the fact that $(\frac{3}{2}-\varepsilon_0)\beta>\frac{\delta}{2}$ and Lemma \ref{lemma-u-D2}, and the boundedness of the last two terms follows from the fact that $(\frac{3}{2}-\varepsilon_0)\beta>\frac{1}{2}$, and Lemmas \ref{lemma-upper jacobi}, \ref{lemma-u-D2}, \ref{lemma-u-D3}, and \ref{hardy-inequality}:
\begin{equation*}
\big|\chi^\sharp \rho_0^{(\frac{1}{2}-\varepsilon_0)\beta} U_t \big|_2+\big|\chi^\sharp \rho_0^{\beta}U_t\big|_\infty\leq C(T)\big(\big|\chi^\sharp \rho_0^{(\frac{3}{2}-\varepsilon_0)\beta} (U_t,D_\eta U_{t})\big|_2+ \big|\chi^\sharp \rho_0^{\frac{3}{2}\beta}(U_t,D_\eta U_{t})\big|_2\big)\leq C(T).
\end{equation*}
\end{proof}

Now we denote 
$
L_*=D_\eta U-\frac{2(1-\delta)}{\delta}\frac{ U}{\eta}
$, and  give  some refined weighted estimates for $U$. 
\begin{Lemma}\label{lemma-u-ell-D2-refine}
There exists a constant $C(T)>0$ such that, for all $t\in [0,T]$,
\begin{equation*}
\big|\chi^\sharp\rho_0^{-(\frac{1}{2}+\varepsilon_0)\beta}L_*\big|_2+\big|\chi^\sharp\rho_0^{(\frac{1}{2}-\varepsilon_0)\beta}D_\eta^2 U\big|_2+|\chi^\sharp\rho_0^\beta D_\eta^2 U|_\infty \leq C(T).
\end{equation*}
\end{Lemma}
\begin{proof}
We divide the proof into three steps.

\textbf{1.}  First, since $\varepsilon_0<\frac{1}{100}$ and $(\frac{3}{2}-\varepsilon_0)\beta>\frac{1}{2}$, we choose fixed $(\iota,\sigma)$ in Lemma \ref{lemma-time-space} such that
\begin{align*}
&\iota+\sigma=\big(\frac{1}{2}-\varepsilon_0\big)\beta,\qquad \iota\in \Big(\frac{\beta+1}{p_*}-\frac{\beta}{2},\delta+\frac{\beta}{2}\Big),\\
&0<\sigma<\min\Big\{(1-\varepsilon_0)\beta-\frac{\beta+1}{p_*},\big(\frac{3}{2}-\varepsilon_0\big)\beta-\frac{1}{2}\Big\}\implies \big(\frac{3}{2}-\varepsilon_0\big)\beta-\sigma>\frac{1}{2}.
\end{align*}
Then it follows from the above, Lemmas \ref{lemma-upper jacobi}, \ref{lemma-time-space}, \ref{lemma-u-D1}--\ref{lemma-u-D2}, \ref{lemma-u-D3}, and \ref{hardy-inequality} that
\begin{equation}
 \label{Ur-refine}
 \begin{aligned}
&\big|\chi^\sharp\rho_0^{-(\frac{1}{2}+\varepsilon_0)\beta}L_*\big|_2\leq C(T)\big(1+\big|\chi^\sharp\rho_0^{(\frac{1}{2}-\varepsilon_0)\beta-\sigma}(D_\eta U,U_t)\big|_2\big)\\
&\leq C(T)\big(1+\big|\chi^\sharp\rho_0^{(\frac{3}{2}-\varepsilon_0)\beta-\sigma}(D_\eta U,D_\eta^2 U,U_t,D_\eta U_{t})\big|_2\big) \leq C(T).
\end{aligned}    
\end{equation}

\textbf{2.} In view of \eqref{v-express2}, \eqref{ell-0} can be  written as
\begin{equation}\label{Urr-J12-J14}
\begin{aligned}
D_\eta^2 U& =\frac{1}{2a_{1}\delta}\varrho^{1-\delta} U_t-\frac{2}{\eta}\Big(D_\eta U-\frac{U}{\eta}\Big)-\frac{\delta}{\beta}\frac{\Lambda}{\varrho^\beta} L_* +\frac{A\gamma}{2a_{1} \delta \beta} \varrho^{1- \delta}\Lambda+\frac{1}{2a_{1}\delta}\varrho^{1-\delta}D_\eta\Phi.
\end{aligned}
\end{equation}

By \eqref{Ur-refine}, $(\frac{1}{2}-\varepsilon_0)\beta>0$, and Lemmas \ref{lemma-v Lp ex}, \ref{non-vac}, \ref{lemma-u-D3}, and \ref{lemma-bound-lambda}--\ref{lemma-u-ell-D2}, we have
\begin{align*}
\big|\chi^\sharp \rho_0^{(\frac{1}{2}-\varepsilon_0)\beta}D_\eta^2 U\big|_2&\leq C(T)\big(\big|\chi^\sharp  (U,D_\eta U)\big|_\infty+\big|\chi^\sharp \rho_0^{(\frac{1}{2}-\varepsilon_0)\beta} U_t\big|_2+ \big|(\Lambda,D_\eta \Phi)\big|_\infty\big)\\
&\quad +C(T)\big|\chi^\sharp\rho_0^{-(\frac{1}{2}+\varepsilon_0)\beta}L_*\big|_2|\Lambda|_\infty \leq C(T).
\end{align*}

\textbf{3.} It follows from Lemmas \ref{lemma-lower bound jacobi}, \ref{lemma-v Lp ex}, \ref{non-vac}, \ref{lemma-u-D3}, and \ref{lemma-bound-lambda}--\ref{lemma-u-ell-D2} that, for all $t\in [0,T]$,
\begin{equation*}
\begin{aligned}
|\chi^\sharp \rho_0^{\beta} D_\eta^2 U|_\infty&\leq C(T)\big(|\chi^\sharp \rho_0^\beta U_t|_\infty+ |(U,D_\eta U)|_\infty+ \big|(\Lambda,D_\eta \Phi)\big|_\infty\big)\leq C(T).
\end{aligned}    
\end{equation*}
\end{proof}

Now, we can establish the third-order elliptic estimate for $U$.
\begin{Lemma}\label{lemma-u-D3-ell}
There exists a constant $C(T)>0$ such that 
\begin{equation*}
\big|\chi^\sharp\rho_0^{(\frac{3}{2}-\varepsilon_0)\beta}(\rho_0^{-\beta}D_\eta\Lambda,D_\eta^3 U)(t)\big|_2\leq C(T)\qquad\text{for all $t\in [0,T]$}.
\end{equation*}
\end{Lemma}
\begin{proof}
We divide the proof into two steps.

\textbf{1.} According to Lemmas \ref{lemma-upper jacobi}, \ref{lemma-bound-lambda}, \ref{lemma-u-ell-D2-refine}, and \ref{hardy-inequality}, the H\"older inequality, and the Minkowski integral inequality, we have
\begin{equation}\label{p1p}
\begin{aligned}
\big|\chi^\sharp\rho_0^{(\frac{1}{2}-\varepsilon_0)\beta} D_\eta\Lambda\big|_2&\leq  C(T)\Big(\int_0^t \big|\chi^\sharp\rho_0^{(\frac{1}{2}-\varepsilon_0)\beta}(D_\eta^2 U,\rho_0^{\beta} D_\eta^3 U)\big|_2 \,\mathrm{d}s+1\Big)\\
&\leq C(T)\Big(\int_0^t \big|\chi^\sharp\rho_0^{(\frac{3}{2}-\varepsilon_0)\beta} D_\eta^3 U\big|_2\,\mathrm{d}s+1\Big).
\end{aligned}
\end{equation}

\textbf{2.}  Applying $\chi^\sharp D_\eta$ to \eqref{Urr-J12-J14}, along with $\Lambda=D_\eta(\varrho^\beta)$, gives
\begin{equation}\label{eq:3'}
\chi^\sharp D_\eta^3 U=L_{20},
\end{equation}
where $L_{20}$ has the control of the following form in view of Lemmas \ref{lemma-lower bound jacobi} and \ref{non-vac}:
\begin{equation*}
\begin{aligned}
L_{20}&\leq C(T)\Big(\chi^\sharp\rho_0^{1-\delta}\Big(|D_\eta U_t|+ \frac{|\Lambda|}{\rho_0^\beta}|U_t|\Big)+ \chi^\sharp|(U,D_\eta U,D_\eta^2 U)|\Big)
+C(T)\chi^\sharp\Big(\frac{|D_\eta\Lambda|}{\rho_0^\beta}+\frac{|\Lambda|^2}{\rho_0^{2\beta}}\Big)|L_*|\\[-2mm]
&\quad + C(T)\chi^\sharp\frac{|\Lambda|}{\rho_0^\beta}|(U,D_\eta U,D_\eta^2 U)|+C(T)\chi^\sharp\rho_0^{1-\delta}\Big(\frac{|\Lambda|}{\rho_0^\beta} |(\Lambda,D_\eta\Phi)| 
+|(D_\eta\Lambda,D_\eta^2\Phi)|\Big).
\end{aligned}
\end{equation*}
Then, it follows from $(\frac{3}{2}-\varepsilon_0)\beta>\frac{1}{2}$, \eqref{DDphi},  \eqref{p1p}, Lemmas \ref{lemma-lower bound jacobi}, \ref{lemma-v Lp ex}, \ref{lemma-u-D3}, and \ref{lemma-bound-lambda}--\ref{lemma-u-ell-D2-refine}  that
\begin{align*}
&\big|\rho_0^{(\frac{3}{2}-\varepsilon_0)\beta}L_{20}\big|_2\\
&\leq C(T)\Big(\big|\chi^\sharp \rho_0^{(\frac{3}{2}-\varepsilon_0)\beta}(D_\eta U_{t},\rho_0^{-\beta} U_t,U,D_\eta U,D_\eta^2 U)\big|_2+\big|\chi^\sharp\rho_0^{(\frac{1}{2}-\varepsilon_0)\beta} D_\eta\Lambda \big|_2 + \big|\chi^\sharp\rho_0^{-(\frac{1}{2}+\varepsilon_0)\beta}L_*\big|_2\Big)\\
&\quad + C(T)\big|\chi^\sharp\rho_0^{(\frac{1}{2}-\varepsilon_0)\beta}(U,D_\eta U,D_\eta^2 U)\big|_2+C(T)\big(|D_\eta\Phi|_\infty+\big|\chi^\sharp\rho_0^{(\frac{1}{2}-\varepsilon_0)\beta} D_\eta\Lambda \big|_2+|\sD_\eta D_\eta\Phi|_\infty\big)\\
&\leq C(T) \Big(\int_0^t \big|\chi^\sharp\rho_0^{(\frac{3}{2}-\varepsilon_0)\beta} D_\eta^3 U\big|_2\,\mathrm{d}s+1\Big).
\end{align*}

Therefore, by \eqref{eq:3'}, the above estimates, and the Gr\"onwall inequality, we obtain the third-order elliptic estimate of $U$, which, along with \eqref{p1p}, yields the desired estimate of $D_\eta\Lambda$.
\end{proof}

\subsubsection{Dissipation estimates}\label{934}

We establish the  weighted $L^2([0,T];L^2)$-estimate for $D_{\eta}^2U_t$.
\begin{Lemma}\label{lemma-ell-D4-ex1}
There exists a constant $C(T)>0$ such that
\begin{equation*}
\int_0^t \big(\big|\chi^\sharp \rho_0^{(\frac{1}{2}-\varepsilon_0)\beta}D_{\eta}U_t\big|_2^2+\big|\chi^\sharp \rho_0^{(\frac{3}{2}-\varepsilon_0)\beta}D_{\eta}^2U_t\big|_2^2\big)\,\mathrm{d}s\leq C(T) \qquad\text{for all $t\in[0,T]$}.
\end{equation*}
\end{Lemma}
\begin{proof}
We divide the proof into two steps.

\textbf{1.} We first claim 
\begin{equation}\label{refine-Utx}
\int_0^t  \big|\chi^\sharp \rho_0^{(\frac{1}{2}-\varepsilon_0)\beta}D_{\eta}U_t\big|_2^2 \,\mathrm{d}s\leq C(T) \qquad\text{for all $t\in[0,T]$}.
\end{equation}

Indeed, applying $\partial_t$ to \eqref{zhongyao} yields
\begin{equation*}
D_\eta U_t-\frac{2(1-\delta)}{\delta}\frac{U_t}{\eta}=L_{21},
\end{equation*}
where $L_{21}$ has control of the following form due to the fact that $\rho_0^\beta\sim 1-r$, \eqref{DDphi}, Lemmas \ref{lemma-lower bound jacobi}, \ref{lemma-v Lp ex}, and \ref{lemma-upper jacobi}--\ref{non-vac}, and $\eqref{tan4}_1$ in Lemma \ref{lemma-u-D3}:
\begin{equation*} 
|\chi^\sharp L_{21}| \leq C(T)\Big(1+ \frac{1}{\rho_0^\delta}\int_r^1  \rho_0^\delta |(U_t,D_\eta U_t,U_{tt})| \mathrm{d}\hat r\Big).
\end{equation*}
Then it follows from  the calculation similar to \eqref{floww} in Lemma \ref{lemma-time-space} and  $\rho_0^\beta\sim 1-r$  that 
\begin{equation*} 
|\chi^\sharp L_{21}|\leq C(T) +C(T)\rho_0^{\frac{\beta-1}{2}}\big|\chi^\sharp\rho_0^\frac{1}{2}(U_t,D_\eta U_t,U_{tt})\big|_2.
\end{equation*}
Hence, multiplying the above by $\rho_0^{(\frac{1}{2}-\varepsilon_0)\beta}$, along with Lemmas \ref{lemma-u-D2} and \ref{lemma-u-D3}, implies that claim \eqref{refine-Utx} holds. Here, we have also used the facts that
\begin{equation*}
(1-\varepsilon_0) -\frac{1}{2\beta}>-\frac{1}{2},\qquad \chi^\sharp\rho_0^{(\frac{1}{2}-\varepsilon_0)\beta+\frac{\beta-1}{2}} \sim \chi^\sharp(1-r)^{(1-\varepsilon_0)-\frac{1}{2\beta}} \in L^2.
\end{equation*}

\textbf{2.} Applying $\partial_t$ to \eqref{Urr-J12-J14}, along with $\eqref{eq:VFBP-La}_1$, we arrive at
\begin{equation}\label{L26}
D_\eta^2 U_t=L_{22},
\end{equation}
where $L_{22}$ has control of the following form due to $\rho_0^\beta\sim 1-r$, \eqref{D Phi expression}, Lemmas \ref{lemma-lower bound jacobi} and \ref{non-vac}, and $\eqref{tan4}_1$ in Lemma \ref{lemma-u-D3}:
\begin{equation*}
\begin{aligned}
|\chi^\sharp L_{22}|&\leq C(T)\chi^\sharp\Big(1+|(D_\eta^2 U,U_t,D_\eta U_t,U_{tt})|+ \frac{|(\Lambda,\Lambda_t)|}{\rho_0^\beta}+\frac{|\Lambda|}{\rho_0^\beta}|(U_t,D_\eta U_t)|+\rho_0^{1-\delta} |D_\eta\Phi|\Big).
\end{aligned}    
\end{equation*}

Then, due to Lemmas \ref{lemma-v Lp ex}, \ref{lemma-u-D3}, and \ref{lemma-bound-lambda}--\ref{lemma-u-ell-D2}, and the facts that $(\frac{1}{2}-\varepsilon_0)\beta>0$ and $(\frac{3}{2}-\varepsilon_0)\beta>\frac{1}{2}$, we obtain
\begin{align*}
\big|\chi^\sharp\rho_0^{(\frac{3}{2}-\varepsilon_0)\beta} L_{22}\big|_2& \leq C(T)\big(1+\big|\chi^\sharp\rho_0^{(\frac{3}{2}-\varepsilon_0)\beta}(D_\eta^2 U,U_t,D_\eta U_t,U_{tt})\big|_2+\big|\chi^\sharp\rho_0^{(\frac{1}{2}-\varepsilon_0)\beta}(U_t,D_\eta U_t)\big|_2\big)\\
&\leq C(T)\big(\big|(r^2\rho_0)^\frac{1}{2} U_{tt}\big|_2+\big|\chi^\sharp\rho_0^{(\frac{1}{2}-\varepsilon_0)\beta}D_\eta U_t\big|_2+1\big),
\end{align*}
which, along with \eqref{refine-Utx}--\eqref{L26}, yields the desired estimate.
\end{proof}

Finally, we establish the $L^2([0,T];L^2)$-estimate for $\chi^\sharp \rho_0^{(\frac{3}{2}-\varepsilon_0)\beta}D_{\eta}^4U$.
\begin{Lemma}\label{lemma-ell-D4-ex2}
There exists a constant $C(T)>0$ such that
\begin{equation}\label{lemma5.15}
\int_0^t\big|\chi^\sharp\rho_0^{(\frac{3}{2}-\varepsilon_0)\beta}D_\eta^4 U \big|_2^2\,\mathrm{d}s \leq C(T) \qquad\text{for all $t\in [0,T]$}.
\end{equation}
\end{Lemma}
\begin{proof}
We divide the proof into three steps.

\smallskip
\textbf{1.} First, we can show that
\begin{equation}\label{control1}
|\chi^\sharp D_\eta \Lambda|_\infty\leq C(T)\Big(\int_0^t\big|\chi^\sharp \rho_0^{(\frac{3}{2}-\varepsilon_0)\beta} D_\eta^4 U\big|_2\,\mathrm{d}s+1\Big).
\end{equation}
Indeed, recall from Lemma \ref{lemma-bound-lambda} that
\begin{equation*}
|\chi^\sharp D_\eta \Lambda|_\infty \leq  C(T)\Big(1+\int_0^t\big|\chi^\sharp(D_\eta^2 U,\rho_0^\beta D_\eta^3 U)\big|_\infty\,\mathrm{d}s\Big).
\end{equation*}
Then \eqref{control1} follows directly from Lemma \ref{lemma-u-ell-D2-refine}, and the following controls due to Lemmas \ref{lemma-upper jacobi}, \ref{lemma-u-ell-D2}, \ref{lemma-u-D3-ell}, and \ref{sobolev-embedding}--\ref{hardy-inequality}:
\begin{equation}\label{control2}
\begin{aligned}
&|\chi^\sharp(D_\eta^2 U,\rho_0^\beta D_\eta^3 U)|_\infty \leq C(T)\big(|\chi^\sharp (D_\eta^2 U,D_\eta^3 U)|_{1}+|\chi^\sharp \rho_0^\beta D_\eta^3 U|_\infty\big)\\
&\leq C(T) \big|\chi^\sharp \rho_0^{(\frac{3}{2}-\varepsilon_0)\beta}(D_\eta^2 U,D_\eta^3 U,D_\eta^4 U)\big|_2\leq C(T) \big(\big|\chi^\sharp \rho_0^{(\frac{3}{2}-\varepsilon_0)\beta} D_\eta^4 U\big|_2+1\big),
\end{aligned}
\end{equation}
where we have used  Lemma \ref{hardy-inequality} by taking $\vartheta=-\varepsilon_0>-\frac{1}{2}$ and $\varepsilon=\varepsilon_0$.

Next, we claim:
\begin{equation}\label{d2-lambda}
\big|\chi^\sharp\rho_0^{(\frac{1}{2}-\varepsilon_0)\beta}D_\eta^2\Lambda\big|_2\leq C(T)\Big(\int_0^t \big|\chi^\sharp\rho_0^{(\frac{3}{2}-\varepsilon_0)\beta}D_\eta^4 U\big|_2\,\mathrm{d}s+1\Big).  
\end{equation}
Indeed, it follows from Lemma \ref{lemma-bound-lambda}, the fact that $\varepsilon_0<\frac{1}{2}$, \eqref{control2}, Lemmas \ref{lemma-u-ell-D2-refine}--\ref{lemma-u-D3-ell}, and the Minkowski integral inequality that 
\begin{align*}
\big|\chi^\sharp\rho_0^{(\frac{1}{2}-\varepsilon_0)\beta}D_\eta^2\Lambda\big|_2&\leq  C(T)\Big(1+\int_0^t \big(\big|\chi^\sharp\rho_0^{(\frac{1}{2}-\varepsilon_0)\beta}D_\eta^3 U\big|_2 + \big|\chi^\sharp\rho_0^{(\frac{3}{2}-\varepsilon_0)\beta}D_\eta^4 U\big|_2\big)\,\mathrm{d}s\Big)\\
&\quad +C(T) \big|\chi^\sharp\rho_0^{(\frac{1}{2}-\varepsilon_0)\beta}(D_\eta^2 U,\rho_0^\beta D_\eta^3 U)\big|_2 \int_0^t |\chi^\sharp D_\eta^2 U|_\infty\,\mathrm{d}s\\
&\leq C(T)\Big(\int_0^t \big|\chi^\sharp\rho_0^{(\frac{3}{2}-\varepsilon_0)\beta}D_\eta^4 U\big|_2\,\mathrm{d}s+1\Big).
\end{align*}
This completes the proof of \eqref{d2-lambda}.

Finally, we can obtain similarly from \eqref{control2} that
\begin{equation}\label{Utwuqiong}
|\chi^\sharp U_t|_\infty \leq C(T) \big|\chi^\sharp \rho_0^{(\frac{3}{2}-\varepsilon_0)\beta}(U_t,D_\eta U_t,D_\eta^2U_t)\big|_2\leq C(T) \big(\big|\chi^\sharp \rho_0^{(\frac{3}{2}-\varepsilon_0)\beta} D_\eta^2 U_t\big|_2+1\big).
\end{equation}

\smallskip
\textbf{2.} Applying $D_\eta^2$ to  \eqref{reform boundary condition}  and dividing the resulting equality by $2a_1\delta\varrho^\beta$, we arrive at
\begin{align}
&\,\cU^\star:=D_\eta^4 U +\big(\frac{\delta}{\beta}+2\big) \frac{D_\eta(\varrho^\beta)}{\varrho^\beta} D_\eta^3 U\nonumber\\
&=\underline{-\Big(1+\frac{2\delta}{\beta}\Big)
\frac{D_\eta\Lambda}{\varrho^\beta}D_\eta^2U
-\frac{\delta}{\beta}
\frac{D_\eta^2\Lambda}{\varrho^\beta}D_\eta U
-\frac{2}{\varrho^\beta}D_\eta^2\Big(\varrho^\beta  D_\eta\big(\frac{U}{\eta}\big)-\frac{(1-\delta)}{\beta} \Lambda \frac{U}{\eta}-\frac{\varrho^{\gamma-\delta} U_t}{4a_1\delta}   \Big)}_{:=L_{23}} \nonumber\\
&\quad+\underline{\frac{1}{2a_1\delta\varrho^\beta}D_\eta^2 \Big(\frac{A\gamma}{\beta}\varrho^{\gamma-\delta} \Lambda+\varrho^{\gamma-\delta} D_\eta\Phi\Big)}_{:=L_{24}},\label{979'}
\end{align} 
where $L_{23}$--$L_{24}$ have controls of the following form due to Lemmas \ref{lemma-lower bound jacobi}, \ref{non-vac}, \ref{lemma-u-D3}, and \ref{lemma-u-ell-D2-refine}:
\begin{align*}
\chi^\sharp L_{23}&\leq C(T)\chi^\sharp\frac{1}{\rho_0^\beta}\Big(|(\Lambda,D_\eta \Lambda)| |D_\eta^2 U| + |(\Lambda,D_\eta \Lambda,D_\eta^2 \Lambda)| + \rho_0^\beta |D_\eta^3 U|+1\Big)\\
&\quad +C(T)\chi^\sharp\rho_0^{1-\delta } \Big(
\frac{|\Lambda|^2}{\rho_0^{2\beta}}|U_t|+\frac{|D_\eta \Lambda|}{\rho_0^{2\beta}} |U_t|+  \frac{|\Lambda|}{\rho_0^\beta} |D_\eta U_t|+ |D_\eta^2 U_t|\Big),\\
\chi^\sharp L_{24}&\leq C(T) \rho_0^{1-\delta }\Big( |(D_\eta^2 \Lambda,D_\eta^3 \Phi)|+\frac{1}{\rho_0^\beta}\big( |\Lambda| |(D_\eta \Lambda ,D_\eta^2\Phi)|+  |D_\eta \Lambda| |D_\eta \Phi|\big) + \frac{|\Lambda|^2}{\rho_0^{2\beta}}|(\Lambda,D_\eta \Phi)|\Big).
\end{align*}

Then, for $L_{23}$, it follows from the facts that $\rho_0^\beta\sim 1-r$, $(\frac{3}{2}-\varepsilon_0)\beta>\frac{1}{2}>\varepsilon_0$, and
\begin{equation*}
\varepsilon_0<\frac{1-\delta}{\gamma-1}\iff -\big(\frac{1}{2}+\varepsilon_0\big)\beta+1-\delta>-\frac{\beta}{2},\qquad\chi^\sharp\rho_0^{-(\frac{1}{2}+\varepsilon_0)\beta+1-\delta}\in L^2,
\end{equation*}
\eqref{Utwuqiong}, and Lemmas \ref{lemma-bound-lambda}--\ref{lemma-u-D3-ell} that
\begin{equation*}
\begin{aligned}
\big|\chi^\sharp\rho_0^{(\frac{3}{2}-\varepsilon_0)\beta} L_{23}\big|_2&\leq C(T)\big(\big|\chi^\sharp\rho_0^{(\frac{1}{2}-\varepsilon_0)\beta}(1,D_\eta^2 U)\big|_2 (1+|\chi^\sharp D_\eta\Lambda|_\infty )+\big|\chi^\sharp\rho_0^{(\frac{1}{2}-\varepsilon_0)\beta}D_\eta^2 \Lambda\big|_2\big)\\
&\quad +C(T) \big|\chi^\sharp\rho_0^{(\frac{3}{2}-\varepsilon_0)\beta}D_\eta^3 U\big|_2\\
&\quad +C(T)\big|\chi^\sharp\rho_0^{-(\frac{1}{2}+\varepsilon_0)\beta+1-\delta}\big|_2(|\Lambda|_\infty^2+|\chi^\sharp D_\eta\Lambda|_\infty)|\chi^\sharp U_t|_\infty \\
&\quad +C(T)\big|\chi^\sharp\rho_0^{(\frac{1}{2}-\varepsilon_0)\beta}(D_\eta U_t,\rho_0^\beta D_\eta^2 U_t\big)\big|_2\big)\\
&\leq C(T)\big(\big|\chi^\sharp\rho_0^{(\frac{1}{2}-\varepsilon_0)\beta}(D_\eta^2\Lambda,D_\eta U_t,\rho_0^\beta D_\eta^2 U_t\big)\big|_2+1\big)\\
&\quad +C(T)|\chi^\sharp D_\eta\Lambda|_\infty\big(\big|\chi^\sharp \rho_0^{(\frac{3}{2}-\varepsilon_0)\beta} D_\eta^2 U_t\big|_2+1\big).
\end{aligned}
\end{equation*}

For $L_{24}$, recall from \eqref{vr-urr-l5}, $\rho_0^\beta\sim 1-r$, and Lemmas \ref{lemma-v Linfty ex} and \ref{lemma-u-D3} that
\begin{equation*}
\chi^\sharp|D_\eta^3 \Phi|\leq C(T)\big(1+ \rho_0^{1-\beta}\big).
\end{equation*}
Then, based on a calculation similar to that of $L_{23}$, we can obtain from the above, \eqref{DDphi}, and Lemma \ref{lemma-v Lp ex} that 
\begin{equation*}
\big|\chi^\sharp\rho_0^{(\frac{3}{2}-\varepsilon_0)\beta} L_{24}\big|_2\leq C(T)\Big(\big|\chi^\sharp\rho_0^{(\frac{1}{2}-\varepsilon_0)\beta}D_\eta^2\Lambda\big|_2+|\chi^\sharp D_\eta\Lambda|_\infty \big(\big|\chi^\sharp \rho_0^{(\frac{3}{2}-\varepsilon_0)\beta} D_\eta^2 U_t\big|_2+1\big)+1\Big).
\end{equation*}

Therefore, it follows from the estimates of $L_{23}$--$L_{24}$, \eqref{control1}, \eqref{d2-lambda}, and \eqref{979'} that
\begin{equation}\label{Ustar}
\big|\chi^\sharp\rho_0^{(\frac{3}{2}-\varepsilon_0)\beta} \cU^\star\big|_2\leq C(T)\cA_1(t)\Big(1+ \int_0^t \big|\chi^\sharp\rho_0^{(\frac{3}{2}-\varepsilon_0)\beta}D_\eta^4 U\big|_2\,\mathrm{d}s\Big)
\end{equation}
for some $\cA_1\in L^2(0,T)$ such that $\|\cA_1\|_{L^2(0,T)}\leq C(T)$.

\smallskip
\textbf{3.} Now, define
\begin{equation}\label{guanxi}
\cU_{\mathrm{cross}}:=(D_\eta^3 U)_r +\big(\frac{\delta}{\beta}+2\big) \frac{(\rho_0^\beta)_r}{\rho_0^\beta} D_\eta^3 U=\eta_r\cU^\star+(\delta+2\beta)\big(\log\frac{\eta^2\eta_r}{r^2}\big)_r D_\eta^3 U.
\end{equation}

Next, thanks to $\eta_t= U$, we have
\begin{equation*}
\big(\log\frac{\eta}{r}\big)_{tr}=\eta_rD_\eta\big(\frac{U}{\eta}\big),\qquad (\log\eta_r)_{tr}=\eta_r D_\eta^2 U.
\end{equation*}
Hence, it follows from Lemmas  \ref{lemma-upper jacobi} and \ref{lemma-u-D3} and \eqref{control2} that
\begin{equation*}
\chi^\sharp\Big|\big(\log\frac{\eta^2\eta_r}{r^2}\big)_r\Big|\leq C(T)\Big(1+\int_0^t |\chi^\sharp D_\eta^2 U|_\infty\,\mathrm{d}s\Big) \leq C(T)\Big(1+ \int_0^t \big|\chi^\sharp\rho_0^{(\frac{3}{2}-\varepsilon_0)\beta}D_\eta^4 U\big|_2\,\mathrm{d}s\Big). 
\end{equation*}
This, together with Lemma \ref{lemma-u-D3-ell} and \eqref{Ustar}--\eqref{guanxi}, gives
\begin{equation}\label{Ustar2}
\big|\chi^\sharp\rho_0^{(\frac{3}{2}-\varepsilon_0)\beta} \cU_{\mathrm{cross}}\big|_2\leq C(T)\cB_1(t)\Big(1+ \int_0^t \big|\chi^\sharp\rho_0^{(\frac{3}{2}-\varepsilon_0)\beta}D_\eta^4 U\big|_2\,\mathrm{d}s\Big)
\end{equation}
for some $\cB_1\in L^2(0,T)$ such that $\|\cB_1\|_{L^2(0,T)}\leq C(T)$.

Finally, since $\chi^\sharp=\zeta-\chi+\zeta^\sharp$, $\rho_0^\beta\sim 1-r$, and $\frac{1}{2}<\frac{3}{2}-\varepsilon_0<\frac{3}{2}+\frac{\delta}{2\beta}$, we  obtain from Lemmas \ref{lemma-upper jacobi}, \ref{ell-inner}, \ref{lemma-u-D3-ell}, and \ref{prop2.1} that
\begin{align*}
\big|\chi^\sharp \rho_0^{(\frac{3}{2}-\varepsilon_0)\beta} D_\eta^4 U\big|_2&\leq C_0\cD_{\mathrm{in}}^\frac{1}{2}(t,U)+\big|\zeta^\sharp \rho_0^{(\frac{3}{2}-\varepsilon_0)\beta} D_\eta^4 U\big|_2\\
&\leq C(T)\big(\cD_{\mathrm{in}}^\frac{1}{2}(t,U) + \big|\zeta^\sharp \rho_0^{(\frac{3}{2}-\varepsilon_0)\beta}\mathcal{U}_{\mathrm{cross}}\big|_2+ \big|\chi^\sharp \rho_0^{(\frac{3}{2}-\varepsilon_0)\beta} D_\eta^3 U\big|_2\big)\\
&\leq C(T)\cC_1(t)\Big(1+\int_0^t\big|\chi^\sharp \rho_0^{(\frac{3}{2}-\varepsilon_0)\beta} D_\eta^4 U\big|_2\,\mathrm{d}s\Big)
\end{align*}
for some $\cC_1\in L^2(0,T)$ such that $\|\cC_1\|_{L^2(0,T)}\leq C(T)$. Here, we have used Lemma \ref{prop2.1} with
\begin{equation*}
\bar d=\rho_0^\beta,\qquad b=\frac{3}{2}-\varepsilon_0,\qquad c=2+\frac{\delta}{\beta},\qquad f=D_\eta^3 U,\qquad \beta=\gamma-1.
\end{equation*}

Finally, the above inequality, along with  the Gr\"onwall inequality, leads to \eqref{lemma5.15}.
\end{proof}

\subsection{Time-weighted estimates of the velocity}\label{sub94}

\begin{Lemma}\label{lemma-qiexiang-time weight}
There exists a constant $C(T)>0$ such that
\begin{equation*}
t\big|(r^2\rho_0)^\frac{1}{2}U_{tt}(t)\big|_2^2+\int_0^t s\big|(r^2\rho_0^\delta)^\frac{1}{2}\sD_\eta U_{tt}\big|_2^2\,\ds\leq C(T) \qquad\text{for all $t\in [0,T]$}.
\end{equation*}
\end{Lemma}

\begin{proof}
 
First, replacing $V$ in \eqref{eq-2order-pre} by $\Lambda$ in view of \eqref{v-express2} and then applying $\partial_{t}$ to the resulting equality, we arrive at
\begin{equation}\label{ssr}
\begin{aligned}
&\,r^2\rho_0U_{ttt}+\big(\frac{\eta}{\sqrt{\eta_r}} \varrho^\frac{\delta}{2} L_{25}\big)_r-   \sqrt{\eta_r}\varrho^\frac{\delta}{2}  L_{26}-\eta\sqrt{\eta_r}\varrho^\frac{1}{2} L_{27}\\
&=2a_1\Big(\eta^2 \varrho^\delta \big(\delta D_\eta U_{tt}- 2(1-\delta) \frac{U_{tt}}{\eta}\big)\Big)_r +4a_1\eta^2\eta_r\varrho^\delta\Big((1-\delta)D_\eta U_{tt}- (2\delta-1) \frac{U_{tt}}{\eta}\Big)\frac{1}{\eta},
\end{aligned}
\end{equation}
where  \begin{align*}
\frac{\eta}{\sqrt{\eta_r}} \varrho^\frac{\delta}{2} L_{25}&= \big(-\eta^2\mathfrak{L}_1+\mathfrak{L}_3 \big)_t+2a_1\delta \eta^2\varrho^\delta(\delta L_* L_{21}+ D_\eta U D_\eta U_t)-4a_1(1-\delta)\varrho^\delta U U_t,\\
\sqrt{\eta_r}\varrho^\frac{\delta}{2}  L_{26}&=4a_1\eta\eta_r\varrho^\delta((1-\delta)D_\eta U-(2\delta-1)\frac{U}{\eta})((1-\delta)D_\eta U_t-(2\delta-1)\frac{U_t}{\eta})\\
&\qquad-4a_1\eta\eta_r\varrho^\delta((1-\delta)D_\eta U D_\eta U_t-(2\delta-1)\frac{U}{\eta}\frac{U_t}{\eta})-\big(2\eta\eta_r\mathfrak{L}_1+\mathfrak{L}_4  \big)_t,\\
\eta\sqrt{\eta_r}\varrho^\frac{1}{2} L_{27}&= 2r^2\rho_0\big(D_\eta \Phi \frac{U}{\eta}\big)_t-\frac{2A\gamma}{\beta}\big(r^2\rho_0\Lambda\frac{U}{\eta}\big)_t.      
\end{align*}
For $L_{25}$--$L_{27}$, we have controls of the following form:
\begin{align*}
|(L_{25},L_{26})|&\leq C_0 \eta\sqrt{\eta_r}\varrho^{\gamma-\frac{\delta}{2}} \big(|\sD_\eta U|^2+|\sD_\eta U_t|\big)+C_0\eta \sqrt{\eta_r}\varrho^\frac{\delta}{2}\big(|\sD_\eta U||\sD_\eta U_t|+|\sD_\eta U|^3\big),\\
|L_{27}|&\leq C_0 \eta\sqrt{\eta_r}\varrho^\frac{1}{2} \big(|\sD_\eta U|^2+|\sD_\eta U_t|\big) |(D_\eta \Phi,\Lambda)| + C_0\eta\sqrt{\eta_r}\varrho^{\gamma-\frac{1}{2}} |\sD_\eta U| |\sD_\eta^2 U|.  
\end{align*}
Clearly, by Lemmas \ref{lemma-v Lp ex}, \ref{lemma-upper jacobi}--\ref{non-vac}, and \ref{ell-inner}--\ref{lemma-bound-lambda}, we can obtain
\begin{equation}\label{29-31}
|(L_{25},L_{26},L_{27})|_2 \leq C(T).      
\end{equation}

Next, multiplying \eqref{ssr} by $U_{tt}$ and integrating over $I$, we have
\begin{equation}\label{dt-j15-17}
\frac{1}{2}\frac{\mathrm{d}}{\dt}\big|(r^2\rho_0)^\frac{1}{2}U_{tt}\big|_2^2+L_{28}=\int_0^1 \Big(\eta \sqrt{\eta_r}\varrho^\frac{\delta}{2}\big(L_{25}D_\eta U_{tt}+ L_{26}\frac{U_{tt}}{\eta}\big)+\eta\sqrt{\eta_r}\varrho^\frac{1}{2} L_{27}U_{tt}\Big)\,\mathrm{d}r,
\end{equation}
where
\begin{equation*}
L_{28}:=2a_1\int_0^1  \eta^2\eta_r \varrho^\delta \Big(\delta|D_\eta U_{tt}|^2-4(1-\delta) \frac{U_{tt}}{\eta} D_\eta U_{tt}+ (4\delta-2) \frac{|U_{tt}|^2}{\eta^2}\Big) \,\mathrm{d}r.
\end{equation*}

Since $L_{28}$ can be handled in the same way as $L_3$ in Step 2.1 of the proof of Lemma \ref{lemma-u-D1}:
\begin{equation*} 
L_{28}\geq 2a_1 c_\delta^*\big|(\eta^2\eta_r\varrho^\delta)^{\frac{1}{2}}\sD_\eta U_{tt}\big|_2^2\geq\frac{1}{C(T)}\big|(r^2\rho_0^\delta)^{\frac{1}{2}}\sD_\eta U_{tt}\big|_2^2,
\end{equation*}
we can obtain from \eqref{29-31}--\eqref{dt-j15-17}, Lemmas \ref{lemma-upper jacobi}--\ref{non-vac}, and H\"older and Young inequalities that
\begin{equation}\label{hk-1}
\frac{\mathrm{d}}{\dt} \big|(r^2\rho_0)^\frac{1}{2}U_{tt}\big|_2^2 +\frac{1}{C(T)}\big|(r^2\rho_0^\delta)^\frac{1}{2}\sD_\eta U_{tt}\big|_2^2\leq C(T)\big(1+\big|(r^2\rho_0)^\frac{1}{2}U_{tt}\big|_2^2\big).
\end{equation}
Multiplying \eqref{hk-1} by $t$ and  integrating over $[\tau,t]$ with $\tau\in (0,t)$, along with Lemma \ref{lemma-u-D3}, gives
\begin{equation}\label{tauk}
t\big|(r^2\rho_0)^\frac{1}{2}U_{tt}(t)\big|_2^2+\frac{1}{C(T)}\int_\tau^t s\Big|(r^2\rho_0^\delta)^\frac{1}{2}\big(D_\eta U_{tt},\frac{U_{tt}}{\eta}\big)\Big|_2^2\,\ds\leq \tau\big|(r^2\rho_0)^\frac{1}{2}U_{tt}(\tau)\big|_2^2+C(T).    
\end{equation}
Thanks to Lemmas \ref{lemma-u-D3} and \ref{bjr}, we can find a sequence $\{\tau_k\}_{k=1}^\infty$ such that $\tau_k\to 0$ and $\tau_k|(r^2\rho_0)^\frac{1}{2}U_{tt}(\tau_k)|_2^2\to 0$ as $k\to\infty$. Hence, taking $\tau=\tau_k$ in \eqref{tauk} and then letting $k\to\infty$, we finally obtain the desired estimate.
\end{proof}

\begin{Lemma}\label{lemma-ell-time weight-in1}
There exists a constant $C(T)>0$ such that
\begin{equation*}
t\cD_{\mathrm{in}}(t,U)+t\cD_{\mathrm{ex}}(t,U) \leq C(T) \qquad\,\, \text{for all $t\in [0,T]$}.
\end{equation*}
\end{Lemma}

\begin{proof}
First,  by \eqref{ell-0-t-0}--\eqref{proof0}, \eqref{proof1}, and Lemma \ref{lemma-qiexiang-time weight}, we have 
\begin{equation*} 
\begin{aligned}
&\sqrt{t}\big|\zeta r(\sD_\eta^2 U_t,\sD_\eta^4 U)\big|_2\leq  C(T).
\end{aligned}
\end{equation*}

Second,  it follows from  Lemma \ref{lemma-qiexiang-time weight} and the same argument as in the proof of \eqref{refine-Utx} that 
\begin{equation*}
\sqrt{t}\big|\chi^\sharp \rho_0^{(\frac{1}{2}-\varepsilon_0)\beta}D_\eta U_t\big|_2\leq C(T),
\end{equation*}
which, along  with \eqref{lemma5.15} and the same proof as in Step 2 of Lemma \ref{lemma-ell-D4-ex2}, leads to
\begin{equation*}
\sqrt{t}\big|\chi^\sharp \rho_0^{(\frac{3}{2}-\varepsilon_0)\beta}(D_\eta^2 U_t, D_\eta^4 U)\big|_2
\leq C(T)\big(\sqrt{t}\big|(r^2\rho_0)^{\frac{1}{2}}U_{tt}\big|_2+1\big)\leq C(T).
\end{equation*}
\end{proof}

\section{Global  Existence  of the Unique Classical Solution}\label{Section-global}
Now we are ready to give the proof for Theorem \ref{Theorem1.1}. 

According to Theorem \ref{local-Theorem1.1}, there exists a unique classical solution $(U,\eta)(t,r)$ of \textbf{IBVP} \eqref{eq:VFBP-La-eta} in $[0,T_*]\times \bar I$ for some $T_*>0$, satisfying \eqref{b1-lo} and \eqref{N111}.  Now, suppose that $\overline{T}_*>0$ is the life-span of $(U,\eta)(t,r)$. Then $\overline{T}_*\geq T_*$. We claim that $\overline{T}_*=\infty$.

Otherwise, if $\overline{T}_*<\infty$, for every $T\in (0,\overline{T}_*)$, we collect the uniform {\it a priori} bounds obtained in  Lemmas \ref{lemma-lower bound jacobi}, \ref{lemma-upper jacobi}, \ref{ell-inner}--\ref{ell-ex}, and \ref{lemma-qiexiang-time weight}--\ref{lemma-ell-time weight-in1} on $[0,T]$.  Since $T<\overline{T}_*$, the constants in these estimates can be bounded in terms of $\overline{T}_*$. Therefore,  
\begin{equation}\label{811}
\begin{aligned}
&\sup_{t\in[0,\overline{T}_*)}(\cE(t,U)+t\cD(t,U))+\int_0^{\overline{T}_*}\cD (s,U)\,\mathrm{d}s\leq C(\overline{T}_*),\\
&\,\,(\eta_r,\frac{\eta}{r})(t,r)\in [C^{-1}(\overline{T}_*),C(\overline{T}_*)] \qquad \text{for all $(t,r)\in [0,\overline{T}_*)\times \bar I$}.
\end{aligned}
\end{equation}
Here and below, $C(\overline{T}_*)\in (1,\infty)$ depends only on $(a_1, A,\gamma,\delta,\varepsilon_0,\rho_0,u_0,\cK_1,\cK_2,G,\overline{T}_*)$ and may
be chosen continuous and increasing in $\overline{T}_*$. Moreover, based on \eqref{811}, the equivalence of the energy functionals established in \cite{jiawenlocal}*{Lemma C.1}, the equation $\eta_t=U$ with $\eta|_{t=0}=r$, and Gr\"onwall inequality, we can also obtain the following estimates recursively:
\begin{equation}\label{811'}
\sup_{t\in[0,\overline{T}_*)}(\mathring\cE(t,U)+t\mathring\cD(t,U))
+\int_0^{\overline{T}_*}\mathring\cD (s,U)\,\mathrm{d}s\leq C(\overline{T}_*),
\end{equation}
where $(\mathring\cE,\mathring\cD)(t,U)$ are defined in the same way as \eqref{E-1}--\eqref{D-2} with $\eta(t,r)$ replaced by $r$. For completeness, we give the details of the dissipative estimates of $U_{rrrr}$:
\begin{equation}\label{diss}
\int_0^{\overline{T}_*} \big|\zeta rU_{rrrr}\big|_2^2\,\mathrm{d}t+
\int_0^{\overline{T}_*} \big|\chi^\sharp\rho_0^{(\frac{3}{2}-\varepsilon_0)(\gamma-1)}  U_{rrrr}\big|_2^2\,\mathrm{d}t\leq C(\overline{T}_*).
\end{equation}
By \eqref{811} and the energy comparison in
\cite{jiawenlocal}*{Lemma C.1}, we have
\begin{equation*}
\mathring\cE(t,U)
\leq C(\overline T_*)\cE(t,U)
\leq C(\overline T_*),
\end{equation*}
which, by the definition of $\mathring\cE$, yields
\begin{equation}\label{63}
\big|\zeta r(U_{rr},U_{rrr})\big|_2\leq C(\overline{T}_*),\quad \big|\chi^\sharp \rho_0^{(\frac{3}{2}-\varepsilon_0)(\gamma-1)}(U_{rr},U_{rrr})\big|_2\leq C(\overline{T}_*), \quad \text{ for all } t\in[0,\overline T_*). 
\end{equation}
First, we can obtain from the region segmentation, \eqref{811}, and Lemmas \ref{sobolev-embedding}--\ref{hardy-inequality} that
\begin{equation}\label{revise 6 c.1 pf p-1}
\sup_{t\in [0,\overline{T}_*)}|D_\eta U|_\infty+\int_0^{\overline{T}_*}|D_\eta^2 U|_\infty^2\,\mathrm{d}t\leq C(\overline{T}_*).
\end{equation}
Since
\begin{equation*}
\eta_{trr}=\eta_{rr} D_{\eta}U+\eta_r^2 D_{\eta}^2U,\qquad \eta_{rr}|_{t=0}=0,
\end{equation*}
we obtain from \eqref{811}, \eqref{revise 6 c.1 pf p-1}, and the H\"older and Gr\"onwall inequalities that
\begin{equation}\label{revise 6 c.1 pf p1}
\sup_{t\in [0,\overline{T}_*)}|\eta_{rr}|_\infty\leq C(\overline{T}_*),
\end{equation}
which, along with $U_r=\eta_rD_\eta U$ and $U_{rr}=\eta_r^2 D_\eta^2 U+ \eta_{rr}D_\eta U$, implies that
\begin{equation}\label{64}
\sup_{t\in [0,\overline{T}_*)}|U_r|_\infty+\int_0^{\overline{T}_*}|U_{rr}|_\infty^2\,\mathrm{d}t\leq C(\overline{T}_*).
\end{equation}

Second, let $w=\zeta r$ or $w=\chi^\sharp\rho_0^{(\frac32-\varepsilon_0)(\gamma-1)}$. Then \eqref{63}, together with the following integral equalities
\begin{equation*}
\partial_r^k \eta(t,r)=\int_0^t \partial_r^k U(s,r)\,\mathrm{d}s\qquad\text{for $k=2,3,4$},    
\end{equation*}
yields that
\begin{equation}\label{revise 6 c.1 pf p2}
\begin{aligned}
\sup_{t\in[0,\overline T_*)}|w\eta_{rrr}(t)|_2\leq C(\overline T_*),\qquad|w\eta_{rrrr}(t)|_2\leq\int_0^t|wU_{rrrr}(s)|_2\,\mathrm{d}s.
\end{aligned}
\end{equation}
Hence, using the identity
\begin{equation*}
\begin{aligned}
&U_{rrrr}=\eta_r^4 D_\eta^4 U+6\eta_r^2\eta_{rr}D_\eta^3 U+(3\eta_{rr}^2+4\eta_r\eta_{rrr})D_\eta^2 U+ \eta_{rrrr}D_\eta U,
\end{aligned}
\end{equation*}
and the preceding estimates, we arrive at
\begin{equation*}
\begin{aligned}
|wU_{rrrr}|_2\leq C(\overline T_*)\big(&|wD_\eta^4U|_2+|\eta_{rr}|_\infty|wD_\eta^3U|_2
+|\eta_{rr}|_\infty^2|wD_\eta^2U|_2\\
&+|w\eta_{rrr}|_2|D_\eta^2U|_\infty+|w\eta_{rrrr}|_2|D_\eta U|_\infty\big).
\end{aligned}
\end{equation*}
It then follows from \eqref{811} and \eqref{63}--\eqref{revise 6 c.1 pf p2} that 
\begin{equation*}
\begin{aligned}
|wU_{rrrr}(t)|_2
\leq C(\overline T_*)\Big(|wD_\eta^4U(t)|_2+1+|D_\eta^2U(t)|_\infty+\int_0^t|wU_{rrrr}(s)|_2\,\mathrm{d}s\Big).
\end{aligned}
\end{equation*}
Squaring this inequality and using H\"older's inequality,
we obtain
\begin{equation*}
\begin{aligned}
|wU_{rrrr}(t)|_2^2
\leq C(\overline T_*)\Big(|wD_\eta^4U(t)|_2^2+1+|D_\eta^2U(t)|_\infty^2+\int_0^t|wU_{rrrr}(s)|_2^2\,\mathrm{d}s\Big),
\end{aligned}
\end{equation*}
which, along with the Gr\"onwall inequality and \eqref{revise 6 c.1 pf p-1}, leads to \eqref{diss}.

Now, for any fixed $\tau\in [\frac{\overline{T}_*}{2},\overline{T}_*)$, set
\begin{equation*}
    u_\star(r):=U(\tau,r),\qquad
\eta_\star(r):=\eta(\tau,r).
\end{equation*}
After the time translation $s=t-\tau$, regard $(u_\star,\eta_\star)$
as the initial data at $s=0$, keeping the original particle labels
$r$ and the reference density $\rho_0$ fixed. By \eqref{811}, we have
\begin{equation}\label{811''}
((\eta_\star)_r,\frac{\eta_\star}{r})(r)\in [C^{-1}(\overline{T}_*),C(\overline{T}_*)],\qquad       \mathcal E(0,u_\star)=\cE(\tau,U)\leq C(\overline{T}_*).
\end{equation}
The density associated with the restarted initial data is $\varrho_\star(r):= \frac{r^2\rho_0(r)}{\eta_\star(r)^2(\eta_\star)_r(r)}=\varrho(\tau,r)$.
Moreover, after a detailed calculation based on \eqref{811'}, we can also obtain 
\begin{equation}
\mathcal{F}(\eta_\star)\leq C(\overline{T}_*).
\end{equation}
Indeed, this can be derived by using $\eta_t=U$, $\eta|_{t=0}=r$, and \eqref{811'}. For example, 
\begin{equation*}
\begin{aligned}
\big|\chi^\sharp\rho_0^{(\frac32-\varepsilon_0)(\gamma-1)}\partial_r^4\eta_\star\big|_2&\le\int_0^{\tau}\big|\chi^\sharp\rho_0^{(\frac32-\varepsilon_0)(\gamma-1)}\partial_r^4U\big|_2\,\mathrm{d}s\\
&\le\sqrt{\tau}\Big(\int_0^{\tau}\big|\chi^\sharp\rho_0^{(\frac32-\varepsilon_0)(\gamma-1)}\partial_r^4U\big|_2^2\,\mathrm{d}s\Big)^{\frac{1}{2}} \le C(\overline{T}_*).
\end{aligned}
\end{equation*}

Hence, setting $\sigma*=C^{-1}(\overline{T}_*)$, $\sigma^*=C(\overline{T}_*)$, and $\mathfrak{M}_0=2C(\overline{T}_*)$, we obtain from Theorem \ref{local-Theorem1.1} that there exists $T_0>0$, depending on $C(\overline{T}_*)$ and $(a_1,A,\gamma,\delta,\varepsilon_0,\rho_0,\cK_1,\cK_2,G)$ but independent of $\tau$, such that the time-translated problem \eqref{eq:VFBP-La-eta-tau} with initial data $(u_\star,\eta_\star)$ at $t=\tau$ admits a unique classical solution $(U,\eta)$ on $[\tau,\tau+T_0]\times\bar I$, satisfying
\begin{equation}\label{b1-lo3}
\begin{aligned}
&\cE(t,U)+(t-\tau)\cD(t,U)\in L^\infty(\tau,\tau+T_0),\qquad  \cD (t,U)\in L^1(\tau,\tau+T_0),\\
& (\eta_r,\frac{\eta}{r})(t,r)\in \big[\frac{1}{2}C^{-1}(\overline{T}_*),\frac{3}{2}C(\overline{T}_*)\big] \qquad \quad \ \ \,\text{for $(t,r)\in [\tau,\tau+T_0]\times \bar I$}.
\end{aligned}
\end{equation}
Choose $\tau<\overline{T}_*$ sufficiently close to $\overline{T}_*$ so that $\overline{T}_*-\tau<T_0$. For every $S\in (\tau,\overline{T}_*)$, the restarted solution and the original solution solve the same problem on $[\tau,S]\times \bar I$ with the same data at $t=\tau$. The uniqueness assertion in Theorem \ref{local-Theorem1.1} implies that the two solutions coincide on this interval. The restarted solution therefore extends $(U,\eta)$ beyond $\overline{T}_*$, contradicting the maximality of $\overline{T}_*$. We conclude that $\overline{T}_*=\infty$.

Finally, for every finite $T>0$, the local regularity conclusions of Theorem \ref{local-Theorem1.1}, applied on overlapping time intervals, ensure that $(U,\eta)$ is a classical solution on $[0,T]\times \bar I$ satisfying \eqref{N111}. Uniqueness follows from the corresponding local uniqueness result. This
completes the proof of Theorem \ref{Theorem1.1}.

\appendix
\section{Some Basic Lemmas}\label{appendix A}

For the convenience of readers, we list some basic results that have been used in this paper. 

The first lemma concerns the separability and density of the weighted Sobolev spaces. 
\begin{Lemma}[\cite{kufner}]\label{W-space}
Let $\vartheta_1,\vartheta_2\in (-1,\infty)$, and let $\mathrm{w}(r)\sim r^{\vartheta_1}(1-r)^{\vartheta_2}$ be a function defined on $I$. Then, for $k\in \NN$, $H^k_{\mathrm{w}}$ is a reflexive separable Banach space. Moreover, if $k\in \mathbb{N}$, $C^\infty(\bar I)$ is dense in $H^k_{\mathrm{w}}$ with respect to its norm. 
\end{Lemma}

The second lemma is on the classical Sobolev embedding theorems.
\begin{Lemma}[\cite{leoni}]\label{sobolev-embedding}
We state the Sobolev embedding theorem in $\mathbb{R}$ and $\mathbb{R}^3$ separately. 
\begin{enumerate}
\item[{\rm (i)}]
Let $J\subset \RR$ be some  bounded open interval and $f=f(r)$ be some function on $J$. Then there exist two positive constants $(s_0,C)$, depending only on  $J$, such that
\begin{align*}
&\|f\|_{L^\infty(J)} \leq s_0\|f\|_{L^1(J)}+C\|f_r\|_{L^1(J)} &&\quad \text{for all }f\in W^{1,1}(J),\\[4pt]
&\|f\|_{L^\infty(J)} \leq s_0\|f\|_{L^2(J)}+C\|f_r\|_{L^2(J)} &&\quad \text{for all }f\in H^1(J),\\
&\|f\|_{L^\infty(J)} \leq s_0\|f\|_{L^2(J)}+ C\|f\|^\frac{1}{2}_{L^2(J)}\|f_r\|^\frac{1}{2}_{L^2(J)} &&\quad \text{for all }f\in H^1(J).
\end{align*}
In particular, $W^{1,1}(J)\into C(\bar J)$ and $H^1(J)\into C(\bar J)$ continuously{\rm ;} moreover, if $f$ vanishes at some point $r_0\in \bar J$, then $s_0=0$ can be chosen in the above.

\item[{\rm (ii)}]Let $E \subset \mathbb{R}^{3}$ be some bounded open set with smooth boundary. Assume that $f\in H^1(E)$. Then there exists a positive constant $C$ depending only on $E$, such that
\begin{equation*}
\|f\|_{L^{6}(E)} \leq C \|f\|_{H^{1}(E)}.
\end{equation*}
In particular, if $E=\RR^3$, then for any $f\in H^{1}(\RR^3)$, there exists a positive universal constant $C$ such that
\begin{equation*}
    \|f\|_{L^6(\RR^3)} \leq C \|\nabla f\|_{L^2(\RR^3)}.
\end{equation*}
\end{enumerate}
\end{Lemma}

The next two lemmas are on the Hardy inequality and some weighted interpolation inequality. In Lemmas \ref{hardy-inequality}--\ref{GNinequality}, we let $0\leq a<b<\infty$ and $J=(a,b)$, and let $d=d(r)$ be some function on $J$, taking one of the following two forms{\rm:}
\begin{equation*}
d(r)=r-a,\quad \text{or} \quad  d(r)=b-r.
\end{equation*}
\begin{Lemma}[\cites{opic1,opic2}]\label{hardy-inequality}
Let $p\in [1,\infty]$ and $\vartheta>-\frac{1}{p}$ $(\vartheta>0\text{ if }p=\infty)$. Then, for any $f$ such that $d^{\vartheta+\frac{1}{2}+\frac{1}{p}} (f,f_r)\in L^2(J)$,
\begin{enumerate}
\item[{\rm (i)}] If $p\in [1,2)$, for any $\varepsilon>0$, there exists a constant $C_1>0$, depending only on $(\varepsilon, p,a,b,\vartheta)$, such that
\begin{equation*}
\|d^{\vartheta+\varepsilon} f\|_{L^p(J)} \leq C_1 \big\|d^{\vartheta+\frac{1}{2}+\frac{1}{p}} (f,f_r)\big\|_{L^2(J)};
\end{equation*}
\item[{\rm (ii)}] If $p\in [2,\infty]$, there exists a constant $C_2>0$, which depends only on $(p,a,b,\vartheta)$ if $p\ne \infty$ and depends only on $(a,b,\vartheta)$ if $p=\infty$, such that
\begin{equation*}
\|d^\vartheta f\|_{L^p(J)} \leq C_2 \big\|d^{\vartheta+\frac{1}{2}+\frac{1}{p}} (f,f_r)\big\|_{L^2(J)}. 
\end{equation*}
\end{enumerate}
\end{Lemma}
 
\begin{Lemma}\label{GNinequality}
Let $p\in [2,\infty]$ and $\vartheta>-\frac{1}{p}$ $(\vartheta>0\text{ if }p=\infty)$. Then, for any $f$ satisfying $d^{\vartheta+\frac{1}{p}} (f,f_r)\in L^2(J)$, there exists a constant $C>0$, which depends only on $(p,a,b,\vartheta)$ if $p\ne \infty$ and depends only on $(a,b,\vartheta)$ if $p=\infty$, such that
\begin{equation*} 
\|d^\vartheta f\|_{L^p(J)} \leq C\big(\big\|d^{\vartheta+\frac{1}{p}} f\big\|_{L^2(J)}+\big\|d^{\vartheta+\frac{1}{p}} f\big\|_{L^2(J)}^\frac{1}{2}\big\|d^{\vartheta+\frac{1}{p}} f_r\big\|_{L^2(J)}^\frac{1}{2}\big).
\end{equation*}
\end{Lemma}
\begin{proof}
It suffices to  prove  the case when  $J=I$ and $d(r)=r$. The other cases can be dealt with  similarly. 

Let  $f\in C^\infty(\bar I)$. First, if $p=2$ and $\vartheta>-\frac{1}{2}$, a direct calculation yields
\begin{equation}\label{A..5}
|r^{\vartheta+j} f|_2^2= \frac{1}{2\vartheta+2j+1}f^2(1)- \frac{2}{2\vartheta+2j+1} \int_0^1 r^{2\vartheta+2j+1} ff_r\,\mathrm{d}r\quad\text{for $j=0$ or $\frac{1}{2}$}, 
\end{equation}
which, along with the H\"older inequality, implies
\begin{equation}\label{AA4}
\begin{aligned}
|r^\vartheta f|_2^2
&= \frac{2\vartheta+2}{2\vartheta+1} \big|r^{\vartheta+\frac{1}{2}} f\big|_2^2+ \frac{2}{2\vartheta+1} \int_0^1 (r-1)r^{2\vartheta+1} ff_r\,\mathrm{d}r\\
&\leq C\big(\big|r^{\vartheta+\frac{1}{2}} f\big|_2^2+ \big|r^{\vartheta+\frac{1}{2}} f\big|_2\big|r^{\vartheta+\frac{1}{2}} f_r\big|_2\big).
\end{aligned}
\end{equation}

Next, if $p=\infty$ and $\vartheta>0$, it follows from the above and  Lemma \ref{sobolev-embedding}  that
\begin{equation}\label{A7}
\begin{aligned}
|r^\vartheta f|_\infty^2&\leq  C\int_0^1 |(r^{2\vartheta} f^2)_r|\,\mathrm{d}r\leq C\int_0^1 r^{2\vartheta-1}f^2\,\mathrm{d}r+C\int_0^1 r^{2\vartheta}|f||f_r|\,\mathrm{d}r\\
&\leq C\big|r^{\vartheta-\frac{1}{2}} f\big|_2^2+C|r^\vartheta f|_2|r^\vartheta f_r|_2\leq C\big(\big|r^\vartheta f\big|_2^2+|r^\vartheta f|_2|r^\vartheta f_r|_2\big).
\end{aligned}
\end{equation}

Finally, if $p\in (2,\infty)$ and $\vartheta>-p^{-1}$, repeating the similar calculations  in \eqref{A..5}-\eqref{AA4}, combined with \eqref{A7}, leads to
\begin{align*}
|r^\vartheta f|_p^p
&= \frac{p\vartheta+2}{p\vartheta+1} \big|r^{\vartheta+\frac{1}{p}} f\big|_p^p+ \frac{p}{p\vartheta+1} \int_0^1 (r-1)r^{p\vartheta+1} |f|^{p-2}ff_r\,\mathrm{d}r\\
&\leq C\big|r^{\vartheta+\frac{1}{p}} f\big|_\infty^{p-2}\big(\big|r^{\vartheta+\frac{1}{p}} f\big|_2^2+ \big|r^{\vartheta+\frac{1}{p}} f\big|_2\big|r^{\vartheta+\frac{1}{p}} f_r\big|_2\big)\leq C \big(\big|r^{\vartheta+\frac{1}{p}} f\big|_2^p+ \big|r^{\vartheta+\frac{1}{p}} f\big|_2^\frac{p}{2}\big|r^{\vartheta+\frac{1}{p}} f_r\big|_2^\frac{p}{2}\big).
\end{align*}
Therefore, we complete the proof for the case when $f\in C^\infty(\bar I)$.

For $f\in H^1_{d^{2\vartheta+\frac{2}{p}}}$, there exists a sequence $\{f^\varepsilon\}_{\varepsilon>0}\subset C^\infty(\bar I)$ due to Lemma \ref{W-space} 
such that
\begin{equation*} 
\big|r^{\vartheta+\frac{1}{p}} (f^\varepsilon- f)\big|_2+\big|r^{\vartheta+\frac{1}{p}} (f^\varepsilon_r- f_r)\big|_2 \to 0 \qquad \text{as }\varepsilon\to 0.
\end{equation*}
Then, via the density argument, the desired inequality holds for all $f\in H^1_{d^{2\vartheta+\frac{2}{p}}}$. 
\end{proof}

The fifth lemma is used to obtain the time-weighted estimates of the velocity.
\begin{Lemma}[\cite{bjr}]\label{bjr}
Let $f\in L^2([0,T]; L^2)$. Then there exists a sequence $\{t_k\}_{k=1}^\infty$ such that
\begin{equation*}
t_k\rightarrow 0, \quad\,\, t_k |f(t_k)|^2_{2}\rightarrow 0 \qquad\,\, \text{as $k\rightarrow\infty$}.
\end{equation*}
\end{Lemma}

Finally, the following lemma is used to obtain the higher-order elliptic estimates.
\begin{Lemma}\label{prop2.1}
Assume that $\bar d=\bar d(r)$ is a function defined on $I$ and  satisfies
\begin{equation}\label{con2.9-pre}
\bar d\in C^1(\bar I)\cap C^2((0,1]),\qquad \frac{1}{K}(1-r)\leq \bar d \leq K(1-r) \ \ \text{for some $K>1$}.
\end{equation}
Let $(b,c)$ be two parameters such that
$
\frac{1}{2}<b\leq \frac{c+1}{2}
$,
and let $f=f(r)\in L^1_{\mathrm{loc}}$ satisfy both $f\in H^1_{\bar d^{2q}}(\frac{1}{2},1)$ for some $q\in [b, \frac{c+1}{2}]$ and 
\begin{equation}\label{con2.10}
\zeta^\sharp \bar d^{b}\cF_{\mathrm{cross}},\ \chi^\sharp\bar d^{b} f\in L^2,\qquad\text{where }\cF_{\mathrm{cross}}:=f_r+c\frac{\bar d_r}{\bar d}f.
\end{equation}
Then there exists a constant $C>0$ depending  only on $(\bar d,b,c)$ such that
\begin{equation}\label{con2.11}
|\zeta^\sharp\bar d^{b}f_r|_2\leq C\big(\big|\zeta^\sharp \bar d^{b}\cF_{\mathrm{cross}}\big|_2+|\chi^\sharp\bar d^{b} f|_2\big).
\end{equation}
\end{Lemma}
\begin{proof}
We divide the proof into three steps.

\textbf{1. Case $f\in C^\infty([\frac{1}{2},1])$.}
It follows from integration by parts, the facts that $\frac{1}{2}<b\leq \frac{c+1}{2}$ and $c>0$, \eqref{con2.9-pre}--\eqref{con2.10}, Lemma \ref{GNinequality}, and the Young inequality that
\begin{align}
|\zeta^\sharp\bar{d}^b f_r|_2^2&=\big|\zeta^\sharp \bar d^{b}\cF_{\mathrm{cross}}\big|_2^2-c^2|\zeta^\sharp\bar{d}^{b-1}\bar{d}_rf|_2^2 - 2c\int_0^1 (\zeta^\sharp)^2\bar{d}^{2b-1}\bar{d}_r  f f_r\,\mathrm{d}r\nonumber\\
&\leq \big|\zeta^\sharp \bar d^{b}\cF_{\mathrm{cross}}\big|_2^2+ c\int_0^1 \big(2\zeta^\sharp(\zeta^\sharp)_r\bar{d}_{r}+(\zeta^\sharp)^2\bar{d}_{rr}\big) \bar{d}^{2b-1} f^2\,\mathrm{d}r  \label{2004} \\
&\leq \big|\zeta^\sharp \bar d^{b}\cF_{\mathrm{cross}}\big|_2^2+C\big(|(\zeta^\sharp)_r\bar{d}^{-1} \bar{d}_{r}|_\infty|\chi^\sharp\bar{d}^{b}f|_2^2+ |\chi^\sharp\bar{d}_{rr}|_\infty\big|\zeta^\sharp\bar{d}^{b-\frac{1}{2}}f\big|_2^2\big)\nonumber\\
&\leq  \big|\zeta^\sharp \bar d^{b}\cF_{\mathrm{cross}}\big|_2^2+C|\chi^\sharp\bar{d}^{b}f|_2^2+\frac{1}{2}|\zeta^\sharp\bar{d}^b f_r|_2^2,\nonumber
\end{align}
which yields \eqref{con2.11}.

\textbf{2. Case  $f\in H^1_{\bar{d}^{2b}}(\frac{1}{2},1)$.} In this case, we can repeat the calculation in \eqref{2004} to derive \eqref{con2.11}, except for justifying the following integral equality: 
\begin{equation}\label{fenbujifen}
\begin{aligned}
-2\int_0^1 (\zeta^\sharp)^2\bar{d}^{2b-1}\bar{d}_r  f f_r\,\mathrm{d}r
&=  \int_0^1 \big(2\zeta^\sharp(\zeta^\sharp)_r\bar{d}_{r}+(\zeta^\sharp)^2\bar{d}_{rr}\big) \bar{d}^{2b-1} f^2\,\mathrm{d}r\\
&\quad\,\, +(2b-1)\int_0^1 (\zeta^\sharp)^2\bar{d}^{2b-2}(\bar{d}_r)^2 f^2\,\mathrm{d}r.
\end{aligned}
\end{equation}
Indeed, thanks to Lemma \ref{W-space}, there exists a sequence $\{f^\varepsilon\}_{\varepsilon>0}\subset C^\infty([\frac{1}{2},1])$ such that
\begin{equation}\label{b6b}
|\chi^\sharp\bar{d}^b (f^\varepsilon- f)|_2+|\chi^\sharp\bar{d}^b(f^\varepsilon_r- f_r)|_2\to 0 \qquad \text{as $\varepsilon\to 0$},
\end{equation}
which, along with Lemma \ref{hardy-inequality}, yields
\begin{equation}\label{b7b}
|\chi^\sharp\bar{d}^{b-1}(f^\varepsilon-f)|_2+\big|\chi^\sharp\bar{d}^{b-\frac{1}{2}}(f^\varepsilon-f)\big|_\infty \to 0 \qquad \text{as $\varepsilon\to 0$}.
\end{equation}
Hence, according to \eqref{b6b}--\eqref{b7b}  and integration by parts for $f^\varepsilon$, we have
\begin{equation*}
\begin{aligned}
-2\int_0^1 (\zeta^\sharp)^2\bar{d}^{2b-1}\bar{d}_r f^\varepsilon f^\varepsilon_r\,\mathrm{d}r
&=  \int_0^1 \big(2\zeta^\sharp(\zeta^\sharp)_r\bar{d}_{r}+(\zeta^\sharp)^2\bar{d}_{rr}\big) \bar{d}^{2b-1} (f^\varepsilon)^2\,\mathrm{d}r\\
&\quad +(2b-1)\int_0^1 (\zeta^\sharp)^2\bar{d}^{2b-2}(\bar{d}_r)^2 (f^\varepsilon)^2\,\mathrm{d}r.
\end{aligned}
\end{equation*}
Letting $\varepsilon\to 0$ implies that \eqref{fenbujifen} holds for all $f\in H^1_{\bar{d}^{2b}}(\frac{1}{2},1)$.

\textbf{3. General case.}
It suffices to establish \eqref{con2.11} when \eqref{con2.10} holds and
\begin{equation*}
b<\frac{c+1}{2}, \qquad  q=\frac{c+1}{2},\qquad \ f\in H^1_{\bar{d}^{c+1}}\big(\frac{1}{2},1\big),
\end{equation*}
due to the fact that $H^1_{\bar{d}^{2q}}(\frac{1}{2},1)\subset H^1_{\bar{d}^{c+1}}(\frac{1}{2},1)$ if $q\leq \frac{c+1}{2}$. Note that, in this case, integration by parts in \eqref{fenbujifen} fails owing to $(\zeta^\sharp)^2\bar{d}^{2b-1}\bar{d}_r ff_r\notin L^1$. 
To overcome this difficulty, set
\begin{equation*}
\vartheta:= \frac{c+1}{2}-b,\qquad \bar{d}_j:=\bar{d}+\frac{1}{j} \quad \text{for $j\in \NN^*$}.
\end{equation*}
We first show a variant of  Lemma \ref{GNinequality}, that is, for all $f\in H^1_{\bar{d}^{c+1}}(\frac{1}{2},1)$,
\begin{equation}\label{2..8}
\big|\zeta^\sharp\bar{d}^\frac{c}{2}\bar{d}_j^{-\vartheta}f\big|_2^2\leq C\big(\big|\chi^\sharp\bar{d}^\frac{c+1}{2}\bar{d}_j^{-\vartheta}f\big|_2^2+\big|\zeta^\sharp\bar{d}^\frac{c+1}{2}\bar{d}_j^{-\vartheta}f\big|_2\big|\zeta^\sharp\bar{d}^\frac{c+1}{2}\bar{d}_j^{-\vartheta}f_r\big|_2\big).
\end{equation}

By  Lemma \ref{W-space} and the proof in Lemma \ref{GNinequality}, it suffices to give  \eqref{2..8}  for $f\in C^\infty[\frac{1}{2},1]$. We can further let $(\bar{d},\bar{d}_j)=(d,d_j)$ with $d=d(r):=1-r$ and $d_j:=d+\frac{1}{j}$, due to
\begin{equation*}
\frac{d}{K}\leq \bar{d} \leq K d,\qquad \frac{d_j}{K}\leq \bar{d}_j \leq K d_j.
\end{equation*} 
It follows from the above reductions, integration by parts, and $(d^{c+1})_r=-(c+1)d^c$ that
\begin{align*}
&\int_0^1 (\zeta^\sharp)^2 d^c d_j^{-2\vartheta} f^2\,\mathrm{d}r\\
&= \frac{2}{c+1} \Big( \int_0^1 \zeta^\sharp(\zeta^\sharp)_r d^{c+1} d_j^{-2\vartheta} f^2\,\mathrm{d}r+\int_0^1 (\zeta^\sharp)^2 d^{c+1} d_j^{-2\vartheta} f f_r\,\mathrm{d}r + \vartheta \int_0^1 (\zeta^\sharp)^2 d^{c+1} d_j^{-2\vartheta-1} f^2 \,\mathrm{d}r\Big)\\
&\leq  \frac{2}{c+1} \Big( \int_0^1 \zeta^\sharp(\zeta^\sharp)_r d^{c+1} d_j^{-2\vartheta} f^2\,\mathrm{d}r+\int_0^1 (\zeta^\sharp)^2 d^{c+1} d_j^{-2\vartheta} f f_r\,\mathrm{d}r+\vartheta \int_0^1 (\zeta^\sharp)^2 d^{c} d_j^{-2\vartheta} f^2 \,\mathrm{d}r\Big),
\end{align*}
which implies \eqref{2..8}:
\begin{equation}\label{2..9}
\begin{aligned}
\int_0^1 (\zeta^\sharp)^2 d^c d_j^{-2\vartheta} f^2\,\mathrm{d}r& \leq  \frac{1}{b} \Big( \int_0^1 \zeta^\sharp(\zeta^\sharp)_r d^{c+1} d_j^{-2\vartheta} f^2\,\mathrm{d}r+\int_0^1 (\zeta^\sharp)^2 d^{c+1} d_j^{-2\vartheta} f f_r\,\mathrm{d}r\Big)\\
&\leq C\big(\big|\chi^\sharp d^\frac{c+1}{2} d_j^{-\vartheta}f\big|_2^2+\big|\zeta^\sharp d^\frac{c+1}{2} d_j^{-\vartheta}f\big|_2\big|\zeta^\sharp d^\frac{c+1}{2} d_j^{-\vartheta}f_r\big|_2\big).
\end{aligned}
\end{equation}

Now, we continue to prove \eqref{con2.11}. It follows from \eqref{con2.10} and $0\leq \bar d/\bar d_j\leq 1$ that
\begin{equation}\label{con2.15}
Q_j:=\zeta^\sharp\bar{d}_j^{-\vartheta}\bar{d}^\frac{c+1}{2} \cF_{\mathrm{cross}} \in L^2 \qquad \text{for any $j\in \NN^*$}.
\end{equation}
Using a density argument similar to that in Step 2, we can show that the following integral equality still holds, {\it i.e.}, for all $f\in H^1_{\bar{d}^{c+1}}(\frac{1}{2},1)$ and $j\in \NN^*$,
\begin{equation}\label{fenbujifen1}
\begin{aligned}
-2\int_0^1 (\zeta^\sharp)^2\bar{d}^{c}\bar{d}_j^{-2\vartheta}\bar{d}_r f f_r\,\mathrm{d}r&=  \int_0^1 \big(2\zeta^\sharp(\zeta^\sharp)_r\bar{d}_{r}+(\zeta^\sharp)^2\bar{d}_{rr}\big) \bar{d}^{c}\bar{d}_j^{-2\vartheta} f^2\,\mathrm{d}r\\
&\quad +\int_0^1 (\zeta^\sharp)^2\big(c\bar{d}^{c-1}\bar{d}_j^{-2\vartheta} -2\vartheta \bar{d}^{c}\bar{d}_j^{-2\vartheta-1}\big)(\bar{d}_r)^2 f^2\,\mathrm{d}r.
\end{aligned}
\end{equation}
Hence, based on \eqref{fenbujifen1} and the calculation similar to \eqref{2004}, we deduce from \eqref{con2.15}  that
\begin{align*}
\big|\zeta^\sharp\bar{d}^\frac{c+1}{2}\bar{d}_j^{-\vartheta} f_r\big|_2^2&=|Q_j|_2^2-c^2\big|\zeta^\sharp\bar{d}^{\frac{c-1}{2}}\bar{d}_j^{-\vartheta}\bar{d}_rf\big|_2^2 - 2c\int_0^1 (\zeta^\sharp)^2\bar{d}^{c}\bar{d}_j^{-2\vartheta}\bar{d}_r  f f_r\,\mathrm{d}r\\
&\leq |Q_j|_2^2+ c\int_0^1 \big(2\zeta^\sharp(\zeta^\sharp)_r\bar{d}_{r}+(\zeta^\sharp)^2\bar{d}_{rr}\big) \bar{d}^{c}\bar{d}_j^{-2\vartheta} f^2\,\mathrm{d}r\notag\\
&\leq |Q_j|_2^2+C\big(|(\zeta^\sharp)_r\bar{d}^{-1} \bar{d}_{r}|_\infty\big|\chi^\sharp\bar{d}^{\frac{c+1}{2}}\bar{d}_j^{-\vartheta}f\big|_2^2+ |\chi^\sharp\bar{d}_{rr}|_\infty\big|\zeta^\sharp\bar{d}^{\frac{c}{2}}\bar{d}_j^{-\vartheta}f\big|_2^2\big),
\end{align*} 
which, along with \eqref{con2.9-pre} and \eqref{2..8}, gives
\begin{equation*}
\big|\zeta^\sharp\bar{d}^{\frac{c+1}{2}}\bar{d}_j^{-\vartheta}f_r\big|_2^2\leq C\big(|Q_j|_2^2+\big|\chi^\sharp\bar{d}^{\frac{c+1}{2}}\bar{d}_j^{-\vartheta}f\big|_2^2\big) \leq C\big(\big|\zeta^\sharp \bar d^{b}\cF_{\mathrm{cross}}\big|_2^2+|\chi^\sharp\bar{d}^{b} f|_2^2\big).
\end{equation*}
Since $C$ is independent of $j$, we can extract a subsequence (still denoted by $j$)  such that
\begin{equation}\label{equ213}
\zeta^\sharp\bar{d}^{\frac{c+1}{2}}\bar{d}_j^{-\vartheta}f_r\to g\qquad \text{weakly in }L^2 \quad \text{as }j\to \infty,
\end{equation}
for some limit function $g\in L^2$, and 
\begin{equation*}
|g|_2\leq \liminf_{j\to \infty}\big|\zeta^\sharp\bar{d}^{\frac{c+1}{2}}\bar{d}_j^{-\vartheta}f_r\big|_2\leq C\big(\big|\zeta^\sharp \bar d^{b}\cF_{\mathrm{cross}}\big|_2+|\chi^\sharp\bar{d}^{b} f|_2\big). 
\end{equation*}

Since $f_r\in L^1_{\mathrm{loc}}$, the Lebesgue dominated convergence theorem also gives
\begin{equation}\label{equ214}
\zeta^\sharp\bar{d}^{\frac{c+1}{2}}\bar{d}_j^{-\vartheta}f_r\to \zeta^\sharp\bar{d}^b f_r\qquad \text{in }L^1_{\mathrm{loc}} \ \ \text{as }j\to\infty.
\end{equation}
Hence, by \eqref{equ213}--\eqref{equ214} and the uniqueness of the limits, we have $g=\zeta^\sharp\bar{d}^b f_r$. 
\end{proof}

\section{Coordinate Transformations}\label{appb}

Generally, it is desirable to consider the Lagrangian formulation, so that we can pull back  \eqref{eq:1.1-vfbp} on the moving domain $\Omega(t)$ to a problem on a fixed domain $\Omega$. To this end, denote by $\boldsymbol{x}=\boldsymbol{\eta}(t,\boldsymbol{y})$ the position of the fluid particle $\boldsymbol{x} \in \Omega(t)$ at time $t\geq 0$ so that
\begin{equation}\label{flow-map-md}
\boldsymbol{\eta}_t(t,\boldsymbol{y})=\boldsymbol{u}(t,\boldsymbol{\eta}(t,\boldsymbol{y})) \ \ \text{for $t>0$},\qquad \text{with $\boldsymbol{\eta}(0,\boldsymbol{y})=\boldsymbol{y}$},
\end{equation}
and $(t,\boldsymbol{y})$ are the 3-D Lagrangian coordinates.
This appendix is devoted to showing the conversion of some Sobolev spaces between the 3-D Lagrangian coordinates $\boldsymbol{y}$ and the corresponding spherical coordinate $r=|\boldsymbol{y}|$ for spherically symmetric functions.

Let $0\leq a<b$, $\cJ:=\{\boldsymbol{y}\in \mathbb{R}^3: \, a\leq |\boldsymbol{y}|< b\}$, and $r\in J:=[a,b)$ with $r=|\boldsymbol{y}|$. Consider a coordinate transformation $\boldsymbol{\xi}=\boldsymbol{\xi}(\boldsymbol{y})\in C^\infty(\bar\cJ)$ such that
\begin{equation*}
\boldsymbol{\xi}(\boldsymbol{y})=\xi(r)\frac{\boldsymbol{y}}{r}\quad\text{with $\xi(r)\geq  0$ and $\xi_r>0$},\qquad\boldsymbol{\xi}:\cJ\to \cG:=\boldsymbol{\xi}(\cJ),\qquad \boldsymbol{y}\mapsto \boldsymbol{x}=\boldsymbol{\xi}(\boldsymbol{y}). 
\end{equation*} 
Assume that $\nabla_{\boldsymbol{y}}\boldsymbol{\xi}$ is a non-singular matrix and $\xi_r(r)>0$ on $J$, and then define
\begin{align*}
D_\xi f=\frac{f_r}{\xi_r},\qquad\cB=(\cB_{ij})_{1\leq i,j\leq 3} \qquad \text{with $\cB_{ij}:=((\nabla_{\boldsymbol{y}}\boldsymbol{\xi})^{-1})_{ij}$},\\[-2pt]
\nabla_\cB f=((\nabla_\cB f)_1,(\nabla_\cB f)_2,(\nabla_\cB f)_3)^\top \qquad \text{with $(\nabla_\cB f)_i=\sum_{k=1}^3\cB_{ki}\partial_{y_k}f$},\\[-3pt]
\boldsymbol{f}=(f_1,f_2, f_3)^\top,\qquad\nabla_\cB \boldsymbol{f}=((\nabla_\cB \boldsymbol{f})_{ij})_{1\leq i,j\leq 3} \qquad \text{with $(\nabla_\cB \boldsymbol{f})_{ij}=\sum_{k=1}^3 \cB_{kj}\partial_{y_k}f_i$},    
\end{align*}
where $f(\boldsymbol{y})=f(r)$ and $\boldsymbol{f}(\boldsymbol{y})=f(r)\frac{\boldsymbol{y}}{r}$ are sufficiently smooth functions. 

Then we have the following coordinate transformations.
\begin{Lemma}\label{lemma-initial}
Let $(g,\boldsymbol{g})(\boldsymbol{x})$ be spherically symmetric functions defined on $\cG$  satisfying 
\begin{equation*}
f(\boldsymbol{y})=g(\boldsymbol{\xi}(\boldsymbol{y}))=g(\boldsymbol{x}),\qquad\boldsymbol{f}(\boldsymbol{y})=\boldsymbol{g}(\boldsymbol{\xi}(\boldsymbol{y}))=\boldsymbol{g}(\boldsymbol{x}).
\end{equation*}
Then, for any $q\in [1,\infty]$, the following statements hold{\rm :}
\begin{enumerate}
\item[{\rm(i)}] Transformations for $(g,f)${\rm:} for $k=1,2,3$, 
\begin{equation*}
\begin{aligned}
\|g\|_{L^q(\cG)}&\sim\|(\det \nabla_{\boldsymbol{y}}\boldsymbol{\xi})^\frac{1}{q} f\|_{L^q(\cJ)}\sim 
\|(\xi^2\xi_r)^\frac{1}{q} f\|_{L^q(J)},\\[2pt]
\|\nabla^kg\|_{L^q(\cG)}&\sim\|(\det \nabla_{\boldsymbol{y}}\boldsymbol{\xi})^\frac{1}{q}\nabla_\cB^k f\|_{L^q(\cJ)}\sim  \|(\xi^2\xi_r)^\frac{1}{q}\mathscr{D}_\xi^{k-1} (D_\xi f)\|_{L^q(J)};
\end{aligned}
\end{equation*}
\item[{\rm(ii)}] Transformations for $(\boldsymbol{g},\boldsymbol{f})${\rm:} for $k=0,1,2,3,4$, 
\begin{equation*}
\|\nabla^{k}\boldsymbol{g}\|_{L^q(\cG)}\sim\|(\det \nabla_{\boldsymbol{y}}\boldsymbol{\xi})^\frac{1}{q}\nabla_\cB^{k} \boldsymbol{f}\|_{L^q(\cJ)}\sim \|(\xi^2\xi_r)^\frac{1}{q}\mathscr{D}_\xi^k f\|_{L^q(J)}.
\end{equation*}
\end{enumerate}
Here, $E\sim F$ denotes $C^{-1}E\leq F\leq CE$ for some constant $C\geq 1$ depending only on the spatial dimension.
\end{Lemma}

\begin{proof}
It suffices to prove the transformations for $(\boldsymbol{g},\boldsymbol{f})$, since $\nabla_{\boldsymbol{y}} f=f_r\frac{\boldsymbol{y}}{r}$ can be regarded as a vector  function $\boldsymbol{h}=h\frac{\boldsymbol{y}}{r}$ with $h=f_r$. Moreover, for simplicity, we only sketch the proof for the highest-order estimates. We divide the proof into two steps.

\textbf{1.} We first prove the case when $\xi(r)=r$. In this case, $\boldsymbol{\xi}(\boldsymbol{y})=\boldsymbol{y}$ and $\cB$ is the $3\times 3$ identity matrix. It follows from direct calculations that
\begin{align*}
&\begin{aligned}
(f_k)_{y_iy_jy_\ell y_p}
&=\frac{y_i y_j y_ky_\ell y_p}{r^5}f_{rrrr}\\
&\quad +\Big(\frac{\delta_{ip} y_j y_ky_\ell+\delta_{jp} y_i y_ky_\ell+\delta_{kp} y_i y_jy_\ell+\delta_{\ell p} y_i y_jy_k}{r^3}\\
&\quad\quad \ \ + \frac{\delta_{i\ell}y_jy_ky_p+\delta_{j\ell}y_iy_ky_p+\delta_{k\ell}y_iy_jy_p}{r^3}\\
&\quad\quad \ \ +\frac{\delta_{ij}y_ky_\ell y_p+\delta_{ik}y_jy_\ell y_p+\delta_{jk}y_iy_\ell y_p}{r^3}-\frac{10y_i y_j y_ky_\ell y_p}{r^5}\Big)\big(\frac{f}{r}\big)_{rrr}\\
&\quad +\Big(\frac{\delta_{i\ell}\delta_{jp}y_k+\delta_{i\ell}\delta_{kp}y_j+\delta_{j\ell}\delta_{ip}y_k+\delta_{j\ell}\delta_{kp}y_i+\delta_{k\ell}\delta_{ip}y_j+\delta_{k\ell}\delta_{jp}y_i}{r}\\
&\quad\quad \ \ +\frac{\delta_{ij}\delta_{kp}y_\ell+\delta_{ij}\delta_{\ell p}y_k+\delta_{ik}\delta_{jp}y_\ell+\delta_{ik}\delta_{\ell p} y_j+\delta_{jk}\delta_{ip}y_\ell+\delta_{jk}\delta_{\ell p}y_i}{r}\\
\end{aligned}\\
&\qquad\quad\quad \ \ \begin{aligned}
&\quad\quad \ \ +\frac{\delta_{ij}\delta_{k\ell}y_p+\delta_{ik}\delta_{j\ell}y_p+\delta_{jk}\delta_{i\ell}y_p}{r}\\
&\quad\quad \ \ -\frac{3(\delta_{ip} y_j y_ky_\ell+\delta_{jp} y_i y_ky_\ell+\delta_{kp} y_i y_jy_\ell+\delta_{\ell p}y_i y_j y_k)}{r^3}\\
&\quad\quad \ \ -\frac{3(\delta_{i\ell}y_jy_ky_p+\delta_{j\ell}y_iy_ky_p+\delta_{k\ell}y_iy_jy_p)}{r^3}\\
&\quad\quad \ \ -\frac{3(\delta_{ij}y_ky_\ell y_p+\delta_{ik}y_jy_\ell y_p+\delta_{jk}y_iy_\ell y_p)}{r^3}+\frac{15y_i y_j y_ky_\ell y_p}{r^5}\Big)\Big(\frac{1}{r}\big(\frac{f}{r}\big)_{r}\Big)_r,   
\end{aligned}
\end{align*}
where $\delta_{ij}$ denotes the Kronecker symbol with indices $(i,j)$: $\delta_{ij}= 1$ if $i=j$, $\delta_{ij}=0$ if $i\neq j$. Then the above expressions yield
\begin{equation*}
|\nabla_{\boldsymbol{y}}^4 \boldsymbol{f}|^2=|f_{rrrr}|^2+20\Big|\big(\frac{f}{r}\big)_{rrr}\Big|^2+ 120 \Big|\Big(\frac{1}{r}\big(\frac{f}{r}\big)_r\Big)_r\Big|^2.
\end{equation*}
Moreover, since $(\frac{f_{rr}}{r})_r=(\frac{f}{r})_{rrr}+2(\frac{1}{r}(\frac{f}{r})_r)_r$, we obtain from the above that
\begin{equation}\label{BB}
|\nabla^4_{\boldsymbol{y}} \boldsymbol{f}|\sim |\sD_r^4 f|.
\end{equation}

Finally, let $\omega_3$ denote the surface area of the unit sphere in $\RR^3$. By the integral identity
\begin{equation*}
\int_\cJ f(\boldsymbol{y})\,\mathrm{d}\boldsymbol{y}= \omega_3\int_J f(r)r^2\,\mathrm{d}r,
\end{equation*}
 we thus obtain the desired conclusions  of this lemma when $\boldsymbol{\xi}(\boldsymbol{y})=\boldsymbol{y}$ from \eqref{BB}.

\textbf{2.} For general $\boldsymbol{\xi}(\boldsymbol{y})$, we can first repeat the calculations in Step 1 with the coordinate $\boldsymbol{x}=\boldsymbol{\xi}(\boldsymbol{y})$ and the function $\boldsymbol{g}(\boldsymbol{x})$. Specifically, if we let $x:=|\boldsymbol{x}|$, then $x=\xi(r)\in W$ and $W:=[\xi(a),\xi(b))$, and we can obtain from \eqref{BB} that
\begin{equation*} 
\|\nabla^4 \boldsymbol{g}\|_{L^q(\cG)}\sim \big\|x^\frac{2}{q}\sD_x^4 g\big\|_{L^q(W)}.
\end{equation*}

Next, using the coordinate transformations $\boldsymbol{x}=\boldsymbol{\xi}(\boldsymbol{y})$ and $x=\xi(r)$, we have $\nabla^k\boldsymbol{g}=\nabla_{\cB}^k\boldsymbol{f}$ and $\partial_x^k g=D_\xi^k f$ for $k=0,1,2,3,4$. Therefore, the following integral identities
\begin{equation*}
\int_\cG g(\boldsymbol{x})\,\mathrm{d}\boldsymbol{x}=\int_\cJ f(\boldsymbol{y})(\det \nabla_{\boldsymbol{y}}\boldsymbol{\xi}) \,\mathrm{d}\boldsymbol{y},\qquad \int_W g(x)\,\mathrm{d}x=  \int_J f(r) \xi_r\,\mathrm{d}r
\end{equation*}
lead to the desired results of this lemma.
\end{proof}

\smallskip
\noindent{\bf Acknowledgments:}   This research is partially supported by National Key R$\&$D Program of China (No. 2022YFA1007300),  National Natural Science Foundation of China under the Grant 12471212, and The Royal Society (UK)-Newton International Fellowships NF170015.

\bigskip
\noindent{\bf Conflict of Interest:} The authors declare that they have no conflict of interest. 

\bigskip
\noindent{\bf Data availability:} Data sharing is not applicable to this article as no datasets were generated or analyzed during the current study.


\smallskip

\begin{thebibliography}{99}



\bibitem{bjr} J. L.  Boldrini,  M. A.  Rojas-Medar, and  E. Fern\'andez-Cara, Semi-Galerkin approximation and regular solutions to the equations of the nonhomogeneous asymmetric fluids, \textit{J. Math. Pures  Appl.} \textbf{82} (2003), 1499--1525.


\bibitem{bd6} D. Bresch and B. Desjardins,  Existence of global weak solutions for a 2D viscous Shallow water equations and convergence to the quasi-geostrophic model, \textit{Commun. Math. Phys.} \textbf{238} (2003), 211--223.





\bibitem{opic1} R. C. Brown and B. Opic, Embeddings of weighted Sobolev spaces into spaces of continuous functions. \textit{Proc. Roy. Soc. London Ser. A} \textbf{439} (1992), 279--296.



\bibitem{Chan} S. Chandrasekhar, {\sl An introduction to the study of stellar structures},  University of Chicago Press, Chicago,  1938.


 



\bibitem{CZZ2} G.-Q. Chen, J. Zhang, and S. Zhu, 
Global well-posedness of the vacuum free boundary problem for the degenerate compressible  Navier-Stokes equations with large  data of spherical symmetry, submitted, 2026, arXiv:2601.06620.





\bibitem{coutand2} D. Coutand, H. Lindblad, and S. Shkoller, A priori estimates for the free-boundary 3D compressible Euler equations in physical vacuum, \textit{Commun. Math. Phys.} \textbf{296} (2010), 559--587.





\bibitem{coutand3} D. Coutand and S. Shkoller, Well-posedness in smooth function spaces for moving-boundary three-dimensional compressible Euler equations in physical vacuum, \textit{Arch. Ration. Mech. Anal.} \textbf{206} (2012), 515--616.

\bibitem{cox}J. Cox and  R. Giuli, {\sl Principles of Stellar Structure}, I, II, Gordon and Breach, New York, 1968.










\bibitem{fu3} E. Feireisl,  {\sl Dynamics of Viscous Compressible Fluids}, Oxford: Oxford University Press, 2004.
	





\bibitem{GL1} X. Gu and Z. Lei, Local well-posedness of the three dimensional compressible Euler-Poisson equations with physical vacuum, \textit{J. Math. Pures Appl.} \textbf{105} (2016), 662--723.

\bibitem{Yan} Y. Guo and B. Pausader, Global smooth ion dynamics in the Euler–Poisson system, \textit{Commun. Math. Phys.} \textbf{303} (2011), 89--125.






\bibitem{Jang-Hadzic} M. Hadžić and J. Jang, Expanding large global solutions of the equations of compressible fluid mechanics, \textit{Invent. Math.} \textbf{214} (2018), 1205--1266.

\bibitem{Jang-Hadzic3} M. Hadžić and J. Jang, Nonlinear stability of expanding star solutions of the radially symmetric mass-critical Euler-Poisson system, \textit{ Comm. Pure Appl. Math.} \textbf{71} (2018), 827--891.


\bibitem{Jang-Hadzic2} M. Hadžić and J. Jang, A class of global solutions to the Euler-Poisson system,  \textit{Comm. Math. Phys.} \textbf{370} (2019),  475--505. 







\bibitem{Jang} J. Jang, Local well-posedness of dynamics of viscous gaseous stars, \textit{Arch. Ration. Mech. Anal.} \textbf{195} (2010),  797--863. 
 


 
\bibitem{Jang-M2} J. Jang and N. Masmoudi, Well-posedness of compressible Euler equations in a physical vacuum, \textit{Comm. Pure Appl. Math.} \textbf{68} (2015), 61--111.

	



\bibitem{kufner} A. Kufner, {\sl Weighted Sobolev Spaces}, John Wiley $\&$ Sons, New York, 1985.
	

\bibitem{leoni} G. Leoni, {\sl A First Course in Sobolev Spaces, 2nd Ed.}, Graduate Studies in Mathematics, vol. 181, American Mathematical Society, Providence, RI, 2017.





\bibitem{LWX2} H.-L. Li, Y. Wang, and Z. Xin, On the vacuum free boundary problem of the viscous Saint-Venant system for shallow water in two dimensions, \textit{Math.  Ann.} \textbf{391} (2025),  3555--3639.



	

\bibitem{Lin} S.-S. Lin,  Stability of gaseous stars in spherically symmetric motions,  \textit{SIAM J. Math. Anal.}  \textbf{28} (1997), 539--569.

\bibitem{LinZeng} Z. Lin and C. Zeng, Separable Hamiltonian PDEs and turning point principle for stability of gaseous stars, \textit{Comm. Pure Appl. Math.}  \textbf{75} (2022), 2511--2572.


\bibitem {lions} P. L. Lions, {\sl Mathematical Topics in Fluid Mechanics: Compressible Models}, Oxford University Press, USA, 1998.


\bibitem{LiuTP} T.-P. Liu, Compressible flow with damping and vacuum, \textit{Japan J. Ind. Appl. Math.} \textbf{13} (1996), 25--32.
	

\bibitem{taiping1} T.-P. Liu and J.A. Smoller, On the vacuum state for the isentropic gas dynamics equations, \textit{Adv. in Appl. Math.} \textbf{1} (1980), 345--359.


\bibitem{taiping2} T.-P. Liu and T. Yang,   Compressible Euler equations with vacuum, \textit{J. Differential  Equations}  \textbf{140} (1997), 223–237.




\bibitem{LXZ3} T. Luo, Z. Xin, and H. Zeng, Nonlinear asymptotic stability of the Lane-Emden solutions for the viscous gaseous star problem with degenerate density dependent viscosities, \textit{Commun. Math. Phys.} \textbf{347} (2016), 657--702.







 
\bibitem{Makino} T. Makino,  On a local existence theorem for the evolution equation of gaseous stars, \textit{Patterns and Waves. Stud. Math. Appl.}  \textbf{18} (1986), 459--479, North-Holland, Amsterdam.



\bibitem{MPI} F.	Merle, P. Rapha$\ddot{\text{e}}$l,  I. Rodnianski, and  J. Szeftel,  On the implosion of a compressible fluid II: Singularity formation,  \textit{ Ann. of Math.}  \textbf{196} (2022),  779--889.
		

\bibitem{nishida} Y. Nishida, Equations of fluid dynamics—free surface problems,  Frontiers of the mathematical sciences: 1985 (New York, 1985). \textit{Comm. Pure Appl. Math.}  \textbf{39} (1986), no. S, suppl.\, S221--S238.


\bibitem{Makino2} M. Okada and T.  Makino,  Free boundary problem for the equation of spherically symmetric motion of viscous gas,  \textit{Japan J. Indust. Appl. Math.} \textbf{10} (1993), 219--235.

\bibitem{Oli1} T. Oliynyk, On the existence of solutions to the relativistic Euler equations in two spacetime dimensions with a vacuum boundary, \textit{Classical Quantum Gravity} \textbf{29} (2012), 28 pp.





\bibitem{opic2} B. Opic and P. Gurka, Continuous and compact imbeddings of weighted Sobolev spaces. II. \textit{Czechoslovak Math. J.} \textbf{39} (1989), 78--94.




\bibitem{OZ} Y. Ou and H. Zeng, Global strong solutions to the vacuum free boundary problem for compressible Navier-Stokes equations with degenerate viscosity and gravity force, \textit{J. Differential  Equations}\, \textbf{259} (2015), 6803--6829.


\bibitem{Rei} G. Rein,  Non-linear stability of gaseous stars, \textit{Arch. Ration. Mech. Anal.} \textbf{168} (2003), 
115--130.


\bibitem{sideris2} S. Shkoller and T. C. Sideris, Global existence of near-affine solutions to the compressible  Euler equations, \textit{ Arch. Ration. Mech. Anal.} \textbf{234} (2019), 115--180.


\bibitem{sideris} T. C. Sideris,  Global existence and asymptotic behavior of affine motion of 3D ideal fluids surrounded by vacuum, \textit{Arch. Ration. Mech. Anal.} \textbf{225} (2017), 141--176. 



\bibitem{Strauss}
W. A. Strauss,  Continua of steadily rotating stars, \textit{Quart. Appl. Math.} \textbf{81} (2023),  413--427.


\bibitem{jiawenlocal} D. Wang,  J. Zhang, and S. Zhu,  On a local existence theorem for the evolution equation of viscous gaseous stars in a physical vacuum, 2026,  arXiv:2606.22822v4.








\end{thebibliography}
\end{document}